\documentclass[a4paper,10pt,reqno]{amsart}
\pdfoutput=1 
\usepackage{babel}
\usepackage{microtype}
\usepackage{mathtools}
\usepackage{amsmath,amsfonts,amsthm,amsfonts,amssymb,stmaryrd}
\usepackage[colorlinks=true, linkcolor=blue, citecolor=red]{hyperref}
\usepackage{enumitem}
\usepackage{mathpazo} 
\usepackage{array, booktabs}
\usepackage{multirow}
\newcolumntype{M}[1]{>{\centering\arraybackslash}m{#1}}
\usepackage{tikz}
\usepackage{setspace,kantlipsum} 

\theoremstyle{plain}
\newtheorem*{theorem*}{Theorem}
\newtheorem{theorem}{Theorem}[section]
\newtheorem{proposition}{Proposition}[section]
\newtheorem{lemma}{Lemma}[section]
\newtheorem{corollary}{Corollary}[section]
\theoremstyle{remark}
\newtheorem{digression}{Digression}[section]
\newtheorem{definition}{Definition}[section]
\newtheorem{remark}{Remark}[section]
\newtheorem{example}{Example}[section]
\newtheorem*{sketchproof}{Sketch of the proof}

\title{Topics in Representation Theory and Riemannian Geometry}
\author{Giovanni Russo}
\address{Mathematics Area, SISSA - Scuola Internazionale Superiore di Studi Avanzati, via Bonomea 265, 34136 Trieste (TS), Italy}
\email{girusso@sissa.it}

\begin{document}
\begin{spacing}{1.135}

\begin{abstract}
These are notes for a Ph.D.\ course I held at SISSA, Trieste, in the Winter 2025.
We review well-known topics in Riemannian geometry where Lie groups play a fundamental role.
Part of the theory of compact connected Lie groups, their invariants, and representations is discussed, with particular emphasis on low dimensional examples.
We go through a number of applications in Riemannian geometry, in particular the classification of Riemannian holonomy groups, and the first construction of exceptional holonomy metrics.
Some more recent advances in the field of Riemannian geometry with symmetries are mentioned.
\end{abstract}

\maketitle
\pagenumbering{gobble}
\tableofcontents
\newpage
\pagenumbering{arabic}
\section*{Introduction}

These notes cover some selected topics in the theory of Lie groups and in Riemannian geometry.
The goal is to introduce the reader to certain aspects of differential geometry where Lie groups play a special role, and ultimately discuss some relevant geometries behind Einstein's theory of gravity, particle physics, string theory, and more generally $M$-theory.
The material is classical, and is thus particularly suitable for young researchers.

Let me summarise the main idea guiding these lectures.
Consider a vector bundle $E \to M$ over a smooth manifold $M$ with standard fibre $V$.
A section $f$ is a map $f \colon M \to E$ with values in $V$ pointwise.
For instance, $f$ may be a tensor field, or a spinor field on $M$.
Suppose a compact connected Lie group $G$ acts linearly on $V$, i.e.\ $V$ is a representation of $G$.
Then decomposing $V$ into the direct sum of irreducible $G$-invariant summands singles out the irreducible $G$-invariant components of $f$.
This process allows one to understand the behaviour of $f$ with respect to the action of $G$, and yields classification results.
This general idea is hereby illustrated on some examples, which serve merely as a motivation for the time being.
The technical language to fully understand them is explained somewhere in these notes.

Riemannian geometries can be distinguished by the behaviour of their curvature tensor.
If $(M,g)$ is a Riemannian $n$-dimensional manifold, the curvature tensor $R$ of $g$ is a section of the symmetric tensor bundle $S^2\Lambda^2 T^*M \to M$, where $T^*M$ is the cotangent bundle of $M$.
Pointwise, $R$ takes values in a subspace $\mathcal T \subset S^2\Lambda^2 \mathbb R^n$ (where $\mathbb R^n$ models each cotangent space of $M$), and $\mathcal T$ comes as a representation of the orthogonal group $\mathrm{O}(n)$.
A decomposition of $\mathcal T$ into irreducible $\mathrm{O}(n)$-summands then gives us complete information on the irreducible orthogonal invariant components of $R$.
We will see how $\mathcal T$ decomposes into the direct sum of three irreducible $\mathrm{O}(n)$-invariant summands
\[\mathcal T = \mathbb R \oplus \mathcal Z \oplus \mathcal W,\]
whence $R$ splits into the sum of three invariant tensors.
Einstein manifolds are those Riemannian manifolds where the $\mathcal Z$-component of $R$ vanishes.
Concretely, this is equivalent to saying that the Ricci tensor $\mathrm{Ric}$ of $g$ satisfies $\mathrm{Ric} = \lambda g$, for $\lambda$ a constant.
Riemannian manifolds with constant sectional curvature correspond to the vanishing of the $\mathcal Z$- and the $\mathcal W$-component of $R$.
Further, if one has $n=2m$ and $(M,g)$ comes with an isometric almost complex structure $J$, then $\mathcal T$ is acted on by the unitary group $\mathrm{U}(m)$, and the above $\mathrm{O}(n)$-decomposition of $\mathcal T$ can be refined. 
The former result for the orthogonal group is described in Berger--Gauduchon--Mazet \cite{berger-gauduchon-mazet} or in Besse \cite{besse, besse2}, whereas the latter for the unitary group is due to Tricerri--Vanhecke \cite{tricerri-vanhecke}.

Similar representation theoretic considerations yield a plethora of geometries of interest, both in mathematics and physics.
Let us mention two examples.

On an almost Hermitian $2m$-manifold $(M,g,J)$, with $m>1$, the Levi-Civita connection $\nabla$ preserves $g$ but not necessarily $J$, i.e.\ $\nabla J$ is not identically zero in general.
One can define the fundamental two-form $\omega \coloneqq g(J{}\cdot{},{}\cdot{})$, and working with $\nabla \omega$ or $\nabla J$ is then equivalent.
Observe that $\nabla \omega$ is a section of the tensor bundle $(T^*M)^{\otimes 3} \to M$. 
The symmetries of $J$ with respect to the metric imply that pointwise $\nabla \omega$ takes values in the space $\mathbb R^{2m} \otimes \mathfrak{u}(m)^{\perp}$:
here $\mathfrak{u}(m)^{\perp}$ is the orthogonal complement of the Lie algebra $\mathfrak{u}(m) \subset \mathfrak{so}(2m)$ inside $\mathfrak{so}(2m)$ with respect to the Killing form on $\mathfrak{so}(2m)$.
The space $\mathbb R^{2m} \otimes \mathfrak{u}(m)^{\perp}$ comes with a linear action of the unitary group $\mathrm{U}(m)$.
A decomposition of $\mathbb R^{2m} \otimes \mathfrak{u}(m)^{\perp}$ into the direct sum of irreducible $\mathrm{U}(m)$-representations gives all possible irreducible unitary invariant components of $\nabla \omega$.
It turns out there are four irreducible $\mathrm{U}(m)$-invariant summands:
\[\mathbb R^{2m} \otimes \mathfrak{u}(m)^{\perp} = \mathcal W_1 \oplus \mathcal W_2 \oplus \mathcal W_3 \oplus \mathcal W_4.\]
Correspondingly, one has a list of sixteen almost Hermitian geometries, as the $\mathcal W_i$-component of $\nabla \omega$ may be zero or not.
This result is due to Gray--Hervella \cite{gray-hervella}.
If all $\mathcal W_i$-components of $\nabla \omega$ vanish, one says $(M,g,J)$ is K\"ahler.
Equivalently $(M,g,J)$ has holonomy contained in $\mathrm{U}(m)$.
If the $\mathcal W_i$-components of $\nabla \omega$ vanish for $i \neq 1$, one says that $(M,g,J)$ is nearly K\"ahler.
We will touch upon nearly K\"ahler geometry in these notes.
The nearly K\"ahler condition amounts to saying that $\nabla \omega$ is totally skew-symmetric.
This geometry is of interest particularly for $m=3$, when nearly K\"ahler (non-K\"ahler) metrics are Einstein (see Gray \cite{gray3}) and play a role in string theory as well.

A similar process can be applied to seven-dimensional manifolds equipped with a $\mathrm G_2$-structure.
The latter is a three-form $\varphi$ pointwise linearly equivalent to a standard one on $\mathbb R^7$, whose stabiliser in $\mathrm{GL}(7,\mathbb R)$ is by definition the compact Lie group $\mathrm G_2$.
The form $\varphi$ induces a Riemannian metric $g$ and an orientation on the manifold, in particular we have the Levi-Civita connection $\nabla$.
Studying the symmetries of $\nabla \varphi$ as above, one finds again sixteen classes of $\mathrm G_2$ geometries.
The condition $\nabla \varphi=0$ is equivalent to saying that $g$ has holonomy contained in $\mathrm G_2$, and forces the metric to be Ricci flat, i.e.\ $\mathrm{Ric}=0$, cf.\ Bonan \cite{bonan}.
One also says that $\varphi$ is parallel when $\nabla \varphi = 0$.
We will discuss in particular the class of nearly parallel $\mathrm G_2$-manifolds, i.e.\ those satisfying $\nabla \varphi = \lambda \star \varphi$ for some constant $\lambda$, and $\star$ is the Hodge star operator induced by $\varphi$.
These are again Einstein manifolds (see Bryant \cite{bryant1}) and have a role in string theory.

With these notes, I would like to introduce the reader to the language behind this sort of results, and describe certain situations where all possible geometries are in a sense \lq\lq governed\rq\rq\ by the representation theory of some compact Lie group.
There are many such instances, which are of course much beyond the scope of this work.
The algebraic theory needed to get into these examples is normally very special of the Lie groups considered.
Seminal work in this direction was done by Weyl \cite{weyl}, who studied the classical Lie groups, their invariants and representations. 
More general material on the representation theory of Lie groups and Lie algebras is well documented, however we will not need much of it.
Foundational geometric applications are found in the French literature, e.g.\ Besse \cite{besse2}.
More recent applications will be given in a number of references in the final section of these lecture notes.

A milestone in the direction we want to pursue is Berger's classification of holonomy groups of certain Riemannian manifolds.
This provides a way to divide Riemannian manifolds into classes based on their holonomy group.
The result is of representation theoretic nature, and is published in Berger's Ph.D.\ thesis in 1955 \cite{berger}.
The list of holonomy groups found by Berger in the Riemannian case is essentially a short list of subgroups of some orthogonal groups.
This motivated the search for examples of Riemannian manifolds with special holonomy that lasted until the end of the 1990s.
The problem that took longer to be solved was to find metrics with holonomy $\mathrm G_2$ and $\mathrm{Spin}(7)$, the so-called exceptional cases.
In the last part of these notes we present the work of Bryant on the construction of the first examples of (incomplete) metrics with exceptional holonomy \cite{bryant}.
This relies deeply on the representation theory of $\mathrm G_2$ and $\mathrm{Spin}(7)$.
It also establishes connections between the world of exceptional holonomy manifolds and nearly K\"ahler manifolds in dimension $6$, and nearly parallel $\mathrm G_2$ manifolds in dimension $7$.

The notes are structured as follows. 
Sections \ref{sec:representation-theory}--\ref{sec:invariant-theory} are on algebraic preliminaries. 
We discuss Lie groups and Lie algebras, restrict to certain classical ones, and illustrate various aspects of their representation theory.
We give a number of examples, e.g.\ orthogonal and unitary groups, and look in particular at $\mathrm G_2$ and $\mathrm{Spin}(7)$.
Section \ref{sec:fibre-bundles-and-connections} is a brief summary of the theory of principal bundles, associated vector bundles, connections, and $G$-structures.
The notion of holonomy group is first presented here.
In Section \ref{sec:riemannian-geometry}, we introduce Riemannian geometry and look at applications of the representation theory of orthogonal groups.
In particular, we discuss the orthogonal decomposition of the Riemannian curvature tensor, and describe a classification of types of torsion tensors of linear connections based on the orthogonal group.
In Section \ref{sec:holonomy} we specialise the presentation of the holonomy group to the Riemannian case.
We see how the holonomy group can be computed by using the so-called General Holonomy Principle.
Next, we discuss symmetric spaces and Berger's classification of Riemannian holonomy groups, from which a list of so-called integrable geometries follows.
Lastly, in Section \ref{sec:non-integrable-geometries}, we look at non-integrable geometries, and discuss in particular nearly K\"ahler and nearly parallel $\mathrm G_2$ manifolds.
As already mentioned, these are behind the construction of the first manifolds with exceptional holonomy by Bryant, which we present in the final part.

\section*{Acknowledgments}

As mentioned in the abstract, these lecture notes grew out of a Ph.D.\ course I held at SISSA, Trieste, in the Winter 2025.
My first thanks go to the students who attended the course, for stimulating questions and comments.
I thank Antonio Lerario for encouraging me to write down these lectures.
Further, I thank Giovanni Bazzoni, Beatrice Brienza, Diego Conti, Anna Fino, Oliver Goertsches, and Antonio Lerario for useful comments on these notes.

I am partially supported by the PRIN 2022 Project (2022K53E57) - PE1 - \lq\lq Optimal transport: new challenges across analysis and geometry\rq\rq\ funded by the Italian Ministry of University and Research.
I am also partially supported by INdAM--GNSAGA.

\newpage
\section*{Notations and conventions}

\begin{itemize}
\item The set $\mathbb N$ is for natural numbers, $\mathbb Z$ is for integer numbers, $\mathbb R$ is for real numbers, $\mathbb C$ for complex numbers, $\mathbb H$ for quaternions, $\mathbb O$ for octonions.
We use the notation $\mathbb Z_p$ for integers modulo $p$.
We write $\mathbb R_{>0}$ for the positive real numbers.
\item We write $\mathfrak{X}(M)$ for the space of vector fields over any manifold $M$, and $\Omega(M)$ for the space of differential forms.
We use the notation $\mathfrak{X}^{inv}(M)$ and $\Omega^{inv}(M)$ for invariant vector fields and differential forms on $M$.
\item If $f$ is a smooth map, then $f_*$ denotes the pushforward of $f$, and $f^*$ is the pullback of $f$.
\item Lie algebras and their subspaces are denoted by Fraktur typefaces $\mathfrak g$, $\mathfrak h$, and so on.
\item The symbol $=$ is also used for isomorphisms (in any category).
\item In general, the symbol $\subset$ means \lq\lq included or equal to\rq\rq. 
We specify strict inclusions case by case.
\item The symbol $\mathrm{id}$ is used together with different subscripts to denote an identity matrix or an identity map.
\item A tensor of type $(m,n)$ is $m$ times covariant and $n$ times contravariant. 
In other words, a tensor of type $(m,n)$ takes $m$ vectors as input and returns an $n$-vector.
\item Let $\alpha$ and $\beta$ be any two tensors. 
We write $\odot$ for the symmetric product \[\alpha \odot \beta = \alpha \otimes \beta + \beta \otimes \alpha.\]
We write $\wedge$ for the skew-symmetric product \[\alpha \wedge \beta = \alpha \otimes \beta - \beta \otimes \alpha.\]
\item A fibration is a short exact sequence of objects
\[1 \longrightarrow K \stackrel{f}\longrightarrow N \stackrel{g}\longrightarrow M \longrightarrow 1.\]
This means $f$ is injective, $\ker g = \mathrm{Im}f = K$, and $g$ is surjective.
\item The symbol $\mathfrak{S}_{X_1,\dots,X_n}$ denotes a cyclic permutation over $X_1,\dots,X_n$.
\item If $\alpha$ is a covariant tensor, $\nabla$ a connection, we set \[\nabla \alpha(X,{}\cdot{},\dots,{}\cdot{}) \coloneqq \nabla_X\alpha.\]
\item We write $\delta_{ij}$ for the Kronecker delta, i.e.\ $\delta_{ij}=0$ for $i \neq j$, and $\delta_{ii}=1$.
\item We write $S_p$ for the group of permutations of $p$ elements.
\item We use the standard notation $\pi_i(M)$ for the $i$-th homotopy group of $M$.
The notation $\pi_i(M,p)$ is used for the $i$-th homotopy group of $M$ at $p$.
\item The standard notation $H^k(M,\mathbb R)$ denotes the $k$-th cohomology group of $M$ with real coefficients.
\end{itemize}

\newpage
\section{Representation theory}
\label{sec:representation-theory}

A \emph{representation} of a group is a group action on a vector space preserving the linear structure.
A group acts on many vector spaces, and may act on vector spaces of the same dimension in different ways.
One of the main goals of the theory is to detect the so-called \emph{irreducible} representations of a group, which are building blocks for all other representations.
In geometry, one is often interested in representations of \emph{Lie groups} and their \emph{Lie algebras}, particularly compact Lie groups.
The story goes roughly as follows. 
Compact connected Lie groups contain tori which are maximal with respect to inclusion, and are thus called \emph{maximal tori}.
The dual Lie algebra of a fixed maximal torus contains a lattice of points called \emph{weights}.
A theorem of Weyl then states that isomorphism classes of irreducible complex representations of certain compact connected Lie groups are in one-to-one correspondence with the so-called \emph{dominant weights}.
Since maximal tori are conjugate, this correspondence does not depend on the choice of a maximal torus.

The goal of this section is to give a quick and informal summary of the above picture for certain compact connected Lie groups.
We introduce ideas, tools, and terminology which will be used throughout. 

\subsection{Lie groups and representations}
\label{subsec:generalities}

Details on the material in this section can be mainly found in Adams \cite{adams}, Br\"ocker--tom Dieck \cite{brocker-tomdieck}, Fulton--Harris \cite{fulton-harris}, and Hall \cite{hall}.
We give specific references where needed.
\begin{definition}
A \emph{Lie group} is a smooth manifold with a group structure such that group multiplication and inversion are smooth maps.
A \emph{homomorphism} of Lie groups is a group homomorphism which is also smooth.
\end{definition}
Unless otherwise stated, all Lie groups will be finite-dimensional and real.
A Lie group $G$ acts on itself by $(g,h) \mapsto gh$, where $g,h \in G$. 
For a fixed $g \in G$, write 
\begin{align*}
L_g & \colon G \to G, \qquad L_g(h) \coloneqq gh, \\
R_g & \colon G \to G, \qquad R_g(h) \coloneqq hg.
\end{align*}
The maps $L_g$ and $R_g$ are called \emph{left} and \emph{right translation} respectively.
Both $L_g$ and $R_g$ are diffeomorphisms.
The composition $L_g \circ R_g{}^{-1}$ is called \emph{conjugation}, and is an automorphism of $G$.
\begin{definition}
A vector field $X \in \mathfrak{X}(G)$ on a Lie group $G$ is \emph{left-invariant} if it is preserved by left translations, i.e.\ $(L_g)_*X=X$ for all $g \in G$.
\end{definition}
The commutator of left-invariant vector fields is left-invariant: let $X$ and $Y$ be left-invariant vector fields on $G$, then for $p \in G$ one computes
\begin{equation*}
[X,Y]_{gp} = [(L_g)_*X,(L_g)_*Y]_{gp} = (L_g)_*[X,Y]_p, \qquad g \in G,
\end{equation*}
and hence $[X,Y]$ is left-invariant. 

The values of a left-invariant vector field are determined by its value at the identity $e_G \in G$, so the space $\mathfrak{X}^{inv}(G)$ of left-invariant vector fields on $G$ can be identified with the tangent space $T_{e_G}G$.
It follows that the $\mathfrak{X}^{inv}(G) \subset \mathfrak{X}(G)$ is finite-dimensional.
The commutator of vector fields imposes the additional algebraic structure of Lie algebra on $\mathfrak{X}^{inv}(G)$.

\begin{definition}
\label{def:lie-algebra-abstract}
A \emph{Lie algebra} $\mathfrak a$ is a vector space equipped with a skew-symmetric bilinear map $[{}\cdot{},{}\cdot{}]$ satisfying the Jacobi identity
\[[X,[Y,Z]]+[Y,[Z,X]]+[Z,[X,Y]]=0, \qquad X,Y,Z \in \mathfrak a.\]
The map $[{}\cdot{},{}\cdot{}]$ is called \emph{Lie bracket}.
A \emph{homomorphism} of Lie algebras is a linear map preserving the Lie brackets, i.e.\ a linear map $f \colon (\mathfrak a_1, [{}\cdot{},{}\cdot{}]_1) \to (\mathfrak a_2, [{}\cdot{},{}\cdot{}]_2)$ such that $[f(X),f(Y)]_2 = f([X,Y]_1)$, for all $X,Y \in \mathfrak a_1$.
\end{definition}
\begin{definition}
\label{def:lie-algebra-lie-group}
The \emph{Lie algebra} $\mathfrak g$ of a Lie group $G$ is the vector space of left-invariant vector fields on $G$, equipped with the Lie bracket induced by the commutator of vector fields.
\end{definition}
\begin{remark}
Since the Lie algebra $\mathfrak g$ of a Lie group $G$ can be identified with the tangent space $T_{e_G}G$, one computes $\mathfrak g$ by differentiating smooth curves in $G$ passing through the identity.
\end{remark}
Here are some examples of Lie groups and their Lie algebras.
\begin{example}
The \emph{group of automorphisms} $\mathrm{Aut}(V) = \{A \in \mathrm{End}(V): \det A \neq 0\}$ of a finite-dimensional vector space $V$. 
The Lie algebra of $\mathrm{Aut}(V)$ is the space of endomorphisms $\mathrm{End}(V)$, equipped with the Lie bracket $[A,B] = A\circ B-B\circ A$.
\end{example}
\begin{example}
After choosing a basis of $V$, $\mathrm{End}(V)$ can be identified with the space $M(n,\mathbb K)$ of $n\times n$ matrices with coefficients in $\mathbb K$ (where $\mathbb K\in \{\mathbb R, \mathbb C, \mathbb H\}$).
Similarly, $\mathrm{Aut}(V)$ can be identified with a general linear group. 
The \emph{general linear group} $\mathrm{GL}(n,\mathbb K)$ is the group of invertible $n\times n$ matrices with coefficients in $\mathbb K$.
The Lie algebra $\mathfrak{gl}(n,\mathbb K)$ is $\mathrm{End}(\mathbb K^n)$ with the Lie bracket as in the previous point.
See also Remark \ref{rmk:quaternionic-general-linear-group}.
\end{example}
\begin{example}
The \emph{special linear group} $\mathrm{SL}(V) \subset \mathrm{Aut}(V)$ of linear transformations preserving any volume form.
It can be identified with the group $\mathrm{SL}(n,\mathbb K)$ of matrices in $\mathrm{GL}(n,\mathbb K)$ with determinant $1$.
Its Lie algebra $\mathfrak{sl}(n,\mathbb K) \subset \mathfrak{gl}(n,\mathbb K)$ is given by traceless $n\times n$ matrices with coefficients in $\mathbb K$.
\end{example}
\begin{example}
If $V=\mathbb R^n$ and comes equipped with a Euclidean structure, the \emph{orthogonal group} $\mathrm{O}(V)$ is the Lie group of linear transformations of $V$ preserving the Euclidean structure.
It can be identified with the subgroup $\mathrm{O}(n) \subset \mathrm{GL}(n,\mathbb R)$ of matrices whose inverse is the transpose.
Its Lie algebra $\mathfrak{o}(n)$ is given by elements in $\mathfrak{gl}(n,\mathbb R)$ which are skew-symmetric with respect to the Euclidean structure, i.e.\ $A^T+A=0$.
\end{example}
\begin{example}
\label{ex:non-compact-symplectic-group}
If $V=\mathbb R^{2n}$ and is equipped with a non-degenerate skew-symmetric bilinear form $\omega$, the \emph{symplectic group} $\mathrm{Sp}(V)$ is the Lie group of linear transformations of $V$ preserving $\omega$.
It can be identified with the subgroup $\mathrm{Sp}(2n,\mathbb R) \subset \mathrm{SL}(2n,\mathbb R)$ of $2n \times 2n$ real matrices $A$ such that $A^T\Omega A = \Omega$, where $\Omega$ is a $2n \times 2n$ matrix representing $\omega$ (see Remark \ref{rmk:non-compact-symplectic-group}).
Its Lie algebra $\mathfrak{sp}(2n,\mathbb R)$ is given by $2n \times 2n$ matrices $A$ such that $A^T\Omega+\Omega A=0$.
\end{example}
\begin{example}
If $V=\mathbb C^n$ and comes equipped with a Hermitian structure, the \emph{unitary group} $\mathrm{U}(V)$ is the Lie group of linear transformations of $V$ preserving the Hermitian structure.
It can be identified with the subgroup $\mathrm{U}(n)\subset \mathrm{GL}(n,\mathbb C)$ of matrices whose inverse is the conjugate transpose. 
Its Lie algebra $\mathfrak{u}(n)$ is given by elements in $\mathfrak{gl}(n,\mathbb C)$ which are skew-Hermitian with respect to the Hermitian structure, i.e.\ $\overline A{}^T+A=0$.
\end{example}
\begin{remark}
\label{rmk:quaternionic-general-linear-group}
When $V$ is a real vector space, we write $\mathrm{GL}_+(V)$ for the subgroup of transformations in $\mathrm{GL}(V)$ with positive determinant.
We have an inclusion $M(n,\mathbb C) \hookrightarrow M(2n,\mathbb R)$ given by
\begin{equation}
\label{eq:monomorphism-c}
A+iB \mapsto \begin{pmatrix} A & -B \\ B & A \end{pmatrix}, \qquad A,B \in M(n,\mathbb R).
\end{equation}
This induces a Lie group monomorphism $\mathrm{GL}(n,\mathbb C) \hookrightarrow \mathrm{GL}_+(2n,\mathbb R)$.
To check that this map is well-defined, view $\mathrm{GL}(2n,\mathbb R)$ in $\mathrm{GL}(2n,\mathbb C)$ and compute
\begin{align*}
\det_{\mathbb R} \begin{pmatrix} A & -B \\ B & A \end{pmatrix} & = \det_{\mathbb R} \Biggl(\begin{pmatrix} -i\mathrm{id}_n & i\mathrm{id}_n \\ \mathrm{id}_n & \mathrm{id}_n \end{pmatrix}\begin{pmatrix} A-iB & 0 \\ 0 & A+iB \end{pmatrix}\begin{pmatrix} \frac{i}{2}\mathrm{id}_n & \frac12\mathrm{id}_n \\ -\frac{i}{2}\mathrm{id}_n & \frac12 \mathrm{id}_n \end{pmatrix}\Biggr) \\
& = |\negthinspace\det_{\mathbb C} (A+iB)|^2 >0.
\end{align*}
Note that matrices in the image of the map \eqref{eq:monomorphism-c} are exactly those matrices $X$ in $M(2n,\mathbb R)$ such that $XJ=JX$, where
\begin{equation}
\label{eq:symplectic-matrix}
J \coloneqq \begin{pmatrix} 0 & -\mathrm{id}_n \\ \mathrm{id}_n & 0 \end{pmatrix}.
\end{equation}
Similarly, there is an inclusion $M(n,\mathbb H) \hookrightarrow M(2n,\mathbb C)$, defined by
\begin{equation}
\label{eq:monomorphism-h}
A+jB \mapsto \begin{pmatrix} A & -\overline B \\ B & \overline A \end{pmatrix}, \qquad A,B \in M(n,\mathbb C).
\end{equation}
Since quaternions do not commute, the determinant of a matrix with quaternionic entries can be defined as the complex determinant of the corresponding matrix in $M(2n,\mathbb C)$.
The group $\mathrm{GL}(n,\mathbb H)$ is then the subgroup of $\mathrm{GL}(2n,\mathbb C)$ of matrices in the image of the map \eqref{eq:monomorphism-h}, see also \cite{aslaksen}.
Note that matrices $X$ in the image of the map \eqref{eq:monomorphism-h} are exactly those matrices $X$ in $M(2n,\mathbb C)$ such that $XJ=J\overline X$.
\end{remark}
\begin{remark}
\label{rmk:non-compact-symplectic-group}
Let $V=\mathbb R^{2n}$ be equipped with a non-degenerate skew-symmetric bilinear form $\omega$.
Up to a change of basis (see e.g.\ \cite{dasilva}), $\omega$ is represented by the matrix
\begin{equation}
\label{eq:omega}
\Omega \coloneqq
\begin{pmatrix}
0 & -\mathrm{id}_n \\
\mathrm{id}_n & 0
\end{pmatrix}.
\end{equation}
It is clear that $\Omega^2 = -\mathrm{id}_{2n}$, so $\Omega^{-1} = -\Omega$. 
Observe that $\omega^n = \omega \wedge \dots \wedge \omega$ is a volume form, so a matrix $A \in \mathrm{Sp}(2n,\mathbb R)$ must have determinant $1$.
Also, note that if $A^T\Omega A = \Omega$, then clearly $(A^T)^{-1}\Omega A^{-1} = \Omega$.
Now taking the inverse of both sides gives $A\Omega^{-1}A^T = \Omega^{-1}$, and hence $A\Omega A^T = \Omega$ as $\Omega^{-1} = -\Omega$.
This shows $\mathrm{Sp}(2n,\mathbb R)$ is stable under conjugate transpose.
\end{remark}
The above examples share two properties, i.e.\ they are linear and reductive in the following sense (cf.\ Knapp \cite{knapp}).
\begin{definition}
A group that can be realised as a closed subgroup of a general linear group is called \emph{linear}.
A closed connected linear group of real or complex matrices is called \emph{reductive} if it is stable under conjugate transpose. 
\end{definition}
\begin{remark}
An example of Lie group that cannot be realised as a matrix group is the metaplectic group \cite{weissman}.
An example of a subgroup in $\mathrm{GL}(2,\mathbb R)$ which is not reductive is given by matrices of the form
\[\begin{pmatrix} 1 & a \\ 0 & 1\end{pmatrix}, \qquad a \in \mathbb R.\]
\end{remark}
We will be mostly interested in compact, connected Lie groups, which can be safely thought of as linear reductive groups.
In fact, the following result can be shown.
\begin{theorem}[Knapp \cite{knapp}]
\label{thm:knapp1}
Any compact connected Lie group can be realised as a linear connected reductive group.
\end{theorem}
The \emph{exponential map} $\exp \colon \mathfrak g \to G$ acts via $\exp(tX) \coloneqq \theta_X(t)$, $t \in \mathbb R$, where $\theta_X$ is the one-parameter subgroup of $G$ corresponding to $X$.
In the linear case, the exponential map has the standard form \[\exp(A) = \sum_{k=0}^{\infty} \frac{A^k}{k!}.\]

Besides $\mathrm{U}(n)$, compact connected examples of linear reductive groups coming in families are the following.
\begin{example}
The \emph{special orthogonal group} $\mathrm{SO}(n) \subset \mathrm{O}(n)$, the group of isometries of the Euclidean space $\mathbb R^n$ preserving any volume form. Its Lie algebra $\mathfrak{so}(n)$ coincides with the Lie algebra $\mathfrak{o}(n)$.
\end{example}
\begin{example}
The \emph{special unitary group} $\mathrm{SU}(n) \subset \mathrm{U}(n)$ of unitary transformations of the Hermitian space $\mathbb C^n$ preserving a complex volume form. Its Lie algebra $\mathfrak{su}(n)$ is given by traceless elements in $\mathfrak{u}(n)$. 
\end{example}
\begin{example}
\label{ex:compact-symplectic-group}
The \emph{compact symplectic group} $\mathrm{Sp}(n) \subset \mathrm{GL}(n,\mathbb H)$ of isometries of the (left or right) vector space $\mathbb H^n$, equipped with its standard symplectic scalar product.
Its Lie algebra $\mathfrak{sp}(n)$ is given by quaternionic skew-Hermitian matrices.
\end{example}
\begin{remark}
For any $n\times n$ matrix $A$, the matrix $A-(\mathrm{Tr}(A)/n)\mathrm{id}_n$ is traceless. 
If $A \in \mathfrak{u}(n)$, writing $A = (A-(\mathrm{Tr}(A)/n)\mathrm{id}_n)+(\mathrm{Tr}(A)/n)\mathrm{id}_n$ gives the decomposition $\mathfrak{gl}(n,\mathbb R) = \mathfrak{sl}(n,\mathbb R) \oplus \mathbb R$.
In particular, $\mathfrak{u}(n) = \mathfrak{su}(n) \oplus \mathbb R$
\end{remark}
\begin{remark}
\label{rmk:alternative-def-spn}
By Remark \ref{rmk:quaternionic-general-linear-group}, the unitary group $\mathrm{U}(n)$ can be defined as the subgroup of $\mathrm{SO}(2n)$ of matrices $A$ such that $AJ=JA$, where $J$ is as in \eqref{eq:symplectic-matrix}. 
We then have an isomorphism
\begin{equation}
\label{eq:unitary-group}
\mathrm{U}(n) = \mathrm{GL}(n,\mathbb C) \cap \mathrm{SO}(2n),
\end{equation}
where $\mathrm{GL}(n,\mathbb C)$ is viewed as a subgroup of $\mathrm{GL}(2n,\mathbb R)$.
Similarly, the compact group $\mathrm{Sp}(n)$ can also be defined as the subgroup of $\mathrm{U}(2n)$ of matrices $A$ such that $AJ = J\overline A$.
We then have an alternative formulation
\begin{equation}
\label{def:alternative-spn}
\mathrm{Sp}(n) = \mathrm{GL}(n,\mathbb H) \cap \mathrm{U}(2n),
\end{equation}
where $\mathrm{GL}(n,\mathbb H)$ is viewed as a subgroup of $\mathrm{GL}(2n,\mathbb C)$.
\end{remark}
\begin{remark}
The symplectic group $\mathrm{Sp}(n)$ in Example \ref{ex:compact-symplectic-group} is compact, whereas the symplectic group $\mathrm{Sp}(2n,\mathbb R)$ in Example \ref{ex:non-compact-symplectic-group} is not.
For instance, $\mathrm{Sp}(2,\mathbb R) = \mathrm{SL}(2,\mathbb R)$. 
We avoid confusion by using the notation $\mathrm{Sp}(n)$ for the compact symplectic group, and $\mathrm{Sp}(2n,\mathbb R)$ for the non-compact (real) one.
\end{remark}
We will need specific facts on certain spin groups, whose general theory is treated in e.g.\ Br\"ocker--tom Dieck \cite{brocker-tomdieck} or Lawson--Michelsohn \cite{lawson-michelsohn}.
Whilst the groups $\mathrm{SU}(n)$ and $\mathrm{Sp}(n)$ are simply connected, $\mathrm{SO}(n)$ is not in general.
For $n\geq 3$, the fundamental group $\pi_1(\mathrm{SO}(n))$ equals $\mathbb Z_2$. 
So, for $n\geq 3$, $\mathrm{SO}(n)$ admits a universal double cover $\mathrm{Spin}(n)$ which is a connected, simply connected, compact Lie group, and the Lie algebras $\mathfrak{so}(n)$ and $\mathfrak{spin}(n)$ coincide. 
For $3 \leq n \leq 6$, there are the following so-called \lq\lq accidental\rq\rq\ isomorphisms
\begin{itemize}
\item $\mathrm{Spin}(3) = \mathrm{SU}(2) = \mathrm{Sp}(1)$,
\item $\mathrm{Spin}(4) = \mathrm{SU}(2) \times \mathrm{SU}(2) = \mathrm{Sp}(1) \times \mathrm{Sp}(1)$, 
\item $\mathrm{Spin}(5) = \mathrm{Sp}(2)$ (see \cite{cadek-vanzura}),
\item $\mathrm{Spin}(6) = \mathrm{SU}(4)$ (see \cite{postnikov}),
\end{itemize}
and there are inclusions $\mathrm{Spin}(3) \subset \mathrm{Spin}(4) \subset \mathrm{Spin}(5) \subset \mathrm{Spin}(6)$.
We now look at the first two isomorphisms, the last two will only be mentioned at times.
\begin{example}
\label{ex:su(2)-sp(1)}
Take $\mathrm{SU}(2)$, the group of $2 \times 2$ unitary matrices with determinant $1$.
An explicit computation shows that matrices in $\mathrm{SU}(2)$ are of the form
\begin{equation*}
\begin{pmatrix}
a & -\overline b \\
b & \overline a
\end{pmatrix}, \quad a,b \in \mathbb C, \quad |a|^2+|b|^2=1.
\end{equation*}
By mapping any such element to $a+bj$, we get an isomorphism $\mathrm{SU}(2)=\mathrm{Sp}(1)$, the latter being the group of unit quaternions.
Topologically, $\mathrm{Sp}(1)$ is a three-sphere $S^3$, and thus $\mathrm{SU}(2)$ is compact, connected, and simply connected.
If $x \in \mathrm{Im}\mathbb H$ is a purely imaginary quaternion, and $q \in \mathrm{Sp}(1)$, then $qxq^{-1}$ is a purely imaginary quaternion.
This gives an action of $\mathrm{SU}(2)$ on $\mathbb R^3 = \mathrm{Im}\mathbb H$ preserving the standard scalar product, and thus there is a Lie group homomorphism $\mathrm{SU}(2) \to \mathrm{SO}(3)$.
This map turns out to be surjective, and $\mathbb Z_2 = \{\pm \mathrm{id}\} \subset \mathrm{SU}(2)$ acts trivially.
By universality of $\mathrm{Spin}(3)$, the groups $\mathrm{SU}(2)$ and $\mathrm{Spin}(3)$ are isomorphic.
\end{example}
\begin{example}
\label{ex:spin(4)}
The isomorphism $\mathrm{Spin}(4)=\mathrm{SU}(2) \times \mathrm{SU}(2)$ can be established by observing that $\mathrm{Sp}(1) \times \mathrm{Sp}(1)$ acts on $\mathbb H=\mathbb R^4$ via $(p,q)\cdot x = px\overline q$ preserving the scalar product $\langle x,y\rangle = \mathrm{Re}(\overline xy)$.
This gives an exact sequence \[1 \to \mathbb Z_2 \to \mathrm{SU}(2) \times \mathrm{SU}(2) \to \mathrm{SO}(4) \to 1.\]
By universality of $\mathrm{Spin}(4)$, we again deduce that $\mathrm{Spin}(4)$ and $\mathrm{SU}(2) \times \mathrm{SU}(2)$ are isomorphic.
\end{example}

On the other hand, for $n=2$ the group $\mathrm{SO}(n) = \mathrm{SO}(2)$ is isomorphic to a circle $S^1$. 
In fact we have
\[\mathrm{SO}(2) = \Biggl\{\begin{pmatrix} \cos(2\pi \alpha) & -\sin(2\pi \alpha) \\ \sin(2\pi \alpha) & +\cos(2\pi \alpha) \end{pmatrix} \in M(2,\mathbb R): \alpha \in \mathbb R\Biggr\},\] 
so $\mathfrak{so}(2)$ is isomorphic to $\mathbb R$.
A product of circles is a \emph{torus} $T = S^1\times \dots \times S^1$, so its Lie algebra is isomorphic to some $\mathbb R^n$ with Lie bracket vanishing identically.
Note that tori are the only compact, connected, Abelian Lie groups, as connected Abelian Lie groups are of the form $T^k \times \mathbb R^n$ for some $k$, $n$.

We will also deal with the compact groups $\mathrm G_2$ and $\mathrm{Spin}(7)$. 
The former can be realised as a subgroup of $\mathrm{SO}(7)$, the latter as a subgroup of $\mathrm{SO}(8)$.
The group $\mathrm G_2$ is the only \lq\lq exceptional\rq\rq\ case we will deal with, in the sense that its Lie algebra corresponds to the \emph{exceptional} Lie algebra $\mathfrak g_2$.

Some aspects of the representation theory we are about to present rely on the simplicity of the Lie groups considered.
Let us recall this terminology.
\begin{definition}
Let $G$ be a group. 
A subgroup $H \subset G$ is called \emph{normal} when $gHg^{-1} \subset H$ for all $g \in G$.
The \emph{centre} $Z(G)$ of $G$ is the subgroup of elements $g \in G$ such that $gh=hg$ for all $h \in G$.
\end{definition}
\begin{remark}
We see that a subgroup $H \subset G$ is normal when it is preserved by the conjugation map.
Note that $gh=hg$ if and only if $ghg^{-1}=h$, so $Z(G)$ is the kernel of the homomorphism $G \to \mathrm{Aut}(G)$, $g \mapsto L_g \circ R_g{}^{-1}$.
\end{remark}
\begin{definition}
Let $\mathfrak g$ be a Lie algebra.
A subspace $\mathfrak a \subset \mathfrak g$ is called an \emph{ideal} when $[X,\mathfrak a] \subset \mathfrak a$ for all $X \in \mathfrak g$.
The \emph{centre} $\mathfrak z(\mathfrak g)$ of $\mathfrak g$ is the subspace of elements $X$ such that $[X,Y]=0$ for all $Y \in \mathfrak g$.
\end{definition}
In general, each finite-dimensional real Lie algebra (as in Definition \ref{def:lie-algebra-abstract}) is the Lie algebra of a simply connected Lie group.
In particular, there is a one-to-one correspondence between ideals of a Lie algebra and connected normal subgroups of the attached Lie group.
Also, if $G$ is a connected Lie group, the Lie algebra of its centre is the centre of its Lie algebra.
The following two definitions then correspond to each other.
\begin{definition}
A Lie algebra is \emph{semisimple} if it is non-Abelian and possesses no Abelian ideal other than the trivial one.
A Lie algebra is \emph{simple} if it is non-Abelian and possesses no non-zero proper ideals.
\end{definition}
\begin{definition}
A connected Lie group is called \emph{semisimple} if it is non-Abelian and possesses no Abelian connected normal subgroup other than the trivial one.
A connected Lie group is called \emph{simple} if it is non-Abelian and has no non-trivial connected normal subgroups.
\end{definition}
\begin{remark}
An abstract group is sometimes called \emph{simple} if it does not admit any proper normal subgroup.
We will stick to our previous definitions throughout.
\end{remark}
\begin{remark}
A semisimple Lie algebra is necessarily centreless, but the converse is not true.
For instance, the two-dimensional Lie algebra with basis $E_1,E_2$ satisfying $[E_1,E_2]=E_1$ is centreless but not semisimple.
It can be shown that a Lie algebra is semisimple if and only if it is a direct sum of simple Lie algebras.
\end{remark}
In the following, $G$ will be a connected, compact, real Lie group. 
As mentioned in the introduction to this section, the representation theory of $G$ goes through the representation theory of a maximal torus.
\begin{definition}
A \emph{torus} of $G$ is a compact connected Abelian subgroup.
A \emph{maximal torus} of $G$ is a torus which is not properly contained in any other torus.
\end{definition}

The following are examples of so-called \emph{standard} (diagonal or block diagonal) maximal tori.
\begin{example}
A maximal torus of $\mathrm{SO}(2n)$ is given by block matrices of the form
\[\begin{pmatrix}
\cos (2\pi \theta_1) & -\sin (2\pi \theta_1) & & \\
\sin (2\pi \theta_1) & +\cos (2\pi \theta_1) & & \\
& & \ddots \\
& & & \cos (2\pi \theta_n) & -\sin (2\pi \theta_n) \\
& & & \sin (2\pi \theta_n) & +\cos (2\pi \theta_n)
\end{pmatrix}.\]
For $\mathrm{SO}(2n+1)$ one can take a similar block matrix, then a $1$ at the bottom right and zeros elsewhere.
\end{example}
\begin{example}
\label{ex:diag-max-torus-unitary-group}
A (diagonal) maximal torus of $\mathrm{U}(n)$ is given by diagonal matrices of the form
\[\begin{pmatrix}
e^{2\pi i\theta_1} & & \\
& \ddots & \\
& & e^{2\pi i\theta_n}
\end{pmatrix}.\]
For $\mathrm{SU}(n)$, one has the same diagonal form with the condition $\theta_1+\dots+\theta_n=0$.
\end{example}
\begin{example}
\label{ex:max-torus-spn}
A (diagonal) maximal torus of $\mathrm{Sp}(n)$ is given by diagonal matrices of the form
\[\begin{pmatrix}
e^{2\pi i\theta_1} & & \\
& \ddots & \\
& & e^{2\pi i\theta_n}
\end{pmatrix}.\]
\end{example}
\begin{theorem}
\label{thm:max-tori}
Any two maximal tori in a compact connected Lie group are conjugate.
Every element of a compact connected Lie group is contained in a maximal torus.
\end{theorem}
For a proof, see \cite[Chapter IV, Section 1]{brocker-tomdieck}.
So, up to conjugation, a maximal torus in $\mathrm{SO}(n)$, $\mathrm{SU}(n)$, and $\mathrm{Sp}(n)$, can be assumed to be of the above standard form.
\begin{remark}
Let $A \in \mathrm{Sp}(n)$ be any element. 
By Theorem \ref{thm:max-tori}, there is a matrix $B \in \mathrm{Sp}(n)$ such that $BAB^{-1}$ is diagonal as in Example \ref{ex:max-torus-spn}.
Then $\det A = 1$ (cf.\ Remark \ref{rmk:quaternionic-general-linear-group}).
In fact, there is an isomorphism $\mathrm{Sp}(n) = \mathrm{GL}(n,\mathbb H) \cap \mathrm{SU}(2n)$.
\end{remark}
As a consequence of Theorem \ref{thm:max-tori}, the following is well-posed.
\begin{definition}
The \emph{rank} of a compact connected Lie group $G$ is the dimension of any maximal torus in $G$.
\end{definition}
\begin{example}
The ranks of $\mathrm{SO}(n)$, $\mathrm{U}(n)$, and $\mathrm{Sp}(n)$ are easily deduced from the above examples of maximal tori.
The rank of $\mathrm G_2$ is $2$ \cite{cartan, killing}. 
The rank of $\mathrm{Spin}(7)$ is $3$, as $\mathrm{SO}(7)$ has rank $3$.
\end{example}
\begin{proposition}
Let $G$ be a compact connected Lie group.
The intersection of all maximal tori in $G$ is the centre $Z(G)$.
Further, $G$ is semisimple if and only if the centre $Z(G)$ is finite.
\end{proposition}
\begin{proof}
Let $T$ be any maximal torus in $G$.
By Theorem \ref{thm:max-tori}, the intersection of all maximal tori is the set $X \coloneqq \bigcap_{g \in G} gTg^{-1}$.
Let $h \in Z(G)$. 
Since $h$ is contained in a maximal torus, and maximal tori are conjugate, $h$ sits in all maximal tori.
This shows $Z(G) \subset X$.
Let now $h \in X$, so that $h \in gTg^{-1}$ for all $g \in G$.
Let $\ell \in G$ be any other element.
Since $\ell$ and $h$ sit in $aTa^{-1}$ for some $a \in G$, they must commute, so $h$ is central.
Therefore, $X \subset Z(G)$, whence the first assertion.

Now, a connected Abelian normal subgroup of $G$ must be a torus $H$, which is contained in a maximal torus.
The normality condition forces $H$ to be in the intersection of all maximal tori, so $H$ is central.
So if $Z(G)$ is finite, $H$ must be trivial, and hence $G$ is semisimple.
Conversely, if $G$ is semisimple, either $Z(G)$ is trivial or it is not connected.
In the latter case, the connected component of the identity is a connected Abelian normal subgroup, and hence must be trivial.
It follows that $Z(G)$ is discrete.
Since $G$ is compact, $Z(G)$ is finite.
\end{proof}
Up to conjugation, the centres of $\mathrm{SO}(n)$, $\mathrm{SU}(n)$, $\mathrm{U}(n)$, and $\mathrm{Sp}(n)$ lie in a standard maximal torus, and are thus easy to compute by hand.
It follows that the groups $\mathrm{SO}(n)$, $\mathrm{SU}(n)$, $\mathrm{Sp}(n)$ have finite centre (the only exception being $\mathrm{SO}(2)$), so they are semisimple.
We will see an alternative way to compute the centres of $\mathrm{SU}(n)$ and $\mathrm{Sp}(n)$ by using Schur's Lemma, see Remark \ref{rmk:unitary-group-not-semisimple}.
The group $\mathrm G_2$ has trivial centre, whereas the centre of $\mathrm{Spin}(7)$ is isomorphic to $\mathbb Z_2$.
These two facts will be seen in Section \ref{subsec:on-g2-spin7}.
So both $\mathrm G_2$ and $\mathrm{Spin}(7)$ are semisimple.
We note that the unitary group $\mathrm{U}(n)$ is in general not semisimple, as its centre is a circle (see also Remark \ref{rmk:unitary-group-not-semisimple} below).
\begin{remark}
The special orthogonal group $\mathrm{SO}(2)$ is not simple, as it is Abelian.
However, the groups $\mathrm{SO}(n)$ for $n$ odd are simple, and $\mathrm{SO}(n)/\mathbb Z_2$ is simple for $n > 4$ even.
The group $\mathrm{SO}(4)$ is not simple.
Indeed, $\mathfrak{so}(4) = \mathfrak{spin}(4) = \mathfrak{su}(2) \oplus \mathfrak{su}(2)$, and $\mathfrak{su}(2)$ is simple by a direct calculation of its ideals.
The groups $\mathrm{SU}(n)$, $\mathrm{Sp}(n)$, $\mathrm G_2$, and $\mathrm{Spin}(7)$ (and more generally $\mathrm{Spin}(n)$ for $n \geq 8$) are simple.
\end{remark}

\begin{definition}
\label{def:killing-form}
The \emph{Killing form} on a Lie algebra $\mathfrak g$ is the map \[(X,Y) \mapsto \mathrm{Tr}(\mathrm{ad}_X \circ \mathrm{ad}_Y),\] where $X,Y \in \mathfrak g$ and $\mathrm{ad}_X(Z) \coloneqq [X,Z]$.
\end{definition}
We are interested in the Killing form when $\mathfrak g$ comes from a compact connected Lie group which is semisimple.
In this case, it is possible to show that the Killing form is negative-definite (see Cartan's criterion \cite[Chapter V, Remark 7.13]{brocker-tomdieck}).
The Killing form will be suitably scaled if needed.

Choose a maximal torus $T$ in $G$.
Let $\mathfrak g$ be the Lie algebra of $G$, and $\mathfrak t$ be the Lie algebra of $T$.
Then $\mathfrak t \subset \mathfrak g$ gets the structure of a Euclidean space, as a suitable multiple of the Killing form restricts to a real-valued positive-definite scalar product on $\mathfrak t \times \mathfrak t$.

The differential of the exponential map at $0 \in \mathfrak g$ is the identity map, and hence the exponential map is a diffeomorphism of a neighbourhood of $0 \in \mathfrak g$ onto a neighbourhood of the identity of $G$.
Therefore, $\ker(\exp)$ must be discrete.
By the Baker--Campbell--Hausdorff formula (see e.g.\ \cite{fulton-harris}), we have $\exp(A+B) = \exp(A)\exp(B)$ when $[A,B]=0$.
So the kernel of the restriction of $\exp$ to $\mathfrak t \subset \mathfrak g$ is a discrete subgroup of $\mathfrak t$, and has maximal rank.

\begin{definition}
The kernel of the exponential map $\exp \colon \mathfrak t \to T$ is called the \emph{integer lattice}.
Its $\mathbb Z$-dual in $\mathfrak t^*$ is the \emph{weight lattice}.
\end{definition}

\begin{example}
\label{ex:su(2)-max-torus}
A maximal torus $T$ in $\mathrm{SU}(2)$ is given by the subgroup of diagonal matrices $\mathrm{diag}(e^{2\pi i\alpha},e^{-2\pi i\alpha})$, and is isomorphic to a circle $S^1$.
The Lie algebra of $T$ is given by matrices of the form $\mathrm{diag}(i\alpha,-i\alpha)$, $\alpha \in \mathbb R$.
The kernel of $\exp \colon i\mathbb R \to S^1$, $\exp(2\pi i \alpha) = e^{2\pi i \alpha}$, is isomorphic to $\mathbb Z$, as well as the weight lattice.
\end{example}
\begin{example}
\label{ex:su(3)-max-torus}
A maximal torus $T$ in $\mathrm{SU}(3)$ is given by the subgroup of diagonal matrices $\mathrm{diag}(e^{2\pi i\theta},e^{2\pi i \varphi},e^{-2\pi i(\theta+\varphi)})$, and is isomorphic to a torus $T^2$.
The Lie algebra of $T$ is given by matrices of the form $\mathrm{diag}(i\theta,i\varphi,-i(\theta+\varphi))$, $\theta,\varphi \in \mathbb R$.
The integer lattice and the weight lattice are then isomorphic to $\mathbb Z^2$.
\end{example}

We now wish to understand the representation theory of $T$, then study the representation theory of $G$.
We start by giving basic definitions.
\begin{definition}
A complex \emph{representation} of a Lie group $G$, or a \emph{$G$-module}, is a Lie group homomorphism $\pi \colon G \to \mathrm{Aut}(V)$, where $V$ is a complex vector space, and automorphisms of $V$ are complex-linear.
\end{definition}
We may simply say that $V$ is a representation of $G$ when the action is understood, or that $G$ acts linearly on $V$. 
One sometimes says that $V$ alone is a \emph{representation space} of $G$.
A real (resp.\ quaternionic) representation is defined analogously by taking $V$ real (resp.\ quaternionic), and $\mathrm{Aut}(V)$ to be given by $\mathbb R$-linear (resp.\ $\mathbb H$-linear) automorphisms.

We take the occasion to define a few standard representations that will be often used.
Let $V$ be a $G$-representation, and let $\rho_1 \colon G \to \mathrm{Aut}(V)$ be the corresponding homomorphism.
Then $G$ acts on the dual $V^*$ in the following way.
Let $\alpha \in V^*$, and $v$ be any vector in $V$. Then $g \in G$ acts on $\alpha$ via
\begin{equation}
\label{eq:dual-action}
(\rho_1^*(g)\alpha)(v) \coloneqq \alpha(\rho_1(g^{-1})v) = \alpha(\rho_1(g)^{-1}v).
\end{equation}
In this way, the pairing between $V$ and $V^*$ is preserved, namely \[(\rho_1^*(g)\alpha)(\rho_1(g)v) = \alpha(v).\]
If $\rho_2\colon G \to \mathrm{Aut}(W)$ is another representation of $G$, then $G$ acts on $V \oplus W$ via $\rho_1\oplus \rho_2$, defined by \[(g,v\oplus w) \mapsto \rho_1(g)v \oplus \rho_2(g)w.\]
Further, $G$ acts on $V \otimes W$ via $\rho_1 \otimes \rho_2$ defined by \[(g,v \otimes w) \mapsto \rho_1(g)v \otimes \rho_2(g)w.\]
Other representations of $G$, e.g.\ on subspaces of the tensor algebra over $V$, are easily deduced by the actions above.
We normally simplify the notations in the following way.
Say $G$ acts linearly on $V$, then $G$ acts on $V^*$ via $(g,\alpha) \mapsto \alpha(g^{-1}{}\cdot{})$.
If $V$ and $W$ are two representations of $G$, then $G$ acts on $v\oplus w \in V \oplus W$ via $(g,v\oplus w) \mapsto gv\oplus gw$, and so on.

\begin{definition}
A representation of $G$ is called \emph{irreducible} if it contains no non-trivial $G$-invariant subspace.
\end{definition}
\begin{definition}
Let $V$ and $W$ be representations of a Lie group $G$.
A linear map $f \colon V \to W$ is said to be \emph{equivariant}, or \emph{intertwining}, when it preserves the Lie group actions, i.e.\ 
\[f(gv) = gf(v), \qquad g \in G, v \in V.\]
Two representations $V$ and $W$ are \emph{equivalent}, or \emph{isomorphic}, if there is an equivariant linear bijection $f \colon V \to W$.
\end{definition}
Choose any Hermitian product $\langle{}\cdot{},{}\cdot{}\rangle$ on $V$. 
Since in our general discussion $G$ is compact, one defines the averaged Hermitian product $\int_G \langle g{}\cdot{},g{}\cdot{}\rangle d\mu_G$, and this is then $G$-invariant (cf.\ e.g.\ \cite{brocker-tomdieck} for integration on Lie groups and Haar measure).
Therefore, it is not restrictive to look at Hermitian representations on which $G$ acts via unitary transformations. 
Such representations always decompose uniquely (up to isomorphism) into orthogonal direct sums of irreducible ones.

\begin{theorem}[Schur's Lemma]
\label{thm:schur-lemma}
Let $V,W$ be irreducible complex representations of a Lie group $G$, and let $f \colon V \to W$ be an equivariant linear map.
Then either $f=0$ or $f$ is an isomorphism. Furthermore, if $V=W$, then $f = \lambda \mathrm{id}_V$, with $\lambda \in \mathbb C$.
\end{theorem}
\begin{proof}
Note that $\ker f$ and $\mathrm{Im}f$ are $G$-invariant, so they must be trivial subspaces of $V$ and $W$ respectively. 
The first assertion then follows.
If $V=W$, note that the characteristic polynomial of $f$ admits a root $\lambda$. 
Then $\ker(f-\lambda \mathrm{id}_V)$ is non-zero and $G$-invariant, and hence must coincide with $V$.
\end{proof}
\begin{remark}
\label{rmk:schur-lemma}
Unlike the last statement, the first conclusion in Schur's Lemma does not make any use of the fact that $V$ and $W$ are complex vector spaces.
Indeed, this part of Schur's Lemma can be carried over to representations over the reals or the quaternions.
More on real and quaternionic representations is discussed in subsection~\ref{subsec:real-quaternionic-representations} below.
The second part of Schur's Lemma can also be carried over to real representations when $f$ is known to admit a real eigenvalue.
\end{remark}

\begin{remark}
\label{rmk:unitary-group-not-semisimple}
As an application of Schur's Lemma, one can show that the unitary group $\mathrm{U}(n)$ is not semisimple. 
Take $A \in \mathrm{U}(n)$ to be a central element, so that $A \colon \mathbb C^n \to \mathbb C^n$ given by left multiplication is an equivariant map.
Since $\mathbb C^n$ is an irreducible representation of $\mathrm{U}(n)$, $A$ must be a multiple of the identity $A = \lambda \mathrm{id}$.
Since $A$ is unitary, $\lambda$ must have unit length.
It follows that the centre of $\mathrm{U}(n)$ is isomorphic to a circle $S^1$.
The same process applied to $A \in \mathrm{SU}(n)$ gives that $A=\lambda \mathrm{id}$, and $\lambda$ is an $n$-th root of unity, so the centre of $\mathrm{SU}(n)$ is isomorphic to $\mathbb Z_n$.
A similar process yields that the centre of $\mathrm{Sp}(n) \subset \mathrm{U}(2n)$ is isomorphic to $\mathbb Z_2$.
\end{remark}

\begin{digression}[Torus representations]
\label{rmk:irreps-torus}
Let us take a torus $T = S^1 \times \dots \times S^1$ ($n$ times) with $S^1 \subset \mathbb C$ the group of complex numbers of unit length.
Let $V$ be an irreducible complex representation of $T$, and take $t \in T$.
Then $t$ gives a linear map $t\colon V \to V$ which is equivariant (because $T$ is Abelian), and hence by Schur's Lemma $t$ acts as a complex multiple of the identity.
Irreducibility forces $V$ to be one-dimensional, so our representation really is a smooth homomorphism $T \to \mathrm{U}(1)=S^1$, a \emph{group character}.
One can argue that all such maps are of the form \[(e^{2\pi i\theta_1},\dots,e^{2\pi i \theta_n}) \mapsto e^{2\pi i (m_1\theta_1+\dots+m_n\theta_n)}, \qquad m_i \in \mathbb Z,\]
The $n$-tuples $(m_1,\dots,m_n) \in \mathbb Z^n$ are essentially the weights of $T$. 
In fact, they can be viewed as linear forms in $\mathfrak t^*$ mapping $(\theta_1,\dots,\theta_n) \in \mathfrak t$ to $m_1\theta_1+\dots+m_n\theta_n$, and take integral values on the integer lattice.
We deduce that there is a one-to-one correspondence between weights of $T$ and irreducible $T$-modules.
\end{digression}
We now start looking at the representation theory of $G$.
Let $\pi \colon G \to \mathrm{U}(V)$ be a unitary representation of $G$. 
Let $T \subset G$ be a maximal torus. 
In Digression~\ref{rmk:irreps-torus} we have seen that irreducible representations of a torus are (complex) one-dimensional, so $V$ decomposes under the action of $T$ into $T$-invariant irreducible subspaces
\begin{equation}
\label{eq:weight-space-decomposition}
V = \bigoplus_{\lambda} V_{\lambda},
\end{equation}
where the $\lambda$'s are weights of $T$. The non-empty subspaces $V_{\lambda}$ are called \emph{weight spaces}, and the \emph{multiplicity} of $\lambda$ as a weight is the number of times the summand $V_{\lambda}$ occurs in the sum \eqref{eq:weight-space-decomposition}.
The set of $\lambda$'s for $V$ is called the \emph{weight system} of $V$, and the decomposition~\eqref{eq:weight-space-decomposition} is called \emph{weight space decomposition}.

\begin{definition}
Let $\mathfrak g$ be a Lie algebra. A complex \emph{representation of $\mathfrak g$} is a Lie algebra homomorphism $\mathfrak g \to \mathfrak{gl}(V)$, where $V$ is a complex vector space.
\end{definition}
The differential $d\pi \colon \mathfrak g \to \mathfrak{u}(V)\subset \mathfrak{gl}(V)$ of $\pi \colon G \to \mathrm{U}(V)$ at the identity defines a representation of $\mathfrak g$.

We now wish to understand how the Lie algebra $\mathfrak g$ acts on the dual space $V^*$.
Let $X = \dot{\gamma}(0)$, $\gamma(0)=e_G$, for some curve $\gamma \colon (-\epsilon,\epsilon) \to G$, with $\epsilon>0$, passing through the identity $e_G$ of the group at $t=0$. Then if $\alpha \in V^*$ we have
\[d\pi(X)(\alpha) = \frac{d}{dt}(\pi(\gamma(t))\alpha)\Bigr|_{\substack{t=0}} = \frac{d}{dt}(\alpha(\pi(\gamma(t))^{-1}{}\cdot{}))\Bigr|_{\substack{t=0}}=-\alpha(d\pi(X){}\cdot{}).\]
Similarly, if $(V,\rho_1)$ and $(W,\rho_2)$ are two representation of $G$, then $G$ acts via $\rho = \rho_1 \otimes \rho_2$ on the tensor product $V \otimes W$. The induced representation of $\mathfrak g$ on $V \otimes W$ is given by the Leibniz rule and is
\begin{equation}
\label{eq:action-lie-algebra}
d\rho(X)(v \otimes w) = d\rho_1(X)v \otimes w+v\otimes d\rho_2(X)w.
\end{equation}
One can extend the action to other representations in a similar fashion.
\begin{digression}[The adjoint representation]
\label{rmk:roots}
As any Lie group, $G$ acts on its Lie algebra $\mathfrak g$ via the \emph{adjoint representation} $\mathrm{Ad} \colon G \to \mathrm{Aut}(\mathfrak g)$ defined as
\[g \in G \mapsto \mathrm{Ad}_g \in \mathrm{Aut}(\mathfrak g), \quad \mathrm{Ad}_g(X)=gXg^{-1}, \qquad X \in \mathfrak g.\]
Note that the differential $\mathrm{ad} \colon \mathfrak g \to \mathrm{End}(\mathfrak g)$ of $\mathrm{Ad} \colon G \to \mathrm{Aut}(\mathfrak g)$ at the identity is the Lie bracket in $\mathfrak g$.
Take $X = \dot{\gamma}(0)$, with $\gamma \colon (-\epsilon,\epsilon) \to G$, $\epsilon>0$, a curve through $e_G$ at $t=0$.
Then the action of $\mathrm{ad}_X = d(\mathrm{Ad})_{e_G}(X)$ on a vector $Y \in \mathfrak g$ is
\begin{align*}
\mathrm{ad}_X(Y) = \frac{d}{dt}\left(\gamma(t)Y\gamma(t)^{-1}\right)\Bigr|_{\substack{t=0}} = XY-YX = [X,Y].
\end{align*}
This is consistent with the notation used in Definition \ref{def:killing-form}.
Restricting $\mathrm{Ad}$ to a maximal torus $T$ with Lie algebra $\mathfrak t$ gives a decomposition 
\begin{equation}
\label{eq:root-space-decomposition}
\mathfrak g = \mathfrak t \oplus \bigoplus_{\alpha} \mathfrak g_{\alpha},
\end{equation}
where each $\mathfrak g_{\alpha}$ is an irreducible real two-dimensional torus-representation. 
Any element $t \in T$ acts on $\mathfrak g_{\alpha}$ via a special orthogonal rotation in $\mathrm{SO}(2) = S^1$ by an angle normalised as $2\pi \alpha(t)$.
The fact that $\mathfrak t$ coincides with the subspace of $\mathfrak g$ on which $T$ acts trivially follows by the the fact that $T \subset G$ is maximal.
The maps $\alpha \colon T \to S^1$ are given by weights $\alpha \in \mathfrak t^*$, and no $\alpha$ vanishes.
In this special case, the weights $\pm \alpha$ are called the \emph{roots} of $G$, and the relative weight spaces are called \emph{root spaces}. 
The decomposition~\eqref{eq:root-space-decomposition} is called \emph{root space decomposition}, and is of course a special case of \eqref{eq:weight-space-decomposition}.
The set of roots of $\mathfrak g$ is normally denoted by $\Delta_{\mathfrak g}$, or simply by $\Delta$ when there is no danger of confusion.
A trivial dimension count gives $\dim G - \dim T$ is even.
\end{digression}
\begin{example}
\label{ex:su2-roots}
Let again $G = \mathrm{SU}(2)$. The Lie algebra $\mathfrak{su}(2)$ is defined by traceless skew-Hermitian matrices, so a basis is given by
\[H = \begin{pmatrix} i & 0 \\ 0 & -i \end{pmatrix}, \qquad X = \begin{pmatrix} 0 & 1 \\ -1 & 0\end{pmatrix}, \qquad Y = \begin{pmatrix} 0 & i \\ i & 0\end{pmatrix}.\]
One checks the structure equations 
\[[H,X]=2Y, \qquad [X,Y]=2H, \qquad [H,Y]=-2X.\] 
If $t=\mathrm{diag}(e^{2\pi i \theta},e^{-2\pi i\theta})$ is an element in the diagonal maximal torus $T \subset G$ (cf.\ Example~\ref{ex:su(2)-max-torus}), one computes $tHt^{-1} = H$ and
\begin{align*}
tXt^{-1} & = +\cos(2\pi(2\theta))X+\sin(2\pi (2\theta))Y, \\
tYt^{-1} & = -\sin(2\pi(2\theta))X+\cos(2\pi(2\theta))Y,
\end{align*}
so that $t$ acts on $\mathrm{Span}_{\mathbb R}\{X,Y\}$ via a proper rotation by an angle $2\theta$.
Note that $T$ acts trivially on $H$: indeed $H$ generates the Lie algebra $\mathfrak t$ of $T$.
So the set of roots $\Delta$ is just given by the vectors $\pm 2e_1$, with $e_1$ a unit length generator of the real line with respect to a positive-definite rescaling of the Killing form, cf.\ Figure~\ref{fig:roots-su(3)}. 
\end{example}
\begin{example}
\label{ex:su3-roots}
Let now $G=\mathrm{SU}(3)$, and consider the matrices
\begin{alignat*}{3}
E_1 & = \begin{pmatrix} 0 & +1 & 0 \\ -1 & 0 & 0 \\ 0 & 0 & 0 \end{pmatrix}, && \quad E_2 =\begin{pmatrix} 0 & +i & 0 \\ +i & 0 & 0 \\ 0 & 0 & 0 \end{pmatrix}, && \quad E_3 =\begin{pmatrix} 0 & 0 & +1 \\ 0 & 0 & 0 \\ -1 & 0 & 0 \end{pmatrix}, \\
E_4 & =\begin{pmatrix} 0 & 0 & +i \\ 0 & 0 & 0 \\ +i & 0 & 0 \end{pmatrix}, && \quad E_5 =\begin{pmatrix} 0 & 0 & 0 \\ 0 & 0 & +1 \\ 0 & -1 & 0 \end{pmatrix}, && \quad E_6 =\begin{pmatrix} 0 & 0 & 0 \\ 0 & 0 & +i \\ 0 & +i & 0 \end{pmatrix}.
\end{alignat*}
These six matrices and \[E_7=\mathrm{diag}(i,i,-2i), \qquad E_8=\mathrm{diag}(i,-i,0),\] are a basis for the Lie algebra $\mathfrak{su}(3)$.
The Killing form can be rescaled so as to have \[(E_i,E_j) = -\mathrm{Tr}(E_iE_j)/2.\] 
We denote by $X^\flat$ the dual vector of $X$ with respect to the Killing form, i.e.\ $X^{\flat}(X) = (X,X)$. 
A computation yields $(E_7,E_7)=3$, $(E_8,E_8)=1$, and $(E_7,E_8)=0$. So if we define 
\[
e_7 \coloneqq E_7/\sqrt 3, \qquad e_8 \coloneqq E_8,
\]
the restriction of the scalar product to the span of $e_7$ and $e_8$ is the standard scalar product.
A maximal (diagonal) torus $T \subset \mathrm{SU}(3)$ is given by matrices 
\[\mathrm{diag}\Bigl(e^{2\pi i\left(\alpha+\beta\right)},e^{2\pi i \left(\alpha-\beta\right)}, e^{2\pi i \left(-2\alpha \right)}\Bigr),\]
which agrees with the choice of $E_7$, $E_8$ up to a constant factor $2\pi$.
The integer lattice is given by vectors $((m+n)/2)E_7+((m-n)/2)E_8$, with $m,n$ integers.
An analogous computation to that in the previous example yields three $T$-invariant subspaces on which $T$ acts respectively via multiplication by \[e^{2\pi i (2\beta)}, \quad e^{2\pi i (3\alpha-\beta)}, \quad e^{2\pi i (3\alpha+\beta)}.\]
Then we get three linear forms $\lambda_i = a_iE_7^{\flat}+b_iE_8^{\flat}$, and the constraints
\begin{align*}
\lambda_1(\alpha E_7+\beta E_8) & = 3a_1\alpha+b_1\beta = 2\beta, \\
\lambda_2(\alpha E_7+\beta E_8) & = 3a_2\alpha+b_2\beta = 3\alpha-\beta, \\
\lambda_3(\alpha E_7+\beta E_8) & = 3a_3\alpha+b_3\beta = 3\alpha+\beta.
\end{align*}
It follows that $\lambda_1 = 2E_8^{\flat}$, $\lambda_2 = E_7^{\flat}-E_8^{\flat}$, $\lambda_3 = E_7^{\flat}+E_8^{\flat}$.
In terms of the orthonormal basis $e_7,e_8$ the roots are
\[\pm (0,2), \qquad \pm (\sqrt 3,-1), \qquad \pm (\sqrt 3,+1).\]
Note that these six points are the vertices of a regular hexagon inside $\mathfrak t^*$, and
\[(0,2) = 2(0,1), \qquad (\sqrt 3,-1) = 2\left(\frac{\sqrt3}{2},-\frac12\right), \qquad (\sqrt 3,+1) = 2\left(\frac{\sqrt3}{2},+\frac12\right),\]
cf.\ Figure~\ref{fig:roots-su(3)}.
\end{example}
\begin{figure}
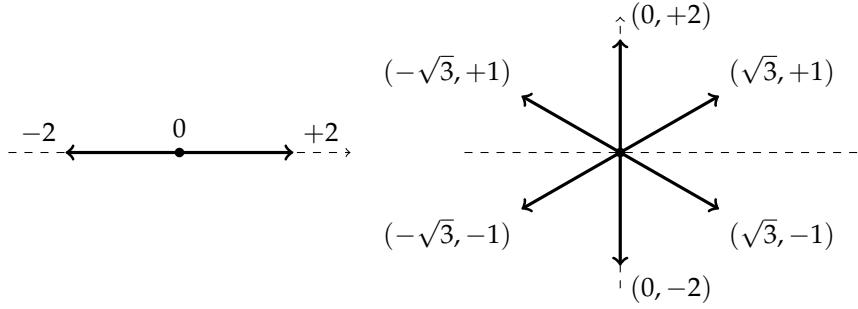

  \tikzpicture
  [scale=1.5]
  \coordinate (suss) at (-4.5,0.5);
  \coordinate (sudx) at (-1.5,0.5);
  \coordinate (gg) at (0.8660254038,-0.7);
  \coordinate (ss) at (0.8660254038,+1.7);
  \coordinate (sx) at (-0.5,0.5);
  \coordinate (dx) at (+3,0.5);
  \coordinate (CC) at (0.8660254038,0.5);
  \coordinate (G) at (-4,0.5);
  \coordinate (H) at (-2,0.5);
  \coordinate (Z) at (-3,0.5);
  \coordinate (A) at (0,0);
  \coordinate (B) at (0.8660254038,-0.5); 
  \coordinate (C) at (1.7320508076,0); 
  \coordinate (D) at (1.7320508076,1); 
  \coordinate (E) at (0.8660254038,1.5); 
  \coordinate (F) at (0,1);
  \draw[<->][very thick][black] (G) -- (H);
  \draw[<->][very thick][black] (A) -- (D);
  \draw[<->][very thick][black] (B) -- (E);
  \draw[<->][very thick][black] (C) -- (F);
  \draw[->][dashed][black] (sx) -- (dx);
  \draw[->][dashed][black] (gg) -- (ss);
  \draw[->][dashed][black] (suss) -- (sudx);
  \fill (A)  node[below left] {$(-\sqrt 3,-1)$};
  \fill (B) node[below right] {$(0,-2)$};
  \fill (C) node[below right] {$(\sqrt 3,-1)$};
  \fill (D) node[above right] {$(\sqrt 3,+1)$};
  \fill (E) node[above right] {$(0,+2)$};
  \fill (F) node[above left] {$(-\sqrt 3,+1)$};
  \fill (G) node[above left] {$-2$};
  \fill (H) node[above right] {$+2$};
  \fill (Z)  circle [radius=1.2pt] node[above][circle] {$0$};
  \fill (CC) circle [radius=1.2pt] node[right][circle] {};
  \endtikzpicture
  \caption{Roots of $\mathrm{SU}(2)$ (left) and $\mathrm{SU}(3)$ (right).}
  \label{fig:roots-su(3)}
\end{figure}
\begin{remark}
A computation of the Lie brackets in Example~\ref{ex:su3-roots} shows that the Lie subalgebra generated by $E_1,E_2,E_8$ is isomorphic to $\mathfrak{su}(2)$.
So there is a Lie monomorphism $\mathfrak{su}(2) \hookrightarrow \mathfrak{su}(3)$, which is the differential of a Lie group immersion $\mathrm{SU}(2) \hookrightarrow \mathrm{SU}(3)$.
The monomorphism $\mathfrak{su}(2) \hookrightarrow \mathfrak{su}(3)$ can be easily written down in matrix form.
This can also be seen by the root space decomposition of $\mathfrak{su}(3)$, as $E_1,E_2$ generate a root space of $\mathfrak{su}(3)$ and $E_8$ lies in a maximal Abelian subalgebra.
In general, from the root system of a Lie group one can read off the presence of Lie subgroups, cf.\ Dynkin \cite{dynkin}.
\end{remark}

From now on, we assume $G$ is also simple and simply connected.
This amounts to saying that representations of $G$ correspond to representations of its Lie algebra, and simplifies our exposition.
In particular, it simplifies our presentation of the fundamental weights and Weyl's correspondence, Theorem \ref{thm:weyl1}.
We refer to \cite{adams, brocker-tomdieck, hall} for a more general picture.

We discuss more on weights, but we first need some extra terminology on roots.
Consider the roots $\Delta$ of $G$ (cf.\ Digression~\ref{rmk:roots}).
Define a subset $\Delta^+ \subset \Delta$ such that 
\begin{enumerate}
\item for each root $\alpha \in \Delta$, only one of the roots $\pm \alpha$ sits in $\Delta^+$,
\item for any distinct $\alpha,\beta \in \Delta^+$ such that $\alpha+\beta \in \Delta$, then $\alpha+\beta \in \Delta^+$.
\end{enumerate}
Such a subset is not unique, and we call a choice of $\Delta^+$ a set of \emph{positive roots}.
\begin{definition}
A positive root is called \emph{simple} if it cannot be written as the sum of two elements of $\Delta^+$.
\end{definition}
A set of positive simple roots is a basis for $\mathfrak t^*$.
\begin{definition}
A weight $\lambda$ is called \emph{dominant} if $(\lambda,\alpha_i) \geq 0$ for all positive simple roots $\alpha_i$.
\end{definition}
The set of dominant weights lies in the closure of a \emph{fundamental} (dual) \emph{Weyl chamber}, which depends on the choice of positive simple roots.
More generally, a Weyl chamber is defined roughly in the following way. 
For any root, consider the hyperplane orthogonal to it. 
The set of all such hyperplanes disconnects $\mathfrak t^*$ into connected components, which one calls \emph{Weyl chambers}.
The fundamental Weyl chamber contains weights that are dominant with respect to the choice of positive simple roots.
\begin{example}
In Example \ref{ex:su2-roots}, a choice of a positive (simple) root is $+2e_1$, so dominant weights correspond to the natural numbers, cf.\ Figure \ref{fig:weights-su(3)}.
\end{example}
\begin{example}
In Example \ref{ex:su3-roots}, a choice of positive roots is \[(0,2), \quad (\sqrt 3,\pm 1),\] and the positive simple roots are then $(0,2)$, $(\sqrt 3,-1)$.
Identify points in $\mathfrak t^*$ with pairs $(x,y)$. The closure of the fundamental dual Weyl chamber for this choice of simple roots is defined by the region \[\left\{(x,y) \in \mathfrak t^*: y \geq 0, y \leq \sqrt 3 x\right\}.\]
Dominant weights correspond to weights in the fundamental Weyl chamber, cf.\ Figure \ref{fig:weights-su(3)}.
Here the dominant weights are found by a calculation based on the integer lattice in Example \ref{ex:su3-roots}.
An easier way to describe dominant weights will be given later, when we introduce the so-called \emph{fundamental} weights.
\end{example}
\begin{figure}
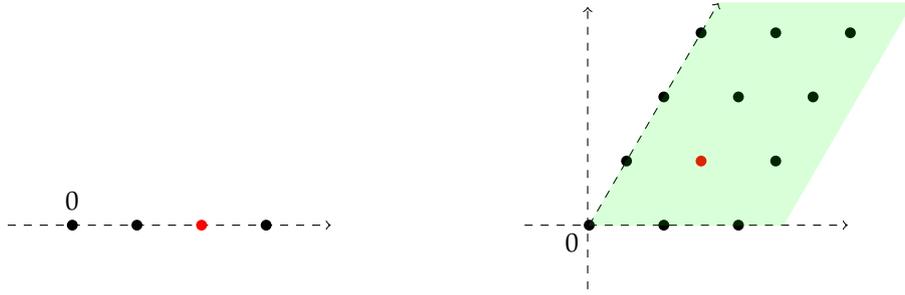

  \tikzpicture
  [scale=1.7]
  \coordinate (weightsu3-0) at (1,0.5);
  \coordinate (weightsu3-1) at (1.2886751346,1);
  \coordinate (weightsu3-2) at (1.5773502692,1.5);
  \coordinate (weightsu3-3) at (1.8660254038,2);
  \coordinate (weightsu3-4) at (1.5773502692,0.5);
  \coordinate (weightsu3-5) at (1.8660254038,1);
  \coordinate (weightsu3-6) at (2.1547005384,1.5);
  \coordinate (weightsu3-7) at (2.443375673,2);
  \coordinate (weightsu3-8) at (2.1547005384,0.5);
  \coordinate (weightsu3-9) at (2.443375673,1);
  \coordinate (weightsu3-10) at (2.7320508076,1.5);
  \coordinate (weightsu3-11) at (3.0207259422,2);
  \coordinate (point-horizontal-axis) at (2.5,0.5);
  \coordinate (point-north-east) at (3.5,2.232);
  \coordinate (linesu3) at (2,2.232);
  \coordinate (suss) at (-3.5,0.5);
  \coordinate (sudx) at (-1,0.5);
  \coordinate (sudxdx) at (-1.5,0.5);
  \coordinate (gg) at (0.98803162,0);
  \coordinate (ss) at (0.98803162,+2.2);
  \coordinate (sx) at (0.5,0.5);
  \coordinate (dx) at (+3,0.5);
  \coordinate (CC) at (0.98803162,0.5);
  \coordinate (G) at (-3,0.5);
  \coordinate (G1) at (-2.5,0.5);
  \coordinate (G2) at (-2,0.5);
  \coordinate (H) at (-2,0.5);
  \coordinate (Z) at (-3,0.5);
  \coordinate (A) at (0,0);
  \coordinate (B) at (0.98803162,-0.5);
  \coordinate (C) at (1.976,0);
  \coordinate (D) at (1.976,1);
  \coordinate (E) at (0.98803162,1.5);
  \coordinate (F) at (0,1);
  \draw[->][dashed][black] (sx) -- (dx); 
  \draw[->][dashed][black] (gg) -- (ss); 
  \draw[->][dashed][black] (suss) -- (sudx);
  \draw[->][dashed][black] (weightsu3-0) -- (linesu3);
  \fill (G) circle node[above left] {};
  \fill (G1) circle [radius=1.2pt] node[above][circle] {};
  \fill[red] (G2) circle [radius=1.2pt] node[above][circle] {};
  \fill (G2) node[above][circle] {};
  \fill (Z)  circle [radius=1.2pt] node[above][circle] {$0$};
  \fill (sudxdx) circle [radius=1.2pt] node[above][circle] {};
  \fill (weightsu3-0) circle [radius=1.2pt] node[below left][circle] {$0$};
  \fill (weightsu3-1) circle [radius=1.2pt] node[above][circle] {};
  \fill (weightsu3-2) circle [radius=1.2pt] node[above][circle] {};
  \fill (weightsu3-3) circle [radius=1.2pt] node[above][circle] {};
  \fill (weightsu3-4) circle [radius=1.2pt] node[above][circle] {};
  \fill[red] (weightsu3-5) circle [radius=1.2pt] node[above left][circle] {};
  \fill (weightsu3-5) node[above left][circle] {};
  \fill (weightsu3-6) circle [radius=1.2pt] node[above][circle] {};
  \fill (weightsu3-7) circle [radius=1.2pt] node[above][circle] {};
  \fill (weightsu3-8) circle [radius=1.2pt] node[above][circle] {};
  \fill (weightsu3-9) circle [radius=1.2pt] node[above][circle] {};
  \fill (weightsu3-10) circle [radius=1.2pt] node[above][circle] {};
  \fill (weightsu3-11) circle [radius=1.2pt] node[above][circle] {};
  \draw [fill, opacity=.15, green] (weightsu3-0) -- (point-horizontal-axis) -- (point-north-east) -- (linesu3) -- cycle;
  \endtikzpicture
  \caption{Dominant weights of $\mathrm{SU}(2)$ (left) and $\mathrm{SU}(3)$ (right). The red dots correspond to dominant roots. The shaded green region is the fundamental dual Weyl chamber for the positive simple roots.}
  \label{fig:weights-su(3)}
\end{figure}
\begin{remark}
In connection with Example~\ref{ex:su3-roots}, let \[\alpha_1 \coloneqq (0,2), \quad \alpha_2 \coloneqq (\sqrt 3,+1), \quad \alpha_3 \coloneqq (\sqrt 3,-1).\]
One computes $(\alpha_i,\alpha_i)=4$, $i=1,2,3$, and $(\alpha_1,\alpha_2)=-(\alpha_1,\alpha_3)=(\alpha_2,\alpha_3)=2$.
Note that the projection of each $\alpha_i$ on another $\alpha_j$ is always half an integer.
In general, it can be either an integer or half an integer.
This is known as \emph{integrality condition}: if $\alpha$ and $\beta$ are simple roots, we always have
\[\langle \alpha,\beta\rangle \coloneqq 2\frac{(\alpha,\beta)}{(\alpha,\alpha)} \in \mathbb Z.\]
These integers are called \emph{Cartan integers}. 
Integrality is a rather strong condition forcing the possible angles between two roots to be in a rather small set of values.
An analysis of all possibilities yields all possible sets of roots for any group of given rank, and consequently a classification of semisimple Lie algebras \cite{dynkin, fulton-harris, humphreys}.
\end{remark}
The set of Weyl chambers has a natural symmetry group, which is finite. 
In order to state some of its properties, let us briefly recall some terminology.
\begin{definition}
Let $G$ be a group acting on a set $X$, and let $x \in X$.
The \emph{orbit} of $x$ is the set $Gx\coloneqq \{gx \in X: g \in G\}$.
The \emph{stabiliser} of $x$ is the set $G_x\coloneqq \{g \in G: gx=x\}$.
The \emph{orbit space} is the quotient $X/G$, the set of all orbits.
\end{definition}
\begin{definition}
The action of a group $G$ on a set $X$ is 
\begin{enumerate}
\item \emph{effective}, or \emph{faithful}, if every non-trivial element of $G$ acts non-trivially on at least one point in $X$;
\item \emph{free}, when all stabilisers are trivial, i.e.\ $gx=x$ for any $x \in X$ implies $g=e_G$;
\item \emph{transitive}, when for any $x,y \in X$ there is $g \in G$ such that $gx=y$, i.e.\ the orbit space $X/G$ is given by a single coset.
\end{enumerate}
\end{definition}
We have the following special case, which is well-known.
\begin{proposition}[Orbit--Stabiliser Theorem]
\label{thm:orbit-stabiliser}
Let $G$ be a finite group acting on a set $X$, and $x \in X$ be any point.
Then $|Gx| = |G|/|G_x|$, i.e.\ the order of the orbit of $x$ multiplied by the order of its stabiliser is the cardinality of $G$.
\end{proposition}
When $X$ is a manifold, we have the following result, which will be needed later (see \cite[Chapter I, Theorem 4.3 and Proposition 4.6]{brocker-tomdieck}).
\begin{proposition}
\label{prop:homogeneous-spaces}
Let $G$ be a Lie group and $X$ be a smooth manifold acted on by $G$ transitively. 
Let $H \subset G$ be the stabiliser of a point in $X$.
Then $G/H$ is a smooth manifold and there is a diffeomorphism $X=G/H$.
\end{proposition}
\begin{remark}
Note that points in the same $G$-orbit have conjugate stabilisers.
Take $p,q \in X$ and assume $p=gq$ for some $g \in G$. 
Let $h$ be an element in the stabiliser of $p$.
Then
\[q=g^{-1}p=g^{-1}hp=g^{-1}hgq,\]
so the map $h \mapsto g^{-1}hg$ maps the stabiliser of $p$ to the stabiliser of $q$.
One then speaks of equivalence classes of stabilisers up to conjugation, rather than the stabiliser of a specific point.
\end{remark}
\begin{definition}
\label{def:homogeneous-space}
A space $X$ acted on by a group $G$ transitively is called \emph{homogeneous}.
The stabiliser of any point is called \emph{generic} or \emph{principal}.
\end{definition}

Let $\alpha$ be a root of $G$, and let $\sigma_{\alpha}$ be the reflection across the hyperplane orthogonal to $\alpha$:
\[\sigma_{\alpha}(\beta) \coloneqq \beta-2\frac{(\beta,\alpha)}{(\alpha,\alpha)}\alpha, \qquad \beta \in \mathfrak t^*.\]
\begin{definition}
The \emph{Weyl group} $W(G)$ of the root system of $G$ is the group generated by all reflections $\sigma_{\alpha}$, for $\alpha$ a root of $G$.
\end{definition}
It turns out that $W(G)$ acts freely and transitively on the set of Weyl chambers, and maps weights to weights.
By Proposition \ref{thm:orbit-stabiliser}, it follows that the number of Weyl chambers is the same as the order of the Weyl group, which is then finite.
The Weyl group of the root system of $G$ is isomorphic to the Weyl group of $G$, defined as follows.
\begin{definition}
The \emph{Weyl group} of $G$ is defined as $W(G) \coloneqq N(T)/T$, where $T \subset G$ is any maximal torus, and $N(T) = \{h \in G: hTh^{-1} \subset T\}$ its normaliser.
\end{definition}
In virtue of Theorem \ref{thm:max-tori}, the latter definition does not depend on the choice of a maximal torus.
For instance, $W(\mathrm{SU}(2))$ is the symmetric group of two elements, so has order $2$, cf.\ Example \ref{ex:su2-roots}.
The group $W(\mathrm{SU}(3))$ is the symmetric group of three elements, and has order $6$, cf.\ Example \ref{ex:su3-roots}. 
These two results can be seen by a direct calculation, but in general it is hard to compute Weyl groups (e.g.\ the Weyl group of the exceptional Lie group $E_8$ has 696,729,600 elements). 
See \cite{brocker-tomdieck} for more examples.

The set of dominant weights may be written as $\mathbb N \omega_1+\dots+\mathbb N\omega_n$, where $\omega_1,\dots,\omega_n$ are the \emph{fundamental weights} defined by 
\[\frac{2(\omega_i,\alpha_j)}{(\alpha_j,\alpha_j)} \coloneqq \delta_{ij},\]
for $\alpha_j$ simple roots.
Fundamental weights define coordinates for each dominant weight.

\begin{example}
\label{ex:fundamental-reps-su3}
It is clear that the only fundamental weight of $\mathrm{SU}(2)$ corresponds to the natural number $+1$.
The fundamental weights for $\mathrm{SU}(3)$ with the orthonormal data given in Example \ref{ex:su3-roots} are $(1/\sqrt 3,1), (2/\sqrt 3,0)$.
Note that both weights lie on the walls of a fundamental Weyl chamber.
The dominant weights are then of the form \[n\left(\frac{1}{\sqrt 3},1\right)+m\left(\frac{2}{\sqrt 3},0\right), \qquad m,n \in \mathbb N.\]
Dominant and fundamental weights of $\mathrm{SU}(2)$ and $\mathrm{SU}(3)$ are depicted in Figure~\ref{fig:fundamental-weights-su(3)}.
\end{example}
\begin{figure}
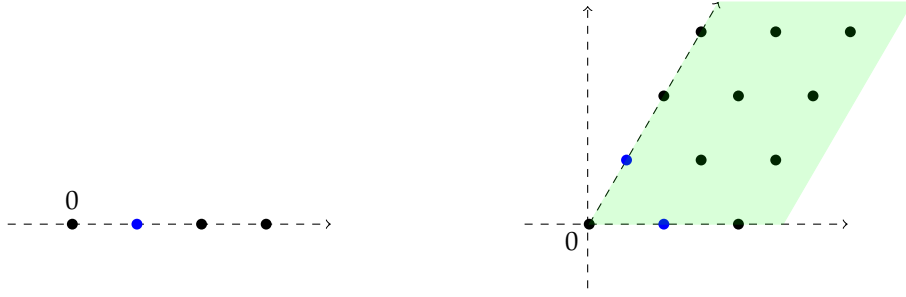

  \tikzpicture
  [scale=1.7]
  \coordinate (weightsu3-0) at (1,0.5);
  \coordinate (weightsu3-1) at (1.2886751346,1);
  \coordinate (weightsu3-2) at (1.5773502692,1.5);
  \coordinate (weightsu3-3) at (1.8660254038,2);
  \coordinate (weightsu3-4) at (1.5773502692,0.5);
  \coordinate (weightsu3-5) at (1.8660254038,1);
  \coordinate (weightsu3-6) at (2.1547005384,1.5);
  \coordinate (weightsu3-7) at (2.443375673,2);
  \coordinate (weightsu3-8) at (2.1547005384,0.5);
  \coordinate (weightsu3-9) at (2.443375673,1);
  \coordinate (weightsu3-10) at (2.7320508076,1.5);
  \coordinate (weightsu3-11) at (3.0207259422,2);
  \coordinate (point-horizontal-axis) at (2.5,0.5);
  \coordinate (point-north-east) at (3.5,2.232);
  \coordinate (linesu3) at (2,2.232);
  \coordinate (suss) at (-3.5,0.5);
  \coordinate (sudx) at (-1,0.5);
  \coordinate (sudxdx) at (-1.5,0.5);
  \coordinate (gg) at (0.98803162,0);
  \coordinate (ss) at (0.98803162,+2.2);
  \coordinate (sx) at (0.5,0.5);
  \coordinate (dx) at (+3,0.5);
  \coordinate (CC) at (0.98803162,0.5);
  \coordinate (G) at (-3,0.5);
  \coordinate (G1) at (-2.5,0.5);
  \coordinate (G2) at (-2,0.5);
  \coordinate (H) at (-2,0.5);
  \coordinate (Z) at (-3,0.5);
  \coordinate (A) at (0,0);
  \coordinate (B) at (0.98803162,-0.5);
  \coordinate (C) at (1.976,0);
  \coordinate (D) at (1.976,1);
  \coordinate (E) at (0.98803162,1.5);
  \coordinate (F) at (0,1);
    \draw[->][dashed][black] (sx) -- (dx); 
  \draw[->][dashed][black] (gg) -- (ss); 
  \draw[->][dashed][black] (suss) -- (sudx);
  \draw[->][dashed][black] (weightsu3-0) -- (linesu3);
  \fill (G) circle node[above left] {};
  \fill[blue] (G1) circle [radius=1.2pt] node[above][circle] {};
  \fill (G2) circle [radius=1.2pt] node[above][circle] {};
  \fill (Z)  circle [radius=1.2pt] node[above][circle] {$0$};
  \fill (sudxdx) circle [radius=1.2pt] node[above][circle] {};
  \fill (weightsu3-0) circle [radius=1.2pt] node[below left][circle] {$0$};
  \fill[blue] (weightsu3-1) circle [radius=1.2pt] node[above][circle] {};
  \fill (weightsu3-2) circle [radius=1.2pt] node[above][circle] {};
  \fill (weightsu3-3) circle [radius=1.2pt] node[above][circle] {};
  \fill[blue] (weightsu3-4) circle [radius=1.2pt] node[above][circle] {};
  \fill (weightsu3-5) circle [radius=1.2pt] node[above left][circle] {};
  \fill (weightsu3-5) node[above left][circle] {};
  \fill (weightsu3-6) circle [radius=1.2pt] node[above][circle] {};
  \fill (weightsu3-7) circle [radius=1.2pt] node[above][circle] {};
  \fill (weightsu3-8) circle [radius=1.2pt] node[above][circle] {};
  \fill (weightsu3-9) circle [radius=1.2pt] node[above][circle] {};
  \fill (weightsu3-10) circle [radius=1.2pt] node[above][circle] {};
  \fill (weightsu3-11) circle [radius=1.2pt] node[above][circle] {};
  \draw [fill, opacity=.15, green] (weightsu3-0) -- (point-horizontal-axis) -- (point-north-east) -- (linesu3) -- cycle;
  \endtikzpicture
  \caption{Dominant weights of $\mathrm{SU}(2)$ (left) and $\mathrm{SU}(3)$ (right). The blue dots correspond to the fundamental weights.}
  \label{fig:fundamental-weights-su(3)}
\end{figure}
There is a partial order on the set of weights. 
\begin{definition}
For two weights $\alpha$ and $\beta$, we write $\alpha \leq \beta$ ($\beta$ is \emph{higher} than $\alpha$) if and only if $\beta-\alpha$ is either a sum of positive roots or zero.
\end{definition}
The next theorem illustrates the profound connection between dominant weights and irreducible representations of a given compact, connected, simply connected, simple Lie group. 
It is known as the \lq\lq Theorem of the highest weight\rq\rq.
\begin{theorem}[Weyl \cite{weyl1}]
\label{thm:weyl1}
There is a bijection $V \mapsto \lambda_V$ between isomorphism classes of irreducible complex representations of $G$ and dominant weights of $G$.
\end{theorem}
Essentially, every irreducible representation $(V,\pi)$ admits a unique \emph{highest weight} $\lambda = \lambda_V$ (which is dominant) in its weight system, and the multiplicity of $\lambda$ is one.
The remaining weights of $V$ have the form $\lambda-\alpha$, where $\alpha$ is a sum of positive roots.
This can be checked explicitly for the examples treated above.
It follows by Theorem \ref{thm:weyl1} that if two irreducible representations have the same highest weight, they must be isomorphic.
Once one has a dominant weight, constructing the corresponding irreducible representation is in general non-trivial (see \emph{Verma modules} e.g.\ in \cite{knapp}).
In special cases, explicit constructions can be described by hand in a fairly easy way.
An example is given by the irreducible representations of $\mathrm{SU}(2)$ (and $\mathrm{SO}(3)$), to be discussed in the next section.

Theorem \ref{thm:weyl1} provides a concrete way of parametrising isomorphism classes of irreducible representations of $G$.
In particular, fundamental weights correspond to the so-called \emph{fundamental representations}.
For instance, the only fundamental representation of $\mathrm{SU}(2)$ is $V = \mathbb C^2$ (standard representation).
There are two fundamental representations of $\mathrm{SU}(3)$, i.e.\ its standard one $U = \mathbb C^3$, and $\Lambda^2U$ (see \cite{adams, brocker-tomdieck} for the representation rings of these and more examples).
In order to get the weights of the corresponding dual representations, one uses the action of the Weyl group on the Weyl chambers.

In the geometric contexts we will be interested in, it is the concrete realisation of an irreducible representation that will matter the most.
The language of weights will be mainly used to exploit the correspondence between dominant weights and irreducible representations.
Irreducibility of a representation will often be understood via different methods, in particular by enumerating invariant quadratic forms on some representation space and invoking different results going back to Weyl.
We refer to \cite{adams,brocker-tomdieck, fulton-harris, hall, humphreys, knapp} for many more details and examples to illustrate the theory above.

\subsection{Irreducible representations of \texorpdfstring{$\mathrm{SU}(2)$}{SU(2)} and \texorpdfstring{$\mathrm{SO}(3)$}{SO(3)}}
\label{subsec:irreps-su2-so3}

We include here a concrete description of the irreducible representations of $\mathrm{SU}(2)$ and $\mathrm{SO}(3)$, which is classical, and can be found in \cite{brocker-tomdieck}.

Consider the special unitary group $\mathrm{SU}(2)$.
This naturally acts on the vector space $V \coloneqq \mathbb C^2$ by left multiplication (in fact, $V$ is the fundamental representation of $\mathrm{SU}(2)$), but also on $\mathbb C$ via the trivial representation.
Let $V_n$ be the space of homogeneous polynomials of degree $n$ in two variables $x_1$ and $x_2$, with generators
\[x_1\negthinspace{}^n, \quad x_1\negthinspace{}^{n-1}x_2, \quad x_1\negthinspace{}^{n-2}x_2\negthinspace{}^2, \quad \dots \quad x_1x_2\negthinspace{}^{n-1}, \quad x_2\negthinspace{}^n.\]
Let now $P \in \mathbb C[x_1,x_2]$ be any polynomial. By viewing it as a polynomial function on $\mathbb C^2$, we define the following action: an element $g \in \mathrm{SU}(2)$ acts on $P$ via the so-called \emph{right regular action} defined by
\[(gP)(x) \coloneqq P(xg), \qquad x = (x_1,x_2).\]
For example, if $g = \left(\begin{smallmatrix} a & b \\ c & d\end{smallmatrix}\right)$ and $P(x_1,x_2) = x_1^ax_2^{n-a}$, then 
\[P(xg) = P(ax_1+cx_2,bx_1+dx_2) = (ax_1+cx_2)^a(bx_1+dx_2)^{n-a}.\]
So $\mathrm{SU}(2)$ acts via homogeneous linear transformations, and hence the spaces $V_n$ are invariant under this action.

\begin{proposition}
The $\mathrm{SU}(2)$-representations $V_n$ are irreducible.
\end{proposition}
\begin{proof}
It is enough to show that each $\mathrm{SU}(2)$-equivariant endomorphism of $V_n$ is a multiple of the identity.
Indeed, if this is true and $V_n$ split into invariant subspaces $V_n = V_n^1 \oplus V_n^2$, then the projections on the first or the second summand should be the identity, which is a contradiction.
Let $A$ be an equivariant endomorphism of $V_n$, and let $g_a = \mathrm{diag}(a,a^{-1}) \in \mathrm{SU}(2)$ be an element of the standard maximal torus of $\mathrm{SU}(2)$. 
Let $P_k(x_1,x_2) \coloneqq x_1^kx_2^{n-k}$.
Then one computes $g_aP_k = a^{2k-n}P_k$, and therefore \[g_aAP_k = Ag_aP_k = Aa^{2k-n}P_k = a^{2k-n}AP_k.\]
Choose $a$ such that all powers $a^{2k-n}$, $0 \leq k \leq n$, are distinct.
One verifies that in this way the $a^{2k-n}$-eigenspace of $g_a$ in $V_n$ is generated by $P_k$.
Therefore, $AP_k = c_kP_k$ for some $c_k \in \mathbb C$.
Now consider the real rotations
\[r_t = \begin{pmatrix} \cos t & -\sin t \\ \sin t & \cos t \end{pmatrix} \in \mathrm{SO}(2) \subset \mathrm{SU}(2), \qquad t \in \mathbb R,\]
and compute $Ar_tP_n = r_tAP_n$. 
One has
\begin{align*}
Ar_tP_n & = A(z_1\cos t+z_2\sin t)^n \\
& = \sum_{k=0}^n \binom{n}{k}(\cos^kt)(\sin^{n-k}t)AP_k \\
& = \sum_{k=0}^n \binom{n}{k}(\cos^kt)(\sin^{n-k}t)c_kP_k.
\end{align*}
On the other hand,
\begin{align*}
r_tAP_n & = \sum_{k=0}^n \binom{n}{k}(\cos^kt)(\sin^{n-k}t)c_nP_k.
\end{align*}
Comparing coefficients, it follows that $c_k = c_n$, so $A = c_n\mathrm{id}$, which proves the claim.
\end{proof}

The number $n$ corresponds to the highest weight of the irreducible representation $V_n$.
It turns out that any irreducible unitary representation of $\mathrm{SU}(2)$ is isomorphic to some $V_n$.

\begin{corollary}
The spaces $V_{2n}$ are irreducible $\mathrm{SO}(3)$-representations.
\end{corollary}
\begin{proof}
As we have seen in Example~\ref{ex:su(2)-sp(1)}, there is a two-to-one universal cover $\pi \colon \mathrm{SU}(2) \to \mathrm{SO}(3)$ with kernel $\{\pm \mathrm{id}\}$.
So if $W$ is an irreducible representation of $\mathrm{SO}(3)$, then $\pi^*W$ is an irreducible representation of $\mathrm{SU}(2)$ on which $-\mathrm{id}$ acts as the identity.
Conversely, if $-\mathrm{id}$ acts as the identity on an $\mathrm{SU}(2)$-representation $V$, then $V$ is also a representation of $\mathrm{SO}(3)$.
By the above definitions, it follows that the spaces $H_n \coloneqq V_{2n}$ are the irreducible representations of $\mathrm{SO}(3)$.
\end{proof}
\begin{remark}
The irreducible $\mathrm{SO}(3)$-representations above can also be realised as invariant spaces of polynomials on the two-sphere $S^2$, the \emph{spherical harmonics}, see Br\"ocker--tom Dieck \cite{brocker-tomdieck} or Faraut \cite{faraut}.
\end{remark}

\subsection{The Clebsch--Gordan formula}
\label{subsec:clebsch-gordan-formula}

Decomposing tensor products of two or more representations into direct sums of irreducible ones is known as \emph{plethysm}.
This is particularly useful in geometric contexts where objects of interest sit in a tensor product of representations.
We look at a few examples of plethysm, which work with weights and some knowledge of the irreducible representations of the Lie group in question.
We then infer a general result on the decomposition of the tensor product of representations for $\mathrm{SU}(2)$ and $\mathrm{SO}(3)$.

First off, the Lie group $\mathrm{SU}(2)$ is simple and simply connected, so its irreducible representations correspond to its Lie algebra representations.
It is more practical to work with the complexified Lie algebra $\mathfrak{su}(2)\otimes \mathbb C$: this is obtained by extending the action of the real scalars on $\mathfrak{su}(2)$ to the complex ones, and by extending the Lie bracket by complex bilinearity.
Next, observe that the complexification $\mathfrak{su}(2) \otimes \mathbb C$ is the complex Lie algebra $\mathfrak{sl}(2,\mathbb C)$ of traceless $2\times 2$ matrices with complex entries.
The Lie algebra $\mathfrak{sl}(2,\mathbb C)$ has a basis given by
\[K = \begin{pmatrix} 1 & 0 \\ 0 & -1 \end{pmatrix}, \qquad U = \begin{pmatrix} 0 & 1 \\ 0 & 0 \end{pmatrix}, \qquad V = \begin{pmatrix} 0 & 0 \\ 1 & 0 \end{pmatrix},\]
and the structure equations are $[K,U]=2U$, $[K,V]=-2V$, $[U,V]=K$.
Recall the structure equations of $\mathfrak{su}(2)$ from Example \ref{ex:su2-roots}, i.e.\ $[H,X]=2Y$, $[H,Y]=-2X$, and $[X,Y]=2H$.
The linear map defined by 
\[K\coloneqq -iH, \qquad U \coloneqq \frac12(X-iY), \qquad V \coloneqq -\frac12(X+iY),\]
is then an isomorphism of Lie algebras $\mathfrak{su}(2) \otimes \mathbb C = \mathfrak{sl}(2,\mathbb C)$.
This is actually a special case of a more general isomorphism $\mathfrak{su}(n) \otimes \mathbb C = \mathfrak{sl}(n,\mathbb C)$.

We also observe that if $\mathfrak{su}(2) \to \mathfrak{gl}(V)$ is any complex representation of $\mathfrak{su}(2)$, then allowing for complex coefficients yields a representation $\mathfrak{sl}(2,\mathbb C) \to \mathfrak{gl}(V)$.
Conversely, restricting the action of $\mathfrak{sl}(2,\mathbb C)$ on $V$ to $\mathfrak{su}(2) \subset \mathfrak{sl}(2,\mathbb C)$ gives a representation of $\mathfrak{su}(2)$ on $V$.

We have seen that the irreducible representations of $\mathrm{SU}(2)$ are classified by the natural numbers, and that the only positive root is $+2e_1$. 
For simplicity, let us identify the weights of $\mathrm{SU}(2)$ with the corresponding natural numbers.
Here is a list of low dimensional irreducible representations with their weight systems.
\begin{center}
\bgroup
\setlength{\tabcolsep}{15pt}
\def\arraystretch{1.25}
\begin{tabular}{|c c c|}
 \hline
 Representation & Dimension & Weights \\ [0.2ex] 
 \hline\hline
 $V_0$ & 1 & $0$ \\ 
 \hline
 $V_1$ & 2 & $-1$, $+1$ \\
 \hline
 $V_2$ & 3 & $-2$,\ $0$,\ $+2$ \\
 \hline
 $V_3$ & 4 & $-3$,\ $-1$,\ $+1$,\ $+3$ \\
 \hline
 $V_4$ & 5 & $-4$,\ $-2$,\ $0$,\ $+2$,\ $+4$\\ 
 \hline
\end{tabular}
\egroup
\end{center}
It is clear by the form of $K$ that the standard representation $V_1=\mathbb C^2$ of $\mathfrak{sl}(2,\mathbb C)$ has weights $\pm 1$.
The structure equations for $\mathfrak{sl}(2,\mathbb C)$ above tell us the roots are $\pm 2$.

Let us compute the decomposition into irreducible $\mathfrak{sl}(2,\mathbb C)$-summands of $V_1 \otimes V_1$, the squared tensor product of the standard representation.
Let $e_1,e_2$ be the standard basis for $V_1$.
A basis for $V_1 \otimes V_1$ is then given by the vectors $\{e_i \otimes e_j\}$, $i,j=1,2$.
Since $Ke_1=e_1$ and $Ke_2=-e_2$, we compute (recall \eqref{eq:action-lie-algebra})
\begin{alignat*}{2}
K(e_1\otimes e_1) & = 2e_1\otimes e_1, \qquad && K(e_1 \otimes e_2) = 0, \\
K(e_2 \otimes e_1) & = 0, \qquad && K(e_2 \otimes e_2) = -2e_2\otimes e_2. 
\end{alignat*}
Note that the weight $0$ has multiplicity two, whereas $\pm 2$ have multiplicity one.
This means that in the decomposition of $V_1 \otimes V_1$ we find one copy of the trivial representation (weight $0$ counted once), and one copy of the adjoint representation (covering the weights $\pm 2$ and the remaining $0$).
We then have a splitting
\[V_1 \otimes V_1=V_2 \oplus V_0,\] 
and this is clearly the decomposition into irreducible summands. 

Next, let us look at $V_1 \otimes V_2$. A basis for this space is $\{e_k \otimes K, e_k \otimes U, e_k \otimes V\}$, with $k=1,2$.
We compute
\begin{alignat*}{2}
K(e_1 \otimes K) & = +e_1 \otimes K, && \qquad K(e_2 \otimes K) = -e_2\otimes K, \\
K(e_1 \otimes U) & = +3e_1 \otimes U, && \qquad K(e_2 \otimes U) = +e_2 \otimes U, \\ 
K(e_1 \otimes V) & = -e_1 \otimes V, && \qquad K(e_2 \otimes V) = -3e_2 \otimes V.
\end{alignat*}
Here $\pm 3$ have multiplicity one, $\pm 1$ have multiplicity two.
Therefore \[V_1 \otimes V_2 = V_3 \oplus V_1.\]
A similar computation for $V_2 \otimes V_2$ shows that $V_2 \otimes V_2 = V_4 \oplus V_2 \oplus V_0$, then one can go on for all tensor products $V_k \otimes V_{\ell}$.

By restricting the action of $\mathfrak{sl}(2,\mathbb C)$ to $\mathfrak{su}(2)$, one obtains representations of $\mathfrak{su}(2)$, and hence of $\mathrm{SU}(2)$.
It seems reasonable to guess that the highest weight of $V_k \otimes V_{\ell}$ is $k+\ell$, and that all other weights are of the form $k+\ell-2m$, where $m$ is a natural number, and $2$ is the only positive root of $\mathrm{SU}(2)$.
This agrees with our discussion on the highest weights of irreducible representations.
Indeed, we have the following result, which can be shown by character theory \cite{brocker-tomdieck}.
\begin{theorem}[Clebsch--Gordan Formula]
\label{thm:clebsch-gordan-formula}
Let $V_k$ be the $(k+1)$-dimensional irreducible representation of $\mathrm{SU}(2)$. 
The tensor product $V_k \otimes V_{\ell}$ decomposes into irreducible summands as
\[V_k \otimes V_{\ell} = \bigoplus_{j=0}^q V_{k+\ell-2j}, \qquad \text{ with } q = \min\{k,\ell\}.\]
If $k$ and $\ell$ are even, the same formula gives the decomposition of $V_k \otimes V_{\ell}$ into irreducible $\mathrm{SO}(3)$-representations.
\end{theorem}

In general, decomposing tensor products can be done in a systematic way by using the theory of weights.
This approach lends itself to computer algebra, see Lie at the link \url{http://wwwmathlabo.univ-poitiers.fr/~maavl/LiE/}.

\subsection{Real and quaternionic representations}
\label{subsec:real-quaternionic-representations}

The theory we have summarised works well over the complex numbers.
However, in geometric applications one is often interested in real representations, or even quaternionic ones.
We now describe connections among complex, real, and quaternionic types, see \cite{adams, brocker-tomdieck}.

Complexification provides a canonical way to pass from a real representation to a complex one, and works as follows.
Let $V$ be a real vector space acted on by a group $G$ via $(g,v) \mapsto gv$, where $g \in G$ and $v \in V$.
Then $V \otimes \mathbb C$ is a complex vector space, where the action of $\mathbb C$ is given by $ w(v \otimes z) \coloneqq v \otimes wz$.
Also, a $G$-action on $V \otimes \mathbb C$ is defined in a natural way by setting $g(v \otimes z) \coloneqq gv \otimes z$.
Note that there is a conjugate-linear, $G$-equivariant map $j \colon V \otimes \mathbb C \to V \otimes \mathbb C$ defined by $j(v \otimes z) = v \otimes \overline z$, and $j^2 = \mathrm{id}_{V \otimes \mathbb C}$.
The map $j$ is an example of \emph{structure map}, in particular a real one.
\begin{definition}
Let $V$ be a $G$-module. A \emph{structure map} $j \colon V \to V$ is a conjugate-linear, equivariant map such that $j^2=\pm \mathrm{id}_V$.
We say that $j$ is a \emph{real} (resp.\ \emph{quaternionic}) structure if $j^2=\mathrm{id}_V$ (resp.\ $j^2=-\mathrm{id}_V$).
\end{definition}

Conversely, a straightforward way to construct a real representation out of a complex one is to restrict the action of the complex scalars to the reals.
However, in some situations (as above) a complex representation $V$ of $G$ comes with a real structure $j$.
Equivariance implies that the eigenspace of $j$ corresponding to eigenvalue $+1$ is automatically a real representation of $G$.

More precisely, given a complex representation $V$ of $G$, we denote by $[\![V]\!]$ the real vector space obtained by restriction of the action of scalars (e.g.\ $[\![\mathbb C]\!] = \mathbb R^2$) and acted on by $G$ in the natural way.
Clearly, $\dim_{\mathbb R} [\![V]\!] = 2\dim_{\mathbb C} V$, and hence $\dim_{\mathbb C}([\![V]\!] \otimes \mathbb C) = \frac12 \dim_{\mathbb R}([\![V]\!] \otimes \mathbb C) = 2\dim_{\mathbb C} V$.
Indeed, it can be shown that there is a canonical isomorphism
\[[\![V]\!] \otimes \mathbb C = V \oplus \overline{V},\]
where $\overline{V}$ has the same additive structure as that of $V$, but scalars act conjugated, see Adams \cite{adams} for details.
After choosing an invariant Hermitian product $\langle{}\cdot{},{}\cdot{}\rangle$ of $V$, the map $\overline V \to V^*$ defined by $v \mapsto \langle v,{}\cdot{}\rangle$ is an isomorphism (provided that $\langle{}\cdot{},{}\cdot{}\rangle$ is conjugate-linear in the first entry and that the group acting is compact).
If $\overline V = V$ we also say that $V$ is \emph{self-conjugate}.
If $V$ comes equipped with a real structure $j$, we denote by $[V]$ the $+1$-eigenvalue of $j$, and there is an equivalence of representations
\[[V] \otimes \mathbb C = V.\] 

If $V$ is a $G$-module over the complex numbers with a structure map $j$ such that $j^2=-\mathrm{id}_V$, we can view $V$ as a $G$-module over the quaternions in a natural way by viewing $j$ as the usual unit quaternion, and setting $k \coloneqq ij=-ji$.
Conversely, if $V$ is a $G$-module over the quaternions $\mathbb H$, we can view it as a $G$-module over the complex numbers with a structure map $j$ such that $j^2=-\mathrm{id}_V$.
The $\mathbb C$-module structure can be set by letting $i$ act on the left, and the structure map $j$ act on the left.
Or else, it can be set by letting $i$ act on the right and the structure map $j$ act on the right. 

\begin{example}
The group $\mathrm{SU}(2)$ has an irreducible representation $V_n$ for all natural numbers $n$, as shown previously.
It is known that $V_n$ admits a structure map $j$.
If $n$ is odd, $j$ is a real structure.
If $n$ is even, $j$ is a quaternionic structure \cite{itzkowitz-rothman-strassberg}.
\end{example}

\newpage
\section{Invariant theory}
\label{sec:invariant-theory}

In studying representations of Lie groups, one is often led to consider tensor invariants.
For instance, inner products are invariants of orthogonal groups.
Volume forms are invariants of special linear and special orthogonal groups.
Invariants provide a remarkable guide to understand the irreducibility of certain representations, in virtue of results going back to Weyl to be discussed.
For instance, suppose the orthogonal group $\mathrm{O}(n)$ acts on $\mathbb R^n$ in the standard way.
The irreducibility of the space of $p$-forms $\Lambda^p(\mathbb R^n)^*$ under the action of $\mathrm{O}(n)$ can be understood via the theory of the invariants of $\mathrm{O}(n)$.
Suppose instead $n$ is even, say $n=2m$, and assume the unitary group $\mathrm{U}(m)\subset \mathrm{O}(n)$ acts on $\mathbb R^n$.
The space of two-forms $\Lambda^2(\mathbb R^n)^*$ turns out to be reducible, but it always contains an irreducible one-dimensional summand corresponding to a unitary invariant.
This is essentially given by a Hermitian structure on $\mathbb R^n=\mathbb C^m$.

The goal of this section is to study representations of orthogonal and unitary groups from the point of view of invariant theory.
We detect invariants and illustrate their role in decomposing certain representations of interest in geometry.
The compact groups $\mathrm G_2$ and $\mathrm{Spin}(7)$ are also introduced.
We first summarise the linear algebra needed, then proceed with general results.
Most of the material in this section is taken from Besse \cite{besse2}, Bryant \cite{bryant}, and Salamon \cite{salamon}.

\subsection{Linear algebra}
\label{subsec:some-linear-algebra}

Let $(V,g)$ be the standard Euclidean $n$-dimensional real vector space.
The Euclidean structure allows one to identify $V$ and its dual $V^*$ via the linear map \[\flat \colon V \to V^*, \quad v \mapsto v^{\flat} \coloneqq g(v,{}\cdot{}).\]
Since $V$ comes as an $\mathrm{O}(n)$-representation, the isomorphism $\flat$ gives an equivalence of $\mathrm{O}(n)$-representations $V = V^*$, as (recall \eqref{eq:dual-action})
\[Av \mapsto g(Av,{}\cdot{}) = g(v,A^{-1}{}\cdot{}) = Av^{\flat}, \qquad A \in \mathrm{O}(n).\]
This also implies we can identify tensors of different type which have the same homogeneous degree, e.g.\ $(2,0)$, $(1,1)$, and $(0,2)$ tensors, and we will tacitly do so in such contexts.
So, for example, a $2$-tensor may be an endomorphism, a bilinear form, or a bivector.
We set $\sharp \coloneqq \flat^{-1}$.

The scalar product $g$ extends to the tensor algebras of $V$ and $V^*$ in a natural way: if $\{e_i\}$, $i=1,\dots,n$, is an orthonormal basis of $V$, then $\{e_{i_1}\otimes \dots \otimes e_{i_k}\}$ is an orthornomal basis of $V \otimes \dots \otimes V$ ($k$ times), and similarly for covectors in $V^* \otimes \dots \otimes V^*$.

\begin{digression}[The trace operator]
\label{digression:trace}
We recall the definition of the trace operator on $V \otimes V$.
There is a well-known isomorphism \[V^* \otimes V \to \mathrm{End}(V), \qquad \alpha \otimes v \mapsto f_{\alpha,v}, \text{ with } f_{\alpha,v}(e) \coloneqq \alpha(e)v.\]
The trace is defined on endomorphisms in the usual way.
Let $v \otimes w \in V \otimes V$, and view it as an element $v^{\flat} \otimes w \in V^* \otimes V$. 
The corresponding endomorphism is $f_{v^{\flat},w}$ and acts as $f_{v^{\flat},w}(e) = g(v,e)w$.
Let now $e_1,\dots,e_n$ be an orthonormal basis of $V$.
The trace of $f_{v^{\flat},w}$ is then 
\[\mathrm{Tr}(f_{v^{\flat},w}) = \sum_{k=1}^n  g(g(v,e_k) w,e_k) = \sum_{k=1}^n v_kw_k = g(v,w).\]
Identify $f_{v^{\flat},w}$ with $v \otimes w$. 
The trace of $v \otimes w$ is then by definition the inner product $g(v,w)$. 
\end{digression}

We now describe the algebra behind the so-called \emph{algebraic curvature tensors}, which arise in Riemannian geometry (see Section \ref{sec:riemannian-geometry}).
These are covariant tensors of type $(4,0)$ satisfying the same symmetries of curvature tensors.
We will need this algebraic set-up later.

Consider the tensor product $\Lambda^*V \otimes \Lambda^*V$ of the exterior algebra over $V$ with itself.
This has a canonical structure of graded algebra: if $x_i, y_i \in \Lambda^*V$, $i=1,2$, set
\[(x_1\otimes y_1)\cdot (x_2\otimes y_2) \coloneqq (x_1\wedge x_2) \otimes (y_1 \wedge y_2).\]
In general the product is not commutative: if $\alpha \in \Lambda^pV \otimes \Lambda^qV$ and $\beta \in \Lambda^rV \otimes \Lambda^sV$, one has
\[\alpha \cdot \beta = (-1)^{pr+qs}\beta\cdot \alpha,\]
where $pr$ comes from the commutation relation of $p$-forms and $r$-forms, similarly for $qs$.
Define the subalgebra \[CV = \bigoplus_p C^pV \coloneqq \bigoplus_p S^2\Lambda^pV\] of $\Lambda^*V \otimes \Lambda^*V$.
It is easy to check that $CV$ is a graded, commutative algebra. 
The restriction of the dot product to $CV$ is denoted by $\owedge$.

\begin{example}
\label{ex:kulkarni-nomizu-product}
We will be interested in the following special case. 
Let $r,s$ be two indecomposable elements of $S^2V$ (i.e.\ $r,s \in C^1V$), and let us compute $r \owedge s \in C^2V = S^2\Lambda^2V$. 
Set $r = r_1 \odot r_2$ and $s=s_1 \odot s_2$.
We compute
\[r\owedge s = (r_1 \wedge s_1) \odot (r_2\wedge s_2) + (r_1\wedge s_2) \odot (r_2\wedge s_1).\]
Let us view $r\owedge s$ as a covariant tensor. 
We can then apply $r \owedge s$ to a quadruple $(x,y,z,t) \in V^4$.
After rearranging all terms, one finds the formula
\begin{align*}
(r\owedge s)(x,y,z,t) & = r(x,z)s(y,t)+r(y,t)s(x,z) \\
& \qquad -r(x,t)s(y,z)-r(y,z)s(x,t).
\end{align*}
Note that the symmetries of this formula reflect the fact that $r \owedge s \in S^2\Lambda^2V$.
The $\owedge$ product of two elements in $S^2V$ as above is called the \emph{Kulkarni--Nomizu product}. 
A straightforward computation gives the identity
\begin{equation}
\label{eq:formula-owedge}
(r \owedge s)(x,y,z,t)+(r \owedge s)(y,z,x,t)+(r \owedge s)(z,x,y,t)=0.
\end{equation}
\end{example}

On $\Lambda^pV$ we have a scalar product induced by $g$ on $V$ defined on indecomposable elements as
\begin{equation}
\label{eq:inner-product-exterior-algebra}
g_{\Lambda}(x_1\wedge \dots \wedge x_p, y_1 \wedge \dots \wedge y_p) \coloneqq \det([g(x_i,y_j)]_{ij}),
\end{equation}
where $[g(x_i,y_j)]_{ij}$ is the matrix whose $ij$-th entry is $g(x_i,y_j)$.
The definition is well-posed. Swapping two $x_i$'s (or $y_i$'s) is the same as swapping two columns in the matrix with coefficients $g(x_i,y_j)$.
Symmetry is clear. Positive-definiteness and non-degeneracy follow by the fact that $\det([g(x_i,x_j)])$ is the squared volume of the polytope generated by $x_1,\dots,x_p$, which vanishes if and only if any two $x_i$, $x_j$ are linearly dependent, i.e.\ $x_1\wedge \dots \wedge x_p=0$.
Note that on two-forms we have the identity
\begin{equation}
\label{eq:inner-product-exterior-algebra2}
g_{\Lambda}(x \wedge y, z \wedge t) = \frac12 (g \owedge g)(x,y,z,t).
\end{equation}
Using the inner product in \eqref{eq:inner-product-exterior-algebra}, we may identify elements of $S^2\Lambda^pV$ with symmetric endomorphisms of $\Lambda^pV$.
Let $r,s$ be as such. Then their standard scalar product $(r,s) \mapsto \mathrm{Tr}(r \circ s)$ induces an inner product on $S^2\Lambda^pV$.

\subsection{Orthogonal groups}
\label{subsec:orthogonal-groups}
We start studying invariants of the orthogonal group and the special orthogonal group.
We introduce the so-called \emph{elementary invariants}, or \emph{products of traces} (cf.\ Gray \cite{gray-tubes} and Berger--Gauduchon--Mazet~\cite[D\'efinition D.IV.33]{berger-gauduchon-mazet}), which are generators for invariant forms.
\begin{definition}
\label{def:products-traces}
Let $P \colon V^{\otimes k} \to \mathbb R$ be a homogeneous polynomial of degree $h$, i.e.\ $P$ has a symmetric multilinear extension to $(V^{\otimes k})^{\otimes h}$.
We call $P$ an \emph{elementary invariant}, or a \emph{product of traces}, if
\begin{enumerate}
\item $kh$ is even, say $kh=2p$, 
\item there is a permutation $\sigma \in S_{2p}$, and
\item there exists an orthonormal basis $\{e_1,\dots,e_n\}$ of $V$ such that 
\[P(R) = \sum_{s_1,\dots,s_p=1}^n \sigma(\otimes_h R)(e_{s_1},e_{s_1},\dots,e_{s_p},e_{s_p}), \qquad R \in V^{\otimes k}.\]
Here $\otimes_h R$ is the element in $V^{\otimes kh}$ that maps $(x_1,\dots,x_{kh})$ to the product
\[R(x_1,\dots,x_k)R(x_{k+1},\dots,x_{2k})\dots R(x_{k(h-1)+1},\dots,x_{kh}).\]
Then $\sigma(\otimes_h R)(x_1,\dots,x_{kh}) \coloneqq (\otimes_h R)(x_{\sigma(1)},\dots,x_{\sigma(kh)})$.
\end{enumerate}
\end{definition}
\begin{remark}
The expression of $P(R)$ does not depend on the orthonormal basis chosen, and hence $P$ is invariant for the orthogonal group.
\end{remark}
\begin{theorem}
The vector space of real homogeneous polynomials $P \colon V^{\otimes k} \to \mathbb R$ invariant under the action of the orthogonal group is generated by elementary invariants.
\end{theorem}
A sketch of the proof that elementary invariants generate the space of invariant homogeneous polynomials on $V^{\otimes k}$ is given in \cite[pag.\ 83]{berger-gauduchon-mazet}, see also references therein.

The next results are due to Weyl, cf.\ Besse, Expos\'e IX \cite{besse2}.
In particular, the next proposition is an essential ingredient behind our decompositions.
\begin{proposition}
\label{prop:invariant-quadratic-forms}
Let $W$ be a subspace of the tensor algebra over $(V,g)$.
The space of $\mathrm{O}(n)$-invariant quadratic forms on $W$ is one-dimensional if and only if $W$ is an irreducible $\mathrm{O}(n)$-representation.
\end{proposition}
\begin{proof}
Assume the space of invariant quadratic forms on $W$ is one-dimensional, and note that the squared norm on $W$ is a generator.
If $f \colon W \to W$ is an $\mathrm{O}(n)$-equivariant map, then the quadratic form $q$ defined by $q(w) = g(f(w),w)$, $w \in W$, is $\mathrm O(n)$-invariant.
But then $q$ is proportional to the above squared norm, so $f$ is forced to be proportional to the identity. 
Now if $W$ split into the direct sum of invariant (non-trivial) subspaces $W_1 \oplus W_2$, the projection $W_1 \oplus W_2 \to W_1$ would be proportional to the identity, contradiction.

Conversely, assume $W$ is an irreducible $\mathrm{O}(n)$-representation.
Let $q$ be an $\mathrm{O}(n)$-invariant quadratic form on $W$, and let $Q$ be the symmetric bilinear form associated to it.
Let now $f \colon W \to W$ be the symmetric endomorphism defined by $Q(v,w) = g(f(v),w)$. 
Since $q$ is $\mathrm{O}(n)$-invariant, $Q$ is bi-invariant, and then $f$ is necessarily $\mathrm{O}(n)$-equivariant.
We know that $f$ is diagonalisable over the reals, and each of its eigenspaces is $\mathrm{O}(n)$-invariant.
By Schur's Lemma, Theorem \ref{thm:schur-lemma}, $f$ admits only one eigenvalue, i.e.\ $f$ is proportional to the identity.
This forces $q$ to be proportional to the restriction of the squared norm on the tensor algebra of $V$ to $W$.
\end{proof}
\begin{remark}
We remark that in Besse, Expos\'e IX \cite{besse2}, only one of the implications in Proposition \ref{prop:invariant-quadratic-forms} is shown.
The statement can be generalised to any compact connected Lie group $G$ by using a $G$-invariant inner product on $V$, which always exists by compactness of $G$, cf.\ also Bredon \cite{bredon}.
\end{remark}
\begin{remark}
\label{rmk:estimate-irreps}
By Proposition \ref{prop:invariant-quadratic-forms}, the number of $\mathrm{O}(n)$-invariant quadratic forms on a subrepresentation $W$ of the tensor algebra of $V$ is an upper bound of the number of the irreducible $\mathrm{O}(n)$-invariant components of $W$.
In the decomposition of $W$, an \emph{isotypic component} may occur, i.e.\ there may be two or more irreducible summands that are equivariantly isomorphic (corresponding to the same highest weight).
If all isotypic components are simple, then the number of quadratic invariants is the same as the number of irreducible summands.
\end{remark}
\begin{proposition}
\label{prop:irreps-orthogonal-group}
Let $(V,g)$ be the standard Euclidean $n$-dimensional space acted on linearly by the orthogonal group $\mathrm{O}(n)$. 
Let $S^2V$ be the space of symmetric $2$-tensors on $V$, and $S_0^2V$ be the traceless elements in $S^2V$. Then $S^2V$ splits into $\mathrm{O}(n)$-invariant irreducible components
\[S^2V = S_0^2V\oplus \mathbb R.\]
Further, for all $p\geq 0$, the space $\Lambda^pV$ is an irreducible $\mathrm{O}(n)$-representation.
\end{proposition}
\begin{proof}
Let us consider $S^2V$ with its inner product given by the trace of the product of two elements.
Note that $g$ is itself an element of $S^2V$, so $S^2V$ contains an $\mathrm{O}(n)$-invariant one-dimensional space.
Let $S_0^2V$ be the orthogonal complement of it, so that $S^2V = S_0^2V \oplus \mathbb Rg$ is an orthogonal decomposition.
In other words, $r \in S_0^2V$ if and only if $\mathrm{Tr}(rg)=0$.
By expressing $rg$ in terms of an orthonormal basis we see that $r$ is traceless.
Obviously $S_0^2V$ is an $\mathrm{O}(n)$-representation, as the trace is invariant under conjugation.

Recall Definition \ref{def:products-traces}.
A computation with $h=k=2$ and $R \in S^2V$ now tells us the only elementary quadratic invariants on $S^2V$ are
\[P_1(R) = \sum_{i,j} R(e_i,e_i)R(e_j,e_j), \qquad P_2(R) = \sum_{i,j} R(e_i,e_j)^2,\]
where $\{e_i\}$, $i=1,\dots,n$, is an orthonormal basis of $V$.
Note that $P_1$ vanishes on $S_0^2V$, as $P_1(R) = \mathrm{Tr}(R)^2$.
Proposition \ref{prop:invariant-quadratic-forms} implies that there are at most two irreducible $\mathrm{O}(n)$-invariant summands in the decomposition of $S^2V$, so $S_0^2V$ must be irreducible.

A similar computation with $k=p$ and $h=2$ shows that up to sign there is only one non-trivial elementary quadratic invariant on $\Lambda^pV$, which is 
\[P(R) = \sum_{i_1,\dots,i_p} R(e_{i_1},\dots,e_{i_p})^2,\]
so $\Lambda^pV$ is necessarily irreducible. 
\end{proof}

\begin{corollary}
Let $(V,g)$ be the standard Euclidean $n$-dimensional space acted on linearly by the orthogonal group $\mathrm{O}(n)$. 
The space of endomorphisms $V \otimes V$ decomposes into irreducible $\mathrm{O}(n)$-invariant summands as 
\[V \otimes V = \Lambda^2V \oplus S_0^2V \oplus \mathbb R.\]
\end{corollary}
\begin{remark}
The above splitting refines the well-known $\mathrm{GL}(n,\mathbb R)$-decomposition into symmetric and skew-symmetric endomorphisms
\[V \otimes V = \Lambda^2V \oplus S^2V.\]
Here the two summands are irreducible for the action of $\mathrm{GL}(n,\mathbb R)$, as can be seen by using Young diagrams and Schur functors, see \cite[Lecture 6]{fulton-harris}.
Incidentally, Schur functors yield decompositions into irreducible $\mathrm{O}(n)$-representations of $V^{\otimes k}$.
\end{remark}
We now discuss more on the special orthogonal group. 
By definition, the special orthogonal group $\mathrm{SO}(n)$ acts on the standard Euclidean space $\mathbb R^n$ preserving any volume form.
If $W$ is an irreducible representation of $\mathrm{O}(n)$, then $\mathrm{SO}(n)$ may act on $W$ irreducibly or not.
In the latter case, $W$ necessarily splits into the direct sum of two (non-trivial) irreducible $\mathrm{SO}(n)$-representations, as $\mathrm{SO}(n)$ has index $2$ in $\mathrm{O}(n)$.
\begin{lemma}
\label{lemma:splitting-so}
If $W$ is an irreducible $\mathrm{O}(n)$-representation, and $W$ splits into irreducible components $W=W_1 \oplus W_2$ under $\mathrm{SO}(n)$, then $W_1$ and $W_2$ have the same dimension.
\end{lemma}
\begin{proof}
Elements in $\mathrm{O}(n)$ with negative determinant can be written as a product of a matrix in $\mathrm{SO}(n)$ and an element $\tau_0 \in \mathrm{O}(n)$ with $\det \tau_0=-1$.
Take $\tau_0 \coloneqq \mathrm{diag}(-1,+1,\dots,+1)$ with respect to the standard orthonormal basis, so $\tau_0^2=\mathrm{id}$.

We claim that $\tau_0(W_i)$, $i=1,2$, is $\mathrm{SO}(n)$-invariant.
Take $x \in \tau_0(W_i)$, $i=1,2$, and $\alpha \in \mathrm{SO}(n)$. 
Then $x = \tau_0(y)$ for some $y \in W_i$, and $\tau_0(\alpha(x)) = \tau_0(\alpha(\tau_0(y)))$. 
But $\det(\tau_0 \circ \alpha \circ \tau_0) = 1$, so $\tau_0 \circ \alpha \circ \tau_0 \in \mathrm{SO}(n)$, and hence $\tau_0(\alpha(\tau_0(y))) \in W_i$, which forces $\alpha(x) \in \tau_0(W_i)$ by applying $\tau_0$ to both sides.
Clearly, $\tau_0(W_i)$ is irreducible for $\mathrm{SO}(n)$.
Now, if we assume $\dim W_1 > \dim W_2$, then $\dim \tau_0(W_1)=\dim W_1>\tfrac12 \dim W$, so $\tau_0(W_1)$ has non-zero intersection with $W_1$.
Also, $\tau_0(W_1) \cap W_1$ and $\tau_0(W_1) \cap W_2$ are invariant subspaces, so irreducibility of $W_1$ and $W_2$ forces $\tau_0(W_1) = W_1$, contradiction.
An analogous argument applies if $\dim W_1 < \dim W_2$, whence $\dim W_1=\dim W_2$.
\end{proof}

\begin{proposition}
\label{prop:decomposition-son}
Let $(V,g)$ be the standard Euclidean $n$-dimensional space acted on linearly by the special orthogonal group $\mathrm{SO}(n)$. 
Let $S^2V$ be the space of symmetric $2$-tensors on $V$, and $S_0^2V$ be the traceless elements in $S^2V$. Then $S^2V$ splits into $\mathrm{SO}(n)$-irreducible components
\[S^2V = S_0^2V\oplus \mathbb R.\]
Further, the space $\Lambda^pV$ is reducible as an $\mathrm{SO}(n)$-representation only if $n=2p$.
\end{proposition}
\begin{proof}
Let $\{e_1,\dots,e_n\}$ be an orthonormal oriented basis of $V$.
Then a basis of $S_0^2V$ is given by the vectors (recall Digression \ref{digression:trace})
\[e_i \odot e_j, \quad i<j, \qquad \text{and} \qquad e_1 \odot e_1 - e_j \odot e_j, \quad 1<j.\]
Recall Lemma \ref{lemma:splitting-so}, and choose $\tau_0$ as in its proof. 
Note that $\tau_0$ acts trivially on the $n-1$ vectors $e_1\odot e_1-e_j\odot e_j$ and on the $\binom{n-1}{2}$ vectors $e_i \odot e_j$ with $1<i<j$.
On the other hand, $\tau_0$ switches the sign of the $n-1$ vectors $e_1 \odot e_j$.
Then a necessary condition for $S_0^2V$ to split is \[n-1=n-1+\binom{n-1}{2},\] which implies $n=1$ or $n=2$.
The only relevant case is $n=2$.
Then $\mathrm{SO}(n) = \mathrm{SO}(2)$ is a circle and acts on $S_0^2V=\mathbb R^2$ by rotations.
Thus $S_0^2V$ is irreducible.

We now consider $\Lambda^pV$. Assume $\Lambda^pV$ splits into the direct sum of two invariant subspaces of the same dimension.
A basis for $\Lambda^pV$ is given by the $\binom{n}{p}$ vectors
\[e_{i_1} \wedge \dots \wedge e_{i_p}, \qquad i_1 <\dots < i_p.\]
Note that $\tau_0$ switches the sign of the vectors $e_1 \wedge e_{i_2} \wedge \dots e_{i_p}$, and leaves all other vectors invariant.
But then a necessary condition for $\Lambda^pV$ to split is
\[\binom{n-1}{p-1}=\binom{n-1}{p}.\]
A computation shows this holds for $n=2p$, so $\Lambda^pV$ is reducible only if $n=2p$.
\end{proof}
\begin{remark}
An alternative way to prove the irreducibility of $\Lambda^pV$ for $n \neq 2p$ follows by the theory of weights, see e.g.\ \cite[Chapter 4]{knapp}.
\end{remark}

Let $\mathrm{SO}(n)$ act on $(V,g)$ in the standard way by left multiplication.
Take the volume form $\mathrm{vol}=e_1\wedge \dots \wedge e_n$ on $(V,g)$.
The action preserves $\mathrm{vol}$ and the Euclidean structure.
Define the \emph{Hodge star operator}
$\star \colon \Lambda^pV \to \Lambda^{n-p}V$ as follows: if $\alpha$ and $\beta$ are $p$-forms, then 
\[\star \alpha \wedge \beta \coloneqq g_{\Lambda}(\alpha,\beta)\mathrm{vol},\]
where $g_{\Lambda}$ is as in \eqref{eq:inner-product-exterior-algebra}. 

The Hodge star operator is clearly an isometry of the exterior algebra, as $\star^2 = (-1)^{p(n-p)}\mathrm{id}$.
If $A \in \mathrm{SO}(n)$, then $g_{\Lambda}(A{}\cdot{},A{}\cdot{})=g_{\Lambda}$, and the condition $\det A=1$ implies $A\mathrm{vol} = (\det A)\mathrm{vol} = \mathrm{vol}$. 
It follows that
\begin{align*}
\star A\alpha \wedge A\beta & = g_{\Lambda}(A\alpha,A\beta)\mathrm{vol} = g_{\Lambda}(\alpha,\beta)\mathrm{vol}, \\
A\star \alpha \wedge A\beta & = A(\star \alpha \wedge \beta) = g_{\Lambda}(\alpha,\beta) A \mathrm{vol} = g_{\Lambda}(\alpha,\beta)\mathrm{vol},
\end{align*}
so $\star A\alpha \wedge A\beta = A\star \alpha \wedge A\beta$. Since this holds for all $\beta$, the action of $A$ commutes with the action of $\star$, namely $\star$ is $\mathrm{SO}(n)$-equivariant.
This implies $\Lambda^pV$ and $\Lambda^{n-p}V$ behave in the same way as $\mathrm{SO}(n)$-representations, and we can restrict attention to those $p\geq0$ such that $p\leq n-p$.
We now know that if $p<n-p$ the space $\Lambda^pV$ is irreducible for $\mathrm{SO}(n)$, so let us take $p=n-p$.
Note that $\star^2 = (-1)^{p^2}\mathrm{id}$. 
So in the special case $n=2p$ and $p$ is even, $\star$ maps $p$-forms to $p$-forms, and has eigenvalues $\pm 1$.
This implies $\Lambda^pV$ splits into the direct sum of $\mathrm{SO}(n)$-invariant eigenspaces
\[\Lambda^pV = \Lambda^p_+V \oplus \Lambda^p_-V,\]
where $\Lambda^p_{\pm}V$ correspond to the eigenvalues $\pm 1$.
\begin{definition}
We say $\Lambda^p_+V$ is the space of \emph{self-dual} $p$-forms, whereas $\Lambda^p_-V$ is the space of \emph{anti self-dual} $p$-forms.
\end{definition}
\begin{proposition}
\label{rmk:lambda-so}
The $\mathrm{SO}(n)$-module $\Lambda^2\mathbb R^n$ with its standard action and the adjoint representation $\mathfrak{so}(n)$ are equivalent.
\end{proposition}
\begin{proof}
Let $g$ be the standard scalar product on $\mathbb R^n$, and let $A \in \mathfrak{so}(n)$ be a skew-symmetric endomorphism.
The identity \[g(AX,Y) = -g(X,AY) = -g(AY,X)\] shows that $g(A{}\cdot{},{}\cdot{})$ is a skew-symmetric bilinear form.
This gives a linear isomorphism $\mathfrak{so}(n) \to \Lambda^2 (\mathbb R^n)^* = \Lambda^2 \mathbb R^n$.
Now let $B \in \mathrm{SO}(n)$ be a special orthogonal matrix. 
We know that $B$ acts on $A \in \mathfrak{so}(n)$ via $A \mapsto BAB^{-1}$. 
Note that
\[g(BAB^{-1}{}\cdot{},{}\cdot{}) = g(AB^{-1}{}\cdot{},B^{-1}{}\cdot{}) = B(g(A{}\cdot{},{}\cdot{})).\]
This chain of identities implies that $\Lambda^2 \mathbb R^n = \mathfrak{so}(n)$ is an equivalence of $\mathrm{SO}(n)$-representations.
The equivalence also identifies the scalar product on $\Lambda^2 \mathbb R^n$ with a multiple of the Killing form on $\mathfrak{so}(n)$.
\end{proof}
\begin{example}
\label{ex:self-dual-forms}
Let $V = \mathbb R^4$ with its standard Euclidean structure.
Let $e_1,\dots,e_4$ be the standard basis of $V$, and let $\mathrm{vol} = e_1 \wedge \dots \wedge e_4$ be the standard volume form. 
A basis for $\Lambda^2_+V$ is given by 
\begin{gather*}
f_1 = e_1\wedge e_2+e_3\wedge e_4, \qquad f_2 = e_1\wedge e_3-e_2\wedge e_4, \\
f_3 = e_1\wedge e_4+e_2\wedge e_3,
\end{gather*}
and a basis of $\Lambda^2_-V$ is easily deduced from this.
By Proposition \ref{rmk:lambda-so}, the images of the bivectors $f_1,f_2,f_3$ in $\mathfrak{so}(4)$ are given by the matrices
\[H = \left(\begin{matrix} 0 & -1 & 0 & 0 \\ 1 & 0 & 0 & 0 \\ 0 & 0 & 0 & -1 \\ 0 & 0 & 1 & 0 \end{matrix}\right), \quad X = \left(\begin{matrix} 0 & 0 & -1 & 0 \\ 0 & 0 & 0 & 1 \\ 1 & 0 & 0 & 0 \\ 0 & -1 & 0 & 0 \end{matrix}\right), \quad Y = \left(\begin{matrix} 0 & 0 & 0 & -1 \\ 0 & 0 & -1 & 0 \\ 0 & 1 & 0 & 0 \\ 1 & 0 & 0 & 0 \end{matrix}\right).\]
One computes $[H,X]=2Y$, $[H,Y]=-2X$, and $[X,Y]=2H$, so we have an equivalence of $\mathrm{SO}(4)$-representations $\Lambda^2_+V = \mathfrak{su}(2) \subset \mathfrak{so}(4)$, cf.\ Example \ref{ex:su2-roots}.
Analogously, $\Lambda^2_-V = \mathfrak{su}(2) \subset \mathfrak{so}(4)$ (the roles of $X$ and $Y$ above are switched).
We then have equivalent decompositions of $\mathrm{SO}(4)$-modules
\begin{align*}
\Lambda^2V & = \Lambda^2_+V \oplus \Lambda^2_-V, \\
\mathfrak{so}(4) & = \mathfrak{su}(2) \oplus \mathfrak{su}(2),
\end{align*}
where each copy of $\mathfrak{su}(2)$ is also an ideal in $\mathfrak{so}(4)$.
\end{example}
\begin{proposition}
The $\mathrm{O}(4)$-modules $\Lambda_+^2V \otimes \Lambda_-^2V$ and $S_0^2V$ are isomorphic. In particular, $\Lambda_+^2V \otimes \Lambda_-^2V$ is $\mathrm{O}(4)$-irreducible, and hence $\Lambda_+^2V$ and $\Lambda_-^2V$ cannot be equivalent as $\mathrm{O}(4)$-modules.
\end{proposition}
\begin{proof}
One way to prove this is to use an explicit basis, cf.\ \cite[Expos\'e XVI]{besse2}.
Let $e_1,\dots,e_n$ be the standard basis of $V$.
Write $e_{ij}$ for $e_i \wedge e_j$.
A basis for $\Lambda_{\pm}^2V$ is given by $\alpha_{\pm} \coloneqq e_{12}\pm e_{34}$, $\beta_{\pm} \coloneqq e_{13}\mp e_{24}$, $\gamma_{\pm} \coloneqq e_{14}\pm e_{23}$.
If $\omega_{\pm}$ are vectors in $\Lambda_{\pm}^2$, the linear map $\Lambda_+^2V \otimes \Lambda_-^2V \to S_0^2V$ defined by
\[\omega_+\otimes \omega_- \mapsto \sum_{i=1}^4 (e_i \lrcorner\ \omega_+) \odot (e_i \lrcorner\ \omega_-)\]
is an isomorphism of $\mathrm{O}(4)$-modules. 

If $\Lambda_+^2V$ and $\Lambda_-^2V$ were equivalent modules, then there would be an equivariant isomorphism $\Lambda_+^2V \to \Lambda_-^2V$ defining an invariant submodule in $\Lambda_+^2V \otimes \Lambda_-^2V$, contradiction.
\end{proof}
By Proposition \ref{prop:decomposition-son}, it follows that $\Lambda_+^2V \otimes \Lambda_-^2V$ is irreducible for $\mathrm{SO}(4)$ as well.
We then have the following special case.
\begin{proposition}
The space of endomorphisms $V \otimes V$ splits into irreducible $\mathrm{SO}(4)$-invariant summands as
\[V \otimes V = \Lambda_+^2V \oplus \Lambda_-^2V \oplus S_0^2V \oplus \mathbb R,\]
where $\Lambda_+^2V$ and $\Lambda_-^2V$ are the spaces of self-dual and anti self-dual two-forms on $V$ respectively.
\end{proposition}

\begin{remark}
Example \ref{ex:self-dual-forms} tells us $\mathfrak{spin}(4) = \mathfrak{so}(4) = \mathfrak{su}(2) \oplus \mathfrak{su}(2)$, which fits with the isomorphism $\mathrm{Spin}(4) = \mathrm{SU}(2) \times \mathrm{SU}(2)$ in Example \ref{ex:spin(4)}.
One sometimes writes $\mathrm{SO}(4) = \mathrm{SU}(2)_+\mathrm{SU}(2)_-$, meaning that the two copies of $\mathrm{SU}(2)$ correspond to self-dual and anti self-dual two-forms.
\end{remark}
\begin{digression}
\label{digression:so(4)}
Note that the equivalence of $\mathrm{SO}(4)$-modules $\Lambda_+^2\mathbb R^4=\mathfrak{su}(2)$ is \emph{not} an equivalence of $\mathrm{SU}(2)$-modules.
In fact, $\mathrm{SU}(2)$ acts trivially on $\Lambda_+^2\mathbb R^4$, whereas $\mathfrak{su}(2)$ is the (irreducible) adjoint representation.

The group $\mathrm{SU}(2)$ acts via left-multiplication on $\mathbb C^2$ preserving the standard Hermitian structure.
It follows that $\mathrm{SU}(2)$ acts as a subgroup of $\mathrm{SO}(4)$ on $\mathbb R^4$ preserving its standard inner product $g$ and a complex structure $J \colon \mathbb R^4 \to \mathbb R^4$.
The latter is a linear isometry such that $J^2=-\mathrm{id}$, so it plays the role of the imaginary unit.
We construct an adapted orthonormal basis $e_1$, $Je_1$, $e_2$, $Je_2$.
Then $e_1 \wedge Je_1 \wedge e_2 \wedge Je_2$ is a natural volume form on $\mathbb R^4$.
Now, essentially by definition, $\mathrm{SU}(2)$ acts on $\mathbb R^4$ and preserves the two-forms
\begin{align*}
\sigma & \coloneqq g(J{}\cdot{},{}\cdot{}) = e_1 \wedge Je_1+e_2\wedge Je_2, \\
\psi_{\mathbb C} & \coloneqq (e_1+iJe_1) \wedge (e_2+iJe_2).
\end{align*}
The latter is a complex volume form of type $(2,0)$ with respect to $J$ (we will discuss forms of this type in subsection \ref{subsec:unitary-groups}). Set
\begin{align*}
\psi_+ & \coloneqq \mathrm{Re}(\psi_{\mathbb C}) = e_1\wedge e_2-Je_1\wedge Je_2, \\
\psi_- & \coloneqq \mathrm{Im}(\psi_{\mathbb C}) = e_1 \wedge Je_2+Je_1 \wedge e_2.
\end{align*}
Then $\sigma$ and $\psi_{\pm}$ clearly correspond to $f_1,f_2,f_3$, and hence are a basis for $\Lambda^2_+\mathbb R^4$.
However, $\mathrm{SU}(2)$ acts trivially on each of them, so $\Lambda_+^2\mathbb R^4$ is a direct sum of three trivial $\mathrm{SU}(2)$-modules \[\Lambda_+^2\mathbb R^4=\mathbb R \sigma \oplus \mathbb R \psi_+ \oplus \mathbb R\psi_-,\] whereas the adjoint representation of $\mathrm{SU}(2)$ on $\mathfrak{su}(2)$ is irreducible (in Weyl's correspondence, it is associated to the dominant weight $+2$).
\end{digression}

\subsection{Unitary groups}
\label{subsec:unitary-groups}

Let $\mathbb C^n$ be equipped with the standard Hermitian scalar product $h$ (conjugate-linear in the first entry).
The group of transformations of $\mathbb C^n$ preserving the Hermitian structure is by definition the unitary group $\mathrm{U}(n)$.

The map $(z_1,\dots,z_n) \in \mathbb C^n \mapsto (\mathrm{Re}(z_1),\dots,\mathrm{Re}(z_n),\mathrm{Im}(z_1),\dots,\mathrm{Im}(z_n))$ identifies $\mathbb C^n$ with $\mathbb R^{2n}$.
The action of the imaginary unit $i$ on $\mathbb R^{2n}$ is then encoded into an endomorphism $J$ of $\mathbb R^{2n}$, which in terms of the standard basis is as in \eqref{eq:symplectic-matrix}.
\begin{definition}
The linear map $J \colon \mathbb R^{2n} \to \mathbb R^{2n}$ is a \emph{complex structure} on $\mathbb R^{2n}$, i.e.\ $J^2=-\mathrm{id}$.
\end{definition}
Let $g$ be the real part of the Hermitian inner product, which corresponds to a scalar product $g$ on $\mathbb R^{2n}$.
Note that $h(iz_1,iz_2)=h(z_1,z_2)$ and $h(iz_1,z_2)=-h(z_1,iz_2)$.
This interplay between the Hermitian structure and the action of the imaginary unit implies that $J$ preserves $g$, so $g(J{}\cdot{},J{}\cdot{}) = g$.
Note that $\sigma \coloneqq g(J{}\cdot{},{}\cdot{})$ is then a two-form, as 
\[\sigma(X,Y) = g(JX,Y) = g(J^2X,JY) = -g(JY,X) = -\sigma(Y,X),\]
and corresponds to the imaginary part of the Hermitian structure, so one writes $h=g+i\sigma$.
Since $J$ is non-degenerate, $\sigma$ is non-degenerate, meaning that $\sigma^n = \sigma \wedge \dots \wedge \sigma$ is a volume form.
Then $\mathbb R^{2n}$ comes with a natural orientation as well.
We lastly recall from Remark \ref{rmk:alternative-def-spn} that there is an isomorphism
\begin{equation*}
\mathrm{U}(n) = \mathrm{SO}(2n) \cap \mathrm{GL}(n,\mathbb C).
\end{equation*}
where $\mathrm{GL}(n,\mathbb C) \subset \mathrm{GL}(2n,\mathbb R)$ is given by matrices commuting with $J$. 
So again, when working on $\mathbb R^{2n}=[\![\mathbb C^n]\!]$, we can identify tensors of the same homogeneous degree.

We now look for the irreducible $\mathrm{U}(n)$-decompositions of the spaces of skew-symmetric forms $\Lambda^k\mathbb R^{2n}$.
In order to do this, we go through the complexification of $\Lambda^k\mathbb R^{2n}$. 
We then give the definition of complex forms of type $(p,q)$, and finally specialise our considerations to the reals.

Let $(V,g,J)$ be the Euclidean space $\mathbb R^{2n}$ as above, with complex structure $J$, and let $\mathrm{U}(n) \subset \mathrm{SO}(2n)$ act on it preserving $g$ and $J$.
Let $\alpha \in \Lambda^1 \coloneqq \Lambda^1 V$ be a one-form. If we view the latter as an element in the complexified space $\Lambda^1\otimes \mathbb C$, we can write
\[\alpha = \frac12 (\alpha-iJ\alpha)+\frac12 (\alpha+iJ\alpha),\]
where $(J\alpha)(x) \coloneqq \alpha(J^{-1}x) = -\alpha(Jx)$. Note that $J^2\alpha = -J\alpha(J{}\cdot{}) = -\alpha$.
Therefore, after extending the action of $J$ to $\Lambda^1 \otimes \mathbb C$ by complex linearity, we have
\begin{align*}
J(\alpha-iJ\alpha) & = J\alpha-iJ^2\alpha = J\alpha+i\alpha = +i(\alpha-iJ\alpha), \\
J(\alpha+iJ\alpha) & = J\alpha+iJ^2\alpha = J\alpha-i\alpha = -i(\alpha+iJ\alpha).
\end{align*}
So $J$ acts on $\Lambda^1 \otimes \mathbb C$ via multiplication by $\pm i$ on forms $\alpha \mp iJ\alpha$.
Define the spaces of complex $(1,0)$- and $(0,1)$-forms
\[\Lambda^{1,0} \coloneqq \{\alpha+iJ\alpha: \alpha \in \Lambda^1\}, \qquad \Lambda^{0,1} \coloneqq \{\alpha-iJ\alpha: \alpha \in \Lambda^1\},\]
so as to have the $J$-decomposition into $\pm i$-eigenspaces \[\Lambda^1 \otimes \mathbb C = \Lambda^{1,0} \oplus \Lambda^{0,1}.\]
Note that if $\beta \in \Lambda^{1,0}$, then there is a real form $\alpha$ such that $\beta = \alpha+iJ\alpha$, and hence $\overline \beta = \alpha-iJ\alpha \in \Lambda^{0,1}$.
So the correspondence $\beta \in \Lambda^{1,0} \to \overline \beta \in \Lambda^{0,1}$ is a linear isomorphism, provided that complex scalars act conjugated on $\Lambda^{0,1}$.
Therefore, we can write $\overline \Lambda{}^{0,1} = (\Lambda^{0,1})^* = \Lambda^{1,0}$, and the above identity can be written as
\begin{equation}
\label{eq:decomposition-lambda1}
\Lambda^1 \otimes \mathbb C = \Lambda^{1,0} \oplus \overline \Lambda{}^{1,0}.
\end{equation}
We lastly note that all spaces considered are unitary representations (as elements in $\mathrm{U}(n)$ commute with $J$), and thus \eqref{eq:decomposition-lambda1} is a $\mathrm{U}(n)$-decomposition.
It follows that \[\Lambda^1 = [\![\Lambda^{1,0}]\!],\] cf.\ subsection \ref{subsec:real-quaternionic-representations}.
Clearly, $\Lambda^1$ is $\mathrm{U}(n)$-irreducible.

We now look at forms of higher order. Define the spaces of $(p,0)$- and $(0,q)$-forms as
\[
\Lambda^{p,0} \coloneqq \Lambda^p(\Lambda^{1,0}), \qquad \Lambda^{0,q} \coloneqq \Lambda^q(\Lambda^{0,1}).
\]
Then $\Lambda^{p,q}$ is the space $\Lambda^{p,0} \otimes \Lambda^{0,q}$.
The action of $J$ can be carried over to each of these spaces by first extending the action of $J$ to the tensor algebra over $V$ and then imposing that it preserves the tensor product. 
We then have relative decompositions into eigenspaces. 
For instance, let $\alpha, \beta$ be complex one-forms.
\begin{enumerate}
\item If $\alpha,\beta \in \Lambda^{1,0}$ or $\alpha,\beta \in \Lambda^{0,1}$, then $J(\alpha \wedge \beta) = J\alpha \wedge J\beta = -\alpha \wedge \beta$;
\item If $\alpha \in \Lambda^{1,0}$ and $\beta \in \Lambda^{0,1}$, then $J(\alpha \wedge \beta) = J\alpha \wedge J\beta = +\alpha \wedge \beta$.
\end{enumerate}
So $J$ acts via multiplication by $-1$ on $\Lambda^{2,0} \oplus \Lambda^{0,2}$, and by $+1$ on $\Lambda^{1,1}$.
We then have the $J$-decomposition into eigenspaces \[\Lambda^2 \otimes \mathbb C = (\Lambda^{2,0} \oplus \Lambda^{0,2}) \oplus \Lambda^{1,1}.\]
Note that $\overline \Lambda{}^{1,1} = \Lambda^{1,1}$, so $\Lambda^{1,1}$ is self-conjugate, and $\overline \Lambda{}^{2,0}=\Lambda^{0,2}$.
Therefore, conjugation on $\Lambda^{1,1}$ is an equivariant conjugate-linear map that squares to the identity.
Following again subsection \ref{subsec:real-quaternionic-representations}, its $+1$-eigenspace $[\Lambda^{1,1}]$ is a real $\mathrm{U}(n)$-representation, and
\[[\Lambda^{1,1}] \otimes \mathbb C = \Lambda^{1,1}.\]
Similarly, by the above observations we have
\[[\![\Lambda^{2,0}]\!] \otimes \mathbb C = \Lambda^{2,0}\oplus \Lambda^{0,2},\]
so we can write
\[\Lambda^2 = [\![\Lambda^{2,0}]\!] \oplus [\Lambda^{1,1}].\]
\begin{example}
\label{ex:decomposition-lambda2}
In $\mathbb R^4$ (so $n=2$), consider the $J$-adapted basis given by $e_1$, $Je_1$, $e_2$, $Je_2$.
A basis for $\Lambda^{1,0}$ is given by $e_1+iJe_1$, $e_2+iJe_2$, and a basis of $\Lambda^{0,1}$ is obtained by conjugating the latter two vectors.
Then a basis for $\Lambda^{1,1}$ is 
\begin{alignat*}{2}
(e_1+iJe_1) & \wedge (e_1-iJe_1), && \qquad (e_1+iJe_1) \wedge (e_2-iJe_2), \\
(e_2+iJe_2) & \wedge (e_1-iJe_1), && \qquad (e_2+iJe_2) \wedge (e_2-iJe_2).
\end{alignat*}
Taking a complex linear combination $\alpha$ of these four vectors and imposing $\alpha = \overline{\alpha}$ forces $\alpha$ to be a real combination of
\[e_1\wedge Je_1, \qquad e_2\wedge Je_2, \qquad e_1\wedge e_2+Je_1\wedge Je_2, \qquad Je_1\wedge e_2-e_1\wedge Je_2.\]
So $[\Lambda^{1,1}]$ is the real four-dimensional vector space generated by these four vectors. Note that $J$ acts as $+\mathrm{id}$ on each of them.
Also, mapping
\begin{align*}
e_1\wedge Je_1 \mapsto e_1\wedge Je_1+e_2\wedge Je_2, \\
e_2\wedge Je_2 \mapsto e_1\wedge Je_1-e_2\wedge Je_2,
\end{align*}
gives a different basis of $[\Lambda^{1,1}]$ where now $\sigma$ corresponds to the first basis vector.
One writes \[[\Lambda^{1,1}] = [\Lambda_0^{1,1}] \oplus \mathbb R \sigma,\] where $[\Lambda_0^{1,1}]$ is the orthogonal complement of $\mathbb R \sigma$ inside $[\Lambda^{1,1}]$ (with respect to the scalar product in \eqref{eq:inner-product-exterior-algebra}). 
A basis for $\Lambda^{2,0}\oplus \Lambda^{0,2}$ is given by $\beta$ and $\overline \beta$, where \[\beta=e_1\wedge e_2-Je_1\wedge Je_2+i(e_1\wedge Je_2+Je_1\wedge e_2).\]
Then $[\![\Lambda^{2,0}]\!]$ is two-dimensional and generated by $e_1\wedge e_2-Je_1\wedge Je_2$, $e_1\wedge Je_2+Je_1\wedge e_2$.
It is clear that $J$ acts as $-\mathrm{id}$ on each of them.
\end{example}
Generalising Example \ref{ex:decomposition-lambda2} to any dimension, we have a $J$-decomposition
\[\Lambda^2 = [\![\Lambda^{2,0}]\!] \oplus [\Lambda^{1,1}_0] \oplus \mathbb R,\]
where $[\Lambda^{1,1}_0]$ is the orthogonal complement of $\mathbb R \sigma$ into $[\Lambda^{1,1}]$ with respect to the inner product \eqref{eq:inner-product-exterior-algebra}.
Another way to understand this decomposition is by means of Lie algebras.
By Proposition \ref{rmk:lambda-so}, we have $\Lambda^2 V = \mathfrak{so}(2n) \supset \mathfrak{u}(n) = \mathfrak{su}(n) \oplus \mathbb R$.
We can write $\mathfrak{so}(2n)=\mathfrak{u}(n)^{\perp} \oplus \mathfrak{u}(n) = \mathfrak{u}(n)^{\perp} \oplus \mathfrak{su}(n) \oplus \mathbb R$, where $\mathfrak{u}(n)^{\perp}$ is the orthogonal complement of $\mathfrak{u}(n)$ inside $\mathfrak{so}(2n)$ with respect to the Killing form. 
So we have
\[\mathfrak{so}(2n) = \mathfrak{u}(n)^{\perp} \oplus \mathfrak{su}(n) \oplus \mathbb R.\]
But matrices in the Lie algebra $\mathfrak{u}(n) \subset \mathfrak{so}(2n)$ act on $\mathbb R^{2n}$ and commute with $J$, so $\mathfrak{u}(n)$ corresponds to $[\Lambda^{1,1}]$ and hence $\mathfrak{u}(n)^{\perp}$ corresponds to $[\![\Lambda^{2,0}]\!]$. 
Furthermore, $\sigma \in \Lambda^2$ corresponds to $J \in \mathfrak{so}(2n)$, so the two orthogonal decompositions 
\begin{align*}
\Lambda^2\mathbb R^{2n} & = [\![\Lambda^{2,0}]\!] \oplus [\Lambda^{1,1}_0] \oplus \mathbb R, \\
\mathfrak{so}(2n) & = \mathfrak{u}(n)^{\perp} \oplus \mathfrak{su}(n) \oplus \mathbb R, 
\end{align*}
are equivalent. 

In general, if $p \neq q$ we have \[[\![\Lambda^{p,q}]\!] \otimes \mathbb C = \Lambda^{p,q} \oplus \Lambda^{q,p}, \qquad [\Lambda^{p,p}] \otimes \mathbb C = \Lambda^{p,p}.\]
As regards the spaces $\Lambda^p \mathbb R^{2n}$, complexifying and decomposing with respect to $J$ yields
\[\Lambda^p\mathbb R^{2n} \otimes \mathbb C = \Lambda^{p,0} \oplus \Lambda^{p-1,1} \oplus \dots \oplus \Lambda^{1,p-1} \oplus \Lambda^{0,p}.\]
For low $p$, the same ideas above lead to the decompositions
\begin{align*}
\Lambda^3 & = [\![\Lambda^{3,0}]\!] \oplus [\![\Lambda^{2,1}]\!], \\
\Lambda^4 & = [\![\Lambda^{4,0}]\!] \oplus [\![\Lambda^{3,1}]\!] \oplus [\Lambda^{2,2}], \\
\Lambda^5 & = [\![\Lambda^{5,0}]\!] \oplus [\![\Lambda^{4,1}]\!] \oplus [\![\Lambda^{3,2}]\!], 
\end{align*}
and so on.
We further note that wedging with $\sigma$ gives a map \[\wedge \sigma \colon [\![\Lambda^{p,q}]\!] \to [\![\Lambda^{p+1,q+1}]\!]\] which is injective for $n>p+q$ and $\mathrm{U}(n)$-equivariant (cf.\ e.g.\ Huybrechts \cite[Chapter 1]{huybrechts} or Weil \cite{weil}),
so each copy of $[\![\Lambda^{p,q}]\!]$ or $[\Lambda^{p,p}]$ with either $p\geq 1$ or $q\geq 1$ contains the image of $[\![\Lambda^{p-1,q-1}]\!]$ under this map.
Let us write \[[\![\Lambda^{p,q}]\!] = [\![\Lambda_0^{p,q}]\!] \oplus [\![\Lambda^{p-1,q-1}]\!], \qquad [\Lambda^{p,p}]=[\Lambda_0^{p,p}] \oplus [\Lambda^{p-1,p-1}].\] 
One then gets the following decompositions by applying the above process multiple times.
\begin{proposition}[Salamon \cite{salamon}]
\label{thm:u(n)-decompositions}
Let $\Lambda^k$ be the space of $k$-forms on the Hermitian vector space $V=(\mathbb R^{2n},g,J)$.
Let the group $\mathrm{U}(n)$ act on $V$ preserving the Hermitian structure.
The following are $\mathrm{U}(n)$-decompositions into irreducible summands:
\begin{align*}
\Lambda^1 & = [\![\Lambda^{1,0}]\!], \\
\Lambda^2 & = [\![\Lambda^{2,0}]\!] \oplus [\Lambda_0^{1,1}] \oplus \mathbb R, \\
\Lambda^3 & = [\![\Lambda^{3,0}]\!] \oplus [\![\Lambda_0^{2,1}]\!] \oplus \Lambda^1, \\
\Lambda^4 & = [\![\Lambda^{4,0}]\!] \oplus [\![\Lambda_0^{3,1}]\!] \oplus [\Lambda_0^{2,2}] \oplus \Lambda^2, \\
\Lambda^5 & = [\![\Lambda^{5,0}]\!] \oplus [\![\Lambda_0^{4,1}]\!] \oplus [\![\Lambda_0^{3,2}]\!] \oplus \Lambda^3.
\end{align*}
The subspaces above are non-zero only when $n \geq 5$. 
\end{proposition}
Irreducible decompositions for higher degree forms are obtained in an analogous manner.
The irreducibility of these decompositions is claimed in Weil \cite{weil} and Salamon \cite[Chapter 3]{salamon}, but essentially follows by the representation theory of the unitary group.
We recall that $\mathrm{U}(n)$ is not semisimple, so it does not immediately fit into our discussion in Section \ref{sec:representation-theory}.
We refer to \v{Z}elobenko \cite{zelobenko} for more details on this case.

\begin{remark}
\label{rmk:u(2)-dec}
Note that the $\mathrm{U}(2)$-decomposition 
\[\Lambda^2\mathbb R^4=[\![\Lambda^{2,0}]\!] \oplus [\Lambda_0^{1,1}] \oplus \mathbb R\]
in Example \ref{ex:decomposition-lambda2} can be refined for $\mathrm{SU}(2)$. 
In fact, $\mathrm{SU}(2)$ is exactly the group acting on $\Lambda^2 \mathbb R^4$ preserving the two generators of $[\![\Lambda^{2,0}]\!]$, and hence the corresponding $\mathrm{SU}(2)$-decomposition is
\[\Lambda^2\mathbb R^4=\mathbb R \oplus \mathbb R \oplus \mathbb R \oplus [\Lambda_0^{1,1}].\]
This complements the content of Digression \ref{digression:so(4)}.
\end{remark}

\begin{digression}
\label{digression:su3}
In dimension $6$, $\psi_{\pm}$ as in Digression \ref{digression:so(4)} are three-forms, so we have different decompositions of $\Lambda^2$ and $\Lambda^3$.
The group $\mathrm{SU}(3)$ acts on $\mathbb C^3$ in the standard way. 
As usual, we identify $\mathbb C^3$ with $\mathbb R^6$ equipped with the standard Hermitian structure, or equivalently the standard scalar product $g$ and complex structure $J$.
Take the $J$-adapted orthonormal basis $e_1,Je_1,e_2,Je_2,e_3,Je_3$. 
Then $\mathrm{SU}(3)$ preserves the forms
\begin{align*}
\sigma & \coloneqq g(J{}\cdot{},{}\cdot{}) = e_1\wedge Je_1+e_2\wedge Je_2+e_3\wedge Je_3, \\
\psi_+ & \coloneqq e_1 \wedge e_2 \wedge e_3-Je_1\wedge Je_2\wedge e_3-e_1\wedge Je_2\wedge Je_3-Je_1\wedge e_2\wedge Je_3, \\
\psi_- & \coloneqq e_1\wedge e_2\wedge Je_3-Je_1\wedge Je_2\wedge Je_3+e_1\wedge Je_2\wedge e_3+Je_1\wedge e_2\wedge e_3. 
\end{align*}
The last two forms are respectively real and imaginary part of the complex volume form $\psi_{\mathbb C} = \psi_++i\psi_-$, which has type $(3,0)$ with respect to $J$.  
This implies in particular that $[\![\Lambda^{3,0}]\!]$ splits as $[\![\Lambda^{3,0}]\!] = \mathbb R \psi_+ \oplus \mathbb R \psi_-$.
Combining this and the content of Remark \ref{rmk:u(2)-dec}, one deduces $[\![\Lambda^{n,0}]\!] = \mathbb R \oplus \mathbb R$ is an irreducible decomposition for $\mathrm{SU}(n)$.
The remaining decompositions in Proposition \ref{thm:u(n)-decompositions} are unchanged.
\end{digression}

\subsection{On \texorpdfstring{$\mathrm G_2$}{G2} and \texorpdfstring{$\mathrm{Spin}(7)$}{Spin7}}
\label{subsec:on-g2-spin7}

We now introduce the groups $\mathrm G_2$ and $\mathrm{Spin}(7)$, which model geometries in dimension $7$ and $8$ respectively.
The material presented in this section is extrapolated from Bryant \cite{bryant}.
For a comprehensive overview of $\mathrm G_2$ and $\mathrm{Spin}(7)$ see e.g.\ the notes of Karigiannis \cite{karigiannis}.

Let $V=\mathbb R^7$. Take any basis $e_1,\dots,e_7$, and let $e^1,\dots,e^7$ be the dual basis.
For simplicity, let us write $e^{ijk} = e^i \wedge e^j \wedge e^k \in \Lambda^3 V^*$ (a similar convention will be used for forms of different degrees).
Define the three-form
\begin{equation}
\label{eq:g2-3-form}
\varphi \coloneqq e^{123}+e^{145}+e^{167}+e^{246}-e^{257}-e^{347}-e^{356}.
\end{equation}
This form encodes one possible multiplication table for the imaginary octonions in the following sense.
The vector space of octonions $\mathbb O$ is isomorphic to $\mathbb R^8$.
Write $e_0,e_1,\dots,e_7$ for a basis of $\mathbb O$, and let $\mathbb O = \mathbb Re_0 \oplus V$ be the splitting into real and imaginary octonions.
By definition of $\varphi$, $e_1e_2=e_3$, $e_2e_3=e_1$, and so on. 
Skew-symmetry of $\varphi$ tells us each $e_k^2$ is real (in fact, $e_0^2=1$ and $e_k^2=-1$ for $k=1,\dots,7$), and that $e_ie_j=-e_je_i$ for $i\neq j$ in $\{1,\dots,7\}$.
Associativity fails as, for instance, $e_1(e_2e_5)=-e_1e_7=+e_6$, but $(e_1e_2)e_5=e_3e_5=-e_6$.
The same algebraic structure can be encoded in a cross product, which we will see below.

Let us define 
\begin{equation*}
\mathrm G_2 \coloneqq \{A \in \mathrm{Aut}(V): A\varphi = \varphi\}.
\end{equation*}
It is a result of Schouten \cite{schouten} that this group is of \lq\lq type $\mathrm G_2$\rq\rq, in the sense that its Lie algebra is isomorphic to the Lie algebra $\mathfrak g_2$ (cf.\ the classification of simple Lie algebras \cite{fulton-harris}, or the original paper by Killing \cite{killing0} on the classification of simple transformation groups).
\begin{theorem}[Bryant \cite{bryant}]
\label{thm:g2}
The group $\mathrm G_2$ is compact, connected, simply connected, simple, centreless, and has dimension $14$.
Also, $V$ is an irreducible representation of $\mathrm G_2$, and $\mathrm G_2$ acts transitively on the spaces of lines in $V$ and two-planes in $V$.
Finally, the group $\mathrm G_2$ is isomorphic to the group of algebra automorphisms of the octonions.
\end{theorem}
\begin{proof}
Let us consider the map $b \colon V \times V \to \Lambda^7V^*$ defined by 
\[6b(x,y)\coloneqq  (x\lrcorner\ \varphi) \wedge (y\lrcorner\ \varphi) \wedge \varphi.\]
We claim that the map $b$ is $\mathrm G_2$-equivariant. If $A \in \mathrm G_2$, we have
\begin{align*}
6b(Ax,Ay) & = (Ax\lrcorner\ \varphi) \wedge (Ay\lrcorner\ \varphi) \wedge \varphi \\
& = \varphi(x,A^{-1}{}\cdot{},A^{-1}{}\cdot{}) \wedge \varphi(y,A^{-1}{}\cdot{},A^{-1}{}\cdot{}) \wedge A\varphi \\
& = A(x\lrcorner\ \varphi) \wedge A(y\lrcorner\ \varphi) \wedge A\varphi \\
& = 6Ab(x,y).
\end{align*}
Set $x=\sum_{i=1}^7 x^ie_i$ and $y=\sum_{i=1}^7 y^ie_i$. Then one computes 
\[b(x,y) = \left(\sum_{i=1}^7 x^iy^i\right) e^{12\dots 7}.\]
Define $g(x,y) \coloneqq \sum_{i=1}^7 x^iy^i$. For $A \in \mathrm G_2$ we compute
\[g(Ax,Ay)e^{12\dots 7} = b(Ax,Ay)=Ab(x,y)=g(x,y)(\det A)^{-1}e^{12\dots 7}.\]
We deduce $g(Ax,Ay) = (\det A)^{-1}g(x,y)$.

On the other hand, we have the formula (recall \eqref{eq:inner-product-exterior-algebra})
\[\det ([g(x_i,x_j)]_{ij}) = (e^{12\dots 7}(x_1\wedge \dots \wedge x_7))^2, \qquad x_i \in V.\]
If we set $x_i=A^{-1}e_i$, then by the above equivariance law for $g$ we compute
\begin{align*}
(\det A)^7 & = \det ((\det A)[g(e_i,e_j)]_{ij}) \\
& = \det ([g(A^{-1}e_i,A^{-1}e_j)]_{ij}) \\
& = (e^{12\dots 7}(A^{-1}e_1\wedge \dots \wedge A^{-1}e_7))^2 = (\det A)^{-2}.
\end{align*}
This forces $(\det A)^9=1$, so $\det A=1$. It follows that any element $A \in \mathrm G_2$ preserves the scalar product $g$ and $\det A=1$, so $\mathrm G_2 \subset \mathrm{SO}(7)$.
Since $\mathrm G_2 \subset \mathrm{SO}(7)$ is closed, it is also compact.
It can be shown that $\mathrm G_2$ contains a subgroup generated by $\mathrm{SO}(4)$ and $\mathrm{SU}(3) \subset \mathrm{SO}(7)$, from which it follows that $\mathrm G_2$ acts transitively on $S^6 \subset V$ (see e.g.\ \cite{bryant-salamon, montgomery-samelson}).
This implies $\mathrm G_2$ acts irreducibly on $V$.

Combining $\varphi$ and $g$ we get a map $\times \colon V \times V \to V$ such that 
\[g(x\times y,z) \coloneqq \varphi(x,y,z).\]
Bilinearity and skew-symmetry of $\varphi(x,y,z)$ in $x$ and $y$ imply that $\times$ is skew-symmetric and bilinear.
Invariance of $\varphi$ implies equivariance of $\times$, i.e.\ $Ax \times Ay = A(x \times y)$ for $A \in \mathrm G_2$.
The map $\times$ is then a cross product on $V$.

Note that $e_1 \times e_2 = e_3$. Let $H \subset \mathrm G_2$ be the stabiliser of $e_1$ and $e_2$.
Then $He_3=H(e_1\times e_2) = He_1 \times He_2=e_1\times e_2=e_3$, so $H$ fixes $e_3$ as well.
If we write $\varphi$ as
\[\varphi = e^{123}+e^1(e^{45}+e^{67})+e^2(e^{46}-e^{57})-e^3(e^{47}+e^{56}),\]
we see that $H$ must act on $\mathbb R^4 = \mathrm{Span}\{e^4,\dots,e^7\}$ preserving the two-forms $e^{45}+e^{67}$, $e^{46}-e^{57}$, $e^{47}+e^{56}$.
By Digression \ref{digression:so(4)}, it follows that $H = \mathrm{SU}(2)$.

Let $\mathcal V_2(\mathbb R^7)$ be the Stiefel manifold of orthonormal $2$-frames in $\mathbb R^7$, or equivalently the space of $2\times 7$ matrices $A$ with $A^TA=\mathrm{id}$.
This is a compact submanifold of $\mathbb R^{14}$ and has dimension $11$.
We have a map $\nu \colon \mathrm G_2 \to \mathcal{V}(\mathbb R^7)$ given by $A \mapsto (Ae_1,Ae_2)$ whose kernel is $H=\mathrm{SU}(2)$.
The fibres of $\nu$ are left cosets of $H$, so $\dim \mathrm G_2 \leq 11+3=14$, with equality if and only if $\nu(\mathrm G_2)$ is open in $\mathcal V_2(\mathbb R^7)$ (i.e.\ $\nu$ is surjective, by compactness of $\mathrm G_2$).
On the other hand, $\dim \Lambda^3 V^*=35$ and $\mathrm{Aut}(V)/\mathrm G_2$ naturally embeds in $\Lambda^3 V^*$ as $\mathrm{GL}(V)\varphi$.
Then $\dim(\mathrm{GL}(V)/\mathrm G_2) \leq 35$, so $49-\dim \mathrm G_2 \leq 35$, so $\dim \mathrm G_2 \geq 14$.
It follows that $\dim \mathrm G_2=14$ and $\nu$ is surjective, i.e.\ $\mathrm G_2$ acts transitively on $2$-frames in $\mathbb R^7$ with generic stabiliser isomorphic to $\mathrm{SU}(2)$.
We then have the homogeneous description $\mathcal V_2(\mathbb R^7) = \mathrm G_2/\mathrm{SU}(2)$ (cf.\ Proposition \ref{prop:homogeneous-spaces}).

If is known that $\mathcal V_2(\mathbb R^7)$ is an $S^5$-bundle over $S^6$, so by the long exact homotopy sequence for the fibration
\[1 \longrightarrow S^5 \longrightarrow \mathcal V_2(\mathbb R^7) \longrightarrow S^6 \longrightarrow 1,\]
one deduces $\pi_i(\mathcal V_2(\mathbb R^7))$ is trivial for $i=0,\dots,4$.
Now, we also have the fibration
\[1 \longrightarrow \mathrm{SU}(2) \longrightarrow \mathrm G_2 \longrightarrow \mathcal V_2(\mathbb R^7) \longrightarrow 1,\]
and we then deduce that $\pi_i(\mathrm G_2)$ is trivial for $i=0,1,2$, whereas $\pi_3(\mathrm G_2) = \mathbb Z$.
Therefore, $\mathrm G_2$ is connected, simply connected, and compact.
By Cartan's criterion \cite[Chapter V, Remark 7.13]{brocker-tomdieck}, $\mathrm G_2$ is semisimple, and its centre $Z(\mathrm G_2)$ is finite.
This implies in particular that $\mathfrak g_2$ is semisimple, i.e.\ it is the direct sum of simple Lie algebras.
The number of its simple summands corresponds to the rank of $\pi_3(\mathrm G_2)=\mathbb Z$, so $\mathfrak g_2$ is simple, and hence $\mathrm G_2$ is simple.

We show that $\mathrm G_2$ has trivial centre.
The characteristic polynomial of any non-trivial element $\lambda \in Z(\mathrm G_2)$ is of degree $7$, so it admits a real root.
Also, if $k = |Z(\mathrm G_2)|$, then $\lambda^k=\mathrm{id}$, so $\pm 1$ are the only possible real eigenvalues of $\lambda$.
Since $\lambda$ is central in $\mathrm G_2$, the map $\lambda \colon V \to V$ is $\mathrm G_2$-equivariant.
But $\mathrm G_2$ acts irreducibly on $V$, so Schur's Lemma, Theorem \ref{thm:schur-lemma}, and Remark \ref{rmk:schur-lemma} imply that $\lambda$ acts as $\pm \mathrm{id}$, whence $Z(\mathrm G_2) \subset \{\pm \mathrm{id}\}$.
But $-\mathrm{id} \not \in \mathrm{SO}(7)$, so $Z(\mathrm G_2)$ is trivial.

We now show two formulas for any $x,y \in V$:
\begin{enumerate}
\item $\lVert x \times y \rVert^2 = \lVert x\rVert^2 \lVert y\rVert^2-g(x,y)^2$, 
\item $x \times (x \times y) = -\lVert x\rVert^2 y + g(x,y)x$.
\end{enumerate}
Since $\mathrm G_2$ acts transitively on orthonormal pairs in $V$, we can restrict to taking $x=x^1e_1$ and $y=y^1e_1+y^2e_2$.
In this case, the two identities are trivial to verify.
On $O \coloneqq \mathbb R e_0\oplus V$, define a $\mathrm G_2$-equivariant multiplication and inner product by
\begin{enumerate}
\item $(x^0e_0+x)(y^0e_0+y) = (x^0y^0-g(x,y))e_0+(x^0y+y^0x+x \times y),$
\item $(x^0e_0+x,y^0e_0+y) = x^0y^0+g(x,y)$.
\end{enumerate}
By the above formulas, $(zw,zw)=(z,z)(w,w)$ and $({}\cdot{},{}\cdot{})$ is positive-definite. 
Then $O$ is isomorphic to the octonions $\mathbb O$ and $\mathrm G_2$ is its automorphism group, so we are done.
\end{proof}
\begin{remark}
The latter characterisation of $\mathrm G_2$ as the automorphism group of the octonions was already known to Cartan \cite{cartan2}.
\end{remark}
\begin{corollary}
The $\mathrm{GL}(V)$-orbit of $\varphi$ in $\Lambda^3V^*$ is open and diffeomorphic to $\mathrm{GL}(V)/\mathrm G_2$.
\end{corollary}
We write $\Lambda_+^3V^*$ for such an orbit, and call it the set of \emph{positive}, or \emph{definite}, or \emph{stable} forms.
The word \lq\lq stable\rq\rq\ is used more generally for forms lying in an open orbit.
Stable forms in dimension $6$, $7$, and $8$ were investigated by Hitchin \cite{hitchin}.

Since $\mathrm{GL}(V)$ has two connected components and $\mathrm G_2$ is connected, $\Lambda_+^3V^*$ has two connected components. 
Each three-form in the orbit of $\varphi$ defines uniquely a Hodge star operator $\star_{\varphi}$ which is $\mathrm G_2$-equivariant, so we also have a four-form $\star_{\varphi} \varphi$.
In this context, we simply write $\star$ for $\star_{\varphi}$.

\begin{figure}
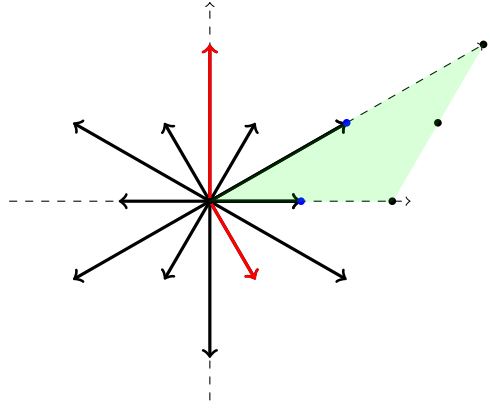

\centering
  \tikzpicture
  [scale=1.2]
  \coordinate (centre) at (0,0);
  \coordinate (little-hex11) at (0.5,0.8660254038);
  \coordinate (little-hex12) at (-0.5,-0.8660254038);
  \coordinate (little-hex21) at (-0.5,0.8660254038);
  \coordinate (little-hex22) at (0.5,-0.8660254038);
  \coordinate (little-hex31) at (-1,0);
  \coordinate (little-hex32) at (1,0);
  \coordinate (big-hex11) at (1.5,0.8660254038);
  \coordinate (big-hex12) at (-1.5,-0.8660254038);
  \coordinate (big-hex21) at (1.5,-0.8660254038);
  \coordinate (big-hex22) at (-1.5,0.8660254038);
  \coordinate (big-hex31) at (0,-1.7320508076);
  \coordinate (big-hex32) at (0,1.7320508076);
  \coordinate (vert-axis-north) at (0,2.2);
  \coordinate (vert-axis-south) at (0,-2.2);
  \coordinate (hor-axis-left) at (-2.2,0);
  \coordinate (hor-axis-right) at (2.2,0);
  \coordinate (end-line) at (3,1.7320508076);
  \coordinate (weight20) at (2,0);
  \coordinate (weight02) at (3,1.7320508076);
  \coordinate (weight11) at (2.5,0.8660254038);
  \draw[<->][very thick][black] (little-hex11) -- (little-hex12);
  \draw[<->][very thick][black] (little-hex21) -- (little-hex22);
  \draw[<->][very thick][black] (little-hex31) -- (little-hex32);
  \draw[<->][very thick][black] (big-hex11) -- (big-hex12);
  \draw[<->][very thick][black] (big-hex21) -- (big-hex22);
  \draw[<->][very thick][black] (big-hex31) -- (big-hex32);
  \draw[->][dashed][black] (vert-axis-south) -- (vert-axis-north);
  \draw[->][dashed][black] (hor-axis-left) -- (hor-axis-right);
  \draw[->][very thick][red] (centre) -- (big-hex32);
  \draw[->][very thick][red] (centre) -- (little-hex22);
  \draw[->][very thick][black] (centre) -- (big-hex11);
  \draw[->][very thick][black] (centre) -- (little-hex32);
  \draw[->][dashed][black] (big-hex11) -- (end-line);
  \fill[blue] (little-hex32) circle [radius=1.2pt];
  \fill[blue] (big-hex11) circle [radius=1.2pt];
  \fill (weight20) circle [radius=1.2pt];
  \fill (weight02) circle [radius=1.2pt];
  \fill (weight11) circle [radius=1.2pt];
  \fill (centre) circle [radius=1pt];
   \draw [fill, opacity=.15, green] (centre) -- (weight20) -- (end-line) -- cycle;
  \endtikzpicture
  \caption{Roots and weights of $\mathrm G_2$. The red roots correspond to a choice of positive simple roots.
  The black dots correspond to some dominant weights, the blue ones are the fundamental weights.
  The shaded green region is the fundamental dual Weyl chamber for the positive simple roots chosen.}
  \label{fig:roots-g2}
\end{figure}

We include here the decomposition of the spaces $\Lambda^k \coloneqq \Lambda^kV^*$ with respect to the action of $\mathrm G_2$.
We write $\Lambda_h^k$ for an irreducible $\mathrm G_2$-invariant submodule of $\Lambda^k$ of dimension $h$.
Since $\Lambda^k$ behaves like $\Lambda^{7-k}$, we give the decompositions only for $k < 7-k$.
\begin{proposition}
\label{prop:g2-decomposition-forms}
We have the following decompositions into irreducible $\mathrm G_2$-invariant modules:
\begin{align*}
\Lambda^2 & = \Lambda_7^2 \oplus \Lambda_{14}^2, \\
\Lambda^3 & = \Lambda_1^3 \oplus \Lambda_7^3 \oplus \Lambda_{27}^3.
\end{align*}
We have isomorphisms $\Lambda_7^2 = V^* = \Lambda_7^3$, $\Lambda_{14}^2 = \mathfrak g_2^{\flat}$, $\Lambda_1^3 = \mathbb R \varphi$, and $\Lambda_{27}^3 = S_0^2V^*$.
\end{proposition}
\begin{proof}
The rank of $\mathrm G_2$ is $2$, so its weights can be viewed as elements in $\mathbb R^2$, Figure \ref{fig:roots-g2}.
We know that $\mathrm G_2$ acts irreducibly on $V = \mathbb R^7$ (it corresponds to dominant weight $(1,0)$, where coordinates are with respect to the fundamental weights), so it acts irreducibly on $\Lambda^1 \simeq \Lambda^6$.
We observe that the adjoint representation $\mathfrak g_2$ of $\mathrm G_2$ is irreducible (it corresponds to dominant weight $(0,1)$).
Recall that $\Lambda^2=\mathfrak{so}(7)$, cf.\ Proposition \ref{rmk:lambda-so}.
Since $\mathfrak g_2 \subset \mathfrak{so}(7)$, and contracting $\varphi$ with any vector in $V$ gives a $\mathrm G_2$-invariant two-form, we have a decomposition
\[\Lambda^2 = \mathfrak g_2^{\flat} \oplus V^*.\]
Note that the linear map $\alpha \mapsto \star (\varphi \wedge \alpha)$ maps $\Lambda^2$ into itself and is $\mathrm G_2$-equivariant.
A computation shows that its eigenvalues are $+2$ and $-1$, and we have identifications
\begin{align*}
\Lambda_7^2 & = \{\alpha \in \Lambda^2: \star (\varphi \wedge \alpha) = +2\alpha\}=\{\varphi(X,{}\cdot{},{}\cdot{}): X \in V\}, \\
\Lambda_{14}^2 & = \{\alpha \in \Lambda^2: \star(\varphi \wedge \alpha)=-\alpha\}=\mathfrak g_2^{\flat}.
\end{align*}
The corresponding decomposition of $\Lambda^5$ then follows.

Next, consider $\Lambda^3$.
We have an invariant trivial one-dimensional submodule generated by $\varphi$.
The map $\mathrm{Aut}(V) \to \Lambda^3$ given by $A \mapsto A\varphi$ is open.
It follows that the linear map $V \otimes V^* \to \Lambda^3$ defined by $e \otimes w \mapsto w \wedge (e \lrcorner\ \varphi)$ is surjective with kernel $\mathfrak g_2 \subset V \otimes V^*$.
Now, $V \otimes V^*$ is isomorphic to $V^* \otimes V^*$, and hence $\Lambda^3 \oplus \mathfrak g_2 = \Lambda_7^2 \oplus \mathfrak g_2 \oplus S^2V^*$.
It follows that $\Lambda^3 = \Lambda_7^2 \oplus S^2V^*$.
Now $S^2V^*= S_0^2V^* \oplus \mathbb R$ as $\mathrm G_2 \subset \mathrm{SO}(7)$, cf.\ Proposition \ref{prop:decomposition-son}.
The $\mathrm G_2$-module $S_0^2V^*$ has dimension $27$, and must contain the irreducible $\mathrm G_2$-module with dominant weight $(2,0)$.
Freudenthal's formula (see e.g.\ \cite{humphreys}) gives that the latter also has dimension $27$. 
Thus $S_0^2V^*$ is irreducible for $\mathrm G_2$ as well, and we have the decomposition into irreducible $\mathrm G_2$-invariant summands
\begin{equation*}
\Lambda^3 = \mathbb R \oplus V^* \oplus S_0^2V^* \eqqcolon \Lambda_1^3 \oplus \Lambda_7^3 \oplus \Lambda_{27}^3.
\end{equation*}
This decomposition is also explained in more detail in \cite{bryant1}.
We have the following characterisations:
\begin{align*}
\Lambda_7^3 & = \{\star (\varphi\wedge \alpha): \alpha \in V^*\}, \\
\Lambda_{27}^3 & = \{\alpha \in \Lambda^3 V^*: \alpha \wedge \varphi=0, \alpha \wedge \star \varphi=0\}.
\end{align*}
The decomposition of $\Lambda^4$ follows. We refer to \cite{bryant1} for further details.
\end{proof}
We now get to the group $\mathrm{Spin}(7)$. 
As any other spin group, $\mathrm{Spin}(7)$ can be realized as a subgroup of a Clifford algebra \cite{lawson-michelsohn}, but we choose a more ad hoc approach here coming from \cite{bryant}.
Let $V=\mathbb R^7$ with basis $e_1,\dots,e_7$ as above. 
Set $V_+\coloneqq \mathbb R e_0 \oplus V$, and $e^i$, $i=0,\dots,7$, be the dual basis of $V_+$.
Note the identity
\begin{equation}
\label{eq:identity-four-forms}
\Lambda^4V_+ = \bigoplus_{k=0}^4 \Lambda^k \mathbb R \otimes \Lambda^{4-k}V = \Lambda^3V \oplus \Lambda^4V.
\end{equation}
Accordingly, define a four-form $\Phi \in \Lambda^4(V_+)^*$ by 
\[\Phi=e^0 \wedge \varphi+\star \varphi.\]
Since $\varphi \wedge \star \varphi = 7e^{12\dots 7}$ and $\star \varphi \wedge \star \varphi=0$, one computes $\Phi^2 = 14e^{012\dots 7}$.
Define the forms
\begin{align}
\label{eq:alpha}
\alpha & = e^{01}+e^{23}+e^{45}+e^{67}, \\
\label{eq:beta}
\beta & = (e^0+ie^1)\wedge (e^2+ie^3) \wedge (e^4+ie^5)\wedge (e^6+ie^7),
\end{align}
and note that
\begin{equation}
\label{eq:Phi}
\Phi = \frac12 \alpha \wedge \alpha+\mathrm{Re}(\beta).
\end{equation}
Define the group $\mathrm{Spin}(7)$ as
\begin{equation*}
\mathrm{Spin}(7) \coloneqq \{A \in \mathrm{Aut}(V_+): A\Phi=\Phi\}.
\end{equation*}
Extend the scalar product on $V$ to a scalar product $\langle{}\cdot{},{}\cdot{}\rangle$ on $V_+$ so that $e_0,\dots,e_7$ form an orthonormal basis.
\begin{theorem}
\label{thm:spin7}
The group $\mathrm{Spin}(7)$ is compact, connected, simply connected, and has dimension $21$.
Also, $\mathrm{Spin}(7)$ acts irreducibly on $V_+$, and transitively on the space of $k$-planes in $V_+$ for $k\neq 4$.
Lastly, $\mathrm{Spin}(7)$ preserves the scalar product on $V_+$, its centre is $\mathbb Z_2$, and $\mathrm{Spin}(7)/\mathbb Z_2=\mathrm{SO}(7)$.
\end{theorem}
\begin{remark}
The latter statement tells us $\mathrm{Spin}(7)$ defined above really is the universal double cover of $\mathrm{SO}(7)$.
\end{remark}
\begin{proof}
We have seen that $\Phi^2$ is a volume form on $V_+$.
Since $\mathrm{Spin}(7)$ preserves $\Phi$, it also preserves $\Phi^2$, so $\mathrm{Spin}(7) \subset \mathrm{SL}(V_+)$.
Let $H$ be the subgroup of $\mathrm{Aut}(V_+)$ preserving $\alpha$ and $\beta$ as in \eqref{eq:alpha}--\eqref{eq:beta}.
By identity \eqref{eq:Phi}, $H$ is a subgroup of $\mathrm{Spin}(7)$.
There is a unique complex structure $J \in \mathrm{End}(V_+)$ such that the forms $e^k+ie^{k+1}$, $k=0,2,4,6$, form a basis for $\Lambda^{1,0}\subset \Lambda^1V_+ \otimes \mathbb C$, and $J$ satisfies the formula $\alpha(x,y)=\langle Jx,y\rangle$.
So $H$ preserves the Hermitian inner product determined by $\alpha$ and $\langle{}\cdot{},{}\cdot{}\rangle$, and the complex volume form $\beta$.
It follows that $H=\mathrm{SU}(4) = \mathrm{Spin}(6)$, so $H$ acts transitively on the space of oriented lines in $V_+$.

Let $G \subset \mathrm{Spin}(7)$ be the subgroup preserving the line $\mathbb R e_0 \subset V_+$.
We claim $G=\mathrm G_2$.
We see that $G$ preserves $\Phi$, $G$ preserves $V$ (the orthogonal complement of $\mathbb R e_0$), and any $g \in G$ satisfies $g^*e^0=\lambda e^0+\gamma$, where $\lambda>0$ and $\gamma(e_0)=0$ (i.e.\ $\gamma \in (\mathbb R e_0)^{\perp}$).
By the splitting \eqref{eq:identity-four-forms} we have at once 
\begin{align*}
g^*\varphi & =\lambda^{-1}\varphi, \\
g^*(\star \varphi) & = \star \varphi-\gamma \wedge g^*\varphi = \star \varphi-\lambda^{-1}\gamma \wedge \varphi.
\end{align*}
The induced map $\tilde g \colon V \to V$ satisfies the same identities.
The first of these two conditions forces $\tilde g = \lambda^{-1/3}a$, for $a \in \mathrm G_2$.
Applying this to the second identity gives $\lambda^{-4/3}\star \varphi = \star \varphi-\lambda^{-1}\gamma \wedge \varphi$.
By our previous decomposition of $\Lambda^4V^*$ under $\mathrm G_2$, this forces $\lambda^{-4/3}=1$ and $\gamma=0$. 
Since $\lambda>0$, we necessarily have $\lambda=1$. Therefore $g^*e^0=e^0$.
It follows that $G=\mathrm G_2$. 
The $\mathrm{Spin}(7)$-orbit of $e_0$ intersects the ray $\mathbb R_{>0}e_0$ only in $e_0$.
Since $H \subset \mathrm{Spin}(7)$ acts transitively on the rays of $V_+$ and $He_0=S^7 \subset V_+$, then $\mathrm{Spin}(7)e_0$ contains $S^7$ and intersects each ray in one point.
It follows that $\mathrm{Spin}(7)$ preserves the scalar product and that $\mathrm{Spin}(7)e_0=S^7$.

There is a fibration $\nu \colon \mathrm{Spin}(7) \to S^7$ given by $\nu(g)=ge_0$ whose fibres are $\mathrm G_2$-cosets, so we have a short exact sequence
\[1 \longrightarrow \mathrm G_2 \longrightarrow \mathrm{Spin}(7) \longrightarrow S^7 \longrightarrow 1.\]
The homotopy sequence of this shows that $\pi_k(\mathrm{Spin}(7))$ is trivial for $i=1,2,3$ and $\pi_3(\mathrm{Spin}(7))=\mathbb Z$.
Thus the Lie algebra $\mathfrak{spin}(7) \subset \mathfrak{so}(V_+) = \mathfrak{so}(8)$ is simple with a negative definite Killing form by Cartan's criterion.
Since $\dim (\mathrm{Spin}(7))=21$ by the short sequence above, we have $\mathfrak{spin}(7)=\mathfrak{so}(7)$ by the classification of Lie algebras, and the fact that $\mathrm{Sp}(3)$ has no irreducible representation of dimension $8$.
Then $\mathrm{Spin}(7)$ is the simply connected cover of $\mathrm{SO}(7)$.
Now, $\mathrm{SO}(7)$ is known to be simple and its fundamental group is isomorphic to $\mathbb Z_2$. 
On the other hand, $\{\pm \mathrm{id}_{V_+}\} \subset \mathrm{Spin}(7)$, so $\{\pm \mathrm{id}_{V_+}\}$ must be the centre of $\mathrm{Spin}(7)$.
Since $\mathrm{Spin}(7)$ acts transitively on $S^7$, it acts irreducibly on $V_+$.
Moreover, the fact that $\mathrm{Spin}(7)$ acts transitively on $k$-planes for $k\neq 4$ now follows immediately from the fact that $\mathrm G_2$ acts transitively on $k$-planes in $V$ for $k \neq 3,4$.
\end{proof}
\begin{remark}
\label{rmk:characterisation-g2-orbits}
We now know that $\mathrm{Spin}(7)/\mathrm G_2=S^7$ by Proposition \ref{prop:homogeneous-spaces}, whence $\mathrm{SO}(7)/\mathrm G_2=\mathbb RP^7$.
The map $\mathrm{GL}(V_+) \to \Lambda^4V_+^*$ given by $A \mapsto A^*\Phi$ has kernel $\mathrm{Spin}(7)$, and hence the $\mathrm{GL}(V_+)$-orbit $\mathrm{GL}(V_+)/\mathrm{Spin}(7)$ is embedded in $\Lambda^4V_+^*$ and has dimension $43$. However, $\dim(\Lambda^4 V_+^*)=70$, so the orbit of $\Phi$ is not open.
\end{remark}
Again, we write $\Lambda_h^k$ for an irreducible $\mathrm{Spin}(7)$-invariant submodule of $\Lambda^k \coloneqq \Lambda^kV_+^*$ of dimension $h$.
Since $\Lambda^k$ behaves like $\Lambda^{8-k}$, we give the decompositions only for $k \leq 8-k$.
\begin{proposition}
We have the following decompositions into irreducible $\mathrm{Spin}(7)$-invariant modules:
\begin{align*}
\Lambda^2 & = \Lambda_7^2 \oplus \Lambda_{21}^2, \\
\Lambda^3 & = \Lambda_8^3 \oplus \Lambda_{48}^3, \\
\Lambda^4 & = \Lambda_1^4 \oplus \Lambda_7^4 \oplus \Lambda_{27}^4 \oplus \Lambda_{35}^4.
\end{align*}
Also, $\Lambda_7^2 = \mathbb R^7 = \Lambda_7^4$ (the standard representation of $\mathrm{SO}(7)$), $\Lambda_{21}^2 = \mathfrak{spin}(7)^{\flat}$, $\Lambda_8^3=V_+$, $\Lambda_1^4 = \mathbb R \Phi$, $\Lambda_{27}^4 = S_0^2\Lambda_7^2$, and $\Lambda_{35}^4 = S_0^2V_+^*$.
\end{proposition}
\begin{proof}
We know that $V_+ = \mathbb R^8$ is irreducible, so $\Lambda^1$ is irreducible.

We have $\mathfrak{spin}(7)^{\flat} \subset \Lambda^2$, and this is an irreducible summand as it corresponds to the adjoint representation of $\mathrm{Spin}(7)$.
Since $\dim \Lambda^2=28$ and $\dim \mathfrak{spin}(7)=21$, the orthogonal complement of $\mathfrak{spin}(7)^{\flat}$ inside $\Lambda^2$ has dimension $7$.
Since $\{\pm \mathrm{id}_{V_+}\}$ acts trivially on $\Lambda^2$, this is a representation of $\mathrm{SO}(7)$.
Since $\mathrm{Spin}(7)$ acts transitively on the space of two-planes in $V_+$, it follows that $\mathrm{Spin}(7)$ cannot fix any element of $\Lambda^2$.
Since the lowest dimension of a non-trivial representation of $\mathrm{SO}(7)$ is $7$, we see that $\Lambda_7^2$ is irreducible, whence the decomposition into irreducible $\mathrm{Spin}(7)$-components
\[\Lambda^2 = \mathbb R^7 \oplus \mathfrak{spin}(7)^{\flat}=\Lambda_7^2 \oplus \Lambda_{21}^2.\]
A computation leads to the following characterisations
\begin{align*}
\Lambda_7^2 & = \{\alpha \in \Lambda^2: \star (\Phi \wedge \alpha)=3\alpha\}, \\
\Lambda_{21}^2 & = \{\alpha \in \Lambda^2: \star (\Phi \wedge \alpha)=-\alpha\}.
\end{align*}

Now we consider $\Lambda^3$. 
We have a $\mathrm{Spin}(7)$-equivariant map $\Lambda^1 \to \Lambda^3$ given by $\alpha \mapsto \star (\Phi \wedge \alpha)$.
Since this map is non-zero, it must be injective, so we denote the image by $\Lambda_8^3$.
On the other hand, the multiplicity formula (see Humphreys \cite{humphreys}) shows that $\Lambda^3V_+^*$ must contain the irreducible $\mathrm{Spin}(7)$-representation of dimension $48$.
Since $8+48=56=\dim \Lambda^3$, we have an irreducible $\mathrm{Spin}(7)$-decomposition 
\[\Lambda^3 = \Lambda_8^3 \oplus \Lambda_{48}^3.\]
We also have the characterisations
\begin{align*}
\Lambda_8^3 & = \{\star (\Phi\wedge \alpha): \alpha \in V_+^*\}, \\
\Lambda_{48}^3 & = \{\beta \in \Lambda^3: \Phi \wedge \beta=0\}.
\end{align*}
Each of these is an effective $\mathrm{Spin}(7)$-representation since $-\mathrm{id}_{V_+}$ acts as minus the identity on $\Lambda^3$.

Finally, we turn to $\Lambda^4$.
If we differentiate the map $\mathrm{GL}(V_+) \to \Lambda^4$ given by $g \mapsto (g^{-1})^*\Phi$ at the identity, we get a map $\mathrm{End}(V_+) \to \Lambda^4$ given by $e \otimes w \mapsto -w \wedge (e \lrcorner\ \Phi)$. 
This map has kernel $\mathfrak{spin}(7) \subset V_+\otimes V_+^*$ since $\mathrm{Spin}(7)$ is the stabiliser of $\Phi$.
We have $\mathrm{End}(V_+) = V_+^* \otimes V_+^* = \Lambda^2 \oplus S^2V_+^*$.
But \[\Lambda^2=\Lambda_7^2 \oplus \Lambda_{21}^2,\] and the latter summand is $\mathfrak{spin}(7)^{\flat}$, so we see that \[(V_+\oplus V_+^*)/\mathfrak{spin}(7)=\Lambda_7^2 \oplus S^2V_+^*.\]
Since $\mathrm{Spin}(7) \subset \mathrm{SO}(8)$, we have a decomposition $S^2V_+^*=\mathbb R \oplus S_0^2V_+^*$.

We show that $S_0^2V_+^*$ is irreducible.
Indeed, $V_+^*$ has highest weight $(\tfrac12, \tfrac12, \tfrac12)$ (with respect to the basis of fundamental weights for $\mathrm{SO}(7)$), so the multiplicity formula shows that $V_+^* \otimes V_+^*$ is the direct sum of the modules with highest weights $(0,0,0)$ (dimension $1$), $(1,0,0)$ (dimension $7$, i.e.\ $\Lambda_7^2$), $(0,1,0)$ (dimension $21$, i.e.\ $\Lambda_{21}^2$), and $(0,0,1)$ (dimension $35$, i.e.\ $\Lambda^3(\Lambda_7^2)$).
Thus $S_0^2V_+^*$, which is of dimension $35$, must be the remaining summand in $V_+^* \oplus V_+^*$.
It follows that, under the above injection, $(V_+\otimes V_+^*)/\mathfrak{spin}(7) \to \Lambda^4$, these modules must go to irreducible modules $\Lambda_1^4$, $\Lambda_7^4$, and $\Lambda_{35}^4$.
Let $\Lambda_{27}^4$ denote the orthogonal complement of the sum of $\Lambda_1^4 \oplus \Lambda_7^4\oplus \Lambda_{35}^4$ in $\Lambda^4$.
We show that $\Lambda_{27}^4$ is $\mathrm{Spin}(7)$-irreducible.
Indeed, consider the map $S^2\Lambda_7^2 \to \Lambda^4$ induced by the wedge product $\alpha \odot \beta \to \alpha \wedge \beta$.
Since $\mathrm{Spin}(7)$ acts as the full orthogonal group on $\Lambda_7^2$, we see that $S^2\Lambda_7^2=\mathbb R \oplus S_0^2\Lambda_7^2$ is an irreducible sum.
In particular, the image of $S_0^2\Lambda_7^2$ must be an irreducible summand of $\Lambda^4$ of dimension $27$.
But this can only be $\Lambda_{27}^4$. 
Thus $\Lambda^4$ has four irreducible summands under the action of $\mathrm{Spin}(7)$.
We have the following characterisations.
Recall that the map $\sharp\colon \Lambda^2 \to \mathfrak{so}(8)\subset V_+ \otimes V_+^*$ is the inverse of $\flat$, and for $a \in V_+\otimes V_+^*$ and $\psi \in \Lambda^4$, let $a\cdot \psi \in \Lambda^4$ denote the result of the pairing $(V_+\otimes V_+^*) \times \Lambda^4 \to \Lambda^4$ given on monomials by $(x\otimes \alpha, \psi) = -\alpha \wedge (x\lrcorner\ \psi)$. Then
\begin{align*}
\Lambda_1^4 & = \mathbb R \Phi, \\
\Lambda_7^4 & = \{(\sharp a)\cdot \Phi: \alpha \in \Lambda_7^2\}, \\
\Lambda_{27}^4 & = \{\psi \in \Lambda^4: \Phi\wedge \psi=0, \star \psi=\psi, \gamma \wedge \psi=0 \text{ for all } \gamma \in \Lambda_7^4\}, \\
\Lambda_{35}^4 & = \{\psi\in \Lambda^4:\star \psi = -\psi\}.
\end{align*}
The claim is proved.
\end{proof}
We are essentially done with the algebraic part.
We finish by remarking that the inclusions $\mathrm{SU}(3) \subset \mathrm G_2 \subset \mathrm{Spin}(7)$ establish connections among $\mathrm{SU}(3)$-geometry in dimension $6$, $\mathrm G_2$-geometry in dimension $7$, and $\mathrm{Spin}(7)$-geometry in dimension $8$.
In particular, this explains the importance of $\mathrm G_2$ in the world of particles and supersymmetry, in particular $M$-theory \cite{agricola2, duff, strominger}.
The latter aims to describe our Universe as a space of eleven dimensions, seven of which have to be rolled up into a manifold of microscopic radius and integrate the four-dimensional space-time introduced by Einstein in his theory of general relativity.
It turns out that $\mathrm G_2$ models the structure of the seven-dimensional space in question, which is then called a \emph{$\mathrm G_2$-manifold}, cf.\ Karigiannis \cite{karigiannis2}.

A $\mathrm G_2$-manifold is a manifold equipped with a $\mathrm G_2$-structure in the first place.
In the next section, we introduce the language to understand, among other things, what a $\mathrm G_2$-structure really is.

\newpage
\section{Fibre bundles}
\label{sec:fibre-bundles-and-connections}

Objects of interest in geometry are often sections of some vector bundle.
The standard fibre of these vector bundles may come as a representation of a Lie group, and hence the theory in the previous sections may help get information on the irreducible invariant components of a section.

Here below we give a quick summary of principal and vector bundles.
The material in this section is explained in greater detail in Joyce \cite{joyce}, Kobayashi--Nomizu \cite[Vol.\ 1]{kobayashi-nomizu}, or Steenrod \cite{steenrod}.
We look in particular at vector bundles and their correspondence with principal bundles, whose fibre is a Lie group.
We can assume throughout our Lie groups to be linear and reductive.
A number of geometric structures related to Lie groups is given. These are examples of \emph{$G$-structures}.
We introduce \emph{connections} on these bundles, which provide a mean to transport data from a tangent space to another via the so-called \emph{parallel transport}.
By using connections, one can define a central algebraic invariant, the \emph{holonomy group}.
This will be particularly relevant in the next sections, when we start looking at \emph{Riemannian geometry}. 
In this context, the holonomy group encodes information on the \emph{curvature} of the space, and can be used to distinguish different Riemannian geometries.

\subsection{Principal and associated bundles}
\label{subsec:principal-and-associated-bundles}
Let $G$ be a Lie group acting freely on the right on a smooth manifold $P$.
Assume the orbit space $M \coloneqq P/G$ is a smooth manifold. 
Let $\pi \colon P \to M$ be the standard projection, and note that $\pi$ is constant on the $G$-orbits.
Also, assume that $P$ is \emph{locally trivial}, i.e.\ for all $x \in M$ there is a neighbourhood $U \subset M$ containing $x$ such that $\pi^{-1}(U)$ is diffeomorphic to $U \times G$.
Specifically, we require that there is a diffeomorphism $\psi \colon \pi^{-1}(U) \to U \times G$ of the form $\psi(u) = (\pi(u),\varphi(u))$ for all $u \in \pi^{-1}(U)$, with $\varphi(ug) = \varphi(u)g$ for all $g \in G$.
\begin{definition}
The quadruple $(P,M,\pi,G)$ is called a \emph{principal bundle with structure group} $G$.
We call $P$ the \emph{total space} and $M$ the \emph{base space} of the bundle.
\end{definition}
\begin{remark}
We also write $\pi \colon P \to M$ with structure group $G$ for $(P,M,\pi,G)$. 
We use a similar notation for other types of bundles.
\end{remark}
\begin{definition}
A \emph{homomorphism} of a principal bundle $(P',M',\pi',G')$ to another principal bundle $(P,M,\pi,G)$ consists of a map $f' \colon P' \to P$, and a Lie group homomorphism $f'' \colon G' \to G$ such that $f'(u'a) = f'(u')f''(a)$, for all $u' \in P'$ and $a \in G'$.
Since the projection map $\pi'$ is constant on the fibres, there is an induced map $f \colon M' \to M$.
Such a homomorphism is called an \emph{embedding} if $f'$ is an embedding, and if $f''$ is injective.
In the latter case, we say that $(P',M',\pi',G')$ is called a \emph{subbundle} of $(P,M,\pi,G)$.
If $(P',M,\pi',G')$ is a subbundle of $(P,M,\pi,G)$, and the map $f$ induced by $f'$ is the identity map, then the subbundle is called a \emph{reduction of the group $G$ to $G'$}.
\end{definition}
A \emph{local section} $\varphi$ of $(P,M,\pi,G)$ is a map defined in a neighbourhood in $M$ with values in $P$, such that $\pi \circ \varphi = \mathrm{id}$.
A section is \emph{global} when it is defined on all of $M$.
\begin{example}
Let $G$ be any Lie group, and $H$ be a closed subgroup of $G$.
Since $G$ acts on itself freely by right multiplication, then $H$ acts on $G$ freely on the right, and hence we have a principal bundle $G \to G/H$ with structure group $H$.
Local triviality follows by the existence of a local section, cf.\ Kobayashi--Nomizu \cite{kobayashi-nomizu}.
\end{example}
\begin{example}
\label{ex:frame-bundle}
Let $M$ be a smooth $n$-manifold. 
At any point $p \in M$ one can consider the set of all ordered bases of the tangent space $T_pM$.
Each such basis is a set of $n$ linearly independent tangent vectors, and can be interpreted as an isomorphism $u_p \colon \mathbb R^n \to T_pM$.
Precomposition with a general linear transformation $g$ in $\mathrm{GL}(n,\mathbb R)$ yields an action of $\mathrm{GL}(n,\mathbb R)$ on any basis: $g \cdot u_p \coloneqq u_p \circ g$.
This is a right action and is free.
If we let $p \in M$ vary, then the set of all possible bases of all tangent spaces becomes a principal bundle $LM$ over $M$ with structure group $\mathrm{GL}(n,\mathbb R)$, the \emph{bundle of linear frames} over $M$.
\end{example}

We now recall the definition of vector bundle.
Let $W$ be a vector space, and $M$ be a smooth manifold. 
Suppose there are a smooth manifold $E$ and a surjective map $\eta \colon E \to M$ whose fibres are vector spaces.
Also, suppose $E$ is \emph{locally trivial}, i.e.\ every point in $M$ has a neighbourhood $U$ such that there is a diffeomorphism $\psi \colon U \times W \to \eta^{-1}(U)$, with $\eta(\psi(x,w)) = x$ for all $x \in U$, $w \in W$, and $\psi_{|x} \colon W \to \eta^{-1}(\{x\})$ is a linear isomorphism.
\begin{definition}
The quadruple $(E,M,\eta,W)$ is called a \emph{vector bundle}.
The manifold $E$ is its \emph{total space}, $M$ its \emph{base space}, $W$ the \emph{standard fibre}.
The dimension of $W$ is the \emph{rank} of the bundle.
\end{definition}
A \emph{local section} $\varphi$ of $(E,M,\eta,W)$ is a map defined in a neighbourhood in $M$ with values in $E$, such that $\eta \circ \varphi = \mathrm{id}$.
A section is \emph{global} when it is defined on all of $M$.

\begin{remark}
In general we will not be too pedantic with the notations, and write simply $E$ for a vector bundle when base manifold and fibre are clear.
\end{remark}

There is a standard procedure to get a vector bundle from a principal one.
Assume we have a principal bundle $P \to M$ with structure group $G$.
Let $V$ be a representation of $G$.
Then we have a left action of $G$ on $P \times V$ defined by $g\cdot (p,v) = (pg^{-1},gv)$.
Let $E \coloneqq P \times_G V$ be the relative orbit space, and let $\pi_E([p,v]) \coloneqq \pi(p)$.
The latter projection is well-defined as $\pi$ is constant on the $G$-orbits.
It turns out that $\pi_E \colon E \to M$ is a vector bundle over $M$ with standard fibre $V$.
\begin{definition}
We say that the quadruple $(E,M,\pi_E,V)$ is the \emph{vector bundle associated to $P$ via the $G$-representation $V$}.
\end{definition}
Conversely, let $E \to M$ be a real vector bundle over $M$ of rank $k$. 
For any point $p \in M$, consider the set of all ordered bases of the fibre $E_p$ over $p$.
By the same idea as in Example \ref{ex:frame-bundle}, one constructs a principal bundle $P \to M$ with structure group $\mathrm{GL}(k,\mathbb R)$.
Such a principal bundle is also called the \emph{principal frame bundle of $E$}.
For more details on these constructions, we refer to Kobayashi--Nomizu \cite[Chapter 1, Section 5]{kobayashi-nomizu} or Joyce \cite[Section 2]{joyce}.
\begin{remark}
\label{rmk:associated-bundles-general-fibre}
The construction of the associated bundle works just as well in the case where the fibre is any smooth manifold acted on by $G$.
\end{remark}
\begin{example}
\label{ex:tensor-bundles}
Let $M$ be a smooth $n$-manifold, and let $LM \to M$ be the canonical frame bundle over $M$ with structure group $\mathrm{GL}(n,\mathbb R)$.
The general linear group $\mathrm{GL}(n,\mathbb R)$ acts on $\mathbb R^n$ by left multiplication, so we can construct the associated vector bundle $LM \times_{\mathrm{GL}(n,\mathbb R)} \mathbb R^n$.
This turns out to be isomorphic to the tangent bundle $TM$ of $M$.
More generally, if we replace $\mathbb R^n$ by the tensor algebra over $\mathbb R^n$, and let $\mathrm{GL}(n,\mathbb R)$ act on it componentwise, the same construction returns the tensor bundle over $M$.
Special examples we will use are the bundles of symmetric tensors and skew-symmetric tensors.
\end{example}
\begin{example}
\label{ex:adjoint-bundle}
Let $(P,M,\pi,G)$ be a principal bundle.
Let $G$ act on its Lie algebra $\mathfrak g$ via the adjoint representation.
The vector bundle associated to $P$ via the adjoint representation $\mathfrak g$ is the \emph{adjoint bundle} $\mathrm{ad}(P) \coloneqq P \times_G \mathfrak g \to M$ over $M$.
\end{example}

\subsection{Connections and holonomy}
\label{sec:principal-connections-holonomy}

Let $(P,M,\pi,G)$ be a principal bundle.
Let $u \in P$, $\pi(u) = p$, and consider the differential $d\pi_u \colon T_uP \to T_pM$.
The kernel of this map is the set of vectors in $T_uP$ which are tangent to the fibre through $u$ (the $G$-orbit of $u$), and hence we have an isomorphism $\ker d\pi_u = \mathfrak g$.
\begin{definition}
\label{def:connections}
A \emph{connection} $\Gamma$ on $P$ is a choice of a complementary vector space $H_u$ at each $u \in P$ such that 
\begin{enumerate}
\item $T_uP = \mathfrak g \oplus H_u$ for all $u \in P$, 
\item the distribution $u \mapsto H_u$ is smooth, 
\item $H_{ua} = (R_a)_*H_u$ for all $a \in G$. 
\end{enumerate}
We say that $V_u \coloneqq \ker d\pi_u$ is the \emph{vertical} space at $u$, and $H_u$ is the \emph{horizontal} space at $u$.
The distributions $u \mapsto V_u$ and $u \mapsto H_u$ are called \emph{vertical} and \emph{horizontal} respectively.
\end{definition}
\begin{remark}
Smooth manifolds are paracompact (i.e.\ every open cover has an open refinement which is locally finite), so connections always exist on smooth manifolds, cf.\ \cite[Chapter II, Section II]{kobayashi-nomizu}.
\end{remark}
\begin{digression}[Fundamental vector fields]
\label{rmk:fundamental-vector-fields}
Let $P$ be any right $G$-manifold, not necessarily a principal bundle.
Let $A \in \mathfrak g$ be a non-zero vector in the Lie algebra of $G$.
Let now $p \in P$. Then $\alpha(t) \coloneqq p\cdot \exp(tA)$ is a curve in $P$ such that $\alpha(0)=p$. 
Set $A_p \coloneqq \alpha'(0)$.
This defines a Lie algebra homomorphism $\mathfrak g \to \mathfrak{X}(P)$ mapping any generator to a vector field on $P$.
If $G$ acts effectively, the map is injective. 
If $G$ acts freely, then each non-zero vector in $\mathfrak g$ is mapped to a vector field with no zeros (see Kobayashi--Nomizu \cite[Chapter 1, Proposition 4.1]{kobayashi-nomizu}).
Vector fields in $P$ obtained by the above map are called \emph{fundamental vector fields} of the $G$-action.
\end{digression}
Let $\Gamma$ be a connection on $(P,M,\pi,G)$, and let $V$ and $H$ be the corresponding vertical and horizontal distributions respectively.
We define a $\mathfrak g$-valued one-form \[\omega \colon TP \to \mathfrak g\] acting in the following way.
Let $u \in P$. Then $T_uP = V_u \oplus H_u$ with respect to $\Gamma$, and $\omega_u$ is the projection $V_u \oplus H_u \to V_u = \mathfrak g$, cf.\ Digression \ref{rmk:fundamental-vector-fields}.
Clearly, $\omega$ vanishes on horizontal vector fields.

Since $G$ acts both on $TP$ and on $\mathfrak g$, it is natural to check the behaviour of $\omega$ under the action of $G$.
Let $X_u \in T_uP = V_u \oplus H_u$, and write $X_u = X_u^v+X_u^h$ for the splitting into vertical and horizontal part.
Then $\omega_u(X_u) = \omega_u(X_u^v) \eqqcolon X$. 
Now, for any $a \in G$, the right action $R_a \colon P \to P$ is the map $u \mapsto ua$. 
Then the differential $(R_a)_{*|u}$ acts on $X_u^v$ as
\begin{equation*}
(R_a)_{*|u}(X_u^v) = \frac{d}{dt}\left(u\exp(tX)a\right)\Bigr|_{\substack{t=0}} = \frac{d}{dt}\left(uaa^{-1}\exp(tX)a\right)\Bigr|_{\substack{t=0}},
\end{equation*}
and hence
\begin{equation}
\label{eq:connection-form-equivariance}
(R_a)_{*|u}(X_u^v) = \frac{d}{dt}\left(ua\exp(ta^{-1}Xa)\right)\Bigr|_{\substack{t=0}},
\end{equation}
so $\omega_{ua}((R_a)_{*|u}(X_u^v)) = a^{-1}Xa = a^{-1}\omega_u(X_u^v)a = a^{-1}\omega_u(X_u)a$.
Note that the equivariance property for the exponential map, cf.\ e.g.\ Fulton--Harris \cite[pag.\ 116]{fulton-harris}.
More succintly, we can write $(R_a)^*\omega = \mathrm{Ad}(a^{-1})\omega$, cf.\ Digression \ref{rmk:roots}
\begin{definition}
The one-form $\omega \colon TP \to \mathfrak g$ is called a \emph{connection one-form} of the connection $\Gamma$ on $(P,M,\pi,G)$. 
\end{definition}

Conversely, given a one-form $\omega \colon TP \to \mathfrak g$ defined as above and satisfying the equivariance property $(R_a)^*\omega = \mathrm{Ad}(a^{-1})\omega$ for all $a \in G$, one defines $H_u \coloneqq \ker \omega_u$ for all $u \in P$ to get a connection $\Gamma$ on $(P,M,\pi,G)$.

Let now $(P,M,\pi,G)$ be a principal bundle with structure group $G$, and choose a connection $\Gamma$ on $P$.
For any point $u \in P$ with $\pi(u)=p$, the differential $d\pi_u \colon T_uP = V_u \oplus H_u \to T_pM$ maps $H_u$ isomorphically onto $T_pM$.
Then any $X \in T_pM$ can be lifted to a vector $X^*$ in $H_u \subset T_uP$ in a natural way.
Since $(R_a)_{*|u}H_u = H_{ua}$, $X^*$ defines a unique vector field on the fibre over $p$ which is invariant under right $G$-multiplication. 
Let $p \in M$ vary to get a smooth vector field $X^*$ on $P$.
Then $X^*$ is called the \emph{horizontal lift} of $X$.

If $X,Y$ are vector fields on $M$, let $X^*,Y^*$ be their horizontal lifts. 
In general $[X^*,Y^*]$ splits into vertical and horizontal part, and by definition $[X,Y]^* = [X^*,Y^*]^h$.
Then there is a skew-symmetric bilinear $\mathfrak g$-valued form $\Omega \colon TM \times TM \to \mathfrak g$ whose action on $(X,Y) \in TM \times TM$ is
\[\Omega(X,Y) \coloneqq \omega([X^*,Y^*]-[X,Y]^*) = \omega([X^*,Y^*]^v).\]
This can be viewed as a section of the adjoint bundle associated to $P$, cf.\ Example \ref{ex:adjoint-bundle} and the equivariance law given by \eqref{eq:connection-form-equivariance}.
\begin{definition}
The two-form $\Omega \colon TM \times TM \to \mathfrak g$ is the \emph{curvature two-form} of the connection $\Gamma$ on $(P,M,\pi,G)$.
\end{definition}
An alternative approach to the derivation of $\Omega$ via exterior covariant differentials is found in Kobayashi--Nomizu \cite[Chapter 2, Section 5]{kobayashi-nomizu}.
In short, $\Omega$ can also be derived as the \emph{exterior covariant differential} $D\omega$ of the connection one-form $\omega$, where $D\omega$ is the standard differential $d\omega$ composed with the projection onto the horizontal distribution in both of its entries. 
In this way, $\Omega$ is initially a two-form on $P$, but descends to a two-form on $M$ with values in $\mathfrak g$, matching the previous definition.
By viewing $\Omega$ as a two-form on $P$, one has the following result.
\begin{proposition}
\label{prop:str-eq-first-bianchi}
If $X,Y$ are vector fields on $P$, then we have the \emph{Cartan's structure equation}
\begin{equation*}
d\omega(X,Y) = -\tfrac12 [\omega(X),\omega(Y)]+\Omega(X,Y),
\end{equation*}
and a \emph{Bianchi identity} $D\Omega=0$.
\end{proposition}
\begin{remark}
\label{rmk:pseudotensorial-forms}
We point out that exterior covariant differentials can only be applied to differential forms satisfying particular equivariant laws.
We refer to Kobayashi--Nomizu \cite{kobayashi-nomizu} for details.
\end{remark}

Lastly, we see how a connection on a principal bundle defines the so-called \emph{parallel transport}.
Take a principal bundle $(P,M,\pi,G)$, and choose a connection $\Gamma$ on it.
Assume $M$ is connected.
Let $p,q \in M$ be two distinct points, and let $\gamma$ be a smooth curve in $M$ joining $p$ and $q$.
The time-derivative $\dot{\gamma}$ is a tangent vector field to the support of $\gamma$, and thus can be lifted to a vector field on the restriction of $P$ to the support of $\gamma$.
Choose a point $u \in P$ in the fibre over $p$, and let $X_u$ be the value of the lifted vector field at $u$.
Then we can integrate to get a curve in $P$ projecting to $\gamma$ at each point, and whose final point in $\pi^{-1}(q)$ is uniquely determined by the integration process.
This new curve is called a \emph{horizontal lift} of $\gamma$.

Let $u$ be a lift of $\gamma$ joining its two endpoints $u_0$ and $u_1$ such that $\pi(u_0)=p$ and $\pi(u_1)=q$.
If we let $u_0$ vary in its fibre, we get a map $\tau \colon \pi^{-1}(p) \to \pi^{-1}(q)$ mapping the initial point of $u_0$ to the unique final point.
It is essentially obvious by Definition \ref{def:connections} that $\tau$ commutes with the action of $G$ on the fibres, which implies that $\tau$ is an isomorphism.
\begin{definition}
The map $\tau \colon \pi^{-1}(p) \to \pi^{-1}(q)$ is called \emph{parallel transport} along $u$. 
\end{definition}
One can compose horizontal paths by concatenation, and hence the set of all parallel transport maps comes with a group structure.

A special case occurs when $\gamma$ is a piecewise smooth loop based at $p$.
Any horizontal lift of $\gamma$ has initial and final point $u_0$ and $u_1$ in the fibre over $p$, but the two points need not coincide.
However, since $G$ acts on each fibre transitively, there is an element $a \in G$ such that $u_0a = u_1$.
Since the $G$-action on $P$ is free, such $a$ is unique.

So, given $u$ in the fibre over $p$, every loop $\gamma$ based at $p$ determines a unique element $a \in G$ via its horizontal lift starting from $u$.
Consider the set $\mathrm{Hol}_{u}(\Gamma)$ of elements $a \in G$ determined in this way.
Concatenating loops corresponds to taking products in $G$, so $\mathrm{Hol}_{u}(\Gamma)$ comes with a group structure.
It turns out that $\mathrm{Hol}_u(\Gamma)$ is a Lie group (see \cite{kobayashi-nomizu} for details).
\begin{definition}
\label{def:principal-hol}
The Lie group $\mathrm{Hol}_u(\Gamma)$ is called the \emph{holonomy group of $\Gamma$ with reference point $u$}.
Its connected component of the identity is called the \emph{restricted holonomy group}, and is denoted by $\mathrm{Hol}^0_u(\Gamma)$.
\end{definition}

\subsection{On vector bundles}
\label{subsec:associated-connections}

Take a principal bundle $(P,M,\pi,G)$, and let $\rho \colon G \to \mathrm{GL}(n,\mathbb R)$ be a representation of $G$.
Let $\Gamma$ be a connection on $(P,M,\pi,G)$ and let $\tau$ be parallel transport with respect to $\Gamma$.
Consider the associated vector bundle $(E,M,\pi_E,\mathbb R^n)$ with $E = P \times_{\mathrm{GL}(n,\mathbb R)} \mathbb R^n$.
We see how $\Gamma$ yields a connection on $E$.

For each $w \in E$ we define the \emph{horizontal subspace} $H_w$ and the \emph{vertical subspace} $V_w$.
The latter is by definition the tangent space to the fibre of $E$ at $w$.
Now, we have a natural projection $P \times \mathbb R^n \to E = (P \times \mathbb R^n)/G$.
Choose $(u,\xi) \in P \times \mathbb R^n$ in the preimage of $w$.
Fix $\xi \in \mathbb R^n$ and consider the resulting function $P \to E$ mapping $v \in P$ to $v\xi \in E$, where $v\xi$ is just the image of $(v,\xi) \in P \times \mathbb R^n$ in $E$ under the above projection map.
The horizontal subspace $H_w$ is by definition the image of the horizontal subspace in $T_uP$.
This construction does not depend on the choice of $(u,\xi)$, and we have $T_wE = V_w \oplus H_w$.

In analogy with the previous case, a curve in $E$ is horizontal if its tangent vector at each point is horizontal. 
A horizontal lift of a curve in $M$ is a horizontal curve in $E$ projecting to the one on $M$.
We then have a parallel transport map, which we denote by $\tau_E$.
Let $\gamma$ be a curve in $M$.
We denote by $\tau_t^{t+h}$ the isomorphism $\pi_E^{-1}(\gamma(t+h)) \to \pi_E^{-1}(\gamma(t))$ defined by $\tau$.
Let $\varphi$ be a section of $E \to M$ and restrict it to the support of $\gamma$.
\begin{definition}
The \emph{covariant derivative} $\nabla_{\dot{\gamma}}\varphi$ of $\varphi$ in the direction of $\dot{\gamma}$ is 
\begin{equation}
\label{eq:covariant-derivative}
\nabla_{\dot{\gamma}(t)}\varphi \coloneqq \lim_{h \to 0} \frac{\tau_t^{t+h}(\varphi(\gamma(t+h)))-\varphi(\gamma(t))}{h} \in \pi_E^{-1}(\gamma(t)).
\end{equation}
\end{definition}
One can extend this definition by replacing $\dot{\gamma}$ with a generic vector field $X$ on $M$, and then $\nabla_X\varphi$ at a point $p$ is $\nabla_{X_p}\varphi$.

A covariant derivative $\nabla_X\varphi$ satisfies the usual properties of a derivation: it is additive both in $X$ and $\varphi$, $\nabla_{\lambda X}\varphi = \lambda \nabla_X\varphi$ when $\lambda$ is a scalar function, and satisfies the Leibniz rule when $\varphi$ is multiplied by a function, i.e.\ $\nabla_X(\lambda\varphi) = X(\lambda)\varphi+\lambda\nabla_X\varphi$.
One can also define covariant derivatives on vector bundles axiomatically based on the above properties of $\nabla$, cf.\ Kobayashi--Nomizu \cite[Volume 1]{kobayashi-nomizu}.
\begin{definition}
A local section $\varphi$ of $E \to M$ on an open subset $U$ of $M$ is called \emph{parallel} if the image of $T_pM$ under $\varphi$ is horizontal for all $p \in U$.
\end{definition}
Note that the curve $\varphi(\gamma(t))$ is horizontal if and only if $\nabla_{\dot{\gamma}}\varphi=0$, so we also say that $\varphi$ is \emph{parallel} with respect to $\dot{\gamma}$ when $\nabla_{\dot{\gamma}}\varphi=0$.

We now discuss the special case where $(P,M,\pi,G)$ is the canonical frame bundle $LM \to M$ with structure group $\mathrm{GL}(n,\mathbb R)$.
Assume we have a connection $\Gamma$ on $LM$.
Let $\omega$ be its connection one-form and $\Omega$ the curvature two-form.
Let $p \in M$ and take $u \in \pi^{-1}(p)$. 
The differential of the projection map defines an epimorphism $d\pi_u \colon T_u(LM) \to T_pM$.
We have seen that $u$ may be regarded as a linear isomorphism $u \colon \mathbb R^n \to T_pM$.
\begin{definition}
The \emph{solder form}, or \emph{canonical one-form}, $\theta \colon T(LM) \to \mathbb R^n$, is the one-form on $LM$ with values in $\mathbb R^n$ such that
\[\theta_u(X) = u^{-1}(d\pi_u(X)), \qquad X \in T_u(LM).\]
\end{definition}
Note that $\theta$ annihilates vertical vector fields, complementing the behaviour of the connection one-form $\omega$. 
The pair $(\omega,\theta)$ maps $T(LM)$ to $\mathfrak{gl}(n,\mathbb R) \oplus \mathbb R^n$, capturing in a sense the entire geometry of the frame bundle.
More precisely, let $E_{ij}$, $i,j=1,\dots,n$, to be the standard basis of endomorphisms in $\mathfrak{gl}(n,\mathbb R)$, and $e_1,\dots,e_n$ be the standard basis of $\mathbb R^n$.
Let $E_{ij}^*$ be the fundamental vector field on $T(LM)$ induced by $E_{ij}$, and let $B(e_i)$, $i=1,\dots,n$ the horizontal vector fields in $T(LM)$ induced by $e_1,\dots,e_n$.
Then the $n(n+1)$ vector fields  $(E_{ij}^*,B(e_k))$ define an \emph{absolute parallelism} in $LM$, i.e.\ they yield a basis of $T_u(LM)$ for all $u \in LM$.

Just like $\omega$, the solder form $\theta$ behaves in a particular way with respect to the action of $\mathrm{GL}(n,\mathbb R)$ on $LM$.
Let $X_u$ be a vector at $T_u(LM)$, $\pi(u)=p$, and recall that $\pi$ is constant on the fibres.
So if $a \in \mathrm{GL}(n,\mathbb R)$, we have
\begin{align*}
\theta_{ua}((R_a)_{*|u}X_u) & = (ua)^{-1}(d\pi_u((R_a)_{*|u}X_u)) \\
& = a^{-1}u^{-1}(d\pi_u(X_u)) \\
& =a^{-1}\theta_u(X_u).
\end{align*}
More succintly, $(R_a)^*\theta = a^{-1}\theta$.
This property allows one to apply the exterior covariant differential $D$ to $\theta$ (cf.\ Remark \ref{rmk:pseudotensorial-forms}).
\begin{definition}
The two-form $\Theta \coloneqq D\theta \colon T(LM) \times T(LM) \to \mathbb R^n$ is called the \emph{torsion two-form}.
\end{definition}
\begin{proposition}
\label{prop:str-equations-second-bianchi}
If $X,Y$ are vector fields on $LM$, we have the structure equation 
\[d\theta(X,Y) = \omega(Y)\theta(X)-\omega(X)\theta(Y)+\Theta(X,Y),\]
and a second \emph{Bianchi identity} $D\Theta = \Omega \wedge \theta$.
\end{proposition}
Take $X,Y$ be two vectors in $T_pM$, and let $X^*,Y^*$ be their horizontal lifts to $T_u(LM)$, where $\pi(u)=p$.
Then $\Theta(X^*,Y^*)$ is a vector in $\mathbb R^n$, and we can map it to $T_pM$ via $u$. 
This operation defines a two-form $T$ on $M$ with values in $TM$ via
\[T(X,Y) \coloneqq u(\Theta(X^*,Y^*)),\]
which is independent of the chosen $u$.
Similarly, for an extra $Z \in T_pM$, set
\[R_{X,Y}Z \coloneqq u(\Omega(X^*,Y^*))(u^{-1}Z).\]
Then $R_{X,Y}Z$ is a vector at $T_pM$.
We call these two tensors the \emph{torsion tensor} and the \emph{curvature tensor} in $TM$.
We have the following result, cf.\ \cite[Chapter III, Theorem 5.1]{kobayashi-nomizu}.
\begin{theorem}
\label{thm:torsion-curvature}
In terms of the covariant derivative $\nabla$ induced by the connection $\Gamma$ on the frame bundle over $M$, the torsion tensor $T$ and the curvature tensor $R$ can be expressed as
\begin{align*}
T(X,Y) & = \nabla_XY-\nabla_YX-[X,Y], \\
R_{X,Y}Z & = \nabla_X\nabla_YZ-\nabla_Y\nabla_XZ-\nabla_{[X,Y]}Z,
\end{align*}
where $X,Y,Z$ are vector fields on $M$.
\end{theorem}
We conclude this part by giving the definition of holonomy group of a connection on $TM$.
Let a connection on $TM$ be given, and let $\nabla$ be its covariant derivative.
Let $p \in M$ be any point, $\gamma$ a piecewise smooth loop based at $p$, and $\tau_{\gamma}$ the parallel transport map along $\gamma$.
In analogy with Definition \ref{def:principal-hol}, we have the following.
\begin{definition}
\label{def:linear-holonomy}
 The Lie group
\[\mathrm{Hol}_p(\nabla) \coloneqq \{\tau_{\gamma} \in \mathrm{GL}(T_pM): \gamma \text{ loop based at } p\}\]
is called the \emph{holonomy group of $\nabla$ with reference point $p$}. 
The \emph{restricted holonomy group} $\mathrm{Hol}_p^0(\nabla)$ based at $p$ is its connected component of the identity.
\end{definition}
\begin{remark}
We sometimes say that $\nabla$ is a \emph{linear} connection on $M$ or on $TM$.
Definition \ref{def:linear-holonomy} can be easily extended to any vector bundle on $M$.
More details are found in Joyce \cite{joyce}.
\end{remark}

\subsection{On \texorpdfstring{$G$}{G}-structures}
\label{subsec:g-structures}

Let $(LM,M,\pi,\mathrm{GL}(n,\mathbb R))$ be the canonical frame bundle over $M$.
Let $G$ be a closed Lie subgroup of $\mathrm{GL}(n,\mathbb R)$.
\begin{definition}
A \emph{$G$-structure} on $M$ is a smooth principal subbundle $P$ of $LM$ with structure group $G$.
If $P$ and $P'$ are $G$-structures over $M$ and $M'$, and $f \colon M \to M'$ is a diffeomorphism, then $f$ is an \emph{isomorphism} of the two $G$-structures if the induced map at the level of frame bundles maps $P$ into $P'$.
If $M=M'$ and $P=P'$ then $f$ is an \emph{automorphism} of the $G$-structure $P$.
\end{definition}

One can think of a $G$-structure as a section of a bundle whose fibre is a quotient $\mathrm{GL}(n,\mathbb R)/G$.
Since $\mathrm{GL}(n,\mathbb R)$ acts freely on $LM$, $G$ acts freely on $LM$.
We then have a bundle $LM/G \to M$ with fibre $\mathrm{GL}(n,\mathbb R)/G$ and such that the coset $uG$ is mapped to $\pi(u)$.
This is well-defined as $\pi$ is $G$-invariant.
We then have another bundle \[\pi_G \colon LM \to LM/G,\] which is just the projection onto the $G$-orbit.
By \cite[Proposition 5.5]{kobayashi-nomizu}, $LM/G$ is isomorphic to the bundle \[LM \times_{\mathrm{GL}(n,\mathbb R)} (\mathrm{GL}(n,\mathbb R)/G)\] associated to $LM$ via the natural left action of $\mathrm{GL}(n,\mathbb R)$ on $\mathrm{GL}(n,\mathbb R)/G$ (cf.\ Remark \ref{rmk:associated-bundles-general-fibre}).
By \cite[Proposition 5.6]{kobayashi-nomizu}, this bundle admits a global section if and only the structure group of $LM$ is reducible to $G$, i.e.\ there is a $G$-structure on $M$, and the correspondence between global sections $M \to LM/G$ and $G$-structures is one-to-one.
This motivates the following.
\begin{definition}
A \emph{$G$-structure} on $M$ is a global section of the bundle $LM/G \to M$.
\end{definition}
\begin{remark}
\label{rmk:extension-g-structures}
We note that if $H\subset G$ are two closed subgroups of $\mathrm{GL}(n,\mathbb R)$, then an $H$-structure $M \to LM/H$ induces uniquely a $G$-structure $M \to LM/G$.
\end{remark}

In many applications, $G$ comes as the stabiliser in $\mathrm{GL}(n,\mathbb R)$ of some tensor, e.g.\ an inner product, a volume form, a complex structure, and so on.
In these cases, the existence of a $G$-structure corresponds to the existence of a particular geometric structure related to $G$ on $M$.
\begin{example}
\label{ex:o(n)-structure}
A \emph{Riemannian manifold} $(M,g)$ is a smooth manifold equipped with a smooth section $g$ of the symmetric tensor bundle $S^2T^*M \to M$ defining a positive-definite inner product on each tangent space.
On a Riemannian $n$-dimensional manifold $(M,g)$, we can single out orthonormal bases of any tangent space.
Let $p \in M$ be fixed, and choose an orthonormal (ordered) basis $(u_1,\dots,u_n)$ of $T_pM$. 
This can be viewed as an isomorphism $\mathbb R^n \to T_pM$.
So any orthonormal basis corresponds to an orthogonal matrix in $\mathrm{O}(n) \subset \mathrm{GL}(n,\mathbb R)$ under this map.
The construction of the frame bundle over $M$ can then be adapted to the group $\mathrm{O}(n)$ to get the \emph{principal bundle of orthonormal frames} $\mathrm{O}(M) \subset LM$.

Conversely, if we have a principal subbundle of $LM$ over $M$ with structure group $\mathrm{O}(n)$, then the fibre over any point $p \in M$ is the set of all possible orthonormal bases of $T_pM$.
Such a basis is an isomorphism $u_p \colon \mathbb R^n \to T_pM$, and we can then set an inner product on $T_pM$ by 
\[g_p(X,Y) = (u_p^{-1}X,u_p^{-1}Y), \qquad X,Y \in T_pM,\]
where $({}\cdot{},{}\cdot{})$ is the standard scalar product on $\mathbb R^n$.
Since $({}\cdot{},{}\cdot{})$ is invariant by $\mathrm{O}(n)$, $g_p$ does not depend on the specific basis $u_p$ chosen: any other basis is of the form $u_p \circ h$, $h \in \mathrm{O}(n)$, and hence 
\[((u_p\circ h)^{-1}X,(u_p\circ h)^{-1}Y)=(h^{-1}u_p^{-1}X,h^{-1}u_p^{-1}Y)=(u_p^{-1}X,u_p^{-1}Y).\]
Letting $p$ vary smoothly we get a smooth Riemannian metric on $M$.
This gives a bijection between Riemannian metrics on $M$ and $\mathrm{O}(n)$-structures over $M$.
A reduction to the group $\mathrm{SO}(n) \subset \mathrm{O}(n)$ corresponds to an orientation on $M$ via a similar process.
\end{example}
\begin{example}
\label{ex:gl(n,c)-structure}
On a $2n$-dimensional real manifold $M$, suppose we are given an \emph{almost complex structure} $J$, i.e.\ an endomorphism of the tangent bundle such that $J^2=-\mathrm{id}$ pointwise.
Let $p \in M$ be any point. 
Let $e_1$ be any non-zero vector in $T_pM$.
Then $Je_1$ and $e_1$ are linearly independent.
If $n=1$ we are done, otherwise take $e_2 \in T_pM$ such that $e_1$, $Je_1$, and $e_2$ are linearly independent. 
Then $e_1,e_2,Je_1,Je_2$ are linearly independent. 
If $n=2$ we are done, otherwise we can go on this way and construct a $J$-adapted basis $e_1,\dots, e_n, Je_1,\dots, Je_n$.
By using a similar reasoning as in the previous example, this corresponds to the existence of a complex structure $J'$ on $\mathbb R^{2n} = T_pM$, given by the standard model \eqref{eq:symplectic-matrix}.
The group acting on $\mathbb R^{2n}$ preserving $J'$ is isomorphic to $\mathrm{GL}(n,\mathbb C)$, so $J$ on $M$ yields a $\mathrm{GL}(n,\mathbb C)$-structure.
Conversely, having a $\mathrm{GL}(n,\mathbb C)$-structure on $M$ implies we can construct a complex structure on each $T_pM = \mathbb C^n = \mathbb R^{2n}$ that varies smoothly on $M$, i.e.\ an almost complex structure on $M$.
\end{example}
\begin{example}
On a $2n$-dimensional real manifold $M$, an \emph{almost symplectic structure} is a non-degenerate two-form $\sigma$ pointwise linearly equivalent to the standard model
\[\sigma = e^1\wedge e^2+\dots+e^{2n-1} \wedge e^{2n}.\]
The presence of $\sigma$ on $M$ is equivalent to having an $\mathrm{Sp}(2n,\mathbb R)$-structure on $M$.
\end{example}
\begin{example}
Recall that $\mathrm{U}(n) = \mathrm{SO}(2n) \cap \mathrm{GL}(n,\mathbb C)$, see \eqref{eq:unitary-group}.
A $\mathrm{U}(n)$-structure on a $2n$-manifold $M$ is an \emph{almost Hermitian structure}, i.e.\ a pair $(g,J)$ of a Riemannian metric and an almost complex structure compatible with $g$ (i.e.\ $J$ is an isometry) pointwise linearly equivalent to standard ones on $\mathbb R^{2n}$.
One can build up the two-form $\sigma \coloneqq g(J{}\cdot{},{}\cdot{})$, which is an almost symplectic structure on $M$.
\end{example}
\begin{example}
An $\mathrm{SU}(n)$-structure on a $2n$-manifold $M$ is an almost Hermitian structure together with a complex volume form of type $(n,0)$ pointwise linearly equivalent to the standard model
\[(e^1+iJe^1) \wedge \dots \wedge (e^n+iJe^n).\]
A low dimensional example was given in Digression \ref{digression:so(4)}. 
We have seen that an $\mathrm{SU}(2)$-structure on a four-dimensional manifold is given by a triple of two-forms $\sigma, \psi_+,\psi_-$, where $\psi_{\mathbb C} = \psi_++i\psi_-$ is a complex form of type $(2,0)$.
One can check the algebraic relations $\sigma \wedge \sigma = \psi_+\wedge \psi_+=\psi_-\wedge \psi_-$, and $\sigma \wedge \psi_+=\sigma \wedge \psi_-=\psi_+\wedge \psi_-=0$.
Digression \ref{digression:su3} gives a similar picture for $\mathrm{SU}(3)$-structures in dimension $6$.
\end{example}
\begin{example}
On a $4n$-dimensional manifold $M$, an \emph{almost hypercomplex} structure is a $\mathrm{GL}(n,\mathbb H)$-structure.
Equivalently, it is a trivial subbundle $Q \subset \mathrm{End}(TM)$ of rank $3$, generated by three almost complex structures $I,J,K$ such that $IJ=-JI=K$. 
In other words, $I,J,K$ behave like the unit quaternions $i,j,k$.
Note that any combination 
\[a_1I+a_2J+a_3K, \qquad a_1^2+a_2^2+a_3^2=1,\]
gives an almost complex structure, so we actually have a two-sphere of almost complex structures.
\end{example}
\begin{example}
\label{ex:spn-str}
Recall that $\mathrm{Sp}(n) = \mathrm{GL}(n,\mathbb H) \cap \mathrm{U}(2n)$, see \eqref{def:alternative-spn}. 
An $\mathrm{Sp}(n)$-structure on a $4n$-dimensional manifold $M$ is an \emph{almost hyperHermitian} structure, i.e.\ a quadruple $(g,I,J,K)$, where each pair $(g,I)$, $(g,J)$, $(g,K)$ is an almost Hermitian structure, and $IJ=-JI=K$.
\end{example}
\begin{example}
Consider the group $\mathrm{Sp}(n) \times \mathrm{Sp}(1)$. 
Take the quotient \[\mathrm{Sp}(n)\mathrm{Sp}(1) \coloneqq (\mathrm{Sp}(n) \times \mathrm{Sp}(1))/\mathbb Z_2,\] where $\mathbb Z_2$ is the subgroup generated by $(-\mathrm{id},-1)$.
This group acts on $\mathbb R^{4n} = \mathbb H^n$ via $[A,q]x \coloneqq Ax\overline q$.
Equip $\mathbb R^{4n}$ with the scalar product $\langle x,y\rangle \coloneqq \mathrm{Re}(\overline x y)$.
This action of $\mathrm{Sp}(n)\mathrm{Sp}(1)$ is faithful and isometric, so one has an embedding \[\mathrm{Sp}(n)\mathrm{Sp}(1) \hookrightarrow \mathrm{SO}(4n).\]
An $\mathrm{Sp}(n)\mathrm{Sp}(1)$-structure on a $4n$-dimensional manifold $M$ is equivalent to having three almost complex structure $I, J, K$ as in Example \ref{ex:spn-str} defined only locally.
More precisely, there is a subbundle $Q \subset \mathrm{End}(TM)$ of rank $3$, which is non-trivial. 
This structure is called \emph{almost quaternionic Hermitian}.
\end{example}
\begin{example}
A \emph{$\mathrm G_2$-structure} on a seven-dimensional manifold $M$ is a three-form $\varphi$ pointwise linearly equivalent to the standard model \eqref{eq:g2-3-form}.
Since $\mathrm G_2 \subset \mathrm{SO}(7)$, a $\mathrm G_2$-structure induces a unique Riemannian metric and an orientation on $M$.
Similarly, a \emph{$\mathrm{Spin}(7)$-structure} on an eight-dimensional manifold $M$ is a four-form $\psi$ pointwise linearly equivalent to the standard one \eqref{eq:Phi}.
Since $\mathrm{Spin}(7) \subset \mathrm{SO}(8)$, a $\mathrm{Spin}(7)$-structure induces a unique Riemannian metric and an orientation on $M$.
\end{example}

Let $P$ be a $G$-structure over a manifold $M$. 
A connection on $LM$ reduces to a connection on $P$ if and only if for each $u \in P$ the horizontal subspace of $T_u(LM)$ lies in $T_uP$.
Conversely, given a connection on $P$, there is a unique connection on $LM$ that reduces to the connection on $P$.

Recall that a connection on $LM$ is equivalent to a connection on $TM$.
We call such a connection on $TM$ \emph{compatible} with the $G$-structure $P$ if the corresponding connection on $LM$ reduces to $P$.
Conversely, every connection on $P$ induces a unique connection on $TM$.

We will be interested in torsion-free linear connections.
Given a $G$-structure $P$ on $M$, we may ask how many torsion-free connections $\nabla$ on $TM$ are compatible with $P$.
If $\nabla$ and $\nabla'$ are two connections on $P$, the difference tensor $\alpha = \nabla'-\nabla$ is a section of $\mathrm{ad}(P) \otimes T^*M$.
Combining Example \ref{ex:tensor-bundles} and \ref{ex:adjoint-bundle} applied to this case, we see that $\mathrm{ad}(P)$ is a vector subbundle of $TM \otimes T^*M$, so $\alpha$ is a type $(2,1)$-tensor.
We compute
\begin{align*}
\nabla'_XY-\nabla'_YX-[X,Y] & = \nabla_XY-\nabla_YX-[X,Y] \\
& \qquad +\alpha(X,Y)-\alpha(Y,X).
\end{align*}
If $\nabla$ is a fixed connection on $P$, the formula shows that $\nabla'$ has vanishing torsion $T'$ if and only if the torsion $T$ of $\nabla$ satisfies
\[T(X,Y)=\alpha(Y,X)-\alpha(X,Y).\]
Further, if such $\nabla'$ exists, then the set of all torsion-free connections on $P$ is in one-to-one correspondence with the vector space of tensors $\alpha$ as above such that $\alpha(X,Y)=\alpha(Y,X)$.

We now get to the concept of \emph{intrinsic torsion} for $P$. 
This essentially represents the obstruction to finding torsion-free connections on a $G$-structure.
The following approach is taken from Joyce \cite{joyce}.
\begin{definition}
\label{def:g-structure-map}
Let $G$ be a Lie subgroup of $\mathrm{GL}(n,\mathbb R)$, and let $V=\mathbb R^n$.
Then $G$ acts effectively on $V$, and $\mathfrak g \subset V \otimes V^*$.
Define the map
\[\sigma \colon \mathfrak g \otimes V^* \to V \otimes \Lambda^2V^*, \qquad \sigma(\alpha_{bc}^a)=\alpha_{bc}^a-\alpha_{cb}^a,\]
in index notation. Define vector spaces $W_1,\dots, W_4$ by
\[W_1\coloneqq V \otimes \Lambda^2V^*, \quad W_2 \coloneqq \mathrm{Im}\ \sigma, \quad W_3 \coloneqq \mathrm{coker}\ \sigma, \quad W_4 \coloneqq \ker \sigma.\]
and let $\rho_j \colon G \to \mathrm{GL}(W_j)$ be the natural representation of $G$ on $W_j$.
\end{definition}
Let $P$ be a $G$-structure on $M$. Recall that $TM = P \times_G V$, and let $\nabla$ be any connection on $TM$.
The torsion of $\nabla$ is a section of the vector bundle associated to $P$ via $\rho_1$.
If $\nabla'$ is another connection, then the difference of the torsion tensors $T'-T$ takes values in the vector bundle associated to $P$ via $\rho_2$, which is a subbundle of the bundle obtained via $\rho_1$.
Since the associated bundle obtained via $\rho_3$ is the quotient bundle $\rho_2/\rho_1$, then the projection of $T'-T$ to $\rho_3$ is zero, and hence $T=T'$.
It follows that the projection of $T$ via $\rho_3$ does not depend on the connection $\nabla$.
\begin{definition}
The \emph{intrinsic torsion} $T^i(P)$ of the $G$-structure $P$ is the projection via $\rho_3$ of the torsion $T$ of any connection $\nabla$ compatible with $P$.
The $G$-structure $P$ is called \emph{torsion-free} if its intrinsic torsion vanishes.
\end{definition}
\begin{remark}
A torsion-free connection $\nabla$ on $P$ exists if and only if $P$ is torsion-free.
In this sense, the intrinsic torsion represents the obstruction to finding a torsion-free connection on $P$.
\end{remark}
\begin{remark}
Any two torsion-free connections differ by a section of the bundle associated to $P$ via $\rho_4$.
So if the intrinsic torsion of $P$ vanishes, then the torsion-free connections on $P$ are in one-to-one correspondence with sections of the latter bundle.
If $\ker \sigma = 0$, there is a unique torsion-free connection.
\end{remark}
\begin{proposition}
\label{prop:torsion-free-holonomy}
Let $M$ be an $n$-manifold, and $G$ be a Lie subgroup of $\mathrm{GL}(n,\mathbb R)$.
Then $M$ admits a torsion-free $G$-structure $P$ if and only if there exists a torsion-free connection $\nabla$ on $TM$ with $\mathrm{Hol}(\nabla) \subset G$.
\end{proposition}
\begin{remark}
This shows that torsion-free $G$-structures on $M$ are closely related to torsion-free connections $\nabla$ on $TM$ with $\mathrm{Hol}(\nabla) = G$.
However, torsion-free $G$-structures are normally easier to handle than torsion-free connections with prescribed holonomy.
In fact, the equation $T^i(P)=0$ is a differential equation, whereas $\mathrm{Hol}(\nabla)=G$ involves both differentiation and integration.
\end{remark}
\begin{example}
\label{ex:uniqueness-levi-civita}
Let $G=\mathrm{O}(n) \subset \mathrm{GL}(n,\mathbb R)$.
The map $\sigma$ in Definition \ref{def:g-structure-map} is then $\sigma \colon \mathfrak{so}(n) \otimes V^* \to V \otimes \Lambda^2V^*$, and in index notation a tensor $\alpha_{bc}^a$ in the domain of $\sigma$ is skew-symmetric in $a$ and $b$.
Assume $\alpha \in \ker \sigma$, then
\[\alpha_{bc}^a=\alpha_{cb}^a=-\alpha_{ab}^c=-\alpha_{ba}^c=\alpha_{ca}^b=\alpha_{ac}^b=-\alpha_{bc}^a.\]
So $\ker \sigma$ is trivial, and since $\mathfrak{so}(n) = \Lambda^2V$, $\sigma$ must be surjective.
Therefore, $W_4=0$, $W_1=W_2$, and $W_3=0$. 
We deduce that every $\mathrm{O}(n)$-structure is torsion-free, and hence on any $\mathrm{O}(n)$-structure there is a unique compatible torsion-free connection, the \emph{Levi-Civita connection}.
\end{example}
\begin{example}
\label{ex:integrable-kahler-str}
Let $G=\mathrm{U}(n)\subset \mathrm{SO}(2n) \cap \mathrm{GL}(n,\mathbb C)$.
By Remark \ref{rmk:extension-g-structures}, a $\mathrm{U}(n)$-structure extends uniquely to an $\mathrm{SO}(2n)$-structure.
Since $\mathfrak{u}(n) \subset \mathfrak{so}(2n)$, the map $\sigma$ is again injective, but the Levi-Civita connection corresponding to the $\mathrm{SO}(2n)$-structure is in general not compatible with the $\mathrm{U}(n)$-structure.
Compatibility occurs when the holonomy of the Levi-Civita connection reduces to $\mathrm{U}(n)$ by Proposition \ref{prop:torsion-free-holonomy}, in which case the $\mathrm{U}(n)$-structure is called \emph{K\"ahler}.
The only case in which compatibility occurs automatically is when $n=1$ (i.e.\ when $\sigma$ is surjective), according to the isomorphism $\mathrm{U}(1)=\mathrm{SO}(2)$ (cf.\ Remark \ref{eq:monomorphism-c}).
So every $\mathrm{U}(1)$-structure is torsion-free.
\end{example}

Finally, we present a second fundamental concept related to $G$-structures, which is \emph{integrability}.
See Kobayashi \cite{kobayashi} for the material here below.
\begin{definition}
A $G$-structure $P$ on an $n$-manifold $M$ is called \emph{integrable} (or \emph{locally flat}) if every point of $M$ has a local neighbourhood $U$ with local coordinates $x_1,\dots,x_n$ such that the local section $(\partial_{x_1}, \dots, \partial_{x_n})$ of $LM$ over $U$ is a local section of $P$ over $U$.
We say that $(x_1,\dots,x_n)$ is \emph{admissible} with respect to the $G$-structure $P$.
\end{definition}
If $(x_1,\dots,x_n)$ and $(y_1,\dots,y_n)$ are two admissible local coordinate systems in two open sets $U$ and $V$ respectively, then the Jacobian $(\partial y_i/\partial x_j)_{i,j}$ is in $G$ at each point of $U \cap V$.
\begin{proposition}
Let $\mathrm K$ be a tensor on $\mathbb R^n$, and $G \subset \mathrm{GL}(n,\mathbb R)$ the stabiliser of $\mathrm K$.
Let $P$ be a $G$-structure over $M$ and $K$ the tensor field on $M$ defined by $\mathrm K$ and $P$ in the natural way.
Then $P$ is integrable if and only if each point of $M$ has a local neighbourhood with local coordinate system $x_1,\dots,x_n$ with respect to which the components of $K$ are constant functions on $U$.
\end{proposition}
\begin{proof}
We have already seen how $\mathrm K$ induces a tensor field $K$ on $M$ 
Let $p \in M$ and $u_p$ be an element in the fibre of $P$ over $p$. 
Then $u_p$ is a linear isomorphism $\mathbb R^n \to T_pM$, and thus extends to a linear isomorphism of the tensor algebras of the two spaces.
Then $K_p$ is defined as the image of $\mathrm K$ under this isomorphism.
Invariance of $\mathrm K$ under $G$ implies that $K_p$ does not depend on the choice of $u_p$.

Assume $P$ to be integrable, and let $x_1,\dots,x_n$ be an admissible local coordinate system. 
The components of $K$ with respect to $x_1,\dots,x_n$ equal the components of $\mathrm K$ with respect to the standard basis of $\mathbb R^n$, so they are constant functions.

Conversely, let $x_1,\dots,x_n$ be a local coordinate system, and assume $K$ has constant components with respect to it.
Note that in general $x_1,\dots,x_n$ need not be admissible.
Consider $(\partial_{x_1},\dots, \partial_{x_n})$ at the origin of this coordinate system.
By a linear change of coordinates, we get another coordinate system $y_1,\dots,y_n$ such that $(\partial_{y_1},\dots,\partial_{y_n})$ at the origin sits in $P$.
It follows that $K$ has constant components with respect to $y_1,\dots,y_n$.
These correspond to the components of $\mathrm K$ with respect to the standard basis of $\mathbb R^n$, as $(\partial_{y_1}, \dots, \partial_{y_n})$ at the origin sits in $P$.
Let $u$ be an element in the fibre of $P$ over $x \in U$.
The components of $K$ with respect to $u$ coincide with the components of $\mathrm K$ with respect to the standard basis of $\mathbb R^n$, and hence with the components of $K$ with respect to $(\partial_{y_1},\dots,\partial_{y_n})$. 
Therefore, the frame $(\partial_{y_1},\dots,\partial_{y_n})$ at $x$ coincides with $u$ modulo $G$, and hence lies in $P$.
\end{proof}
\begin{example}
We have seen in Example \ref{ex:o(n)-structure} that there is a one-to-one correspondence between Riemannian metrics and $\mathrm{O}(n)$-structures.
An $\mathrm{O}(n)$-structure is integrable if and only if the corresponding Riemannian metric is flat, i.e.\ the Levi-Civita connection $\nabla$ has vanishing curvature, cf.\ Theorem \ref{thm:torsion-curvature}. 
The idea of the proof is the following.
If the $\mathrm{O}(n)$-structure is integrable, then there are local coordinate with respect to which the corresponding metric has constant coefficients, and hence its curvature vanishes.
As for the converse, if the curvature of $g$ vanishes, then there exists a parallel frame on $M$ in a neighbourhood of each point, so the metric tensor $g$ is locally constant. 
\end{example}
\begin{example}
Consider an almost symplectic $2n$-dimensional manifold $(M,\sigma)$.
Suppose $d\sigma=0$, i.e.\ $\sigma$ is a \emph{symplectic} form.
In this case, $(M,\sigma)$ is called a \emph{symplectic manifold}.
Then by Darboux Theorem (cf.\ \cite[Theorem 8.1]{dasilva}) every point in $M$ has a coordinate neighbourhood on which $\sigma$ has constant coefficients.
So $d\sigma=0$ is the integrability condition for the $\mathrm{Sp}(2n,\mathbb R)$-structure induced by $\sigma$.
\end{example}
\begin{example}
\label{ex:integrability-complex-structure}
Let $M$ be a smooth $2n$-manifold equipped with an almost complex structure $J$.
This is equivalent to having a $\mathrm{GL}(n,\mathbb C)$-structure over $M$.
The latter is integrable if and only if it comes from a complex structure on $M$.
By the Newlander--Nirenberg Theorem \cite{newlander-nirenberg}, integrability is equivalent to the vanishing of the \emph{Nijenhuis tensor} of $J$, which (up to multiples) is defined by 
\begin{equation*}
N_J(X,Y) \coloneqq [X,Y]-[JX,JY]+J[JX,Y]+J[X,JY].
\end{equation*} 
Useful references for complex and K\"ahler geometry are Huybrechts \cite{huybrechts}, or the notes by Moroianu \cite{moroianu}.
\end{example}
\begin{remark}
One way to see how the Nijenhuis tensor arises is the following.
If $(M,J)$ is a $2n$-dimensional almost complex manifold, the decomposition $\Lambda^1 \mathbb R^{2n} \otimes \mathbb C = \Lambda^{1,0} \oplus \Lambda^{0,1}$ implies that the space of complex one-forms $\Omega_{\mathbb C}^1(M)$ splits as a direct sum
\[\Omega_{\mathbb C}^1 = \Omega^{1,0} \oplus \Omega^{0,1}.\]
Consider the differential of a $(1,0)$ form $\alpha$, i.e.\ $J\alpha=i\alpha$. Since $d\alpha$ is a two-form, and we have a $J$-decomposition
\[\Omega_{\mathbb C}^2(M) = \Omega^{2,0} \oplus \Omega^{1,1} \oplus \Omega^{0,2},\]
then $d \colon \Omega^{1,0} \to \Omega_{\mathbb C}^2$ splits as 
\[d=d^{1,0}+d^{0,1}+d^{-1,2},\]
where $d^{i,j} \colon \Omega^{1,0} \to \Omega^{1+i,j}$.
We are interested in understanding how $d^{-1,2}$ acts on $\alpha$. Since $d^{-1,2}\alpha \in \Omega^{0,2}$, its values are determined by evaluating $d\alpha$ on pairs of $(0,1)$-vector fields $X+iJX, Y+iJY$.
Recall that $\alpha(X+iJX)=\alpha(Y+iJY)=0$ as $\alpha \in \Omega^{1,0}$. We then compute
\begin{align*}
d\alpha(X+iJX,Y+iJY) & = (X+iJX)(\alpha(Y+iJY))-(Y+iJY)(\alpha(X+iJX)) \\
& \qquad -\alpha([X+iJX,Y+iJY]) \\
& = -\alpha([X,Y]-[JX,JY]+i[X,JY]+i[JX,Y]).
\end{align*}
By writing the argument of $\alpha$ as 
\[N_J(X,Y)+i([JX,Y]+iJ[JX,Y])+i([X,JY]+iJ[X,JY]),\]
and noting that $\alpha$ vanishes on $[JX,Y]+iJ[JX,Y]$ and $[X,JY]+iJ[X,JY]$, we see that
\[d^{-1,2}\alpha(X+iJX,Y+iJY) = -\alpha(N_J(X,Y)).\]
So $d^{-1,2}\alpha=0$ if and only if the Nijenhuis tensor of $J$ vanishes.
A similar process can be repeated for $d^{2,-1}\alpha$ when $\alpha \in \Omega^{0,1}$. 
It follows that the Nijenhuis tensor of $J$ vanishes if and only if the exterior differential can be written as \[d=d^{1,0}+d^{0,1} \eqqcolon \partial+\overline{\partial},\] the latter being a standard notation in complex geometry.
\end{remark}
\begin{proposition}
If a $G$-structure $P$ over $M$ is integrable, then its intrinsic torsion vanishes.
\end{proposition}
\begin{proof}
Let $U$ be a coordinate chart with admissible coordinate system $x_1,\dots,x_n$.
Let $\omega_U$ be the connection form restricted to $P_{|U}$ defining a flat connection on $U$ such that the vector fields $\partial_{x_1},\dots,\partial_{x_n}$ are parallel.
Cover $M$ by a locally finite family of such open sets $U$.
Let $\pi \colon P \to M$ be the standard projection.
By taking a partition of unity $\{f_U\}$ subordinate to $\{U\}$, one defines a connection one-form $\omega$ by
\[\omega = \sum_{U} \pi^* f_U\omega_U.\]
Since the above frame is parallel, the torsion of $\omega$ vanishes.
\end{proof}
\begin{example}
Obviously a Riemannian structure admits a (unique) compatible torsion-free connection.
Complex and symplectic manifolds admit compatible torsion-free connections.
\end{example}
The above result shows that in general if a $G$-structure is integrable then it is also torsion-free.
The torsion-free condition for a $G$-structure can be interpreted as a first-order integrability condition.
So torsion-free $G$-structures are sometimes called \emph{$1$-flat}, whereas integrable $G$-structures are called \emph{flat}.
We refer the reader to \cite{bryant, guillemin} for an account on \emph{$k$-flat} $G$-structures, integrability of $G$-structures, and the relation to \emph{Spencer homology}.

\newpage
\section{Riemannian geometry}
\label{sec:riemannian-geometry}

On smooth manifolds, one can always impose a \emph{Riemannian metric}, namely a smooth symmetric $2$-tensor that is positive-definite pointwise.
These spaces appear very naturally in geometry and physics, particularly in general relativity or in string theory.
One feature of Riemannian manifolds is that they come with a natural torsion-free connection, the \emph{Levi-Civita} connection, and one can study the curvature tensor of it under the action of the orthogonal group.
Distinguished metrics coming out from the decomposition of the curvature into irreducible orthogonal invariant components are metrics with \emph{constant sectional curvature} and \emph{Einstein} metrics.

In this section we understand the Riemannian curvature tensor and its behaviour with respect to orthogonal groups.
We also illustrate other applications of the representation theory of the orthogonal group, e.g.\ we illustrate how torsion tensors are classified into orthogonal invariant types.
In particular, we look at \emph{connections with torsion}.
Most of the material is taken from Besse \cite{besse, besse2} and Berger--Gauduchon--Mazet \cite{berger-gauduchon-mazet}.

\subsection{Riemannian manifolds}
\label{subsec:riemannian-manifolds}
We start with the following definition.
\begin{definition}
A \emph{Riemannian manifold} is a pair $(M,g)$, where $M$ is a smooth $n$-dimensional manifold, and $g$ a smooth section of the symmetric tensor bundle $S^2T^*M \to M$ defining a positive-definite inner product on $TM$ pointwise.
We also say that $g$ is a \emph{metric tensor} (or simply a \emph{metric}) on $M$.
\end{definition}
The presence of a metric on $M$ allows one to single out local orthonormal frames, so that the structure group of the general frame bundle $LM$ over $M$ can be reduced to the orthogonal group $\mathrm{O}(n) \subset \mathrm{GL}(n,\mathbb R)$.
Each tangent space $T_pM$ is then acted on by the orthogonal group $\mathrm{O}(n)$, which preserves the scalar product $g_p$ on $T_pM$. 
\begin{remark}
The identification of $TM$ and $T^*M$ via the metric tensor will normally be implicitly used, cf.\ subsection \ref{subsec:some-linear-algebra}.
\end{remark}

We have already seen that for a given $\mathrm{O}(n)$-structure there is a unique compatible torsion-free connection, the Levi-Civita connection, cf.\ Example \ref{ex:uniqueness-levi-civita}.
The following is a reformulation of this result.

\begin{theorem}[Fundamental Theorem of Riemannian Geometry]
On any Riemannian manifold $(M,g)$, there is a unique torsion-free linear connection $\nabla$ parallelising $g$, i.e.\ $\nabla g = 0$.
\end{theorem}
\begin{remark}
The condition $\nabla g = 0$ is also known as $\nabla$ being a \emph{metric connection}. 
\end{remark}
\begin{proof}
Let $X,Y,Z \in \mathfrak{X}(M)$ be vector fields on $M$. Metric and torsion-free condition amount to saying
\begin{align*}
X(g(Y,Z)) & = g(\nabla_XY,Z)+g(Y,\nabla_XZ), \\
[X,Y] & = \nabla_XY-\nabla_YZ.
\end{align*}
By the above requirements, the sum of the three identities
\begin{align*}
X(g(Y,Z)) & = g(\nabla_XY,Z)+g(Y,\nabla_XZ), \\
-Y(g(Z,X)) & = -g(\nabla_YZ,X)-g(Z,\nabla_YX), \\
Z(g(X,Y)) & = g(\nabla_ZX,Y)+g(X,\nabla_ZY),
\end{align*}
gives \emph{Koszul formula}
\begin{align*}
2g(\nabla_ZX,Y) & = X(g(Y,Z))-Y(g(Z,X))+Z(g(X,Y)) \\
& \qquad +g(Z,[Y,X])+g(X,[Y,Z])+g(Y,[Z,X]),
\end{align*}
which expresses the general action of $\nabla$.
We see that $\nabla$ is uniquely determined, so the claim is proved.
\end{proof}

We are particularly interested in the curvature tensor of $\nabla$, cf.\ Theorem \ref{thm:torsion-curvature}, i.e.\ the type $(3,1)$-tensor field $R$ on $M$ defined by 
\begin{equation*}
R_{X,Y}Z = \nabla_X\nabla_YZ-\nabla_Y\nabla_XZ-\nabla_{[X,Y]}Z.
\end{equation*}
It will be convenient to view $R$ as a type $(4,0)$-tensor by the identity
\[R(X,Y,Z,W) \coloneqq g(R_{X,Y}Z,W), \qquad X,Y,Z,W \in \mathfrak{X}(M),\]
in fact we will do so from now on.
It is easily checked that $R$ satisfies the following symmetries:
\begin{enumerate}
\item $R(X,Y,Z,W) = -R(Y,X,Z,W) = -R(X,Y,W,Z) = R(Z,W,X,Y)$.
\item $R(X,Y,Z,W)+R(Y,Z,X,W)+R(Z,X,Y,W)=0$.
\end{enumerate}
The first symmetries can be reinterpreted by saying that $R$ is a section of the tensor bundle $S^2 \Lambda^2 T^*M$.
The second identity is known as the \emph{first Bianchi identity} and is derived from $D\Theta = \Omega \wedge \theta$ in Proposition \ref{prop:str-equations-second-bianchi}
\begin{remark}
\label{rmk:second-bianchi-identity}
A \emph{second Bianchi identity} $\mathfrak{S}_{X,Y,Z} (\nabla_XR)_{Y,Z}W=0$ also holds, and is a consequence of $\nabla$ being torsion-free. 
This is derived by $D\Omega=0$ in Proposition \ref{prop:str-eq-first-bianchi}.
The second Bianchi identity is not interesting for us at the moment, but the restrictions it imposes on the symmetries of the curvature tensor are relevant in the classification of Riemannian holonomy groups, cf.\ subsection \ref{subsec:berger-theorem} below.
\end{remark}
By Proposition \ref{rmk:lambda-so} and the above symmetries, we may write \[R(X,Y,Z,W) = R(X \wedge Y,Z \wedge W),\] and since we have symmetry in the pairs $X \wedge Y$, $Z \wedge W$, $R$ is determined by its values $R(X \wedge Y,X\wedge Y)$.
For $X$ and $Y$ linearly independent at a point $p \in M$, define (cf.\ \eqref{eq:inner-product-exterior-algebra})
\begin{equation}
\label{eq:sectional-curvature}
\sigma_p(X,Y) \coloneqq \frac{R_p(X \wedge Y,X\wedge Y)}{\lVert X \wedge Y\rVert_p^2} = \frac{R_p(X \wedge Y,X\wedge Y)}{g_p(X,X)g_p(Y,Y)-g_p(X,Y)^2}.
\end{equation}
The definition of $\sigma_p(X,Y)$ is independent of the basis of the plane generated by $X$ and $Y$.
Hence at any point $p \in M$, $\sigma_p$ is defined on the Grassmannian of two-planes in $T_pM$.
It is clear by the above observations that $\sigma$ and $R$ contain the same amount of information.
\begin{definition}
The map $\sigma$ in \eqref{eq:sectional-curvature} is the \emph{sectional curvature} of $(M,g)$.
\end{definition}
It is interesting to note that $\sigma$ takes a constant value $k$ on all planes for all points in $M$ if and only if
\begin{align*}
R(X,Y,Z,W) & = k(g(X,Z)g(Y,W)-g(Y,Z)g(X,W)) \nonumber \\
& = \frac{k}{2}(g\owedge g)(X,Y,Z,W), 
\end{align*}
where the latter identity is derived by the expression of the Kulkarni--Nomizu product, cf.\ Example \ref{ex:kulkarni-nomizu-product}.
Indeed, the tensor \[R'(X,Y,Z,W)\coloneqq g(X,Z)g(Y,W)-g(Y,Z)g(X,W)\] has the same symmetries of $R$, and is thus determined by its values on $Z=X$ and $W=Y$.
If we assume $R(X,Y,X,Y) = kR'(X,Y,X,Y)$, then $R=kR'$.
The converse is trivial.
This motivates the following.
\begin{definition}
Let $(M,g)$ be a Riemannian manifold, and $R$ the curvature tensor of $g$.
If there is a constant $k \in \mathbb R$ such that \[R = \frac{k}{2} g\owedge g,\] we say that $(M,g)$ has \emph{constant sectional curvature}.
\end{definition}
\begin{definition}
If $(M,g)$ is a Riemannian manifold with constant sectional curvature $k$, we say that $(M,g)$ is \emph{elliptic}, \emph{hyperbolic}, or \emph{flat} if $k>0$, $k<0$, or $k=0$ respectively.
Such spaces are called \emph{space forms}.
\end{definition}
\begin{theorem}[Killing--Hopf \cite{hopf, killing}] 
The universal cover of a complete space form with curvature $k$ is isometric (up to homotheties) to a sphere if $k>0$, a hyperbolic space if $k<0$, or the standard Euclidean space if $k=0$.
\end{theorem}
If $R$ is a Riemannian curvature tensor, then one defines the \emph{Ricci curvature} as the trace
\[\mathrm{Ric}(X,Y) \coloneqq \mathrm{Tr}(R(X,{}\cdot{})Y{}\cdot{}) = \sum_{k=1}^n R(X,E_k,Y,E_k),\]
where $\{E_k\}$ is a local orthonormal frame. 
Note that $\mathrm{Ric}$ is symmetric.
By using the metric, one can also define the Ricci endomorphism via $g(\mathrm{Ric}(X),Y) \coloneqq \mathrm{Ric}(X,Y)$.
Then the trace of the Ricci endomorphism is the \emph{scalar curvature} 
\[s \coloneqq \sum_{k=1}^n g(\mathrm{Ric}(E_k),E_k) = \sum_{k=1}^n \mathrm{Ric}(E_k,E_k).\]
\begin{definition}
Let $(M,g)$ be a Riemannian manifold. If $\mathrm{Ric}=\lambda g$ for a constant $\lambda$, then $(M,g)$ is called \emph{Einstein}.
If in particular $\lambda=0$, we say that $(M,g)$ is \emph{Ricci-flat}.
\end{definition}

\begin{digression}
\label{digression:curvature-low-dimension}
In general, $\mathrm{Ric}$ and $s$ contain only partial information on $R$ (or equivalently $\sigma$).
However, there are two low dimensional cases where we can say more on the relationship among $R$, $\mathrm{Ric}$, and $s$.
In these two cases, we will see that the first Bianchi identity follows from the other symmetries of $R$, cf.\ Remark \ref{rmk:bianchi-identity-trivial-low-dimension} below.
Hereafter, for any choice of a local frame $E_1,\dots,E_n$ we denote by $R_{ijk\ell}$ the function $R(E_i,E_j,E_k,E_{\ell})$.

In dimension $2$ the curvature tensor $R$, the Ricci tensor $\mathrm{Ric}$, and the scalar curvature $s$ are equivalent.
This is easily seen as follows. 
At any point, take an orthonormal basis $E_1,E_2$. Then $R$ is specified by only one coefficient $R_{1212}$.
A computation gives $\mathrm{Ric}(E_i,E_i) = R_{1212}$ for $i=1,2$, $\mathrm{Ric}(E_1,E_2)=0$, and $s=2R_{1212}$.
Then $\mathrm{Ric}=\tfrac12 sg$.
It is clear that the knowledge of $R$ is equivalent to that of $\mathrm{Ric}$ or $s$.
We then see that in dimension $2$ the metric $g$ is Einstein if and only if its scalar curvature (or sectional curvature) is constant.
Also, in this dimension the scalar curvature is twice the \emph{Gaussian curvature}, i.e.\ the sectional curvature of any two-plane.

A similar computation shows that in dimension $3$, the curvature tensor $R$ and the Ricci tensor $\mathrm{Ric}$ are equivalent.
At any point, take again an orthonormal basis $E_1,E_2,E_3$.
Since $R \in S^2\Lambda^2\mathbb R^3 = S^2 \mathbb R^3 = \mathbb R^6$, the curvature $R$ is given by six coefficients $R_{1212}, R_{1313}, R_{2323}, R_{1213}, R_{1223}, R_{1323}$.
The Ricci tensor is specified by six coefficients as well.
A straighforward computation gives 
\begin{alignat*}{2}
\mathrm{Ric}(E_1,E_1) & = R_{1212}+R_{1313}, && \qquad \mathrm{Ric}(E_1,E_2) = R_{1323}, \\
\mathrm{Ric}(E_2,E_2) & = R_{1212}+R_{2323}, && \qquad \mathrm{Ric}(E_1,E_3) = R_{1232}, \\ 
\mathrm{Ric}(E_3,E_3) & = R_{1313}+R_{2323}, && \qquad \mathrm{Ric}(E_2,E_3) = R_{1213}.
\end{alignat*}
The system is invertible, so the knowledge of $\mathrm{Ric}$ is again equivalent to that of $R$.
Note that if $R$ has constant sectional curvature, then \[\mathrm{Ric}(E_1,E_1) = \mathrm{Ric}(E_2,E_2) = \mathrm{Ric}(E_3,E_3),\] and they all equal a constant, so the metric is Einstein. 
Conversely, if the metric is Einstein, one can invert the above system to get that the only non-zero coefficients of $R$ are $R_{1212}, R_{1313}, R_{2323}$, and they are all constant.
This implies that the sectional curvature is constant.
In higher dimension, $\mathrm{Ric}$ contains less information than $R$.
\end{digression}

\subsection{Orthogonal decompositions}
\label{sec:orthogonal-decompositions}

Take a Riemannian $n$-dimensional manifold $(M,g)$, and let $\nabla$ be any linear connection on $TM$. 
The formula
\begin{align*}
(\nabla_X g)(Y,Z) = X(g(Y,Z))-g(\nabla_XY,Z)-g(Y,\nabla_XZ)
\end{align*}
shows that $\nabla_Xg$ is a symmetric $2$-tensor, and hence $\nabla g$ takes values in $T^*M \otimes S^2T^*M$.
Let $V$ be the model of each tangent space of $M$, and recall $V$ comes as a representation of the orthogonal group $\mathrm{O}(n)$.
Pointwise $\nabla g$ takes values in $V \otimes S^2V$, where we have identified $V$ and its dual, as usual.

Let now $\omega$ be any tensor in $V \otimes S^2V$, and let $\omega_{ijk}$ be its components with respect to an orthonormal basis.
Define $Y_2^1$ to be the subspace of $V \otimes S^2V$ of tensors $\omega$ satisfying the three conditions
\[\omega_{ijk}+\omega_{jki}+\omega_{kij}= \sum_{\ell=1}^n \omega_{\ell \ell i} = \sum_{\ell=1}^n \omega_{i\ell \ell}=0.\]
\begin{proposition}{\cite[Expos\'e XVI]{besse2}}
The $\mathrm{O}(n)$-module $V \otimes S^2V$ decomposes into the direct sum of irreducible $\mathrm{O}(n)$-invariant summands
\[V \otimes S^2V = S_0^3V \oplus (V \oplus V) \oplus Y_2^1.\]
\end{proposition}
\begin{sketchproof}
The tensor product can be decomposed by hand.
There are five $\mathrm{O}(n)$-invariant quadratic forms on $V \otimes S^2V$, so according to Remark \ref{rmk:estimate-irreps} there at most five irreducible summands in $V \otimes S^2V$. Let $\omega \in V \otimes S^2V$, and write $\omega_{ijk}$ for its components with respect to an orthonormal basis. Then the five quadratic forms are computed from Definition \ref{def:products-traces} and turn out to be
\begin{alignat*}{2}
Q_1(\omega) & = \sum_{i,j,k} \omega_{ijk}^2, && \qquad Q_2(\omega) = \sum_{i,j,k} \omega_{ijk}\omega_{jik}, \\
Q_3(\omega) & = \sum_{i,j,k} \omega_{iij}\omega_{jkk}, && \qquad Q_4(\omega) = \sum_{i,j,k} \omega_{iij}\omega_{kjk}, \\
Q_5(\omega) & = \sum_{i,j,k} \omega_{ijj}\omega_{ikk}.
\end{alignat*}
Three of these quadratic forms correspond to the isotypic component $V\oplus V$. \qed
\end{sketchproof}
The result describes all possible values taken by $\nabla g$ for any linear connection on $M$.
Working with a metric connection is equivalent to be working in the trivial submodule in $\{0\} \subset V \otimes S^2V$.
We will always be interested in metric connections.

When $n$ is even, it may be that $(M,g)$ comes equipped with a compatible almost complex structure $J$, in which case we recall that $(g,J)$ is an almost Hermitian structure on $M$.
Take again $\nabla$ to be any connection on $TM$. 
Consider the two-form $\sigma \coloneqq g(J{}\cdot{},{}\cdot{})$.
The general formula
\[(\nabla_X\sigma)(Y,Z) = X(\sigma(Y,Z))-\sigma(\nabla_XY,Z)-\sigma(Y,\nabla_XZ)\]
shows that $\nabla_X\sigma$ is skew-symmetric, and hence $\nabla\sigma$ takes values in $T^*M \otimes \Lambda^2T^*M$.
If $V$ is the model of each tangent space, pointwise $\nabla \sigma$ takes values in $V \otimes \Lambda^2V$.
We are then led to consider the decomposition of $V \otimes \Lambda^2V$ under the action of the orthogonal group.

There is another case where we need to consider the same decomposition, namely when we work with torsion tensors on Riemannian manifolds, cf.\ Agricola \cite[Proposition 2.1]{agricola} and references therein.
We have seen in Theorem \ref{thm:torsion-curvature} that a torsion tensor $T$ gets two vectors $X$ and $Y$ as input and returns
\[T(X,Y)=\nabla_XY-\nabla_YX-[X,Y].\]
Using the metric $g$ on $M$, we can view the same tensor as a totally covariant tensor by
\[T(X,Y,Z) \coloneqq g(X,T(Y,Z)).\]
The skew-symmetry of $T$ tells us this new tensor takes values in $T^*M \otimes \Lambda^2T^*M$.
The following result follows from Besse \cite[Expos\'e XVI]{besse2}.
\begin{proposition}
\label{eq:torsion-tensors}
For $n \geq 3$, the $\mathrm{O}(n)$-module $V \otimes \Lambda^2V$ decomposes into the direct sum of irreducible $\mathrm{O}(n)$-invariant summands
\begin{equation*}
V \otimes \Lambda^2V = \Lambda^3V \oplus V \oplus Y_2^1.
\end{equation*}
\end{proposition}
\begin{sketchproof}
The proof proceeds as above, but we now only have three independent $\mathrm{O}(n)$-invariant quadratic forms.
If $\omega_{ijk}$ are the components of $\omega \in V \otimes \Lambda^2V$, then we have the forms
\begin{alignat*}{2}
Q_1(\omega) & = \sum_{i,j,k} \omega_{ijk}^2, && \qquad Q_2(\omega) = \sum_{i,j,k} \omega_{ijk}\omega_{jik}, \\
Q_3(\omega) & = \sum_{i} \omega_{iik}^2.
\end{alignat*}
So in this case there are three irreducible $\mathrm{O}(n)$-submodules. \qed
\end{sketchproof}
\begin{remark}
Note that the identity $\sum_{\ell} \omega_{i\ell\ell}=0$ characterising tensors in $Y_2^1$ is trivially verified in this case.
More generally, it can be proved that $V \otimes \Lambda^kV$ splits into irreducible orthogonal invariant summands \[V \otimes \Lambda^kV=\Lambda^{k+1}V \oplus \Lambda^{k-1}V \oplus Y_2^1,\] where $Y_2^1$ has a more general characterisation, see \cite[Expos\'e XVI]{besse2} for details.
\end{remark}
By using the latter result, we can distinguish eight classes of connections defined by the properties of their torsion tensor.
For instance, a connection $\nabla$ has \emph{vectorial torsion} if its torsion tensor takes values in $TM$, corresponding to the irreducible submodule $V$ in Proposition \ref{eq:torsion-tensors}.
Every metric connection on a surface is of vectorial type: this is easily seen by $V \otimes \Lambda^2V = V \otimes \mathbb R = V$ when $\dim V = 2$.
Connections with \emph{skew-symmetric torsion} are those whose torsion tensor lies in $\Lambda^3TM$, corresponding to the irreducible submodule $\Lambda^3V$ in Proposition \ref{eq:torsion-tensors}.
In physics, this type of connections appear in superstring theory, $M$-theory \cite{curio-kors-lust, duff, lust-theisen, strominger} and weak holonomy theories \cite{agricola-friedrich}.
We refer to the lecture notes of Agricola \cite[Section 2]{agricola} for further details, examples, and more references.

\subsection{Algebraic curvature tensors}
\label{susec:decomposition-algebraic-curvature-tensors}

In the following, we denote by $(V,\langle{}\cdot{},{}\cdot{}\rangle)$ the vector space $\mathbb R^n$ with its standard scalar product.
We take the whole digression on the Riemannian curvature tensor to a more abstract level by considering the subspace $\mathcal T$ of tensors in $S^2 \Lambda^2 V^*$ satisfying the first Bianchi identity. 
We call $\mathcal T$ the space of \emph{algebraic curvature tensors} on $V$.
In other words, we only look at the formal properties of Riemannian curvature tensors, regardless of any geometric context they might come from.
Since we look at algebraic properties, the second Bianchi identity does not play any role here, as it involves differential operations.
The material in this section can be found in \cite{berger-gauduchon-mazet, besse}.

The vector space $V$ with its Euclidean structure models each tangent space of a Riemannian manifold, so it comes as a representation of the orthogonal group $\mathrm{O}(V)$.
We can then identify the representations $V$ and $V^*$.
The action of $\mathrm{O}(V)$ on $S^2\Lambda^2V \subset V^{\otimes 4}$ is obtained by extending the action to all four factors.

Let $\mathcal S$ be the subspace of symmetric tensors $S^2V$ inside $V^{\otimes 2}$.
We define two maps corresponding to Ricci and scalar curvature respectively.
Set
\begin{alignat*}{2}
\rho & \colon \mathcal T \to \mathcal S, && \qquad \rho(R)(X,Y) = \mathrm{Tr}(R(X,{}\cdot{})Y{}\cdot{}), \\
\tau & \colon \mathcal T \to \mathbb R, && \qquad \tau(R) = \mathrm{Tr}(\rho(R)).
\end{alignat*}
Conversely, the Kulkarni--Nomizu product provides a way to construct an element of $\mathcal T \subset S^2\Lambda^2V$ from two elements of $S^2V$.

Let us now introduce the \emph{Bianchi map} $b \colon V^{\otimes 4} \to V^{\otimes 4}$ defined by 
\begin{equation*}
b(R)(X,Y,Z,W) \coloneqq \frac13 (R(X,Y,Z,W)+R(Y,Z,X,W)+R(Z,X,Y,W)).
\end{equation*}
It is clear that $b$ is $\mathrm{GL}(V)$-equivariant (i.e.\ $b(gR)=gb(R)$ for $g \in \mathrm{GL}(V)$), then one verifies $b^2=b$, and that $b$ maps $S^2\Lambda^2V$ into itself.
Also, one easily checks that $b(R)(X,Y,Z,W)$ is skew-symmetric in $X$ and $Z$, which implies $\mathrm{Im}b \subset \Lambda^4V$.
On the other hand, if $\alpha \in \Lambda^4V$, then $b(\alpha)=\alpha$, so $\mathrm{Im}b = \Lambda^4V$.
So we have a $\mathrm{GL}(V)$-decomposition
\[S^2\Lambda^2V \otimes \Lambda^kV = \ker b \oplus \mathrm{Im}b = \mathcal T \oplus \Lambda^4V.\]
\begin{remark}
\label{rmk:bianchi-identity-trivial-low-dimension}
If $\dim V = 2$ or $3$, then $\Lambda^4V=0$, and hence $b=0$. 
So the first Bianchi identity follows from the other symmetry conditions in these dimensions.
\end{remark}
We know that $\Lambda^4V$ is irreducible as an $\mathrm{O}(V)$-module, cf.\ Proposition \ref{prop:irreps-orthogonal-group}.
We then wish to understand how $\mathcal T$ decomposes under the action of $\mathrm{O}(V)$.
By identity \eqref{eq:formula-owedge}, we see that $g \owedge$ is a linear map $S^2V \to \mathcal T$.
Now if $\dim V > 2$, by evaluating $g \owedge k$ on an orthonormal basis $\{E_i\}$ and imposing $g \owedge k=0$, we find $k=0$ (recall that since $g \owedge k$ lies in $S^2\Lambda^2V$, it is enough to compute $(g\owedge k)(E_i,E_j,E_i,E_j)$). 
So $g \owedge \colon S^2V \to \mathcal T$ is injective.
Also, if $A \in \mathrm{O}(V)$, we have \[g \owedge Ak=Ag \owedge Ak=A(g \owedge k),\] so $g \owedge \colon S^2V \to \mathcal T$ is also equivariant, and $\mathcal T$ contains irreducible modules $(\mathbb Rg \oplus S_0^2V) \owedge g$.
We now show that $\mathcal T$ has exactly three orthogonal invariant submodules by enumerating quadratic invariants on $\mathcal T$.
There are at least three quadratic forms defined on $\mathcal T$, namely the squared norms (in terms of an orthonormal basis)
\begin{align*}
|R|^2 & = \sum_{i,j,k,\ell} R_{ijk\ell}^2, \quad |\rho(R)|^2 = \sum_{j,k}\biggl(\sum_i R_{ijik}\biggr)^2, \quad |\tau(R)|^2 = \biggl(\sum_{i,j} R_{ijij}\biggr)^2.
\end{align*}
We call these three quantities \emph{fundamental forms}.

\begin{theorem}
The vector space of $\mathrm{O}(V)$-invariant quadratic forms on $\mathcal T$ is generated by the three fundamental forms 
$|R|^2$, $|\rho(R)|^2$, $|\tau(R)|^2$. It is of dimension $1$, $2$, or $3$ according to $\dim V=2$, $3$, or more than $3$.
\end{theorem}
\begin{sketchproof}
A detailed discussion is found in \cite{berger-gauduchon-mazet}, but the proof essentially amounts to writing down all possible products of traces for $k=4$ and $h=2$, cf.\ Definition \ref{def:products-traces}.
The second statement follows by Digression \ref{digression:curvature-low-dimension}. \qed
\end{sketchproof}
As a consequence, we have the following result.
\begin{theorem}
If $\dim V \geq 4$, the $\mathrm{O}(V)$-module $\mathcal T$ decomposes into irreducible $\mathrm{O}(V)$-modules
\begin{equation}
\label{eq:decomposition-algebraic-curvature-tensors}
\mathcal T = \mathbb R \oplus \mathcal Z \oplus \mathcal W,
\end{equation}
where $\mathbb R = \mathbb R g \owedge g$, $\mathcal Z = g \owedge S_0^2V$, and $\mathcal W$ is the orthogonal complement of $\mathbb R \oplus \mathcal Z$ in $\mathcal T$.
\end{theorem}
\begin{proof}
There are four orthogonal quadratic invariants on $S^2\Lambda^2V=\mathcal T \oplus \Lambda^4V$, namely
\begin{alignat*}{2}
P_1(T) & \coloneqq \sum_{ijk\ell} (T_{ijk\ell})^2, && \qquad P_2(T) = \sum_{ijk\ell} T_{ijk\ell}T_{ikj\ell}, \\
P_3(T) & \coloneqq \sum_{ijk\ell} T_{iji\ell}T_{kjk\ell}, && \qquad P_4(T) = \sum_{ijk\ell}T_{ijij}T_{k\ell k\ell}.
\end{alignat*}
This follows by Definition \ref{def:products-traces} with $k=4$ and $h=2$, then considering the symmetries of tensors in $S^2\Lambda^2V$.
Since $\Lambda^4V$ is irreducible, $\mathcal T$ must be the direct sum of three invariant modules by Proposition \ref{prop:invariant-quadratic-forms}.
By the discussion above, $\mathcal T$ contains two irreducible invariant submodules isomorphic to $\mathbb R$ and $S_0^2V$, so the statement follows.
\end{proof}
\begin{remark}
By the formulas in the previous section, it follows that the $\mathbb R$-part of $R \in \mathcal T$ is given by 
\[\frac{s}{2n(n-1)} g \owedge g,\]
where $n=\dim V$. Also, $\rho(R)$ is a symmetric $2$-tensor by definition, so its traceless part is $\rho(R)-\frac{s}{n}g$. 
Since for $h \in S^2V$ we have the formula $\rho(g \owedge h) = (n-2)h+\mathrm{Tr}(h)g$, we then find that the $\mathcal Z$-part of $R$ is 
\[\frac{1}{n-2}\left(\rho(R)-\frac{s}{n}g\right) \owedge g.\]
Therefore, we can write the general formula
\begin{equation}
\label{eq:curvature-general-formula}
R = \frac{s}{2n(n-1)} g \owedge g+\frac{1}{n-2}\left(\rho(R)-\frac{s}{n}g\right) \owedge g+W.
\end{equation}
The tensor $W$ is the $\mathcal W$-component of $R$ and is called the \emph{Weyl tensor}.
This is invariant under conformal transformations $g \mapsto ug$, $u>0$ a positive function, cf.\ \cite{besse, besse2}.
Its expression can be derived by formula \eqref{eq:curvature-general-formula}.
\end{remark}

\begin{remark}
Consider a Riemannian $n$-dimensional manifold $(M,g)$, with $n \geq 3$.
Apply the orthogonal decomposition to the curvature tensor $R$ of $(M,g)$.
Suppose that the $\mathcal Z$-component of $R$ vanishes at each point of $M$.
This is equivalent to the metric $g$ being Einstein.
In fact, if $\mathrm{Ric}_p = \lambda_pg_p$ at all points $p \in M$ for constants $\lambda_p$, then $\mathrm{Ric} = \lambda g$ for a constant $\lambda$ in virtue of \cite[Theorem 1.97]{besse}. 

Suppose now that both the $\mathcal Z$- and the $\mathcal W$-components of $R$ vanish at each point of $M$.
Then in particular $(M,g)$ is Einstein and has constant scalar curvature. 
Then $(M,g)$ has constant sectional curvature.
\end{remark}

\subsection{Curvature in low dimension}
\label{subsec:curvature-in-low-dimension}

A curvature tensor in dimension greater than $3$ splits into the sum of three invariant components, and dimensions $2$ and $3$ are more special due to the relationship among the curvature tensor and its traces.
We then reinterpret our previous discussion on $R$ in dimension $2$ and $3$ using the orthogonal decomposition of $R$.
We also look at the decomposition of $R$ under the special orthogonal group in dimension $4$, which is interesting on oriented Riemannian manifolds.

Let $(M,g)$ be a Riemannian $n$-dimensional manifold, $R$ the curvature tensor of $g$, and let $p \in M$ be any point.
If $n=2$, then $S^2\Lambda^2T_p^*M$ is one-dimensional, and formula \eqref{eq:curvature-general-formula} implies \[R=\frac{s}{4}g \owedge g.\] 
It follows immediately that $\mathrm{Ric}=\frac{s}{2}g$.
If $n=3$, then $S^2\Lambda^2T_p^*M$ is six-dimensional, and formula \eqref{eq:curvature-general-formula} reads 
\[R=\frac{s}{12}g\owedge g+\left(\mathrm{Ric}-\frac{s}{3}g\right)\owedge g.\]
Since knowing $\mathrm{Ric}$ implies knowing $s$, we see that $\mathrm{Ric}$ determines $R$.
Applying again \cite[Theorem 1.97]{besse}, we have the following result.
\begin{proposition}
In dimension $2$ and $3$, a Riemannian manifold is Einstein if and only if it has constant sectional curvature.
\end{proposition}
Let us consider the decomposition of $\mathcal T$ with respect to the special orthogonal group.
In dimension greater than $4$, the decomposition \eqref{eq:decomposition-algebraic-curvature-tensors} is irreducible for $\mathrm{SO}(n)$, cf.\ Kirillov \cite{kirillov}.
When $\dim V = 4$, then $\Lambda^2V$ decomposes as $\Lambda^2V=\Lambda_+^2V\oplus \Lambda_-^2V$ (Example \ref{ex:self-dual-forms}), and hence we have an $\mathrm{SO}(4)$-decomposition
\[V \otimes V = \Lambda^2V \oplus S^2V = \Lambda_+^2V\oplus \Lambda_-^2V \oplus S_0^2V \oplus \mathbb R.\]

We mentioned in Example \ref{ex:self-dual-forms} that in dimension $4$ there is an equivalence of $\mathrm{O}(4)$-modules 
\[\Lambda_+^2V \otimes \Lambda_-^2V=S_0^2V,\]
so $\Lambda_+^2V \otimes \Lambda_-^2V$ is $\mathrm{O}(4)$-irreducible by Proposition \ref{prop:irreps-orthogonal-group}.
In Equation \eqref{eq:decomposition-algebraic-curvature-tensors}, we have seen that $\mathcal Z$ corresponds to $S_0^2V$, so $\mathcal Z$ corresponds to $\Lambda_+^2V \otimes \Lambda_-^2V$.
Now, we have the identity
\[\Lambda^2 \otimes \Lambda^2 = (\Lambda_+^2 \otimes \Lambda_+^2)\oplus (\Lambda_+^2\otimes \Lambda_-^2) \oplus (\Lambda_-^2 \otimes \Lambda_-^2),\]
whence
\begin{align}
\label{eq:first-decomposition-four-dimensions}
S^2\Lambda^2 & = S^2\Lambda_+^2 \oplus (\Lambda_+^2\otimes \Lambda_-^2)\oplus S^2\Lambda_-^2 \nonumber \\
& = S_0^2\Lambda_+^2 \oplus (\Lambda_+^2\otimes \Lambda_-^2)\oplus S_0^2\Lambda_-^2 \oplus \mathbb R \oplus \mathbb R,
\end{align}
where the $\mathbb R$ summands are generated by the identity elements $\mathrm{id}_{\pm}\colon \Lambda_{\pm}^2 \to \Lambda_{\pm}^2$, whose sum is $\frac12 g \owedge g$, cf.\ \eqref{eq:inner-product-exterior-algebra2}.
Note also that if $\alpha \in \Lambda^2 = \Lambda_+^2\oplus \Lambda_-^2$, then $\star \alpha = \star \alpha_++\star \alpha_- = \alpha_+-\alpha_-$, so the Hodge star operator corresponds to the difference of the identity maps $\mathrm{id}_{\pm}$.
Now, $\star$ maps $\mathbb R$ to $\Lambda^4V$, and hence the $\mathbb R \oplus \mathbb R$ part in \eqref{eq:first-decomposition-four-dimensions} corresponds to $\Lambda^4V \oplus \mathbb R g \owedge g$ in the general decomposition of $S^2\Lambda^2V$ in \eqref{eq:decomposition-algebraic-curvature-tensors}.
Since $\mathcal Z = \Lambda_+^2V \otimes \Lambda_-^2V$, the space of Weyl tensors must be $S_0^2\Lambda_+^2\oplus S_0^2\Lambda_-^2$.
Both summands here are irreducible $\mathrm{SO}(4)$-representations by the following result.

\begin{theorem}[Singer--Thorpe \cite{singer-thorpe}]
The $\mathrm{SO}(4)$-representation $S^2\Lambda^2V$ decomposes into the direct sum of irreducible $\mathrm{SO}(4)$-summands as
\[S^2\Lambda^2V = \mathbb R \oplus \mathbb R \oplus (\Lambda_+^2V \otimes \Lambda_-^2V) \oplus S_0^2\Lambda_+^2V \oplus S_0^2\Lambda_-^2V.\]
\end{theorem}
It follows that the space of Weyl tensors $\mathcal W$ splits into the direct sum of two irreducible submodules $\mathcal W_{\pm} \coloneqq S_0^2\Lambda_{\pm}^2V$, which is the only difference to the general case.
We have the following characterisations of the irreducible components above:
\begin{align*}
\mathcal ZV & = \{R \in S^2\Lambda^2V: \star R = -\star R\}, \\
\mathcal W_+V & = \{R \in S_0^2\Lambda^2V: \star R = R\star = R\}, \\ 
\mathcal W_-V& = \{R \in S_0^2\Lambda^2V: \star R = R\star = -R\}.
\end{align*}
\begin{corollary}
For a four-dimensional Riemannian manifold $(M,g)$, the following are equivalent:
\begin{itemize}
\item $(M,g)$ is Einstein, 
\item for any local orientation, $R \star = \star R$, 
\item for any $p \in M$ and any two-plane in $T_pM$, its sectional curvature is the same as the sectional curvature of its orthogonal two-plane in $T_pM$.
\end{itemize}
\end{corollary}
The above correspondence and its corollary are proved in \cite{besse2}, Expos\'e IX. We refer to Besse \cite{besse, besse2} for more applications. 
We also mention that a decomposition of the curvature tensor for the unitary group refining the orthogonal one can be found in Tricerri--Vanhecke \cite{tricerri-vanhecke}.

We conclude this section with a well-known result relating the curvature and the topology of a manifold.
Complete, connected Einstein spaces with positive scalar curvature are automatically compact with finite fundamental group.
This is a consequence of the following result.
\begin{theorem}[Bonnet--Myers \cite{myers}]
\label{thm:bonnet-myers}
Let $(M,g)$ be a complete, connected Riemannian $n$-dimensional manifold.
Suppose there is a constant $r>0$ such that 
\[\mathrm{Ric} \geq \frac{n-1}{r^2}g.\]
Then $\mathrm{diam}(M,g) \leq \pi r$, and $M$ is compact with finite fundamental group.
\end{theorem}
Here $\mathrm{diam}(M,g)$ is the supremum of the lengths of the curves joining any two given points in $(M,g)$.
This result will be useful for certain considerations later.

\newpage
\section{Holonomy}
\label{sec:holonomy}

A connection on a manifold is used to transport geometric data along a curve, and thus connect tangent spaces.
This allows one to define covariant differentiation.
On Riemannian manifolds we have a canonical metric and torsion-free connection, so we have a natural notion of parallel transport, and a corresponding Riemannian holonomy group.
Metric connections with torsion are of interest as well.

Hereafter we discuss \emph{Riemannian holonomy}, and we see how to compute it via the \emph{General Holonomy Principle}.
The idea behind this principle is illustrated on a number of $G$-structures.
For details on the material see e.g.\ Besse \cite[Chapter 10]{besse} or Joyce \cite[Section 2]{joyce}.

\subsection{Riemannian holonomy}
\label{subsec:riemannian-holonomy}
We start by reviewing the notion of parallel transport in the Riemannian context, specifically for the Levi-Civita connection.
Certain considerations can be generalised to different types of linear connections.

Let $(M,g)$ be a Riemannian $n$-manifold, and $\nabla$ be the Levi-Civita connection on it.
If $K$ is a tensor field on $M$, then $K$ is parallel along a curve $\gamma$ when $\nabla_{\dot{\gamma}} K=0$.
Here $\dot{\gamma}$ is the time-derivative of $\gamma$.
We remark that $\gamma$ can be taken piecewise smooth.
An analysis of $\nabla_{\dot{\gamma}}K=0$ in a local coordinate system where $\gamma$ is smooth shows that this is a system of ordinary differential equations in the first derivatives of $K$.
\begin{example}
\label{ex:geodesics}
Let $(x^1,\dots,x^n)$ be a local coordinate system on an $n$-dimensional Riemannian manifold $(M,g)$.
Let $\nabla$ be the Levi-Civita connection, and $K$ a vector field on $M$.
Suppose that $\gamma$ is a smooth curve with support contained in the above coordinate chart.
Let $\partial_{\mu}$ be the coordinate vector field $\partial/\partial x^{\mu}$.
Using the Einstein convention, we can write
\begin{equation*}
\nabla_{\dot{\gamma}}K = \nabla_{\dot{\gamma}^{\mu}\partial_{\mu}}(K^{\nu}\partial_{\nu}) = \dot{\gamma}^{\mu}\nabla_{\partial_{\mu}}(K^{\nu}\partial_{\nu}) = \dot{\gamma}^{\mu}(\partial_{\mu}K^{\lambda}+K^{\nu}\Gamma_{\mu\nu}^{\lambda})\partial_{\lambda},
\end{equation*}
where $\nabla_{\partial_{\mu}}\partial_{\nu} \eqqcolon \Gamma_{\mu\nu}^{\lambda}\partial_{\lambda}$. Thus $K$ is parallel along $\gamma$ when $\dot{\gamma}^{\mu}(\partial_{\mu}K^{\lambda}+K^{\nu}\Gamma_{\mu\nu}^{\lambda})=0$ for all $\lambda$ and $t$ (all functions involved in the latter identity are evaluated at $\gamma(t)$).
A special case occurs when $K = \dot{\gamma}$: the above parallel condition then becomes
\[\nabla_{\dot{\gamma}}\dot{\gamma}=\ddot{\gamma}^{\lambda}+\dot{\gamma}^{\mu}\dot{\gamma}^{\nu}\Gamma_{\mu\nu}^{\lambda}=0.\]
A curve $\gamma$ satisfying this equation is called a \emph{geodesic}.
Geodesics in Riemannian manifolds are the equivalent of straight lines in a flat space.
A geodesic is uniquely determined by a point $p \in M$ and a tangent vector $X \in T_pM$, and one writes $t \mapsto \exp_p(tX)$ for the corresponding curve (see also the exponential map in the Riemannian context \cite{kobayashi-nomizu}).
\end{example}
Take $\gamma \colon [0,1] \to M$ to be a smooth curve with support inside a given coordinate system, and assume $\gamma(0)=p$ and $\gamma(1)=q$.
If the value of a tensor $K$ is specified at $p$, then Cauchy's Theorem guarantees that there is a unique local solution of $\nabla_{\dot{\gamma}}K=0$, and hence a unique value of $K$ at any point in the support of $\gamma$.
This defines a function from the tensor algebra of $T_pM$ to the tensor algebra of $T_{\gamma(t)}M$ for all $t\in (0,1]$, in particular of $T_qM$.
The map $\tau_{\gamma} \colon T_pM \to T_qM$ obtained in this way is a parallel transport map.

If $X$ and $Y$ are two tangent vectors at a point $p \in M$, and $\gamma$ is a curve as above joining $p$ to another point $q \in M$, we can transport $X$ and $Y$ along $\gamma$ in a parallel way until we reach $q$.
The metric condition for $\nabla$ implies that the inner product of $X$ and $Y$ remains constant along $\gamma$, as
\begin{equation}
\label{eq:scalar-products-preserved-under-parallel-transport}
\dot{\gamma}(g(X,Y))=g(\nabla_{\dot{\gamma}}X,Y)+g(X,\nabla_{\dot{\gamma}}Y)=0.
\end{equation}
If in particular $\gamma$ is taken to be a loop based at a point $p \in M$, $\tau_{\gamma}$ is an isometry of $T_pM$.
In this case, moving backwards along $\gamma$ gives the inverse of $\tau_{\gamma}$, and concatenating loops gives a composition operation of parallel transport maps.
Recall Definition \ref{def:linear-holonomy} of the holonomy group for $\nabla$.
By \eqref{eq:scalar-products-preserved-under-parallel-transport}, $\mathrm{Hol}_p(\nabla)$ is a subgroup of $\mathrm{O}(n)$, and is a Lie group by a result of Freudenthal \cite{kobayashi-nomizu}. 
\begin{definition}
The representation of $\mathrm{Hol}_p(\nabla) \to \mathrm{O}(T_pM)$ on the tangent space $T_pM$ is called the \emph{holonomy representation}.
\end{definition}
\begin{remark}
It is clear that the holonomy group does not come as an abstract Lie group, but as a Lie group with a natural orthogonal representation on each tangent space.
Recall that $\nabla$ is equivalent to a connection on $LM$ compatible with the $\mathrm{O}(n)$-structure induced by $g$.
The image of the holonomy group of the latter connection (as in Definition \ref{def:principal-hol}) via the associated bundle construction is exactly $\mathrm{Hol}_p(\nabla)$ (see Joyce \cite[Proposition 2.3.7]{joyce}).
\end{remark}
Mapping the homotopy class of any loop $\gamma$ at $p$ to the parallel transport map $\tau_{\gamma}$ gives a group homomorphism
\begin{equation}
\pi_1(M,p) \mapsto \mathrm{Hol}_p(\nabla)/\mathrm{Hol}_p^0(\nabla),
\end{equation}
which is obviously surjective. 

If we assume $M$ to be connected, it is easy to see that holonomy groups at different points are isomorphic (in fact, they are conjugate to one another), so that one can speak of \emph{the} holonomy group $\mathrm{Hol}(\nabla)$ (defined up to conjugation) without reference point.
In this sense, the holonomy group is a global invariant of the Riemannian manifold.

If we further assume $M$ to be simply connected, and thus orientable, then the holonomy group is a connected group, as loops can be shrunk to a point.
Then $\mathrm{Hol}(\nabla)=\mathrm{Hol}^0(\nabla)$, so $\mathrm{Hol}(\nabla)$ must sit in $\mathrm{SO}(n) \subset \mathrm{O}(n)$.

\begin{remark}
We remark that $\mathrm{Hol}^0(g)$ is a \emph{closed} subgroup of $\mathrm{SO}(n)$, but this is a non-trivial fact.
The holonomy group of a Riemannian manifold need not be closed---see Wilking \cite{wilking} for examples of compact Riemannian manifolds with non-compact holonomy.
\end{remark}

By using the theory of principal bundles, one can prove a deep relationship between the curvature tensor of a Riemannian metric and the holonomy Lie algebra. 
This is stated precisely in the following theorem, cf.\ Kobayashi--Nomizu \cite[Chapter II, Theorem 9.1]{kobayashi-nomizu}, or Ambrose--Singer \cite{ambrose-singer} for the original paper.
In the Riemannian case, the result can be stated as follows for any metric connection (which may or may not be the Levi-Civita connection).
\begin{theorem}[Ambrose--Singer, 1953]
\label{thm:ambrose-singer}
Let $(M,g)$ be a Riemannian manifold, $\nabla$ any linear connection, $R$ its curvature tensor, and $\mathrm{Hol}_p(\nabla)$ the holonomy group of $\nabla$ at $p \in M$.
The Lie algebra of $\mathrm{Hol}_p(\nabla)$ is equal to the subspace of endomorphisms of $T_pM$ spanned by all elements of the form \[\tau_{\gamma}^{-1} \circ R_{\tau_{\gamma}X,\tau_{\gamma}Y}\circ \tau_{\gamma},\] where $X,Y \in T_pM$, and $\tau_{\gamma}$ is the parallel transport map along a piecewise smooth loop $\gamma$ based at $p$.
\end{theorem}
\begin{remark}
\label{rmk:curvature-ambrose-singer}
To make this result plausible, recall Proposition \ref{rmk:lambda-so} and the symmetries of $R$.
We had $R_{X,Y} \in \Lambda^2 TM$, so pointwise $R_{X,Y} \in \mathfrak{so}(n)$.
On the other hand, the Lie algebra $\mathfrak{hol}_p(\nabla)$ of $\mathrm{Hol}_p(\nabla)$ sits inside $\mathfrak{so}(n)$.
Since the trace is invariant under conjugation, the endomorphisms \[\tau_{\gamma}^{-1} \circ R_{\tau_{\gamma}X,\tau_{\gamma}Y}\circ \tau_{\gamma}\] are traceless, and their skew-symmetry is obvious as $\tau_{\gamma}$ is an isometry.
The theorem essentially states that $R_p \in S^2(\mathfrak{hol}_p(\nabla))$ at each point $p \in M$.
The proof presented is taken from Besse \cite[Chapter 10]{besse}.
\end{remark}
\begin{proof}
Fix an orthonormal frame $f_0$ at $p \in M$. 
Consider the frames $f$ obtained by transporting $f_0$ in a parallel way along all curves starting from $p$.
The set $P$ of such frames is a principal bundle over $M$ with structural group $\mathrm{Hol}_p(\nabla)$.
Denote by $\mathfrak g$ the Lie subalgebra of $\mathfrak{o}(n)$ built up as follows.
Given $p \in M$, vectors $X, Y \in T_pM$, and an orthonormal frame $f$ of $P$ at $p$, the curvature endomorphism $R_{X,Y}$ and $f$ generate an element in $\mathfrak{o}(n)$.
Let $\mathfrak g$ be the Lie subalgebra of $\mathfrak{o}(n)$ generated by all such elements when $p$ runs through $M$, $X$ and $Y$ through $T_pM$, and $f$ through the frames of $P$ at $p$.
Then $\mathfrak g$ is a subalgebra of $\mathfrak{hol}_p(\nabla)$. 

We claim $\mathfrak{hol}_p(\nabla)=\mathfrak g$.
Consider the following distribution $D$ of $P$.
At every $f \in P$, the vector space $D(f)$ is the direct sum of the horizontal subspace of $P$ at $f$ and of the vertical one defined by $\mathfrak g$ at $f$.
The distribution $D$ turns out to be integrable in the sense of Frobenius \cite[Chapter I]{kobayashi-nomizu}.
Consider the maximal leaf $L(f_0)$ through $f_0$.
By the construction of $P$ and the definition of $D$, we have $L(f_0)=P$, and therefore $\mathfrak{hol}_p(\nabla)=\mathfrak g$.
\end{proof}
\begin{remark}
The theorem has a theoretical importance in that it establishes a link between the holonomy and the curvature of a Riemannian manifold.
However, it is not too effective to compute the holonomy group, cf.\ Besse \cite[Section 10.61]{besse}.
\end{remark}

\subsection{The General Holonomy Principle}
\label{subsec:the-general-holonomy-principle}

One of the reasons why the holonomy group of a manifold is a remarkable object is that the invariants of the holonomy representation correspond to covariantly constant tensor fields on the manifold.
So studying the holonomy of a connection $\nabla$ is equivalent to studying tensors which are constant with respect to $\nabla$.
This allows one to compute the holonomy group effectively.

The following is also known as the \emph{General Holonomy Principle}, cf.\ Agricola \cite[Theorem 2.7]{agricola}, Besse \cite[Chapter 10, 10.19]{besse}, or Joyce \cite[Proposition 2.5.2]{joyce}.
We state the result in the context of Riemannian geometry as in Besse, but as one can see from the other references, the same result holds for more general linear connections as well.
\begin{proposition}
\label{prop:general-holonomy-principle}
Let $(M,g)$ be a connected Riemannian manifold, and $\nabla$ the Levi-Civita connection. The following are equivalent:
\begin{enumerate}
\item There is a tensor field $S$ on $M$ which is invariant under parallel transport.
\item There is a tensor field $S$ on $M$ which is parallel, i.e.\ $\nabla S = 0$.
\item There are a point $p \in M$ and a tensor $S_p$ on $T_pM$ which is invariant under the holonomy representation on the tensor algebra of $T_pM$.
\end{enumerate}
\end{proposition}
\begin{proof}
Let $S$ be a tensor field invariant under parallel transport. 
Then $\nabla S = 0$, cf.\ subsection \ref{subsec:associated-connections}, in particular \eqref{eq:covariant-derivative}.

Assume now $S$ is a covariantly constant tensor field.
Let $p \in M$ and let $S_p$ be the value of $S$ at $p$.
Take any element $\tau_{\gamma} \in \mathrm{Hol}(\nabla)$. 
Since $\nabla S = 0$, then $\tau_{\gamma}S_p = S_p$. 

Finally, let $S_p$ on $T_pM$ be any tensor which is invariant under the holonomy representation on the tensor algebra of $T_pM$.
For any $q \in M$ and any curve $\gamma$ joining $p$ and $q$, let $S_q \coloneqq \tau_{\gamma}(S_p)$. 
We claim that the definition of $S$ does not depend on the choice of $\gamma$. 
If $\delta$ is another curve from $p$ to $q$, set \[S_q' \coloneqq \tau_{\delta}(S_p).\]
Then $\tau_{\delta}(S_p) = \tau_{\gamma}\tau_{\gamma^{-1}\delta}(S_p)$, and since $\gamma^{-1}\delta$ is a loop, $\tau_{\gamma^{-1}\delta}$ acts trivially on $S_p$. Hence
\[S_q' = \tau_{\delta}(S_p) = \tau_{\gamma\gamma^{-1}\delta}(S_p) = \tau_{\gamma}\tau_{\gamma^{-1}\delta}(S_p) = \tau_{\gamma}(S_p) = S_q.\]
Then $S$ is clearly invariant under parallel transport.	
\end{proof}
\begin{corollary}
\label{cor:holonomy-principle}
The number of parallel global tensor fields on $M$ coincides with the number of trivial representations of the holonomy representation on the fibres.
Under the same assumptions of Proposition \ref{prop:general-holonomy-principle}, the holonomy group $\mathrm{Hol}(\nabla)$ is a subgroup of $\{g \in \mathrm{O}(n): g\alpha = \alpha\}$, where $\alpha$ is a tensor corresponding to any parallel section on $(M,g)$.
\end{corollary}
By using the General Holonomy Principle, we introduce a number of geometric structures which are classical. 
These can all be characterised by the vanishing of the Levi-Civita covariant derivative of certain structure tensors.
\begin{example}
Let $(M,g)$ be an orientable Riemannian manifold equipped with a metric connection.
In general, its holonomy group is a subgroup of the orthogonal group $\mathrm{O}(n)$.
Let $p \in M$ and fix an orthonormal basis of $T_pM$, which we identify with a dual basis.
We can then define a volume form on $T_pM$. 
Now, define a volume form on $M$ by parallel transport of the volume form on $T_pM$.
By Corollary \ref{cor:holonomy-principle}, the holonomy group of $(M,g)$ reduces to a subgroup of $\mathrm{SO}(n)$.
Conversely, if the holonomy reduces to $\mathrm{SO}(n)$, then it leaves invariant a top degree global form, and $M$ is necessarily orientable.
\end{example}
\begin{example}
A Riemannian manifold is flat if and only if there is a parallel local frame around any point.
By the General Holonomy Principle the holonomy group is discrete. 
This can also be seen by the Ambrose--Singer Theorem, as the holonomy algebra is trivial in this case.
\end{example}
\begin{example}
\label{ex1:kahler-manifolds}
Let $(M,g)$ be a Riemannian $2n$-dimensional manifold.
Assume $M$ comes with an almost complex structure $J$ compatible with $g$ and such that $\nabla J=0$, where $\nabla$ is the Levi-Civita connection.
In this case, the holonomy group of $(M,g)$ reduces to $\mathrm{U}(n) = \mathrm{SO}(2n) \cap \mathrm{GL}(n,\mathbb C)$.
The condition $\nabla J=0$ implies that the almost complex structure is integrable, i.e.\ its Nijenhuis tensor vanishes, cf.\ Example \ref{ex:integrability-complex-structure}.
In fact, since $\nabla$ is torsion-free we can write
\begin{align*}
N_J(X,Y) & = \nabla_XY-\nabla_YX-\nabla_{JX}JY+\nabla_{JY}JX \\
& \qquad +J(\nabla_{JX}Y-\nabla_YJX)+J(\nabla_XJY-\nabla_{JY}X) \\
& = \nabla_XY-\nabla_YX-J\nabla_{JX}Y+J\nabla_{JY}X \\
& \qquad +J(\nabla_{JX}Y-J\nabla_YX)+J(J\nabla_XY-\nabla_{JY}X) =0.
\end{align*}
Thus $M$ gets a complex structure.
Recall that $\sigma\coloneqq g(J{}\cdot{},{}\cdot{})$ is the fundamental two-form of the almost Hermitian structure.
Then $d\sigma = \mathcal A(\nabla \sigma) = \mathcal A(g((\nabla_{{}\cdot{}} J){}\cdot{},{}\cdot{})=0$, where $\mathcal A$ is the skew-symmetrisation operator (cf.\ \cite[Chapter III, Corollary 8.6]{kobayashi-nomizu}).
It follows that $\sigma$ is a symplectic form, so $M$ gets a symplectic structure as well.

Conversely, if the holonomy reduces to $\mathrm{U}(n)$, then there is a complex structure on $(M,g)$ which is invariant under parallel transport for the Levi-Civita connection.
We will say more on this geometry in subsection \ref{subsec:berger-theorem}.
\end{example}
The following definition was anticipated in Example \ref{ex:integrable-kahler-str}.
\begin{definition}
\label{def:kahler-manifolds}
An almost Hermitian $2n$-dimensional manifold with holonomy $\mathrm{U}(n)$ is called \emph{K\"ahler}.
\end{definition}
\begin{example}
\label{ex:complex-volume-form}
On a K\"ahler manifold $(M,g,J)$ of real dimension $2n$, a \emph{complex volume form} $\Psi$ is by definition an exterior form of type $(n,0)$ with $\Psi \neq 0$ and parallel with respect to the Levi-Civita connection.
The holonomy group in this case reduces to $\mathrm{SU}(n) \subset \mathrm{U}(n)$.
Conversely, the General Holonomy Principle implies that if the holonomy reduces to $\mathrm{SU}(n)$, then there is a parallel complex volume form $\Psi$ on $M$.
\end{example}
\begin{definition}
\label{def:calabi-yau}
An almost Hermitian $2n$-dimensional manifold with holonomy $\mathrm{SU}(n)$ is called \emph{Calabi--Yau}.
\end{definition}
The holonomy reduction to a special unitary group has remarkable consequences on the Ricci curvature.
\begin{proposition}[Iwamoto \cite{iwamoto}, Lichnerowicz \cite{lichnerowicz}]
\label{prop:special-kahler}
Let $(M,g,J)$ be a K\"ahler manifold of complex dimension $n$, and let $\nabla$ be the Levi-Civita connection. 
Consider the conditions
\begin{enumerate}
\item $(M,g)$ is Ricci flat,
\item there is a complex volume form on $(M,g,J)$,
\item $\mathrm{Hol}(\nabla) \subset \mathrm{SU}(n)$.
\end{enumerate}
Then (2) and (3) are equivalent and imply (1). Conversely, (1) implies that the restricted holonomy of $(M,g)$ reduces to $\mathrm{SU}(m)$.
Moreover, if $M$ is simply connected, (1) implies (2) and (3).
\end{proposition}
If property (1) above is satisfied, one sometimes says that $(M,g,J)$ is a \emph{special K\"ahler manifold}.
Note that the terminology \lq\lq special K\"ahler\rq\rq\ is also used for K\"ahler manifolds with specific torsion-free connections, see e.g.\ \cite{freed}.

\begin{example}
\label{ex:almost-hyperhermitian}
Let $(M,g)$ be a $4n$-dimensional Riemannian manifold equipped with three almost complex structures $I$, $J$, $K$ such that $IJ=-JI=K$.
The structure $(g,I,J,K)$ on $M$ is called \emph{almost hyperHermitian}.
The compact symplectic group $\mathrm{Sp}(n) \subset \mathrm{O}(4n)$ acts on each tangent space preserving the entire structure.
It follows by the General Holonomy Principle that the holonomy of the Levi-Civita connection $\nabla$ reduces to $\mathrm{Sp}(n)$ if and only if $I$, $J$, and $K$ are $\nabla$-parallel.
\end{example}
\begin{definition}
\label{def:hyperkahler-manifold}
An almost hyperHermitian $4n$-dimensional manifold with holonomy $\mathrm{Sp}(n)$ is called \emph{hyperK\"ahler}.
\end{definition}
\begin{remark}
Since $\mathrm{Sp}(n) \subset \mathrm{SU}(2n)$, hyperK\"ahler manifolds are special K\"ahler manifolds, and thus Ricci flat by Proposition \ref{prop:special-kahler}.
\end{remark}

\begin{example}
On a $4m$-manifold $(M,g,I,J,K)$ as in Example \ref{ex:almost-hyperhermitian}, assume that $I$, $J$, and $K$ are not parallel with respect to the Levi-Civita connection $\nabla$.
Let us instead require that there are three one-forms $\alpha$, $\beta$, $\gamma$ locally defined on $(M,g)$ such that for every $X$ we have the identities
\begin{align*}
\nabla_X I & = +\gamma(X)J-\beta(X)K, \\
\nabla_XJ & = -\gamma(X)I+\alpha(X)K, \\
\nabla_XK & = +\beta(X)I-\alpha(X)J.
\end{align*}
The three endomorphisms $I$, $J$, $K$ define locally a rank $3$ subbundle of the vector bundle of endomorphisms over $M$.
Define $\omega_L \coloneqq g(L{}\cdot{},{}\cdot{})$, for $L \in \{I,J,K\}$. 
The above equations imply
\begin{align*}
\nabla_X (\omega_I\wedge \omega_I) & = 2g(\nabla_XI{}\cdot{},{}\cdot{}) \wedge \omega_I = 2(\gamma(X)\omega_J\wedge \omega_I-\beta(X)\omega_K\wedge \omega_I), \\
\nabla_X (\omega_J\wedge \omega_J) & = 2g(\nabla_XJ{}\cdot{},{}\cdot{}) \wedge \omega_J = 2(-\gamma(X)\omega_I\wedge \omega_J+\alpha(X)\omega_K\wedge \omega_J), \\
\nabla_X (\omega_K\wedge \omega_K) & = 2g(\nabla_XK{}\cdot{},{}\cdot{}) \wedge \omega_K = 2(\beta(X)\omega_I\wedge \omega_K-\alpha(X)\omega_J\wedge \omega_K),
\end{align*}
and the sum of left left-hand sides is then
\[\nabla_X(\omega_I^2+\omega_J^2+\omega_K^2) = 0.\]
The four-form \[\omega_I^2+\omega_J^2+\omega_K^2\] is called the \emph{fundamental four-form} of the almost quaternionic Hermitian structure on $M$, and is globally defined.
In this case, the holonomy of the Levi-Civita connection is $\mathrm{Sp}(1)\mathrm{Sp}(m)$.
\end{example}
\begin{definition}
\label{def:quaternionic-kahler-manifold}
An almost hyperHermitian $4n$-manifold with holonomy $\mathrm{Sp}(n)\mathrm{Sp}(1)$ is called \emph{quaternionic K\"ahler}.
\end{definition}
\begin{remark}
Quaternionic K\"ahler manifolds are automatically Einstein for $m\geq2$, see Besse \cite[Theorem 14.39]{besse} or Berger \cite{berger2}.
The quaternionic projective space $\mathbb HP^m$ is an example of quatenionic K\"ahler (not hyperK\"ahler) manifold.
\end{remark}

\begin{remark}
\label{rmk:quaternion-kahler-not-kahler}
Unlike $\mathrm{Sp}(m)$, the group $\mathrm{Sp}(m)\mathrm{Sp}(1)$ is not a subgroup of $\mathrm{U}(2m)$ (cf.\ Dynkin \cite{dynkin}), so quaternionic K\"ahler manifolds are \emph{not} K\"ahler.
In fact, $\mathrm{Sp}(m)\mathrm{Sp}(1)$ is a maximal subgroup of $\mathrm{SO}(4m)$, see Gray \cite{gray}.
\end{remark}
\begin{example}
A $\mathrm G_2$-structure on a seven-manifold $M$ is a three-form pointwise equivalent to a standard three-form on $\mathbb R^7$, whose stabiliser in $\mathrm{GL}(7,\mathbb R)$ is exactly $\mathrm G_2$.
A $\mathrm G_2$-structure induces a Riemannian metric $g$ on $M$ in a canonical way.
A seven-dimensional manifold $M$ with a $\mathrm G_2$-structure $\varphi$ has holonomy contained in $\mathrm G_2$ if and only if $\nabla \varphi=0$, where $\nabla$ is the Levi-Civita connection for $g$.
It was shown by Bonan \cite{bonan} that a seven-manifold with holonomy in $\mathrm G_2$ is Ricci flat.
\end{example}
\begin{example}
A $\mathrm{Spin}(7)$-structure on a Riemannian eight-dimensional manifold $(M,g)$ is a four-form $\psi$ pointwise equivalent to a standard four-form on $\mathbb R^8$, whose stabiliser in $\mathrm{GL}(8,\mathbb R)$ is exactly $\mathrm{Spin}(7)$.
An eight-dimensional manifold $M$ with a $\mathrm{Spin}(7)$-structure $\psi$ has holonomy contained in $\mathrm{Spin}(7)$ if and only if $\nabla \psi=0$, where $\nabla$ is the Levi-Civita connection for $g$.
It follows again by Bonan \cite{bonan} that a holonomy $\mathrm{Spin}(7)$ manifold is Ricci flat.
\end{example}
When the holonomy is any of the groups $\mathrm{U}(n)$, $\mathrm{SU}(n)$, $\mathrm{Sp}(n)$, $\mathrm{Sp}(n)\mathrm{Sp}(1)$, we say it is \emph{special}.
It if $\mathrm G_2$ or $\mathrm{Spin}(7)$, we say it is \emph{exceptional}.
If it is not contained in any of these groups, we say it is \emph{generic}.

\subsection{Symmetric spaces}
\label{subsec:symmetric-spaces}

Consider the following question: classify Riemannian manifolds $(M,g)$ whose Riemannian curvature $R$ is parallel with respect to the Levi-Civita connection, i.e.\ $\nabla R = 0$.
This problem was completely solved by \'E.\ Cartan.
By the General Holonomy Principle, having parallel curvature is equivalent to the holonomy representation leaving the curvature tensor invariant.
One is then led to an algebraic problem about curvature tensors and orthogonal representations.
Here below we give a brief summary of Cartan's results.

\begin{definition}
A Riemannian manifold $(M,g)$ is called \emph{irreducible} if it cannot be written locally as a Riemannian product of two or more Riemannian factors.
\end{definition}
Let us restrict to the case where $(M,g)$ is irreducible, complete, and simply connected.
Cartan found two distinct series of possible cases.
Each element of a series is a pair of spaces, a compact one and a non-compact one.
\begin{itemize}
\item The first series is given by pairs $(G,G^{\mathbb C})$, where $G$ is a compact, simple, simply connected Lie group, and $G^{\mathbb C}$ the complex group associated to $G$ (see Remark \ref{rmk:complex-group} below).
\item The second series is given by pairs $(G/H,G^*/H)$, where $G$ is a non-compact, simple, simply connected Lie group, $G^*$ the compact form of $G$ (cf.\ Remark \ref{rmk:compact-real-form} below), and $H$ a maximal compact connected subgroup of $G$ (unique up to conjugation). 
\end{itemize}
\begin{remark}
\label{rmk:complex-group}
A \emph{complex Lie group} is a complex analytic manifold whose operations of multiplication and inversion are holomorphic.
The group $G^{\mathbb C}$ can be viewed as the group of $\mathbb C$-algebra homomorphisms $\mathcal F(G,\mathbb C) \to \mathbb C$, where $\mathcal F(G,\mathbb C)$ is a $\mathbb C$-algebra of so-called \emph{representative functions}, cf.\  \cite[Chapter III, Section 8]{brocker-tomdieck}. 
There is a map $\iota \colon G \to G^{\mathbb C}$ such that $\iota(g)$ is evaluation at $g$. 
The group $G^{\mathbb C}$ is called \emph{complexification} of $G$.
The Lie algebra of $G^{\mathbb C}$ coincides with the complexified Lie algebra $\mathfrak g \otimes \mathbb C$, where $\mathfrak g$ is the Lie algebra of $G$.
For instance, $\mathrm{U}(n)^{\mathbb C} = \mathrm{GL}(n,\mathbb C)$, $\mathrm{SU}(n)^{\mathbb C} = \mathrm{SL}(n,\mathbb C)$, $\mathrm{O}(n)^{\mathbb C} = \mathrm{O}(n,\mathbb C)$, $\mathrm{SO}(n)^{\mathbb C} = \mathrm{SO}(n,\mathbb C)$, $\mathrm{Sp}(n)^{\mathbb C} = \mathrm{Sp}(n,\mathbb C)$.
Note that $\mathrm{O}(n,\mathbb C)$ is defined as $\{A \in \mathrm{GL}(n,\mathbb C): A^TA=\mathrm{id}\}$, and $\mathrm{SO}(n,\mathbb C) = \mathrm{O}(n,\mathbb C) \cap \mathrm{SL}(n,\mathbb C)$.
Also, $\mathrm{Sp}(n,\mathbb C) = \{A \in \mathrm{GL}(2n,\mathbb C): A^T\Omega A=\Omega\}$, where $\Omega$ is as in \eqref{eq:omega}.
\end{remark}
\begin{remark}
\label{rmk:compact-real-form}
Every complex semisimple Lie algebra has a split real form (unique up to isomorphism) and a compact real form (see the Cartan decomposition of a semisimple Lie algebra \cite{helgason, knapp}).
The latter gets a negative-definite Killing form, and corresponds to a compact Lie group.
For instance, the Lie algebra $\mathfrak{sp}(n,\mathbb C)$ has a split real form $\mathfrak{sp}(n,\mathbb R)$ and a compact real form $\mathfrak{sp}(n)$.
See also the correspondence with Satake diagrams \cite{satake}.
\end{remark}
It follows that there is a link between manifolds with $\nabla R = 0$ and real forms of simple Lie groups, which were classified by Cartan himself \cite{cartan}.
So the classification of Riemannian manifolds with $\nabla R = 0$ is reduced to the classification of real forms of simple Lie algebras (cf.\ Helgason \cite{helgason} for a modern proof).

Before stating Cartan's results, let us define what a geodesic symmetry is.
Recall Example \ref{ex:geodesics}.
\begin{definition}
\label{def:geodeisc-symmetry}
Let $(M,g)$ be a Riemannian manifold, and let $p \in M$ be any point.
A \emph{geodesic symmetry} $s_p$ at $p$ is defined as a diffeomorphism of a neighbourhood $U \ni p$ into itself, fixing $p$, and mapping $\exp_p(X)$ to $\exp_p(-X)$, where $X \in T_pM$.
\end{definition}
\begin{remark}
In Definition \ref{def:geodeisc-symmetry} there is actually no dependence on the neighbourhood $U$, as the definition of $s_p$ involves infinitesimal data only.
The symmetry $s_p$ is \emph{involutive}, i.e.\ $s_p \circ s_p = \mathrm{id}_p$, and by definition its differential at $p $ is $-\mathrm{id}_{T_pM}$.
\end{remark}
\begin{definition}
Let $(M,g)$ be a Riemannian manifold. We say that $(M,g)$ is \emph{locally symmetric} if its geodesic symmetries are isometries.
We say that $(M,g)$ is \emph{globally symmetric} if its geodesic symmetries can be extended to global isometries on $M$.
\end{definition}
Cartan's results can be summarised as follows, see Besse \cite[Section 10.72]{besse}.
\begin{theorem}[Cartan]
\label{thm:cartan}
For a Riemannian manifold $(M,g)$, the following conditions are equivalent.
\begin{enumerate}
\item The curvature tensor of $g$ is parallel with respect to the Levi-Civita connection, i.e.\ $\nabla R = 0$.
\item $(M,g)$ is locally symmetric.
\end{enumerate}
Furher, if $(M,g)$ is complete, the following are equivalent.
\begin{enumerate}[resume]
\item $(M,g)$ is globally symmetric.
\item The space $M$ is homogeneous, i.e.\ $M$ is diffeomorphic to $G/H$, where $G$ is a connected Lie group, $H$ a compact subgroup of $G$, and there is an involutive automorphism $\sigma$ of $G$ for which, if $S$ denotes the fixed point set of $\sigma$ and $S_e$ its connected component of the identity, one has $S_e \subset H \subset S$. Furthermore, the Riemannian metric on $G/H$ is $G$-invariant.
\end{enumerate}
If $(M,g)$ satisfies (3)--(4), it also satisfies (1)--(2).
If $(M,g)$ satisfies (1) or (2) and if it is simply connected and complete, then it also satisfies (3)--(4).
\end{theorem}
The connection with real forms of simple Lie algebras comes from (4), and the involution is the one coming from complex conjugation of the complex form, cf.\ \cite{helgason}.
We refer to \cite{besse} for a discussion on the proof of Theorem \ref{thm:cartan}.

A symmetric space decomposes uniquely into a Riemannian product of \emph{irreducible symmetric spaces}, i.e.\ symmetric spaces that cannot be decomposed further into products of symmetric ones.
Also, the universal cover of a simply connected symmetric space is a symmetric space.
Let us then assume $(M,g)$ is a simply connected irreducible symmetric space.
\begin{proposition}[Besse \cite{besse}]
\label{prop:holonomy-symmetric-spaces}
For an irreducible simply connected symmetric space $G/H$, the holonomy group is $H$, acting by the adjoint representation.
\end{proposition}
The main point for us is that for symmetric spaces the holonomy group is essentially known by Proposition \ref{prop:holonomy-symmetric-spaces}.
We refer to \cite{besse, helgason} for more details and explicit examples.

\subsection{Berger's Theorem}
\label{subsec:berger-theorem}

We now consider a second question: what are the possible holonomy groups of Riemannian manifolds? 
In order to give a concise answer to this question we impose some harmless restrictions.

First, we state the following result, cf.\ de Rham \cite{derham}, or Kobayashi--Nomizu \cite[Volume 1, Chapter IV, Section 6]{kobayashi-nomizu}.
\begin{theorem}[de Rham decomposition]
\label{thm:derham}
A connected, simply connected, complete Riemannian manifold $(M,g)$ is isometric to the direct product $M_0 \times M_1 \times \dots \times M_k$, where $M_0$ is a Euclidean space (possibly zero-dimensional), and $M_1,\dots,M_k$ are all simply connected, complete, irreducible Riemannian manifolds. 
\end{theorem}
\begin{remark}
\label{rmk:derham}
The holonomy group of $(M,g)$ then admits a decomposition $\{1\} \times H_1 \times \dots \times H_k$.
That $M_i$ is \emph{irreducible} here means also that $H_i$ acts irreducibly on each tangent space of $M_i$ (see e.g.\ \cite[Corollary 3.2.5]{joyce}).
\end{remark}

We have seen that the holonomy group of a Riemannian manifold $(M,g)$ is connected when $M$ is simply connected.
So the holonomy group will be a subgroup of a special orthogonal group.
By Theorem \ref{thm:derham} and Remark \ref{rmk:derham}, we can safely assume $(M,g)$ is irreducible.
Finally, the holonomy groups of symmetric spaces are deduced by the classification of symmetric spaces by Cartan, so we can assume $(M,g)$ is not locally symmetric.

We are now ready to state the answer to the question above. 
The following is also known as \emph{Berger's Theorem}, and the list of groups in the statement is known as \emph{Berger's list}.
\begin{theorem}[Berger, 1955 \cite{berger}]
\label{thm:berger-thm}
Let $(M,g)$ be a simply connected Riemannian $n$-dimensional manifold.
Assume $(M,g)$ is irreducible and not locally symmetric. 
Then the holonomy group of $(M,g)$ equals one of the following groups:
\begin{itemize}
\item $\mathrm{SO}(n)$,
\item $\mathrm{U}(m) \subset \mathrm{SO}(2m)$, for $n=2m$ and $m \geq 2$,
\item $\mathrm{SU}(m) \subset \mathrm{SO}(2m)$, for $n=2m$ and $m \geq 2$,
\item $\mathrm{Sp}(m) \subset \mathrm{SO}(4m)$, for $n=4m$ and $m \geq 2$,
\item $\mathrm{Sp}(m)\mathrm{Sp}(1) \subset \mathrm{SO}(4m)$, for $n=4m$ and $m \geq 2$,
\item $\mathrm G_2 \subset \mathrm{SO}(7)$, for $n=7$,
\item $\mathrm{Spin}(7) \subset \mathrm{SO}(8)$, for $n=8$.
\end{itemize}
\end{theorem}
\begin{sketchproof}
Let $\nabla$ be the Levi-Civita connection.
Since $M$ is simply connected, the holonomy group is connected, so $\mathrm{Hol}(\nabla)$ is a closed connected subgroup of $\mathrm{SO}(n)$.
Since $\mathrm{Hol}(\nabla)$ acts irreducibly on $\mathbb R^n$, the holonomy representation is completely determined by the dimension $n$ (up to conjugacy by an orthogonal transformation).
The classification of Lie groups \cite{fulton-harris} then yields a list of candidates for $\mathrm{Hol}(\nabla)$, but the Ambrose--Singer Theorem \ref{thm:ambrose-singer} imposes further restrictions (see Remark \ref{rmk:curvature-ambrose-singer}).
Specifically, $R \in S^2(\mathfrak{hol}(g))$, and $R$ satisfies two Bianchi identities, cf.\ subsection \ref{subsec:riemannian-manifolds}.
Berger's list is obtained by testing all possible groups satisfying these restrictions.
We also refer to Merkulov--Schwachh\"ofer \cite{merkulov-schwachhofer} for more details.
\end{sketchproof}
\begin{remark}
The restrictions on $m$ are imposed to avoid repetitions or trivial cases. 
For instance, for $m=1$ we have $\mathrm{U}(1)=\mathrm{SO}(2)$, $\mathrm{SU}(1)$ is trivial, $\mathrm{Sp}(1)=\mathrm{SU}(2)$, and $\mathrm{Sp}(1)\mathrm{Sp}(1)=\mathrm{SU}(2)\mathrm{SU}(2)=\mathrm{SO}(4)$.
\end{remark}
\begin{remark}
A different proof by Simons \cite{simons} relies on the fact that, under the same assumptions of the theorem, $\mathrm{Hol}(\nabla)$ must act transitively on the unit sphere in $\mathbb R^n$.
The list of groups acting transitively and effectively on spheres was found by Montgomery--Samelson \cite{montgomery-samelson} (cf.\ Borel \cite{borel}).
This list contains more groups than those appearing in Berger's list (i.e.\ $\mathrm{Sp}(m)\mathrm{U}(1)$ and $\mathrm{Spin}(9)$), so a proof of the theorem is obtained by showing that these two cases cannot occur, cf.\ Besse \cite[pag.\ 303--305]{besse} and Salamon \cite[pag.\ 149--151]{salamon}.
\end{remark}
\begin{remark}
As for Simons, the original statement from Berger's work \cite[Theorem 3, p.\ 318]{berger} contains $\mathrm{Spin}(9)$ in dimension $16$.
It was proved by Alekseevskii \cite{alekseevskii} and Brown--Gray \cite{brown-gray} that a Riemannian manifold of dimension $16$ with holonomy group $\mathrm{Spin}(9)$ is necessarily symmetric.
\end{remark}
Berger's result in an abstract result on Riemannian holonomy groups that may occur as holonomy groups of Riemannian manifolds.
The problem of finding explicit examples of Riemannian manifolds with holonomy groups those appearing in Berger's list has a longer history, which we now discuss.
The \lq\lq holonomy $H$\rq\rq\ condition can be rephrased by using the General Holonomy Principle (Proposition \ref{prop:general-holonomy-principle}) in terms of the vanishing of the Levi-Civita covariant derivative of certain structure tensors.
Also, having special holonomy imposes curvature restrictions, as one would expect from the Ambrose--Singer Theorem.

\begin{example}
\label{ex:kahler-geometry}
Case $\mathrm{Hol}(\nabla) = \mathrm{U}(m)$. 
These are K\"ahler manifolds, and we have already seen that they get a complex and symplectic structure, see Example \ref{ex1:kahler-manifolds}.
Examples of K\"ahler manifolds are $\mathbb C^n$ with its standard Hermitian structure, tori $\mathbb C^n/(\mathbb Z^n+i\mathbb Z^n)$ with metric inherited by $\mathbb C^n$, or the complex projective space $\mathbb CP^n$ with the \emph{Fubini--Study metric}. See \cite[Vol.\ 2]{kobayashi-nomizu} or \cite{griffiths-harris} for more details.
\end{example}

\begin{example}
\label{ex:calabi-yau}
Case $\mathrm{Hol}(\nabla)=\mathrm{SU}(m)$, the so-called Calabi--Yau metrics.
Clearly these are K\"ahler metrics, with the additional data of a complex (parallel) volume form, so in this sense they are special K\"ahler, cf.\ Example \ref{ex:complex-volume-form}.
Proposition \ref{prop:special-kahler} implies Calabi--Yau metrics are Ricci flat.
The first incomplete examples of special K\"ahler manifolds were given by Calabi in 1960 \cite{calabi}.
Complete (unpublished) examples were obtained again by Calabi in 1970.
Compact ones came only later in 1978 with the solution of the Calabi conjecture by Yau \cite{yau}.
For a non-trivial example with holonomy exactly $\mathrm{SU}(m)$ see \cite[p.\ 302]{besse}. 
The four-dimensional case ($m=2$), K3 surfaces and tori are the only compact Calabi--Yau manifolds.
The six-dimensional case ($m=3$) is particularly relevant in superstring theory, as it is conjectured to provide a model for a compactification of the four-dimensional space-time to a ten-dimensional manifold \cite{hu}.
\end{example}

\begin{example}
Case $\mathrm{Hol}(\nabla)=\mathrm{Sp}(m)$, the so-called hyperK\"ahler manifolds. 
Recall that in this case we have a triple of complex structures $(I,J,K)$ compatible with $g$ and satisfying the quaternionic relations.
The inclusion $\mathrm{Sp}(m)\subset \mathrm{SU}(2m)$ implies that any complex structure in the two-sphere given by $I$, $J$, $K$ is K\"ahler.
Again as in the K\"ahler case, the three two-forms $\sigma_I$, $\sigma_J$, $\sigma_K$ are symplectic.
Ricci flatness follows from the inclusion $\mathrm{Sp}(m)\subset \mathrm{SU}(2m)$.
An obvious example is $\mathbb H^n$ with its flat Euclidean metric. 
The first non-trivial example, found by Eguchi and Hanson \cite{eguchi-hanson}, is a hyperK\"ahler metric on the cotangent bundle of the two-sphere.
The result was generalised by Calabi, who showed that the cotangent bundle of $\mathbb CP^m$ admits a complete non-flat hyperK\"ahler metric \cite{calabi}.
Any compact hyperK\"ahler four-dimensional manifold is either a K3 surface or a compact torus, essentially because of the isomorphism $\mathrm{Sp}(1)=\mathrm{SU}(2)$.
Compact examples of hyperK\"ahler eight-dimensional manifolds were obtained in 1981. 
For higher $m$, examples were found later, see Beauville \cite{beauville} for references.
\end{example}

\begin{example}
Case $\mathrm{Hol}(\nabla)=\mathrm{Sp}(m)\mathrm{Sp}(1)$, namely quaternionic K\"ahler manifolds.
As discussed in Remark \ref{rmk:quaternion-kahler-not-kahler}, these are not K\"ahler metrics. 
However, they are examples of Einstein, not Ricci-flat, metrics (Ricci flatness would imply that the holonomy be $\mathrm{Sp}(m)$).
At present, there are no known examples of compact quaternionic K\"ahler manifolds that are not locally symmetric. 
Symmetric examples were classified by Wolf, which is why they are also called \emph{Wolf spaces} \cite{wolf}.
The quaternionic projective space $\mathbb HP^m$ is an example of Wolf space.
A conjecture of LeBrun--Salamon asserts that all complete quaternionic-K\"ahler manifolds of positive scalar curvature are symmetric.
For negative scalar curvature, Galicki--Lawson \cite{galicki-lawson} and LeBrun \cite{lebrun} showed that not locally-symmetric examples exist in abundance.
\end{example}

\begin{example}
Case $\mathrm{Hol}(\nabla)=\mathrm G_2$ or $\mathrm{Spin}(7)$, the so-called exceptional cases.
The three-form and the four-form defined by $\mathrm G_2$- and $\mathrm{Spin}(7)$-structures respectively are parallel in these two cases.
This implies Ricci flatness \cite{bonan}.
The existence of Riemannian spaces with exceptional holonomy was settled in the positive only in the 1980s.
Incomplete metrics with exceptional holonomy were first found by Bryant \cite{bryant}.
Complete examples were found by Bryant--Salamon \cite{bryant-salamon}.
Compact examples were found by Joyce \cite{joyce2, joyce3} by algebro-geometric methods.
\end{example}

The above geometries are sometimes called \emph{integrable}, as the \lq\lq holonomy $H$\rq\rq\ condition is equivalent to the integrability of the corresponding $H$-structure.
Note that apart from the cases corresponding to $\mathrm{SO}(n)$ and $\mathrm G_2$, all other integrable geometries appear in even dimension.

It may happen that the Levi-Civita connection does not preserve the entire structure, so the geometry is non-integrable.
For instance, on an almost Hermitian $2n$-dimensional manifold $(M,g,J)$, $\nabla J$ need not vanish identically.
If $\sigma$ is the fundamental two-form of the structure, $\nabla \sigma$ is equivalent to $\nabla J$, but it is totally covariant and hence a little easier to handle.
Let $V$ be the model of each tangent space of $M$.
Pointwise, $\nabla \sigma$ sits in $V^{\otimes 3}$, but also satisfies certain symmetries.
Since the unitary group $\mathrm{U}(n)$ acts on $V$ preserving the Hermitian structure induced by $g$ and $J$, it is possible to study its representation on the space of tensors satisfying the same symmetries of $\nabla \sigma$ inside $V^{\otimes 3}$.
This gives complete information on the possible values taken by $\nabla \sigma$, and yields a list of non-integrable $\mathrm{U}(n)$-geometries. 
We shall see more in the next section.

\newpage

\section{Non-integrable geometries}
\label{sec:non-integrable-geometries}

We take a step beyond the integrable geometries singled out by Berger and glance at the world of \emph{non-integrable} geometries.
These are described by $G$-structures on Riemannian manifolds where the Levi-Civita connection does not preserve the entire structure.

We first discuss general classification results of non-integrable geometries.
Then we present two examples of non-integrable geometries, which are called \emph{nearly K\"ahler} and \emph{nearly parallel $\mathrm G_2$}.
The former is most relevant in dimension $6$, the latter only occurs in dimension $7$.
These geometries are relevant in the construction of holonomy $\mathrm G_2$ and $\mathrm{Spin}(7)$ manifolds.
We provide several references, and in particular follow Bryant \cite{bryant} for the construction of holonomy $\mathrm G_2$ and $\mathrm{Spin}(7)$ manifolds.

\subsection{Classification results}
\label{subsec:classification-results}

There are many classification results of non-integrable geometries.
We mention some of them here below, together with relevant references.
Here we are mostly interested in $\mathrm{U}(n)$-structures and $\mathrm G_2$-structures.
We illustrate the general idea behind this sort of classification results for these two cases only, the others are essentially analogous.

\begin{example}
The classification of $\mathrm{U}(n)$-structures $(g,J)$ in dimension $2n$ is due to Gray--Hervella \cite{gray-hervella}.
As anticipated, this is obtained by considering the symmetries of the tensor $\nabla \sigma$, where $\nabla$ is the Levi-Civita connection and $\sigma$ is the fundamental two-form of the structure.
We discuss this example in more detail below.
\end{example}
\begin{example}
The classification of $\mathrm{SU}(n)$-structures in dimension $2n$ is due to Cabrera \cite{cabrera}.
It is shown in particular that all the information on the intrinsic torsion of an $\mathrm{SU}(n)$-structure is contained in the exterior derivatives of the fundamental two-form and the complex volume form.
The low dimensional case $n=3$ is particularly interesting for geometric and physical reasons, see e.g.\ Chiossi--Salamon \cite{chiossi-salamon} on the intrinsic torsion of $\mathrm{SU}(3)$-structures, and Bedulli--Vezzoni \cite{bedulli-vezzoni} on the Ricci-tensor for $\mathrm{SU}(3)$-structures.
A spinorial description of $\mathrm{SU}(3)$-structures is given in \cite{agricola-et-al}.
\end{example}
\begin{example}
The classification of $\mathrm{Sp}(n)$-structures was done by Cabrera--Swann \cite{cabrera-swann}.
It is shown that there are at most 144 classes of $\mathrm{Sp}(2)$ geometries in dimension $8$, and at most 167 classes of $\mathrm{Sp}(n)$ geometries in dimension $4n>8$.
Information on the curvature for $4n=8$ is found in Salamon \cite{salamon}. 
\end{example}
\begin{example}
The classification of $\mathrm{Sp}(n)\mathrm{Sp}(1)$-structures was done by Cabrera \cite{cabrera3, cabrera-swann3}.
It turns out there are $16$ classes of $\mathrm{Sp}(2)\mathrm{Sp}(1)$ geometries, and $64$ classes of $\mathrm{Sp}(n)\mathrm{Sp}(1)$ geometries for $n>2$.
Swann \cite{swann, swann2} proved that the information on the covariant derivative of the fundamental four-form is contained in its exterior derivative.
For information on curvature see Cabrera--Swann \cite{cabrera-swann2}.
\end{example}
\begin{example}
The classification of $\mathrm G_2$-structures in dimension $7$ was initiated by Fern\'andez--Gray \cite{fernandez-gray}, and then continued by Fern\'andez--Iglesias \cite{fernandez-iglesias}, Cabrera \cite{cabrera2}, and Cabrera--Monar--Swann \cite{cabrera-monar-swann}. The classification yields $16$ classes of $\mathrm G_2$ geometries.
A study of the curvature tensor for $\mathrm G_2$-manifolds is given by Cleyton--Ivanov \cite{cleyton-ivanov}, see also Bryant \cite{bryant1} for the Ricci tensor of $\mathrm G_2$-structures.
\end{example}
\begin{example}
The classification of $\mathrm{Spin}(7)$-structures in dimension $8$ was done by Fern\'andez \cite{fernandez}.
There are only $4$ classes of $\mathrm{Spin}(7)$ geometries.
Information on curvature tensor and Ricci tensor are found in Salamon \cite{salamon} or Karigiannis \cite{karigiannis3}.
\end{example}
For an account on the intrinsic torsion of these structures, see e.g.\ Salamon \cite{salamon} and Fino \cite{fino}.
In odd dimension, we mention the classification of $\mathrm{U}(n) \times 1$-structures by Chinea--Gonzalez \cite{chinea-gonzales}.
These are called \emph{almost contact metric manifolds}, and are beyond the scope of these notes.
It is shown there are $2^{12}$ classes of $\mathrm{U}(n) \times 1$-structures in dimension $2n+1$.

Let us now describe in more detail the cases of $\mathrm{U}(n)$ and $\mathrm G_2$.
On a $2n$-dimensional manifold with a $\mathrm{U}(n)$-structure $(g,J)$ and fundamental two-forms $\sigma$, the covariant derivative $\nabla \sigma$ has the following symmetries:
\begin{enumerate}
\item $(\nabla_X \sigma)(Y,Z) = -(\nabla_X \sigma)(Z,Y)$,
\item $(\nabla_X \sigma)(JY,JZ) = -(\nabla_X\sigma)(Y,Z)$.
\end{enumerate}
Therefore, pointwise $\nabla \sigma$ sits in $\mathbb R^{2n} \otimes [\![\Lambda^{2,0}]\!]$, where $[\![\Lambda^{2,0}]\!] \subset \Lambda^2 \mathbb R^{2n}$ is the $-1$-eigenspace of $J$, cf.\ subsection \ref{subsec:unitary-groups}.
The tensor product $\mathcal W \coloneqq \mathbb R^{2n} \otimes [\![\Lambda^{2,0}]\!]$ is a $\mathrm{U}(n)$-module, and can thus be decomposed into the direct sum of irreducible representations.
By applying an adapted theory of invariant quadratic forms, Gray--Hervella \cite{gray-hervella} showed that in general $\mathcal W$ splits into the orthogonal direct sum of four irreducible summands:
\[\mathcal W = \mathcal W_1 \oplus \mathcal W_2 \oplus \mathcal W_3 \oplus \mathcal W_4.\]
Therefore, $\nabla \sigma \in \mathcal W$ splits accordingly into the sum of four invariant summands, which may or may not be zero.
We then have $2^4=16$ possibilities for the values of $\nabla \sigma$, and one defines sixteen classes of almost Hermitian geometries accordingly.
Gray and Hervella provide explicit descriptions of each of these classes.
Note that the trivial submodule $\{0\} \subset \mathcal W$ corresponds to the K\"ahler case, as $\nabla \sigma = 0$, or equivalently $\nabla J = 0$.
We will be interested in the first summand $\mathcal W_1$, in which case $\nabla \sigma$ is totally skew-symmetric, or equivalently $\nabla J$ is skew-symmetric.
In this case, the geometry is called \emph{nearly K\"ahler}.

On a seven-dimensional manifold with a $\mathrm G_2$-structure specified by a three-form $\varphi$, one can consider the covariant $(4,0)$-tensor $\nabla \varphi$.
Here $\nabla$ is the Levi-Civita connection for the Riemannian metric $g$ induced by $\varphi$. 
A first property of $\nabla \varphi$ we note is the skew-symmetry in the last three arguments:
\[(\nabla_X\varphi)(Y,Z,W) = -(\nabla_X\varphi)(Z,Y,W)=-(\nabla_X\varphi)(Y,W,Z).\]
This implies $\nabla \varphi \in \mathbb R^7 \otimes \Lambda^3\mathbb R^7$ pointwise.
Recall that in the proof of Theorem \ref{thm:g2}, a cross product was defined by the identity
\[\varphi(X,Y,Z) \eqqcolon g(X\times Y,Z).\]
Fern\'andez and Gray \cite{fernandez-gray} consider the space $\mathcal T$ of all algebraic tensors $\alpha$ in $\mathbb R^7 \otimes \Lambda^3\mathbb R^7$ such that
\[\alpha(X,Y \wedge Z \wedge (Y \times Z))=0.\]
Clearly $\mathcal T$ is a $\mathrm G_2$-module and $\nabla \varphi$ sits in it, see \cite[Lemma 5.1]{fernandez-gray}.
Again, $\mathcal T$ splits as the orthogonal direct sum of four invariant irreducible $G_2$-modules
\[\mathcal T = \mathcal T_1 \oplus \mathcal T_2 \oplus \mathcal T_3 \oplus \mathcal T_4.\]
The trivial submodule $\{0\} \subset \mathcal T$ corresponds to holonomy $\mathrm G_2$.
Again, we have sixteen classes of $\mathrm G_2$ geometries, and explicit descriptions are provided in \cite{fernandez-gray}.
We remark that the result can be rephrased in terms of $d\varphi$ and $d(\star \varphi)$, with $\star \varphi$ the Hodge star of $\varphi$.
In particular, the holonomy $\mathrm G_2$ case corresponds to having both $d\varphi=0=d(\star \varphi)$.
We will be interested in the module $\mathcal T_1$. 
If $\nabla \varphi$ takes non-zero values only in $\mathcal T_1$ then $\nabla \varphi = \lambda \star \varphi$, where $\lambda$ is a non-zero constant.
The class $\mathcal T_1$ is called the class of \emph{nearly parallel $\mathrm G_2$ structures}.
We now review some of the literature on nearly K\"ahler and nearly parallel $\mathrm G_2$ manifolds.

\subsection{Nearly K\"ahler manifolds}
\label{subsec:nearly-kahler-manifolds}

Consider an almost Hermitian manifold $(M,g,J)$, where $g$ is a Riemannian metric, and $J$ a compatible almost complex structure.
Let $\sigma \coloneqq g(J{}\cdot{},{}\cdot{})$ be the fundamental two-form of the structure, and $\nabla$ the Levi-Civita connection.
Following Gray--Hervella \cite{gray-hervella}, $\nabla J$ may be skew-symmetric, i.e.\ $(\nabla_XJ)X=0$ for all vector fields $X$ on $M$.
\begin{definition}[Gray \cite{gray4}]
\label{def:nk}
We say that $(M,g,J)$ is a \emph{nearly K\"ahler} manifold if $\nabla J$ is skew-symmetric.
We say that $(M,g,J)$ is a \emph{strict nearly K\"ahler} manifold if it is nearly K\"ahler and not K\"ahler.
\end{definition}
We now show that the every nearly K\"ahler manifold of dimension $2$ or $4$ is necessarily K\"ahler.

The following is a general fact for almost Hermitian structures. A proof can be found in \cite[Lemma 1.2.1]{russo}, but it is essentially an exercise.
\begin{lemma}
\label{lemma:moving-j}
For each triple $X,Y,Z$ of vector fields on $(M,g,J)$ we have
\[\nabla\sigma(JX,Y,Z)=\nabla\sigma(X,JY,Z)=\nabla\sigma(X,Y,JZ).\]
\end{lemma}

\begin{proposition}[Gray \cite{gray2}]
If $\dim M = 2$ or $4$, a nearly K\"ahler structure on $M$ is K\"ahler.
\end{proposition}
\begin{proof}
Assume $(g,J)$ is a nearly K\"ahler structure on $M$, and that $\dim M = 2$. 
Let $X$ be any non-zero vector field on an open subset $U \subset M$.
Then $JX$ is orthogonal to $X$ at each point of $U$, and $(\nabla_XJ)JX=-J(\nabla_XJ)X=0$.
This implies $\nabla J=0$.

Assume then $\dim M = 4$. 
Take $X,JX,Y,JY$ orthonormal fields on an open subset of $M$.
We have $g((\nabla_XJ)Y,X)=-\nabla\sigma(Y,X,X)=0$, and similarly $g(\nabla_XJ)Y,Y)=0$.
Then $(\nabla_XJ)Y$ must be orthogonal to $X$ and $Y$.
By Lemma \ref{lemma:moving-j}, $g((\nabla_XJ)Y,JX)=-\nabla\sigma(JY,X,X)=0$, and similarly $g((\nabla_XJ)Y,JY)=\nabla\sigma(JX,Y,Y)=0$.
It follows that $\nabla J=0$.
\end{proof}

In dimension $6$, an example of strict nearly K\"ahler structure is given by an almost Hermitian structure $(g,J)$ on the six-sphere $S^6$ built via the octonions.
An explicit construction is due to Fr\"olicher \cite{frolicher}. 
Subsequently, Fukami and Ishihara \cite{fukami-ishihara} proved that $\nabla J$ is skew-symmetric.
We discuss this nearly K\"ahler structure in detail here below.

In dimension $6$, a nearly K\"ahler structure is a special $\mathrm{U}(3)$-structure.
A result by Carri\'on \cite{carrion} states that this $\mathrm{U}(3)$-structure reduces further to an $\mathrm{SU}(3)$-structure with special intrinsic torsion.
\begin{theorem}[Carri\'on \cite{carrion}]
Let $(M,g,J)$ be an almost Hermitian six-dimensional manifold, and $\sigma = g(J{}\cdot{},{}\cdot{})$ be the fundamental two-form of the structure.
Then $(M,g,J)$ is nearly K\"ahler if and only if there exist a constant function $\mu$ on $M$ and a complex $(3,0)$-form $\psi_{\mathbb C} = \psi_++i\psi_-$ such that
\begin{equation}
\label{eq:nk-str-eq}
d\sigma = 3\mu\psi_+, \qquad d\psi_- = -2\mu\sigma \wedge \sigma.
\end{equation}
\end{theorem}
The result should be compared with the general torsion equations for $\mathrm{SU}(3)$-structures \cite{chiossi-salamon}.
Up to scaling the metric, one can set $\mu=1$, so we can assume the nearly K\"ahler structure equations \eqref{eq:nk-str-eq} to be of the form
\begin{equation}
\label{eq:nk-str-eq-scaled}
d\sigma=3\psi_+, \qquad d\psi_-=-2\sigma \wedge \sigma.
\end{equation}
So in dimension $6$ we can think of a nearly K\"ahler structure as a particular $\mathrm{SU}(3)$-structure $(\sigma,\psi_{\pm})$ (cf.\ Digression \ref{digression:su3}) satisfying the structure equations \eqref{eq:nk-str-eq-scaled}.
Note that $\psi_+$ essentially measures the failure of the manifold to be symplectic.

\begin{theorem}[Gray, \cite{gray3}]
\label{thm:nk-einstein}
Nearly K\"ahler six-manifolds are Einstein with positive scalar curvature.
\end{theorem}
Specifically, a nearly K\"ahler six-manifold characterised by the structure equations \eqref{eq:nk-str-eq-scaled} (resp.\ \eqref{eq:nk-str-eq}) satisfies $\mathrm{Ric}=5g$ (resp.\ $\mathrm{Ric}=5\mu^2g$).
A slightly streamlined and self-contained proof of this fact, together with equivalent definitions of nearly K\"ahler manifolds, is found in \cite{russo}.

\begin{remark}
Theorem \ref{thm:nk-einstein} partly explains the interest for nearly K\"ahler manifolds by physicists \cite{friedrich-ivanov}.
Another fact that makes these manifolds interesting for physicists is that they come with \emph{real Killing spinors} \cite{bar, bfgk, grunewald}, but we do not go into details in this direction.
We only remark that the existence of real Killing spinors imposes restrictions on the Weyl tensor of $(M,g)$ and implies Theorem~\ref{thm:nk-einstein}, cf.\ \cite{bfgk}.
\end{remark}
\begin{remark}
Since $\mathrm{SU}(3) \subset \mathrm{U}(3)$, an $\mathrm{SU}(3)$-structure on $M$ extends uniquely to a $\mathrm{U}(3)$-structure by Remark \ref{rmk:extension-g-structures}.
It follows that specifying $(\sigma,\psi_{\pm})$ satisfying \eqref{eq:nk-str-eq-scaled} is equivalent to specifying $(g,J)$ as in Definition \ref{def:nk}.
\end{remark}

We now summarise the construction of the standard nearly K\"ahler structure on $S^6$. 
This explicit description is found e.g.\ in \cite{russo2}.
We view $S^6$ inside $\mathbb R^7 = \mathrm{Im}\mathbb O$ in the standard way.
Recall from subsection \ref{subsec:on-g2-spin7} that, by definition, the group $\mathrm G_2$ is the subgroup of $\mathrm{GL}(7,\mathbb R)$ acting on $\mathbb R^7$ preserving the three-form
\[\varphi=e^{123}+e^{145}+e^{167}+e^{246}-e^{257}-e^{347}-e^{356},\]
where $\{e^i\}_{i=1,\dots,7}$ can be taken to be the dual basis of the standard basis of $\mathbb R^7$.
The form $\varphi$ induces the standard scalar product $\langle{}\cdot{},{}\cdot{}\rangle$ and a volume form on $\mathbb R^7$, so a Hodge star operator $\star$.
Also, $\mathrm G_2$ acts on $\mathbb R^7$ as a subgroup of $\mathrm{SO}(7)$, and there is a transitive action of $\mathrm G_2$ on the unit sphere $S^6\subset \mathbb R^7$.
The stabiliser of the point $(1,0,\dots,0) \in S^6$ is isomorphic to the special unitary group $\mathrm{SU}(3)$, whence a diffeomorphism $S^6=\mathrm G_2/\mathrm{SU}(3)$.

The sphere $S^6$ inherits a Riemannian metric via the standard immersion $\iota \colon S^6 \hookrightarrow \mathbb R^7$.
We now construct an almost complex structure on $S^6$.
Define the cross-product $\times \colon \mathbb R^7 \times \mathbb R^7 \to \mathbb R^7$ by 
\[\langle X\times Y,Z\rangle \coloneqq \varphi(X,Y,Z).\]
Let $N$ be the unit normal vector field to $S^6 \subset \mathbb R^7$, and define $J\colon \mathbb R^7 \to \mathbb R^7$ by $JX \coloneqq N \times X$.
Take any point $p \in S^6$ and $X \in T_pS^6$. Then $\langle JX,N\rangle = \langle N \times X,N\rangle = \varphi(N,X,N)=0$, so $J$ is a linear map $T_pS^6 \to T_pS^6$.
It can then be checked that this restriction of $J$ to the tangent bundle of $S^6$ is such that $J^2=-\mathrm{id}$ pointwise.
\begin{proposition}
The differential forms \[\sigma \coloneqq g(J{}\cdot{},{}\cdot{}), \qquad \psi_+\coloneqq \iota^*\varphi, \qquad \psi_-\coloneqq -\iota^*(N\lrcorner \star \varphi)\] give a nearly K\"ahler structure on the six-sphere.
\end{proposition}
\begin{proof}
The proof is taken from \cite[Proposition 3.1]{russo2}.
Let $(x^k)$, $k=1,\dots,7$, be the standard coordinates on $\mathbb R^7$.
Let the one-form $dx^k$ be the dual of the coordinate vector field $\partial_{x^k}$ for all $k$.
We then have
\[\varphi=dx^{123}+dx^{145}+dx^{167}+dx^{246}-dx^{257}-dx^{347}-dx^{356}.\]
Let $N=\sum_{k=1}^7 x^k\partial_{x^k}$, with $(x^1)^2+\dots+(x^7)^2=1$, be the unit normal to the sphere.
A computation shows that $d\langle J{}\cdot{},{}\cdot{}\rangle=3\varphi$.
Pulling back this identity to $S^6$ yields $d\sigma=3\psi_+$.
Further, $d\psi_-=-\iota^*d(N\lrcorner \star \varphi)$.
By the expression of $\star \varphi$ it follows that $d(N\lrcorner \star \varphi)=4\star \varphi$, and again the restrictions to $S^6$ are the same.
So the claim now is $4\iota^*\star \varphi=2\sigma \wedge \sigma$.
Observe that this identity is invariant under the action of $\mathrm G_2$.
Up to a rotation in $\mathrm G_2$ mapping $N$ to $\partial_{x^7}$, we find \[\sigma = N \lrcorner\ \varphi = \partial_{x^7} \lrcorner\ \varphi = dx^{16}-dx^{25}-dx^{34},\]
so $\sigma \wedge \sigma = -2(dx^{1256}+dx^{1346}-dx^{2345})$, and $\iota^*\star \varphi = dx^{2345}-dx^{1346}-dx^{1256}$ is unchanged. 
This shows $d\psi_-=-2\sigma\wedge \sigma$.
\end{proof}

Now for some examples. 
Besides $S^6=\mathrm G_2/\mathrm{SU}(3)$, there are other three examples of homogeneous nearly K\"ahler manifolds.
Let $(L,U)$ be a pair where $L$ is a linear one-dimensional subspace of $\mathbb C^3$, and $U$ is a linear two-dimensional subspace of $\mathbb C^3$ containing $L$.
The \emph{flag manifold} of $\mathbb C^3$, denoted by $F_{1,2}(\mathbb C^3)$, is by definition the set of pairs $(L,U)$. 
This gets a transitive action of $\mathrm{U}(3)$ (resp.\ $\mathrm{SU}(3)$) with generic stabiliser isomorphic to $T^3$ (resp.\  $T^2$). 
Further, we have homogeneous nearly K\"ahler structures on the complex projective space $\mathbb CP^3$ and the product of three-spheres $S^3 \times S^3$.
In the latter cases one has the homogeneous descriptions $\mathbb CP^3=\mathrm{Sp}(2)/\mathrm{Sp}(1)\mathrm{U}(1)$ and $S^3 \times S^3 = \mathrm{SU}(2)^3/\mathrm{SU}(2)_{\Delta}$, where $\mathrm{SU}(2)_{\Delta}$ denotes the diagonal embedding of $\mathrm{SU}(2)$ in $\mathrm{SU}(2)^3$, i.e.\ $h \in \mathrm{SU}(2) \mapsto (h,h,h) \in \mathrm{SU}(2)^3$.
For explicit constructions and details, see \cite{russo2}.
These examples were originally found by Gray--Wolf \cite{gray-wolf}. The next result tells us there are no other homogeneous compact nearly K\"ahler manifolds in dimension $6$.
\begin{theorem}[Butruille \cite{butruille}]
\label{thm:butruille}
The only four homogeneous connected, complete nearly K\"ahler manifolds are $S^6=\mathrm G_2/\mathrm{SU}(3)$, the flag manifold $F_{1,2}(\mathbb C^3) = \mathrm{SU}(3)/T^2$, the complex projective space $\mathbb CP^3=\mathrm{Sp}(2)/\mathrm{Sp}(1)\mathrm{U}(1)$, and the product of three-spheres $S^3\times S^3=\mathrm{SU}(2)^3/\mathrm{SU}(2)_{\Delta}$.
\end{theorem}
\begin{remark}
Theorem \ref{thm:bonnet-myers} implies that completeness and connectedness are sufficient to have compactness and finite fundamental group in Theorem \ref{thm:butruille}.
Then, up to lifting the nearly K\"ahler structure to the universal cover, one can assume the manifold to be simply connected.
\end{remark}
From the point of view of symmetries, it is then natural to look for nearly K\"ahler structures with cohomogeneity one. 
\begin{definition}
A Lie group $G$ acts on a manifold $M$ with \emph{cohomogeneity one} if $M$ contains a $G$-orbit of codimension $1$.
\end{definition}
In the cohomogeneity one case, the orbit space $M/G$ is one-dimensional, and is homeomorphic to either $\mathbb R$, $S^1$, $[0,+\infty)$, or $[0,1]$, see \cite{mostert}.
Orbits corresponding to the interior of these orbit spaces via the projection $M \to M/G$ are called \emph{principal}.
Other orbits are called \emph{exceptional} (when they have the same dimension of the principal orbits) or \emph{singular} (when their dimension is lower than the dimension of principal orbits).

Possible symmetry groups acting with cohomogeneity one on nearly K\"ahler six-manifolds appear in the following result.
\begin{theorem}[Podest\`a--Spiro \cite{podesta-spiro}]
\label{thm:podesta-spiro}
Let $(M,g,J)$ be a simply connected, complete strict nearly K\"ahler six-manifold.
Let $G$ be a compact connected Lie group acting on $(M,g)$ by isometries and with cohomogeneity one.
Then two cases occur:
\begin{enumerate}
\item $G$ is $\mathrm{SU}(3)$ or $\mathrm{SO}(4)$ and there is an equivariant diffeomorphism $M=S^6$,
\item $G$ is $\mathrm{SU}(2)\times \mathrm{SU}(2)$, and there is an equivariant diffeomorphism $M = \mathbb CP^3$ or $M=S^3\times S^3$.
\end{enumerate}
\end{theorem}
\begin{remark}
All these cases do occur by restricting the known transitive actions.
Whether inhomogeneous, cohomogeneity one actions exist is a much more difficult question.
The group $\mathrm{SU}(3)$ acts naturally on $\mathbb C^3$, and hence on $\mathbb C^3 \oplus \mathbb R = \mathbb R^7$.
The inclusion $\mathrm{SU}(3) \subset \mathrm G_2$ implies that $\mathrm{SU}(3)$ acts on $S^6 \subset \mathbb R^7$.
This action has two fixed points, whereas all other orbits are five-dimensional spheres diffeomorphic to $\mathrm{SU}(3)/\mathrm{SU}(2)$.
The cohomogeneity one action of $\mathrm{SO}(4)$ on $S^6$ is the restriction of the transitive action of $\mathrm G_2$ to a subgroup $\mathrm{SO}(4) \subset \mathrm G_2$.
Since $\mathbb CP^3$ has a transitive action of $\mathrm{Sp}(2)=\mathrm{Spin}(5) \supset \mathrm{Spin}(4)=\mathrm{SU}(2) \times \mathrm{SU}(2)$, the cohomogeneity one action of $\mathrm{SU}(2)\times \mathrm{SU}(2)$ is obtained by restriction of the $\mathrm{Sp}(2)$-action.
The case of $S^3\times S^3$ is essentially trivial.
\end{remark}

The classification in Theorem \ref{thm:podesta-spiro} turns out to be incomplete. 
Indeed, Foscolo and Haskins \cite{foscolo-haskins} noted that a nearly K\"ahler structure with a cohomogeneity one action of $\mathrm{SU}(2)\times \mathrm{SU}(2)$ may exist on $S^2 \times S^4$, although none is known at present.
A second omission from Podest\`a--Spiro's list is $G=T^2 \times \mathrm{SU}(2)$ (and $M=S^3\times S^3$), which again is obtained by restricting the action of $\mathrm{SU}(2)^3$ in the homogeneous case (here the factor $T^2$ is a maximal torus in $\mathrm{SU}(2)^2$).
The corresponding Lie algebra $\mathbb R^2\oplus \mathfrak{su}(2)$ is erroneously ruled out in \cite[Lemma 3.2]{podesta-spiro}.
This last case is currently being studied by Swann and the author to understand whether inhomogeneous cohomogeneity one examples with $T^2 \times \mathrm{SU}(2)$-symmetry occur.

A recent groundbreaking result in the theory was obtained by Foscolo and Haskins in 2017.
\begin{theorem}[Foscolo--Haskins \cite{foscolo-haskins}]
There are inhomogeneous, cohomogeneity one nearly K\"ahler structures on $S^6$ and $S^3\times S^3$.
\end{theorem}
In both cases, the symmetry group of the structure is $\mathrm{SU}(2) \times \mathrm{SU}(2)$, and the principal orbit type is $\mathrm{SU}(2)\times \mathrm{SU}(2)/\mathrm{U}(1)_{\Delta}$.
At present, no more compact examples of nearly K\"ahler six-manifolds are known.

On lower symmetry, we have local results by Madnick \cite{madnick}, who studies nearly K\"ahler six-manifolds with cohomogeneity two.
Nearly K\"ahler six-manifolds with $T^3$- or $T^2$-symmetry were studied by Moroianu--Nagy \cite{moroianu-nagy} and Russo--Swann \cite{russo-swann} respectively.
In both cases, the torus-actions admit \emph{multi-moment maps}, which are generalisations of moment maps to geometries with a closed invariant $k$-form \cite{madsen-swann}.
\begin{remark}
On certain homogeneous nearly K\"ahler six-manifold with two-torus symmetry, the union of all orbits with special stabiliser modulo the torus-action has the structure of a graph \cite{russo2}.
This happens when $M=G/H$ and $G$ and $H$ have equal rank---so it does not happen only for $S^3\times S^3 = \mathrm{SU}(2)^3/\mathrm{SU}(2)_{\Delta}$.
The number of vertices of the graph corresponds to the Euler characteristic of the manifold (cf.\ Kobayashi \cite{kobayashi-fixed-points}), whereas the number of edges is counted by the weights of the torus action at fixed points.
Incidentally, one has the following result of Wang \cite{wang}: for any homogeneous manifold $M=G/H$, if $G$ and $H$ have equal rank, then the Euler characteristic is the ratio $|W(G)|/|W(H)|$ of the cardinalities of the Weyl groups of $G$ and $H$ (if the ranks differ, the Euler characteristic of $M$ vanishes).
The graph contains topological and geometric information on the manifold.
The correspondence between certain torus-actions on manifolds and graphs is well-known, and is a matter of \emph{GKM theory}, named after Goresky--Kottwitz--MacPherson \cite{gkm}.
For more information on the homogeneous and the cohomogeneity one case, see e.g.\ \cite{goertsches-loiudice-russo, guillemin-holm-zara} and references therein.
\end{remark}

We finally remark that the structure theory of nearly K\"ahler manifolds relies partly on the six-dimensional case.
\begin{theorem}[Nagy \cite{nagy}]
Let $(M,g,J)$ be a complete, simply connected, strict nearly K\"ahler manifold.
Then $(M,g)$ is a Riemannian product whose factors belong to one of the following classes:
\begin{enumerate}
\item nearly K\"ahler six-manifolds, 
\item homogeneous nearly K\"ahler spaces,
\item twistor spaces over quaternionic K\"ahler manifolds with positive scalar curvature equipped with the canonical nearly K\"ahler metric.
\end{enumerate}
\end{theorem}
The decomposition given coincides with the de Rham decomposition of the manifold $(M,g)$, cf.\ Theorem \ref{thm:derham}.
The homogeneous nearly K\"ahler spaces that may occur are divided into four types, and discussed in more detail in \cite{nagy}.
For a complete classification of homogeneous nearly K\"ahler spaces, see instead Gonz\'alez D\'avila--Cabrera \cite{gonzalez-cabrera}.

\subsection{Nearly parallel \texorpdfstring{$\mathrm G_2$}{G2} manifolds}
\label{subsec:nearly-parallel-g2-manifolds}

We recall that a $\mathrm G_2$-structure on a manifold $M$ is a reduction $Q$ of the principal bundle of linear frames to structure group $\mathrm G_2 \subset \mathrm{SO}(7)$.
This is equivalent to having a three-form $\varphi$ on $M$ linearly equivalent to a standard definite three-form on $\mathbb R^7$ pointwise.
Recall that $\varphi$ induces a Riemannian metric $g$ and a volume form, so one can construct a Hodge star operator $\star_{\varphi}$, which we just denote by $\star$.
\begin{definition}
\label{def:np-form}
Let $M$ be a seven-manifold with a $\mathrm G_2$-structure given by a three-form $\varphi$.
Then $(M,\varphi)$ is \emph{nearly parallel} if $d\varphi = \lambda \star \varphi$, for a constant $\lambda \in \mathbb R \setminus \{0\}$.
\end{definition}
\begin{remark}
Note that our definition in the previous section was $\nabla \varphi = \lambda \star \varphi$.
However, $\nabla$ is torsion-free and the nearly parallel condition tells us $\nabla \varphi$ is totally skew-symmetric, so $d\varphi$ is proportional to $\star \varphi$ as well.
\end{remark}
\begin{theorem}
Nearly parallel $\mathrm G_2$ manifolds are Einstein with positive scalar curvature.
\end{theorem}
The proof follows for instance by Bryant \cite[Section 4.5.3, formula (4.30)]{bryant1}.
An alternative proof follows by the fact that nearly parallel $\mathrm G_2$ manifolds admit real Killing spinors \cite{bfgk}.
\begin{example}
A first example of nearly parallel $\mathrm G_2$ structure occurs on the seven-sphere $S^7$. 
Let $V = \mathbb R^8 = \{(x_0,\dots,x_7): x_i \in \mathbb R\}$ be the standard Euclidean eight-dimensional vector space, and let $\mathrm{SO}(8)$ act on it in the standard way.
Let $dx^0,\dots,dx^7$ be the standard coframe, and denote by $\partial_i$ the vector dual to $dx^i$.
Write $V = \mathbb R \partial_0\oplus \mathbb R^7$. On $\mathbb R^7$ we have
\begin{align*}
\varphi_0 & = dx^{123}-dx^1(dx^{45}+dx^{67})-dx^2(dx^{46}+dx^{75})-dx^3(dx^{47}+dx^{56}), \\
{*}\varphi_0 & = dx^{4567}-dx^{23}(dx^{45}+dx^{67})-dx^{31}(dx^{46}+dx^{75})-dx^{12}(dx^{47}+dx^{56}),
\end{align*}
where ${*}$ is the Hodge star operator on $\mathbb R^7$ induced by $\varphi_0$.
The standard inner product on $\mathbb R^7$ coincides with the one induced by $\varphi_0$.
Let us set \[\psi \coloneqq dx^0 \wedge \varphi_0+{*}\varphi_0,\]
where ${*}$ is the Hodge star operator on $V$.
The group $\mathrm{Spin}(7)$ can be defined as the stabiliser of $\psi$ in $\mathrm{SO}(8)$.
One can see that $\mathrm{Spin}(7)$ acts transitively on the seven-sphere $S^7 \subset \mathbb R^8$, and the generic stabiliser is isomorphic to $\mathrm G_2$, whence a diffeomorphism $S^7=\mathrm{Spin}(7)/\mathrm G_2$ (cf.\ the proof of Theorem \ref{thm:spin7}).

Let $\iota \colon S^7 \hookrightarrow \mathbb R^8$ be the standard immersion, and let $N$ be the unit normal vector field on $S^7$.
In coordinates, $N = \sum_{k=0}^7 x_k\partial_k$, with $\sum_{k=0}^7 x_k^2=1$. 
Set
\[\varphi \coloneqq \iota^*(N \lrcorner\ \psi), \qquad \phi \coloneqq \iota^*\psi.\]
Since $\mathrm{Spin}(7)$ acts on $\mathbb R^8$ preserving $N$ and $\psi$, the action induced on $S^7$ preserves $\varphi$ and $\phi$.
A standard computation gives $d(N \lrcorner\ \psi) =4\psi$, and hence $d\varphi = 4\phi$ on $S^7$.
We claim that $\phi = \star\varphi$, where $\star$ is the Hodge star operator on $S^7$ defined with respect to the metric induced by $\varphi$ and the natural volume form $\mathrm{vol}_{S^7}=\iota^*(N \lrcorner\ dx^{01\dots 7})$ on $S^7$.
One computes $\psi \wedge \psi = 14dx^{01\dots 7}$. Since $N \lrcorner\ (\psi \wedge \psi) = 2(N \lrcorner\ \psi) \wedge \psi$, we have
\[(N \lrcorner\ \psi) \wedge \psi = 7N \lrcorner\ dx^{01\dots 7}.\]
Pulling back this identity to $S^7$ yields $\varphi \wedge \phi = 7\mathrm{vol}_{S^7}$.
Note that $\phi$ at any point of $S^7$ is a $\mathrm G_2$-invariant four-form, and thus must be proportional to $\star \varphi$, say $\phi = \lambda \star \varphi$, cf.\ Proposition \ref{prop:g2-decomposition-forms}.
But then $\varphi \wedge \phi = \lambda \varphi \wedge \star \varphi = 7\lambda \mathrm{vol}_{S^7}$, where the latter equality is checked by evaluating both sides at $(1,0,\dots,0) \in S^7$.
Hence $\lambda=1$, and thus $\phi = \star \varphi$.
It follows that $\varphi$ is a nearly parallel $\mathrm G_2$-structure on $S^7$.
\end{example}

There is a second nearly parallel $\mathrm G_2$-structure on $S^7$, which comes from a \emph{$3$-Sasakian} structure on it (although we do not touch upon Sasakian geometry here, see Boyer--Galicki \cite{boyer-galicki}).
The seven-sphere with this second nearly parallel $\mathrm G_2$-structure is called \emph{squashed} sphere.

Compact homogeneous nearly parallel $\mathrm G_2$ manifolds are classified in \cite{friedrich-kath-moroianu-semmelmann}, see also \cite{reidegeld}.
One has in particular the \emph{Aloff--Wallach spaces} $X_{k,\ell} = \mathrm{SU}(3)/\mathrm{U}(1)_{k,\ell}$, where the embedding of $\mathrm{U}(1)$ in $\mathrm{SU}(3)$ is determined by two integers $k$ and $\ell$, and the \emph{Berger space} $\mathrm{SO}(5)/\mathrm{SO}(3)$ with an $\mathrm{SO}(5)$-invariant nearly parallel $\mathrm G_2$ structure.
Cleyton and Swann \cite{cleyton-swann} classified manifolds admitting a nearly parallel $\mathrm G_2$ structure with a simple Lie group of automorphisms acting with cohomogeneity one.
They showed that the standard sphere $S^7$ and the real projective space $\mathbb RP^7$ are the only complete examples that occur.
Podest\`a \cite{podesta} investigated the existence of $\mathrm{SU}(2)^3$-invariant nearly parallel $\mathrm G_2$-structures with cohomogeneity one on $S^3 \times \mathbb R^4$, proving the existence of a one-parameter family of such structures. 
Nearly parallel $\mathrm G_2$ manifolds with torus symmetry are being studied by Swann and the author \cite{russo-swann-toric}.

There is one last point we have not mentioned yet which will be relevant in the next section.
The Riemannian cones over nearly K\"ahler six-manifolds and nearly parallel $\mathrm G_2$ seven-manifolds come with a special structure.
\begin{definition}
Let $(M,g)$ be any Riemannian manifold. 
The \emph{Riemannian cone} over $M$ is $C(M)=(\mathbb R_{>0} \times M, dt^2+t^2g)$, where $t$ is a global coordinate on $\mathbb R_{>0}$.
\end{definition}
If $(M,g,J)$ is a nearly K\"ahler six-manifold, the Riemannian cone over $M$ is a seven-manifold and comes equipped with a natural $\mathrm G_2$-structure.
Recall that the nearly K\"ahler condition can be rephrased in terms of the fundamental two-form $\sigma$ and a complex volume form $\psi_{\mathbb C}=\psi_++i\psi_-$.
The nearly K\"ahler structure equations are
\[d\sigma = 3\psi_+, \qquad d\psi_- = -2\sigma \wedge \sigma.\]
Now, on $C(M)$ define a three-form $\varphi$ by
\[\varphi \coloneqq t^2dt \wedge \sigma+t^3\psi_+\]
It turns out that the Hodge star of it is $\star \varphi = \frac12 t^4\sigma \wedge \sigma+t^3\psi_-\wedge dt$.
One checks that the nearly K\"ahler structure equations are equivalent to the closedness of $\varphi$ and $\star \varphi$, which means $C(M)$ has holonomy $\mathrm G_2$ (cf.\ subsection \ref{subsec:classification-results}).
This general fact is behind one of the results that will be presented in the next section.

A similar process can be applied to certain nearly parallel $\mathrm G_2$ manifolds to obtain Riemannian eight-dimensional cones with holonomy $\mathrm{Spin}(7)$.
This is the case e.g.\ for the homogeneous space $\mathrm{SO}(5)/\mathrm{SO}(3)$.
In the case $C(M)$ has holonomy exactly equal to $\mathrm{Spin}(7)$, the nearly parallel $\mathrm G_2$ structure is called \emph{proper}.
When the structure is not proper and $M \neq S^7$, the holonomy of the cone reduces either to $\mathrm{SU}(4)=\mathrm{Spin}(6)$ or to $\mathrm{Sp}(2)=\mathrm{Spin}(5)$.
In the former case, $M$ admits a \emph{Sasakian structure}, whereas in the latter case $M$ admits a \emph{$3$-Sasakian} structure.
We refer to Boyer--Galicki \cite{boyer-galicki} for details.

\subsection{Manifolds with exceptional holonomy}
\label{sec:manifolds-with-exceptional-holonomy}

We now see how six-dimensional nearly K\"ahler and nearly parallel $\mathrm G_2$ structures are used to construct manifolds with exceptional holonomy.
The entire subsection is based on work of Bryant \cite{bryant}.

If $(M,g')$ is a Riemannian manifold of dimension $n$, the associated cone metric on $\mathbb R_{>0}\times M$ is given by 
\[g=dt^2+t^2g', \qquad t>0,\]
where $t$ is the natural coordinate on $\mathbb R_{>0}$. 
A computation shows that $g$ is flat if and only if $g'$ has constant sectional curvature $+1$. 
Moreover, $g$ is Ricci-flat if and only if $\mathrm{Ric}_{g'}=(n-1)g'$.

We now look at the construction of the first Riemannian seven-manifold with holonomy $\mathrm G_2$.
We first show the following lemma.
\begin{lemma}
\label{lemma:parallel-forms}
Let $M$ be a seven-dimensional connected and simply connected manifold.
Let $g$ be a metric on $M$ whose holonomy is a subgroup of $\mathrm G_2$, and let $\nabla$ be the Levi-Civita connection.
Suppose that there are no non-zero $\nabla$-parallel one-forms on $M$.
Then the holonomy of $g$ is exactly $\mathrm G_2$.
\end{lemma}
\begin{proof}
Take $p \in M$ and let $H_p$ be the holonomy group of $g$ at $p$.
Suppose $H_p$ acts reducibly on $T_pM$.
By Theorem \ref{thm:derham}, $M$ splits locally as a Riemannian product $(M_1 \times M_2,g_1+g_2)$, and we can assume $0<\dim M_1 \leq 3$.
By assumption, $(M,g)$ is Ricci-flat, so $(M_1,g_1)$, $(M_2,g_2)$ are both Ricci-flat.
Since $M_1$ has dimension at most $3$, then the Ricci flat condition is equivalent to the vanishing Riemannian tensor, so $(M_1,g_1)$ is flat (see Digression \ref{digression:curvature-low-dimension}).
So at every point $p \in M$ there are $\nabla$-parallel one-forms locally defined around $p$. 
Since $M$ is simply connected, these $\nabla$-parallel one-forms are globally defined, contradiction.
Then $H_p$ must act irreducibly on $T_pM$.
Since $H_p \subset \mathrm G_2$, then $H_p=\mathrm G_2$ by Berger's Theorem, Theorem \ref{thm:berger-thm}.
\end{proof}
\begin{proposition}
The Riemannian cone over the flag manifold $\mathrm{SU}(3)/T^2$ with its invariant homogeneous nearly K\"ahler metric has holonomy $\mathrm G_2$.
\end{proposition}
\begin{remark}
Before the proof, we give an additional piece of information which we have not discussed yet.
Let $G$ be any Lie group with Lie algebra $\mathfrak g$.
Its canonical one-form is the $\mathfrak g$-valued form $A \mapsto A$, where $A \in \mathfrak g$.
Its differential is given by the Maurer--Cartan equations, and encodes the structure constants of $\mathfrak g$ \cite[Chapter I, Section 3]{kobayashi-nomizu}.
\end{remark}
\begin{proof}
Let us consider the group $\mathrm{SU}(3)$.
Its canonical one-form $k$ is $\mathfrak{su}(3)$-valued, and hence we have $\mathrm{Tr}(k)=0$, $dk=-k\wedge k$, and $k$ is skew-hermitian.
Let us write 
\[k=\begin{pmatrix}
-i\theta_3 & -\overline \omega^1 & i\omega^2 \\
\omega^1 & -i\theta_2 & -\overline \omega^3 \\
i\overline \omega^2 & \omega^3 & -i\theta_1
\end{pmatrix},\]
where the $\theta$'s are left-invariant real-valued one-forms, and the $\omega$'s are left-invariant complex-valued one-forms on $\mathrm{SU}(3)$.
Since $k$ is traceless, then $\theta_1+\theta_2+\theta_3=0$.
The identity $dk=-k\wedge k$ implies
\begin{align*}
d\omega^i & = i(\theta_j-\theta_k) \wedge \omega^i-i\overline \omega^j \wedge \overline \omega^k, 
\end{align*}
for cyclic permutations $(ijk) = (123)$.
Let $T^2$ be the standard diagonal maximal torus of $\mathrm{SU}(3)$.
The above formulas imply that 
\begin{align*}
\Omega & \coloneqq \frac{i}{2}(\omega^1\wedge \overline \omega^1+\omega^2\wedge \overline \omega^2+\omega^3\wedge \overline \omega^3), \\
\Psi & \coloneqq \omega^1 \wedge \omega^2 \wedge \omega^3, 
\end{align*}
are well-defined on the quotient $\mathrm{SU}(3)/T^2$, and satisfy
\begin{equation}
\label{str-eq:nk-flag}
d\Omega = 3\mathrm{Re}\Psi, \qquad d\Psi = -2i\Omega \wedge \Omega.
\end{equation}
The metric $g'=\omega^1 \odot \overline \omega^1+\omega^2\odot \overline \omega^2+\omega^3 \odot \overline \omega^3$ on $\mathrm{SU}(3)/T^2$ is $\mathrm{SU}(3)$-invariant and the quotient map $\mathrm{SU}(3) \to \mathrm{SU}(3)/T^2$ is a Riemannian submersion for a suitable multiple of the Killing for on $\mathrm{SU}(3)$.
Note that the forms $(\Omega,\Psi)$ are equivalent to an $\mathrm{SU}(3)$-structure on $\mathrm{SU}(3)/T^2$, and the structure equations \eqref{str-eq:nk-flag} tell us it is the standard nearly K\"ahler structure.

On the product $M\coloneqq \mathbb R_{>0} \times (\mathrm{SU}(3)/T^2)$, let $t$ be a global coordinate on the $\mathbb R_{>0}$-factor, and consider the forms
\begin{align*}
\phi & = t^2dt \wedge \Omega+t^3\mathrm{Re}(\Psi) = \frac13 d(t^3\Omega), \\
\psi & = \frac12 t^4\Omega \wedge \Omega-t^3dt \wedge \mathrm{Im}(\Psi) = -\frac14 d(t^4\mathrm{Im}(\Psi)).
\end{align*}
We have d$\phi=0=d\psi$.
On $\mathbb R_{>0} \times \mathrm{SU}(3)$, define the forms $\{\eta^1,\dots,\eta^7\}$ by the formulas 
\begin{align*}
\eta^1 \coloneqq dt, \qquad \eta^2+i\eta^3 \coloneqq t\omega^1, \qquad \eta^4+i\eta^5 \coloneqq t\omega^2, \qquad \eta^6+i\eta^7 \coloneqq t\omega^3.
\end{align*}
If we pull $\phi$ and $\psi$ back to $\mathbb R_{>0} \times \mathrm{SU}(3)$, one computes
\begin{align*}
\phi & = \eta^1 \wedge (\eta^2 \wedge \eta^3+\eta^4\wedge \eta^5+\eta^6\wedge \eta^7) \\
& \qquad +\mathrm{Re}((\eta^2+i\eta^3)\wedge (\eta^4+i\eta^5)\wedge (\eta^6+i\eta^7)) \\
& = \eta^{123}+\eta^{145}+\eta^{167}+\eta^{246}-\eta^{257}-\eta^{347}-\eta^{356}.
\end{align*}
So $\phi$ is non-degenerate on $M$, i.e.\ $\phi \in \Omega_+^3(M)$ (recall the notation in subsection \ref{subsec:on-g2-spin7}). 
Moreover, the associated metric is $g \coloneqq (\eta^1)^2+\dots+(\eta^7)^2 = dt^2+t^2g'$.
So $(M,g)$ is the Riemannian cone over $(\mathrm{SU}(3)/T^2,g')$.
The associated volume form is $\eta^{12\dots 7}$. 
A computation shows that $\psi=\star_{\phi}\phi$.
Since $d\phi=d\psi=0$, the $\mathrm G_2$-structure induced on $M$ is torsion-free, so the holonomy of the corresponding metric is in $\mathrm G_2$.
We now claim it is exactly $\mathrm G_2$.

We have a fibration $\mathrm{SU}(3) \to \mathrm{SU}(3)/T^2$. The homotopy exact sequence of it gives
\[\pi_0(\mathrm{SU}(3)/T^2)=0=\pi_1(\mathrm{SU}(3)/T^2), \qquad \pi_2(\mathrm{SU}(3)/T^2)=\mathbb Z \oplus \mathbb Z.\]
In particular, $\mathrm{SU}(3)/T^2$ is connected, simply connected but not homeomorphic to $S^6$, so $g'$ cannot have constant sectional curvature.
In particular, $(M,g)$ cannot be flat. 
Let $\nabla$ be the Levi-Civita connection on $(M,g)$.
If we let $\mathcal P$ be the vector space of $\nabla$-parallel one-forms on $M$, then $\dim \mathcal P<6$ (otherwise, the holonomy of $g$ would be trivial and $(M,g)$ would be flat).
Now $\mathrm{SU}(3)$ acts on $M$ by isometries, and hence it acts via pullback on $\mathcal P$.
This is a representation $\rho \colon \mathrm{SU}(3) \to \mathrm{Aut}(\mathcal P)$.
Since $\mathrm{SU}(3)$ has no non-trivial representations of dimension less than $6$, it follows tha $\rho$ is the trivial representation.
Thus each parallel one-form on $M$ is necessarily $\mathrm{SU}(3)$-invariant.
Suppose $\mathcal P$ contains non-zero elements, and let $\alpha \in \mathcal P$ be any of them.
Let $u=(1,eT^2) \in M$.
Since $\alpha$ is invariant under $\mathrm{SU}(3)$, it follows that $\alpha_{|u} \in T_u^*M$ is fixed by the action of $T^2$ on $T_u^*M$.
But $T_u^*M=\mathbb R \oplus \mathbb C^3$, where $T^2$ acts trivially trivially on the $\mathbb R$-factor and on the $\mathbb C^3$ factor by
\[(e^{i\theta_1},e^{i\theta_2},e^{i\theta_3})\cdot (z_1,z_2,z_3) = (e^{i(\theta_2-\theta_3)}z_1,e^{i(\theta_3-\theta_1)}z_2,e^{i(\theta_1-\theta_2)}z_3),\]
with $e^{i(\theta_1+\theta_2+\theta_3)}=1$.
Therefore, the only $T^2$-invariant elements of $T_u^*M$ are the multiples of $dt_{|u}$.
It follows that the only possible non-zero elements of $\mathcal P$ are of the form $\alpha=f(t)dt$.
Since these are not parallel, $\mathcal P$ has no non-zero element.
We are then under the hypotheses of Lemma \ref{lemma:parallel-forms}, so $(M,g)$ has holonomy $\mathrm G_2$.
\end{proof}
\begin{remark}
The fact that the Riemannian cone over a nearly K\"ahler manifold has holonomy $\mathrm G_2$ was shown in the general case by B\"ar \cite{bar}.
\end{remark}

We now look at the construction of the first Riemannian eight-dimensional manifold with holonomy $\mathrm{Spin}(7)$.
\begin{lemma}
Let $M$ be a connected and simply connected eight-manifold.
Let $g$ be a metric on $M$ with holonomy in $\mathrm{Spin}(7)$.
Let $\nabla$ be the Levi-Civita connection.
Assume there are no non-zero one-forms or two-forms on $M$.
Then $\mathrm{Hol}(\nabla)$ is $\mathrm{Spin}(7)$.
\end{lemma}
\begin{proof}
Recall the notations in subsection \ref{subsec:on-g2-spin7} on the $\mathrm{Spin}(7)$ part.
Let $u \in M$ be fixed and let $H_u \subset \mathrm{SO}(T_uM)$ be the holonomy group of $g$ at $u$.
By assumption, we can choose an isometry $T_uM \to V_+$ so that $H_u \subset \mathrm{Spin}(7)$.
Let $\beta$ denote the parallel four-form on $M$ so that $\beta_u = \Phi \in \Lambda^4(V_+^*)$.

We claim that $H_u$ acts irreducibly on $T_uM$.
If $H_u$ preserved an orthogonal splitting $T_uM = \xi_u \oplus \eta_u$, then by parallel translation, we would have a parallel vector bundle splitting $TM=\xi \oplus \eta$.
If $\xi$ had rank $1$, we would be able to construct a parallel vector field, and hence a parallel one-form on $M$.
If $\xi$ had rank $2$, then the dual of the unit volume section of $\Lambda^2\xi$ would be a parallel two-form.
If $\xi$ had rank $3$, we could let $e$ be the unit volume section of $\Lambda^3\xi$, and construct the parallel one-form $e \lrcorner\ \beta$.
Since $\mathrm{Spin}(7)$ acts transitively on three-planes in $V^*$, we can assume that $e_u=e_0\wedge e_1 \wedge e_2 \in \Lambda^3V_+$.
Since $(e_0 \wedge e_1 \wedge e_2) \lrcorner\ \Phi=\omega^3 \neq 0$ in $V_+^*$, then $e \lrcorner\ \beta$ is a parallel non-zero one-form.
Thus, the assumptions imply that if $H_u$ acts reducibly on $T_uM$, then it can only preserve a splitting $T_uM=\xi_u\oplus \eta_u$, where $\dim \xi_u=\dim \eta_u=4$.
Moreover, $H_u$ must act irreducibly on each of $\xi_u$ and $\eta_u$.
By Theorem \ref{thm:derham}, we have $H_u=H_u^1\times H_u^2$, where $H_u^1 \subset \mathrm{SO}(\xi_u) = \mathrm{SO}(4)$, and $H_u^2 \subset \mathrm{SO}(\eta_u)=\mathrm{SO}(4)$. 
Since any connected proper subgroup of $\mathrm{SO}(4)$ preserves a two-form in $\Lambda^2 \mathbb R^4$, we must have $H_u^1=H_u^2=\mathrm{SO}(4)$.
In particular, $H_u=\mathrm{SO}(4) \times \mathrm{SO}(4)$.
But the latter group has rank $4$ and $\mathrm{Spin}(7)$ has rank $3$, contradiction.
Therefore, $H_u$ acts irreducibly on $T_uM$.

Since $H_u \subset \mathrm{Spin}(7)$, $g$ is Ricci-flat.
It follows that $(M,g)$ is not isometric to any irreducible locally symmetric space.
By Berger's Theorem \ref{thm:berger-thm}, $H_u$ must be one of the groups $\mathrm{Sp}(2)$, $\mathrm{SU}(4)$, $\mathrm{Spin}(7)$, $\mathrm{Sp}(2)\mathrm{Sp}(1)$, $\mathrm{U}(4)$, or $\mathrm{SO}(8)$.
The three latter subgroups are not subgroups of $\mathrm{Spin}(7)$.
The groups $\mathrm{Sp}(2)$ and $\mathrm{SU}(4)$ preserve two-forms on $\mathbb R^8$.
The only possibility is $H_u=\mathrm{Spin}(7)$.
\end{proof}
We use the representation theory of $\mathrm{SO}(3)$, cf.\ subsection \ref{subsec:irreps-su2-so3}, to construct a manifold with holonomy $\mathrm{Spin}(7)$.
Recall that the irreducible representations of $\mathrm{SO}(3)$ are certain irreducible $\mathrm{SU}(2)$-representations, and can be realised as invariant spaces of harmonic polynomials on $S^2$.
Let $H_d$ denote the space of harmonic homogeneous polynomials of degree $d$ in the variables $x^1,x^2,x^3$.
The Clebsch--Gordan formula, Theorem \ref{thm:clebsch-gordan-formula}, gives the tensor product decomposition
\[H_m \otimes H_n=H_{m-n} \oplus H_{m-n+1} \oplus H_{m-n+2} \oplus \dots \oplus H_{m+n}, \qquad m \geq n.\]
If $m=n$, we have the following specific formulas
\begin{align*}
S^2H_m & = H_0 \oplus H_2 \oplus H_4 \oplus \dots \oplus H_{2m}, \\
\Lambda^2H_m & = H_1 \oplus H_3 \oplus H_5 \oplus \dots \oplus H_{2m-1}.
\end{align*}
Since $\dim H_2=5$, the representation $H_2$ gives rise to an embedding $\mathrm{SO}(3) \subset \mathrm{SO}(5)$.
We let $M\coloneqq \mathrm{SO}(5)/\mathrm{SO}(3)$.
Let $e \in M$ be the $\mathrm{SO}(3)$-coset of the identity.
We know that $T_eM=\mathfrak{so}(5)/\mathfrak{so}(3)$ as $\mathrm{SO}(3)$-modules.
Now $\mathfrak{so}(5)=\Lambda^2H_2=H_1\oplus H_3$ as $\mathrm{SO}(3)$-modules, and $\mathfrak{so}(3) = H_1$, so $T_eM=H_3$.
Thus $M$ is an isotropy irreducible homogeneous space.
We will be interested in the ring of $\mathrm{SO}(5)$-invariant forms $\Omega^{inv}(M) \subset \Omega(M)$.
This ring is isomorphic to the ring $\Lambda^{inv}(T_e^*M) \subset \Lambda(T_e^*M)$ of $\mathrm{SO}(3)$-invariant forms.
Thus, we must compute the ring of invariant forms on $T^*M=H_3$.

There is an embedding $\mathrm{SO}(3) \subset \mathrm G_2$ \cite{wolf1}, so $\mathrm G_2/\mathrm{SO}(3)$ is an isotropy irreducible homogeneous space.
This means $\mathfrak g_2/\mathfrak{so}(3)$ is an irreducible $\mathrm{SO}(3)$-representation.
By dimension count, it follows that $\mathfrak g_2/\mathfrak{so}(3)=H_5$.
Thus $\mathfrak g_2=H_1\oplus H_5$.
Since $\mathrm G_2 \subset \mathrm{SO}(V)$, $V$ is an $\mathrm{SO}(3)$-module via this restriction.
Since $\dim V = 7$, either $V=H_3$ or $V$ is a sum of modules $\{H_0,H_1,H_2\}$.
In the latter case, $\Lambda^2V$ is a sum of modules $\{H_0,H_1,H_2,H_3\}$.
In particular, $H_5$ would not occur as a summand in $\Lambda^2V=\mathfrak{so}(V) \supset \mathfrak g_2$.
Since $\mathfrak g_2 \supset H_5$, we conclude that $V=H_3$, i.e.\ $\mathrm{SO}(3)$ acts irreducibly on $V$.
Since $\Lambda^2V = V \oplus \mathfrak g_2$ as $\mathrm G_2$-modules, we must have $\Lambda^2V=H_1\oplus H_3 \oplus H_5$ as $\mathrm{SO}(3)$-modules.
Thus, $\mathrm{SO}(3)$ does not fix any two-forms in $\Lambda^2H_3=\Lambda^2V$.
Since $\mathrm G_2$ fixes a three-form in $\Lambda^3V$, the action of $\mathrm{SO}(3)$ must also fix this form.
On the other hand, since $\Lambda^3V$ is a quotient of $\Lambda^2V \otimes V = (H_1\oplus H_3\oplus H_5)\otimes H_3$, the Clebsch--Gordan formula shows that $\Lambda^3V$ could have at most one $H_0$-summand.
Thus, $\mathrm{SO}(3)$ fixes only one three-form on $V$ up to multiples.

It follows that $\Lambda^{inv}H_3\subset \Lambda H_3$ is generated by $1 \in \Lambda^0H_3$, $\varphi \in \Lambda^3H_3$, $\star \varphi \in \Lambda^4H_3$, and $\star 1 \in \Lambda^7H_3$.
This gives rise to generators $\{1,\psi,\star \psi,\star 1\}$ in $\Omega^{inv}(M)$.
Further, $\psi \in \Omega_+^3(M)$ and, after a suitable scaling of $\psi$ and a change of orientation (if necessary), we may assume $\star \psi = \star_{\psi}\psi$.
We know that $\Omega^{inv}(M)$ is closed under exterior differentiation.
Hence $d\psi = \lambda \star \psi$ for some constant $\lambda$.
If $\lambda=0$, we would have $H^3(M,\mathbb R)=\mathbb R$, contradicting the fact that $\mathrm{SO}(5)/\mathrm{SO}(3)$ is a rational homology sphere \cite{berger3}.
So $\lambda \neq 0$, and by scaling we may set $\lambda=4$ (possibly after reversing the orientation).

Now, on $\mathrm{SO}(5)/\mathrm{SO}(3)$ we have an invariant metric $g'$ and a compatible three-form $\psi \in \Omega_+^3(\mathrm{SO}(5)/\mathrm{SO}(3))$ satisfying 
\begin{enumerate}
\item $d\psi=4\star \psi$, 
\item locally, there is an orthonormal coframe $\{\omega^1,\dots,\omega^7\}$ so that
\begin{align*}
\psi & =\omega^{123}+\omega^{145}+\omega^{167}+\omega^{246}-\omega^{257}-\omega^{347}-\omega^{356}, \\
\star \psi & = \omega^{4567}+\omega^{2367}+\omega^{2345}+\omega^{1357}-\omega^{1346}-\omega^{1247}-\omega^{1256}.
\end{align*}
\end{enumerate}
Now consider the manifold $M_+\coloneqq\mathbb R_{>0} \times M$ with cone metric $g=dt^2+t^2g'$.
\begin{proposition}
The Riemannian cone $(M_+,g)$ has holonomy $\mathrm{Spin}(7)$.
\end{proposition}
\begin{proof}
Consider the four-form on $M_+$ given by 
\[\Phi = t^3dt \wedge \psi+t^4\star \psi = \frac14 d(t^4\psi).\]
We clearly have $d\Phi=0$ and $\Phi$ is a positive four-form on $M_+$.
If we set $\eta^0\coloneqq dt$ and $\eta^i\coloneqq t\omega^i$ (where $\omega^1,\dots,\omega^7$ is a $g'$-orthonormal coframe as above), then $\Phi$ induces a $\mathrm{Spin}(7)$-structure on $M_+$ with underlying metric is 
$(\eta^0)^2+\dots+(\eta^7)^2=g$. 
Since $d\Phi=0$, this $\mathrm{Spin}(7)$-structure is torsion-free, so the holonomy of $g$ is a subgroup of $\mathrm{Spin}(7)$.
We claim that this holonomy group is not a propert subgroup.
Since $M \neq S^7$, $g'$ cannot have constant sectional curvature, whence $g$ is not flat.
Next, since $d\Phi=0$, we know that $\Phi$ is $g$-parallel.
It follows that the triple cross product $P \colon TM_+ \times TM_+ \times TM_+ \to TM_+$ given by 
$\langle P(X,Y,Z),W\rangle = \Phi(X,Y,Z,W)$ for $X,Y,Z,W \in T_pM_+$ must be parallel.

We now show that there are no non-zero $g$-parallel one-forms on $M_+$.
Let $\mathcal P$ be the space of parallel one-forms on $M_+$.
We have $\dim \mathcal P < 8$, because $g$ is not flat.
If we had $\dim \mathcal P \geq 5$, then we could apply the cross product operation to triples from $\mathcal P$ and generate a spanning set of one-forms, since successive cross-products taken from any rank $5$ subspace of $V_+$ generate $V_+$.
Therefore, $\dim \mathcal P \leq 4$.
Now $\mathrm{SO}(5)$ acts by isometries on $M$ and hence a fortiori on $M_+$.
Via pullback, this induces a representation of $\mathrm{SO}(5)$ on $\mathcal P$.
Since $\dim \mathcal P \leq 4$, this representation is necessarily trivial, i.e.\ the parallel one-forms on $M_+$ are $\mathrm{SO}(5)$-invariant.
However, since there are no non-zero $\mathrm{SO}(5)$-invariant one-forms on $M=\mathrm{SO}(5)/\mathrm{SO}(3)$, it follows that any parallel one-form on $M_+$ must vanish when pulled back to any level set $t=c$, for $c$ constant.
Thus, all elements of $\mathcal P$ are of the form $f(t)dt$.
It can be shown that $f(t)dt$ is parallel if and only if $f=0$. 
Thus $\mathcal P$ has no non-zero element, and there are no non-zero parallel one-forms on $M_+$.

Now let $u=(1,e) \in \mathbb R_{>0} \times M=M_+$, and let $H_u \subset \mathrm{SO}(T_uM_+)$ be the holonomy group of $g$ at $u$.
We have $H_u \subset \mathrm{Spin}(7) \subset \mathrm{SO}(8)$.
Let $\mathfrak h_u \subset \mathfrak{spin}(7) \subset \mathfrak{so}(8)=\Lambda^2T_uM_+$ be the corresponding inclusions of Lie algebras.
Now $\mathrm{SO}(3)$ acts on $M_+$ as a group of isometries fixing $u$, whence all of the above algebras are $\mathrm{SO}(3)$-modules.
Since $T_uM_+=\mathbb R \oplus H_3$ as $\mathrm{SO}(3)$-representations, we have
\[\Lambda^2T_uM_+=\Lambda^2(\mathbb R \oplus H_3) = H_3 \oplus \Lambda^2H_3 = H_3 \oplus (H_1 \oplus H_3 \oplus H_5).\]
For dimension reasons, we see that $\mathfrak{so}(8)/\mathfrak{spin}(7)=H_3$, so we must have $\mathfrak{spin}(7)=H_1\oplus H_3\oplus H_5$, and $\mathfrak h_u$ must be some non-trivial direct sum of these three summands.
We already know that $\mathfrak g_2=H_1\oplus H_5\subset \Lambda^2H_3\subset \mathfrak{spin}(7)$, but $\mathrm G_2\subset \mathrm{Spin}(7)$ fixes a line in $V_+$.
Thus $\mathfrak h_u \not \subset H_1 \oplus H_5$, as we have already seen that $H_u$ cannot fix a line in $T_uM$.
Thus $H_3 \subset \mathfrak h_u$.
If $H_3$ were a subalgebra of $\mathfrak{spin}(7)$, we would have $[H_3,H_3]=H_3$, since the other possibility $[H_3,H_3]=0$ would violate the fact that $\mathfrak{spin}(7)$ only has rank $3$.
However, $[H_3,H_3]=H_3$ and the fact that $H_3$ is an irreducible $\mathrm{SO}(3)$-module would force $H_3$ to be a simple seven-dimensional Lie algebra.
Since there are none of these, we see that $[H_3,H_3] \not \subset H_3$.
For similar reasons, we see that $[H_5,H_5] \not \subset H_5$.
Thus $[H_5,H_5]=H_1\oplus H_5$. 
These remarks combine to show that either $\mathfrak h_u=\mathfrak{spin}(7)$ or $\mathfrak h_u=H_1\oplus H_3$.
In this latter case, it is easily seen that we must have $H_u=\mathrm{Sp}(2)\subset \mathrm{Spin}(7)$.
However, if $H_u=\mathrm{Sp}(2)$, then the space of $g$-parallel two-forms would have dimension $3$ (see \cite{bourguignon}).
A similar argument to the one applied to one-forms above now shows that every parallel two-form on $M_+$ is $\mathrm{SO}(5)$-invariant.
Since $M$ has no non-zero $\mathrm{SO}(5)$-invariant one-forms or two-forms, it easily follows that $M_+$ has no non-zero $\mathrm{SO}(5)$-invariant two-forms.
We conclude that $\mathfrak h_u=H_1\oplus H_3$ is impossible, so $\mathfrak h_u=\mathfrak{spin}(7)$.
This forces $H_u=\mathrm{Spin}(7)$.
\end{proof}
\begin{remark}
As above, the fact that the Riemannian cone over certain nearly parallel $\mathrm G_2$ manifolds has holonomy $\mathrm{Spin}(7)$ was shown in general by B\"ar \cite{bar}.
\end{remark}

\end{spacing}
\end{document}